\documentclass[11pt,reqno]{amsart}
\usepackage{amssymb,amsmath,amsfonts,amsthm,enumerate,stmaryrd,tensor,mathtools,dsfont,upgreek,bbm,mathrsfs,nameref,microtype,xcolor,bm,tikz,stmaryrd,todonotes,graphicx}
\usetikzlibrary{arrows}
\usepackage[left=0.70 in, right=0.70 in,top=0.69 in, bottom=0.69 in]{geometry}

\usepackage{imakeidx}
\usepackage{hyperref,cleveref}%

\numberwithin{equation}{section}
\definecolor{popblue}{RGB}{55,115,255}
\definecolor{lightbl}{RGB}{155,205,255}
\definecolor{depthbl}{RGB}{145,215,255}
\definecolor{fancyre}{RGB}{225,55,115}
\definecolor{darkblu}{RGB}{15,75,185}
\definecolor{mellowy}{RGB}{225,225,35}

\renewcommand{\tilde}[1]{\widetilde{#1}}
\renewcommand{\Bar}{\overline}

\renewcommand{\S}{\mathbb{S}}
\newcommand{\R}{\mathbb{R}}
\newcommand{\N}{\mathbb{N}}
\newcommand{\Z}{\mathbb{Z}}

\newcommand{\C}{\mathbb{C}}

\newcommand{\F}{\mathbb{F}}

\newcommand{\T}{\mathbb{T}}

\newcommand{\imp}{\;\Rightarrow\;}
\newcommand{\m}{\mathrm}

\newcommand{\lv}{\lVert}
\newcommand{\rv}{\rVert}

\newcommand{\al}{\alpha}
\newcommand{\be}{\beta}
\newcommand{\es}{\varnothing}

\newcommand{\ep}{\varepsilon}
\newcommand{\f}{\frac}
\newcommand{\sig}{\sigma}
\newcommand{\gam}{\gamma}
\newcommand{\del}{\delta}

\newcommand{\bn}{\binom}

\newcommand{\pd}{\partial}

\newcommand{\grad}{\nabla}
\newcommand{\bpm}{\begin{pmatrix}}
\newcommand{\epm}{\end{pmatrix}}

\newcommand{\loc}{\m{loc}}

\newcommand{\emb}{\hookrightarrow}

\renewcommand{\bigstar}{%
{
  \tikz[baseline=(center), scale=0.2] {
    \coordinate (center) at (0,0);
    \filldraw (0,0.4) -- (0.1,0) -- (0,-0.4) -- (-0.1,0) -- cycle;
    \filldraw (0.4,0) -- (0,0.1) -- (-0.4,0) -- (0,-0.1) -- cycle;
  }
}
}

\renewcommand{\le}{\leqslant}
\renewcommand{\ge}{\geqslant}

\newcommand{\norm}[1]{\left\lv#1\right\rv}
\newcommand{\bnorm}[1]{\Big\lv#1\Big\rv}

\newcommand{\tnorm}[1]{\lv#1\rv}

\newcommand{\bp}[1]{\Big(#1\Big)}
\renewcommand{\sp}[1]{\big(#1\big)}
\newcommand{\tp}[1]{(#1)}

\newcommand{\babs}[1]{\Big|#1\Big|}

\newcommand{\tabs}[1]{|#1|}

\newcommand{\bsb}[1]{\Big[{#1}\Big]}
\newcommand{\ssb}[1]{\big[{#1}\big]}
\newcommand{\tsb}[1]{[{#1}]}

\newcommand{\tcb}[1]{\{{#1}\}}

\providecommand{\tbr}[1]{\langle #1 \rangle}

\newcommand{\z}[1]{\mathring{#1}}

\renewcommand{\bf}[1]{\mathbf{#1}}
\newcommand{\ii}{\m{i}}

\newtheorem{prop}{\color{popblue}{Proposition}}[section]
\newtheorem{thm}[prop]{\color{popblue}{Theorem}}
\newtheorem{defn}[prop]{\color{popblue}{Definition}}
\newtheorem{lem}[prop]{\color{popblue}{Lemma}}
\newtheorem{coro}[prop]{\color{popblue}{Corollary}}
\newtheorem{rmk}[prop]{\color{popblue}{Remark}}

\newtheorem{innercustomthm}{\color{popblue}{Theorem}}
\newenvironment{customthm}[1]
{\renewcommand\theinnercustomthm{#1}\innercustomthm}
{\endinnercustomthm}

\author{Noah Stevenson}
\address{
	Department of Mathematics\\
	Princeton University\\
	Princeton, NJ 08544, USA
}
\email[N. Stevenson]{stevenson@princeton.edu}
\thanks{N. Stevenson was supported by an NSF Graduate Research Fellowship}
\author{Ian Tice}
\address{
	Department of Mathematical Sciences\\
	Carnegie Mellon University\\
	Pittsburgh, PA 15213, USA
}
\email[I. Tice]{iantice@andrew.cmu.edu}
\thanks{I. Tice was supported by an NSF Grant (DMS \#2204912). }

\DeclareFontFamily{U}{cbgreek}{}
\DeclareFontShape{U}{cbgreek}{m}{n}{
  <-6>    grmn0500
  <6-7>   grmn0600
  <7-8>   grmn0700
  <8-9>   grmn0800
  <9-10>  grmn0900
  <10-12> grmn1000
  <12-17> grmn1200
  <17->   grmn1728
}{}
\DeclareFontShape{U}{cbgreek}{bx}{n}{
  <-6>    grxn0500
  <6-7>   grxn0600
  <7-8>   grxn0700
  <8-9>   grxn0800
  <9-10>  grxn0900
  <10-12> grxn1000
  <12-17> grxn1200
  <17->   grxn1728
}{}

\DeclareRobustCommand{\Qoppa}{%
  \text{\usefont{U}{cbgreek}{\normalorbold}{n}\symbol{21}}%
}
\DeclareRobustCommand{\sampi}{%
  \text{\usefont{U}{cbgreek}{\normalorbold}{n}\symbol{27}}%
}
\DeclareRobustCommand{\Sampi}{%
  \text{\usefont{U}{cbgreek}{\normalorbold}{n}\symbol{23}}%
}

\DeclareRobustCommand{\kkappa}{%
  \text{\usefont{U}{cbgreek}{\normalorbold}{n}\symbol{107}}%
}
\makeatletter
\newcommand{\normalorbold}{%
  \ifnum\pdf@strcmp{\math@version}{bold}=\z@ bx\else m\fi
}
\makeatother

\title[Stationary wave solutions to shallow water equations]{Stationary wave solutions to two dimensional viscous shallow water equations: theory of small and large solutions}

\begin{document}

\begin{abstract}

We study a system of forced viscous shallow water equations with nontrivial bathymetry in two spatial dimensions.  We develop a well-posedness theory for small but arbitrary forcing data, as well as for a fixed data profile but large amplitude.  In the latter case, solutions may actually fail to exist for large amplitude, but in this case we prove that one of three physically meaningful breakdown scenarios occurs.  Through the use of implicit function theorem techniques and a priori estimates, we construct both spatially periodic and solitary (non-periodic but spatially localized) solutions.  The solitary case is substantially more complicated, requiring a delicate analysis in weighted Sobolev spaces.  To the best of our knowledge, these results constitute the first general construction of stationary wave solutions, large or otherwise, to the viscous shallow water equations and the first general analysis of large solitary wave solutions to any viscous free boundary fluid model.
\end{abstract}
\maketitle
\section{Introduction}

\subsection{The shallow water equations with applied force and bathymetry}\label{intro_-0}

In this paper we study a system of shallow water equations in two spatial dimensions that incorporates the effects of gravity, capillarity,  viscosity, laminar drag, applied forces, and variable bathymetry (bottom topography):
\begin{equation}\label{time-dependent variable-batheymetry shallow water equations}
    \begin{cases}
        \pd_th+\grad\cdot(\tp{h-b}v)=0,\\
        \pd_t\tp{(h-b)v}+\grad\cdot\tp{\tp{h-b}v\otimes v}+\al v-\mu\grad\cdot\tp{\tp{h-b}\mathbb{S}v}+\tp{h-b}\tp{g-\sig\Delta}\grad h={\upkappa} F,\\
        F=\int_b^hf(\cdot,y)\;\m{d}y+\tp{h-b}\grad\tp{\varphi(\cdot,h)}+\tp{\varphi I-\Xi}\tp{\cdot,h}\grad h+\xi(\cdot,h).
    \end{cases}
\end{equation}
Here the unknowns are the tangential fluid velocity field $v:\R^+\times\R^2\to\R^2$ and the fluid height $h:\R^+\times\R^2\to\R$.  The bottom topography is modeled via the bathymetry function $b:\R^2\to\R$.  We will always assume that the fluid height lies above the bathymetry, $h>b$, which means there are no `dry regions'.  The coefficient of laminar friction is $\al\in\R^+$, the viscosity is $\mu\in\R^+$, and the shallow water viscous stress tensor is
\begin{equation}\label{shallow water viscous stress tensor}
    \mathbb{S}v=\grad v+\grad v^{\m{t}}+2\tp{\grad\cdot v}I.
\end{equation}
The gravitational acceleration in the vertical direction is $g\ge0$ and the surface tension coefficient is $\sig\in\R^+$. The vector field $F:\R^+\times\R^2\to\R^2$, with units of length, is the profile of the forcing applied to the system while the coefficient $\upkappa\in\R$, with units of acceleration, indicates the strength of the applied forcing.

The system~\eqref{time-dependent variable-batheymetry shallow water equations} is a viscous variant of the classical Saint-Venant equations ($\sig=\mu={\upkappa}=0$). Its solutions approximate those of the free boundary incompressible Navier-Stokes system in the regimes in which either the fluid depth is very small or the solution has large characteristic wavelength.  As a result, the model is physically, mathematically, and practically relevant in describing the flow of a thin fluid film over the graphical surface determined by $b$. For details of the derivation of \eqref{time-dependent variable-batheymetry shallow water equations} from the Navier-Stokes system see Appendix \ref{appendix on derivation of SWE w/ bathymetry}.  

There are formal similarities between system~\eqref{time-dependent variable-batheymetry shallow water equations} and the barotropic compressible Navier-Stokes equations with density-dependent viscosity, in which $h-b$ and $v$ play the roles of the density and velocity, respectively.  This correspondence dictates how we shall refer to the individual equations in the shallow water system.  The first equation in~\eqref{time-dependent variable-batheymetry shallow water equations} is dubbed the continuity equation, and it describes how the surface height $h$ is transported and stretched by flow.  The second vector equation is called the momentum equation, and it asserts a balance of the advection term $\pd_t\tp{\tp{h-b}v}+\grad\cdot\tp{\tp{h-b}v\otimes v}$ by the sum of laminar drag $-\al v$, dissipation $\mu\grad\cdot\tp{\tp{h-b}\mathbb{S}v}$, gravity-capillary effects $-\tp{h-b}\tp{g-\sig\Delta}\grad h$, and applied force ${\upkappa}F$.  Notice that the role played by the bathymetry $b$ is only to modify the fluid depth $h-b$, which shows up as a sort of `variable coefficient'.  We emphasize that system~\eqref{time-dependent variable-batheymetry shallow water equations} is quasilinear and of mixed-type, as it involves both dissipative and hyperbolic-dispersive structures.

The third equation of~\eqref{time-dependent variable-batheymetry shallow water equations} shows that the term $F$ is a function of both the bathymetry and the height.  The precise form recorded there is derived by taking the shallow water limit of the free boundary Navier-Stokes equations acted on by a bulk force and surface stress. The function $f:\R^+\times\R^3\to\R^2$ is derived from the tangential part of the bulk force, while the functions $\varphi:\R^+\times\R^3\to\R$, $\Xi:\R^+\times\R^3\to\R^{2\times 2}_{\m{sym}}$, and $\xi\in\R^+\times\R^3\to\R^2$ are pieces of a surface stress tensor.  Again, we refer the reader to Appendix~\ref{appendix on derivation of SWE w/ bathymetry} for the precise details of the derivation.

\begin{figure}[!h]
    \centering
    \scalebox{1.42}{\includegraphics{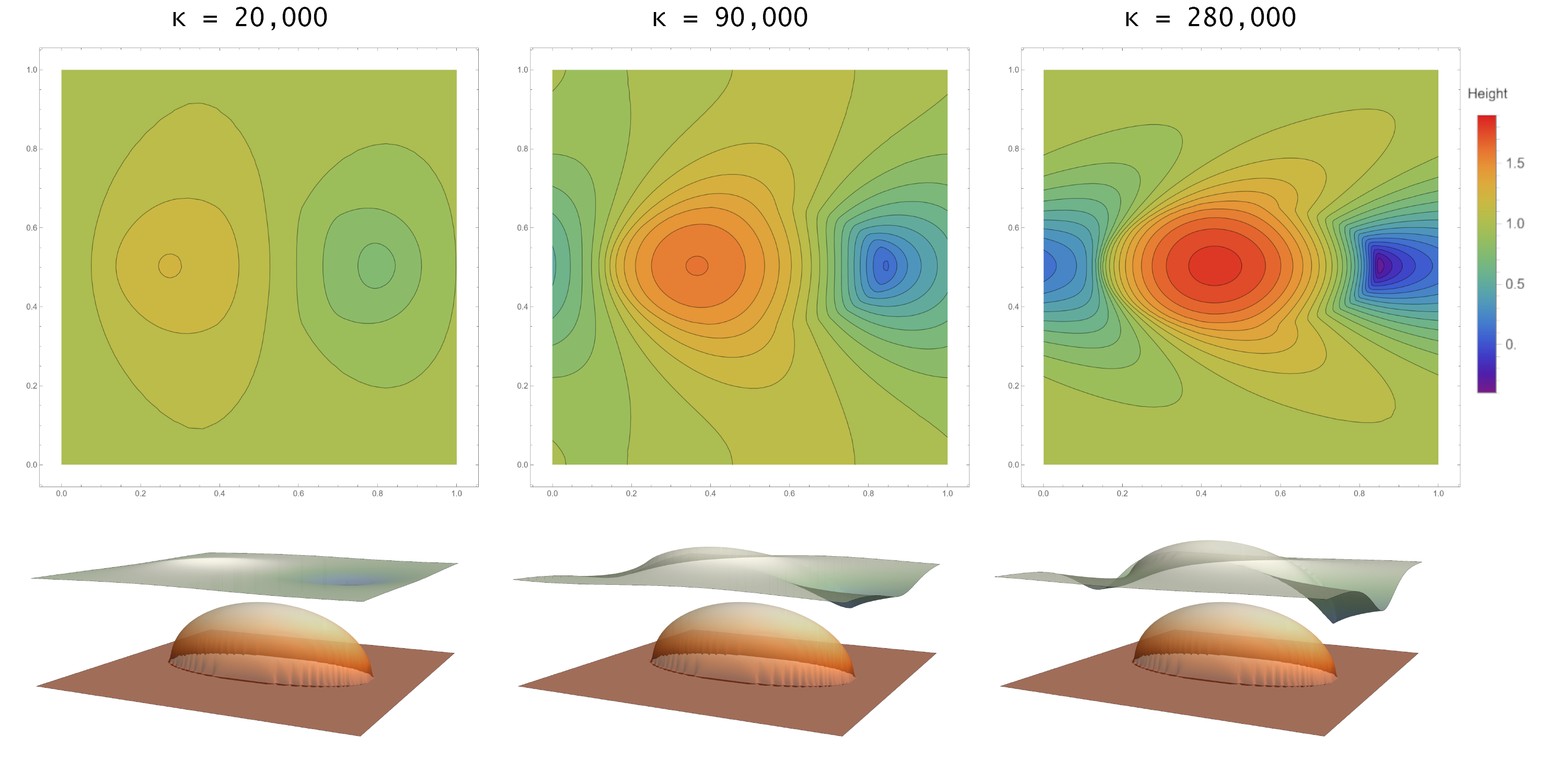}}
    \caption{
    Depicted here are the height data from numerical simulations of periodic stationary solutions to~\eqref{time-dependent variable-batheymetry shallow water equations}, in contour plot form (top) and graph form (bottom), for three values of applied forcing strength ${\upkappa}\in\tcb{20000,90000,280000}$. The bathymetry $b$ (shown in orange on the bottom) is a fixed half-ellipse with base at height $-2$ and maximum at height $0$. The forcing profile $F$ is a component of gravity in the $e_1$ direction, i.e. $\varphi=0$, $\Xi=0$, $\xi=0$, and $f=e_1$. The surface satisfies $h=1+h_0$ for an average zero perturbation $h_0$. The other physical parameters are $\al=0.65$, $g=0.5$, $\mu=0.3$, and $\sig=1.1$. This is meant to depict stationary shallow water flow over an inclined bathymetry.
    }
    \label{the progression figure}
\end{figure}

Our principal goal in this paper is to develop a well-posedness theory for stationary (time independent) solutions to~\eqref{time-dependent variable-batheymetry shallow water equations}.  We will construct both spatially periodic and non-periodic but localized and decaying solutions; we will refer to the former as the `periodic case' and the latter as the `solitary case'.  We will first construct small solutions for arbitrary data in a fixed open neighborhood of zero and then employ global implicit function theorem techniques to study how solutions behave as the amplitude of a fixed forcing profile is increased.  In the latter case we will derive quantitative `alternatives' that show that either the curves of solutions can be continued for all forcing data amplitudes or else degenerate in a specific, physically meaningful, manner.

Although we consider general $F$ of the form specified in~\eqref{time-dependent variable-batheymetry shallow water equations}, several specific choices warrant mentioning. One can model gravity-driven flow in the direction $\nu\in\S^1$ down an incline with the choice $F = \tp{h-b}\nu$, i.e. $f=\nu$, $\varphi=0$, $\Xi=0$, and $\xi=0$.  This scenario is depicted in Figures~\ref{the progression figure} and~\ref{the div-curl figure}.  On the other hand, the choice $F=(h-b)\grad p$, which is achieved by $f=0$, $\varphi=p$, $\Xi=p I$, and $\xi=0$ for some scalar surface pressure $p:\R^2\to\R$, models a local change in pressure above the free surface.  

\subsection{Equilibria and stationary reformulation}\label{intro_-1}

We now aim to  identify the equilibrium solutions to~\eqref{time-dependent variable-batheymetry shallow water equations} and then  reformulate~\eqref{time-dependent variable-batheymetry shallow water equations} perturbatively about these equilibria in a nondimensional fashion.  When $F=0$ and $b$ is a bounded function, the admissible steady solutions are constant and given by
\begin{equation}
    v_{\m{eq}}=0,\quad h_{\m{eq}}=\m{H}+\sup b\quad\text{for any choice of }\m{H}\in\R^+.
\end{equation}
Fixing an arbitrary choice of $\m{H}\in\R^+$, we may  nondimensionalize by selecting units with the effective normalization $(\m{H},\mu,\sig)\mapsto\tp{1,1,1}$.  This is a convenient choice, as positive viscosity and capillarity are crucial for our analytic techniques.  Define the length and time
\begin{equation}\label{nondimensional part 1}
    \m{L}=\mu^2/\sig,\quad\m{T}=\mu^3/\sig^2
\end{equation}
and the new (negative) bathymetry $\be$, velocity $\upupsilon$, and free surface height $\upeta$ as follows:
\begin{equation}\label{nondimensional part 2}
    b(x)=\sup b-\m{H}\be\tp{x/\m{L}},\quad v(t,x)=\tp{\m{L}/\m{T}}\upupsilon\tp{t/\m{T},x/\m{L}},\quad h(t,x)=\m{H}\tp{1+\upeta\tp{t/\m{T},x/\m{L}}}+\sup b.
\end{equation}
In turn, define the new data profile $\Upphi$, inverse slip coefficient $A$, capillary number $G$, and inverse Froude number $\kappa$ via
\begin{equation}\label{nondimensional part 3}
F(t,x)=\m{H}\Upphi\tp{t/\m{T},x/\m{L}},\quad A=\f{\al\m{T}}{\m{H}}=\f{\al\mu^3}{\m{H}\sig^2},\quad G=\f{g\m{H}\m{T}^2}{\m{L^2}}=\f{g\m{H}\mu^2}{\sig^2},\quad\kappa=\f{{\upkappa}\m{T}^2}{\m{L}}=\f{{\upkappa}\mu^4}{\sig^3}.
\end{equation}
Notice that $\be\ge 0$ with $\inf\be=0$; we work with the nonnegative quantity $\be$ here for convenience, but one should think of $-\be$ as the nondimensional description of the bathymetry. The positive fluid depth condition $h>b$ translates in these new variables to $1 + \upeta +\be > 0$. In terms of the quantities~\eqref{nondimensional part 2} and~\eqref{nondimensional part 3}, system~\eqref{time-dependent variable-batheymetry shallow water equations} becomes
\begin{equation}\label{time dependent nondimensionalized equations}
    \begin{cases}
        \pd_t\upeta+\grad\cdot(\tp{1+\be+\upeta}\upupsilon)=0,\\
        \pd_t\tp{\tp{1+\be+\upeta}\upupsilon}+\grad\cdot\tp{\tp{1+\be+\upeta}\upupsilon\otimes\upupsilon}+A\upupsilon-\grad\cdot\tp{\tp{1+\be+\upeta}\mathbb{S}\upupsilon}+(1+\be+\upeta)\tp{G-\Delta}\grad\upeta=\kappa\Upphi.
    \end{cases}
\end{equation}

As previously mentioned, our focus in this work is the not the full time dependent system~\eqref{time dependent nondimensionalized equations}, but rather the particular class of stationary (and thus global-in-time) solutions generated by generic time independent forcing.  We are therefore led to make the  stationary ansatz  $\pd_t\upeta=0$, $\pd_t\upupsilon=0$, and $\pd_t\Upphi=0$, which we formulate for new stationary unknowns as $\upeta(t,x)=\pmb{\eta}(x)$, $\upupsilon(t,x)=\pmb{u}(x)$, and $\Upphi(t,x)=\Phi(x)$.  In terms of the new unknowns, the system \eqref{time dependent nondimensionalized equations} becomes
\begin{equation}\label{stationary nondimensionalized equations}
    \begin{cases}
        \grad\cdot\tp{\tp{1+\be+\pmb{\eta}}\pmb{u}}=0,\\
        \grad\cdot\tp{\tp{1+\be+\pmb{\eta}}\pmb{u}\otimes\pmb{u}}+A\pmb{u}-\grad\cdot\tp{\tp{1+\be+\pmb{\eta}}\mathbb{S}\pmb{u}}+\tp{1+\be+\pmb{\eta}}\tp{G-\Delta}\grad\pmb{\eta}=\kappa\Phi.
    \end{cases}
\end{equation}
The final form of the momentum equation source term $\Phi$ in~\eqref{stationary nondimensionalized equations}, which we formally write as
\begin{equation}\label{formal form of the forcing phi}
    \Phi=\phi(\cdot,\pmb{\eta})+\grad\tp{\tp{1+\be+\pmb{\eta}}\psi\tp{\cdot,\pmb{\eta}}}+\tau(\cdot,\pmb{\eta})\grad\pmb{\eta},
\end{equation}
not only depends on $\pmb{\eta}$ but also involves the generic functions $\phi:\R^3\to\R^2$, $\psi:\R^3\to\R$, and $\tau:\R^3\to\R^{2\times 2}_{\m{sym}}$. This captures the general structure of the forcing present in the third equation of~\eqref{time-dependent variable-batheymetry shallow water equations} after the rescaling and the stationary ansatz. See Section~\ref{subsubsection for data discussion} for a more thorough discussion on these relationships.

We now discuss the necessity of nontrivial forcing in~\eqref{stationary nondimensionalized equations} for the generation of nontrivial solutions. As part of the proof of the later result Proposition~\ref{prop on maximally locally analytic curve of solutions}, we establish the following power-dissipation identity for solutions to system~\eqref{stationary nondimensionalized equations} with trivial forcing ($\kappa=0$):
\begin{equation}\label{the power dissipation identity}
    \int A|\pmb{u}|^2+\tp{1+\be+\pmb{\eta}}\tp{2^{-1}\tabs{\grad\pmb{u}+\grad\pmb{u}^{\m{t}}}^2+2\tabs{\grad\cdot\pmb{u}}^2}=0.
\end{equation}
Here the domain of integration is understood as $\R^2$ when the solutions are solitary and as a fundamental domain when the solutions are periodic.  When $A>0$ and $1+\be+\pmb{\eta}>0$, identity~\eqref{the power dissipation identity} implies that $\pmb{u}=0$. In turn from~\eqref{stationary nondimensionalized equations} we may then deduce that $\grad\pmb{\eta}=0$, and hence the solution is trivial. This tells us that nontrivial solitary or periodic stationary solutions to~\eqref{stationary nondimensionalized equations} cannot exist unless they are generated by a nontrivial stationary applied force or stress, meaning $\kappa \Phi \neq0$. 

\subsection{Previous work}

A number of closely related variants of our shallow water system have been derived and studied in the literature: see, for instance \cite{MR1324142,RevModPhys.69.931,MR1867089,MR1821555,MR1975092,MR2118849,MR2397996,MR4105349}.  The model we employ is most similar to those derived in the articles of Marche~\cite{MR2281291}, Bresch~\cite{MR2562163}, and Mascia~\cite{mascia_2010}; the article \cite{MR2562163} also contains a very thorough survey of the dynamic problem for \eqref{time-dependent variable-batheymetry shallow water equations} and related models with flat bathymetry. As our paper's focus is the stationary wave problem over smooth bottom topography, we will focus this literature review on work that accounts for bathymetry or studies time independent solutions.

Intimately related to the problem of stationary waves is that of traveling waves.  Within the context of the shallow water equations (both inviscid and viscous) there is a large existing literature analyzing special types of one dimensional traveling waves down an incline, known as roll waves. Dressler~\cite{MR0033717} gives the first construction of discontinuous roll wave solutions to the inviscid Saint-Venant equations. Merkin and Needham~\cite{1984RSPSA.394..259N,MR0853213} augment to Dressler's model an energy dissipation term and study the resulting roll waves; Hwang and Chang~\cite{MR0890282} provide further analyses. Chang, Cheng, and Prokopiou~\cite{PCC91} include capillary effects and produce solitary roll waves. A complete theory of linear and nonlinear stability of one dimensional roll wave solutions to both the inviscid and viscous shallow water equations is developed in the works~\cite{MR2784868,MR2813829,MR3600832,MR3219516,MR3449905,MR3933411}. The existence of two dimensional viscous gravity-capillary roll waves is addressed by Stevenson~\cite{stevenson2024shallow}. We also mention that Stevenson and Tice~\cite{stevenson2023shallow} establish the well-posedness for the small amplitude solitary traveling wave problem for system~\eqref{time-dependent variable-batheymetry shallow water equations} over trivial bathymetry; however large or stationary waves were left unaddressed. 

We next discuss the shallow water system with explicit bathymetric effects.  To the best of our knowledge there are no results for our specific shallow water model~\eqref{time-dependent variable-batheymetry shallow water equations}; nevertheless, there are some results on related models. Cherevko and Chupakhin~\cite{MR2502107,MR2531748} study shallow water equations on a rotating sphere and provide stationary solutions. De Valeriola and Van Schaftingen~\cite{MR3101789} study point vortex type stationary solutions and their desingularization within a system of shallow water equations with bathymetry called the lake equations; this is extended to multiple vortices by Cao and Liu~\cite{MR4285533}. Marangos and Porter~\cite{MR4340438} study linear shallow water theory over structured bathymetry. Ketcheson and Quezada de Luna~\cite{MR4252864} study the dispersive properties of linearized water waves over a periodic bathymetry and compare numerically with the Saint-Venant equations.

As the shallow water equations are derived from the free boundary incompressible Navier-Stokes equations, our work is also linked to the stationary and traveling wave problems for free boundary fluids. For the inviscid free boundary Euler equations, the water wave problem has been the target of intense study for well over a century. The interested reader is referred to the survey articles of Toland~\cite{Toland_1996}, Groves~\cite{Groves_2004}, Strauss~\cite{Strauss_2010}, and Haziot, Hur, Strauss, Toland, Wahl\'en, Walsh, and Wheeler~\cite{MR4406719}.

In contrast, the study of stationary and traveling wave solutions to viscous free boundary problems has only recently commenced.  An essential feature of these variants of the problem is that waves are generated by the application of stationary or traveling external forces and stresses. For experimental studies regarding these types of traveling waves we refer, for instance, to the work of Akylas, Cho, Diorio, and Duncan~\cite{CDAD_2011,DCDA_2011}, Duncan and Masnadi~\cite{MD_2017}, and Cho and Park~\cite{PC_2016,PC_2018}. 

Leoni and Tice~\cite{MR4630597} give the first mathematical construction of strictly traveling wave solutions to the free boundary incompressible Navier-Stokes equations. This result is generalized by Stevenson and Tice~\cite{MR4337506,stevenson2023wellposedness} to incompressible fluids with multiple layers and to compressible free boundary flows, and by Koganemaru and Tice~\cite{MR4609068,MR4785303} to periodic or inclined fluids and flows obeying the Navier slip boundary condition. Nguyen and Tice~\cite{MR4690615} construct traveling wave solutions to free boundary Darcy flow and, in the periodic case, analyze their stability. The existence of solitary stationary waves requires more difficult analysis than the strictly traveling cases; Stevenson and Tice~\cite{MR4787851} give the first construction of solitary stationary traveling wave solutions to the free boundary incompressible Navier-Stokes system. All of the aforementioned results on stationary and traveling waves in viscous fluids regard small amplitude solutions. The first general constructions of large amplitude traveling wave solutions are carried out by Nguyen~\cite{nguyen2023largetravelingcapillarygravitywaves} and Brownfield and Nguyen~\cite{MR4797733} for free boundary Darcy flow.

Within the context of these previous results on viscous traveling and stationary waves, one sees that the current paper furthers the theory by settling two outstanding questions about the shallow water system. First, in the case of both Reynolds and Bond numbers finite we resolve the problem of well-posedness for small amplitude stationary waves. Second, we provide the first study of generic large stationary wave solutions, further expanding the study of viscous traveling wave problems past the realm of small solutions. Another key feature of our results is that we handle both the cases of periodic and solitary waves. For the latter case our work stands as the first construction of large solitary waves solutions within the viscous stationary and traveling wave problem literature.

\subsection{Main results,  discussion, and outline}\label{section on main results and discussion}

We now come to the statement of our main theorems, the proofs of which can be found in Section~\ref{section on conclusions}.  Our results split into two categories based on whether we assume the solutions are periodic or not, with the latter case resulting in spatially localized solitary waves.  Within each category we prove two results, one on small solutions and one on large solutions. The former establishes the well-posedness of the system~\eqref{stationary nondimensionalized equations} with a fixed bathymetry $\be$ and $\kappa=1$, and forcing $\Phi$ generic and small (near $0$). The large data result then explores what happens for a fixed but arbitrary choice of forcing data profile but variable forcing strength $\kappa\in\R$.  Before stating any of these, a discussion of how we handle the forcing term $\Phi$ in~\eqref{stationary nondimensionalized equations} is required.

\subsubsection{Data technicalities}\label{subsubsection for data discussion}

 Equation~\eqref{formal form of the forcing phi} shows that the forcing profile $\Phi$ of the second equation in~\eqref{stationary nondimensionalized equations} depends nonlinearly on the free surface function $\pmb{\eta}$.  As we discuss in Appendix~\ref{appendix on derivation of SWE w/ bathymetry}, this dependence is not arbitrary and has a specific functional form $\pmb{\eta} \mapsto \Phi$ derived from the base free boundary Navier-Stokes system.  In specifying this dependence we run into some technical issues that merit discussion.

The first issue we encounter is that we would like to impose just enough regularity on $\Phi$ to produce strong solutions, say $\pmb{u}\in H^2_{\loc}(\R^2;\R^2)$ and $\pmb{\eta} \in H^3_{\loc}(\R^2)$ with $\Phi \in L^2_{\loc}(\R^2;\R^2)$.  Given that $\Phi$ contains terms of the form $\phi(\cdot,\pmb{\eta})$, we then run into the usual headache of dealing with (partial) composition with maps that are no better than measurable on a product.  Of course, this is an old problem, and a typical solution is to force the maps defining $\Phi$ to be Carath\'eodory functions, which imposes a.e. continuity in the second variable.    

The second issue is that we aim to produce solutions with the implicit function theorem for maps between Banach spaces: the standard local one for small solutions, and a specialized global one for large solutions. At a minimum, this requires continuous differentiability for the (Nemytskii) map  $\mathcal{X} \supseteq U \ni \pmb{\eta} \mapsto \Phi \in \mathcal{Y}$, where $\mathcal{X}$ and $\mathcal{Y}$ are (to be determined) Banach spaces of functions and $U \subseteq \mathcal{X}$ is an open set. Again, this is not a new problem, and one solution is to impose enough regularity on the maps defining the dependence of $\Phi$ on $\pmb{\eta}$ to invoke some version of the omega lemma (e.g. Lemma 2.4.18 in~\cite{MR960687}, or the main results of \cite{IKT_2013}) to get the needed continuous differentiability between $\mathcal{X}$ and $\mathcal{Y}$. For example, for the component term $\pmb{\eta} \mapsto \phi(\cdot,\pmb{\eta})$ we would expect to need $\phi$ to be at least continuously differentiable in its second argument, surpassing the Carath\'eodory mandate.

It turns out that the nonlinear structure of~\eqref{stationary nondimensionalized equations} allows us to formulate solutions as the zeros of a real analytic map between Banach spaces.  Exploiting this analytic structure is advantageous, as it allows us to employ an analytic global implicit function theorem of Chen, Walsh, and Wheeler~\cite{chen2023globalbifurcationmonotonefronts} (see also Theorem~\ref{thm on analytic global implicit function theorem} below), which gives us stronger information than other global implicit function theorems (e.g. Section II.6 in~\cite{MR2859263}). Furthermore, real analyticity also presents a convenient, if somewhat tricky to explain, mechanism for dealing with the above two issues.  We now aim to explain our technique by focusing on the simple case $\Phi = \phi(\cdot,\pmb{\eta})$.

Fix the map $\phi : \R^2 \times I \to \R^2$ for some interval $I \subseteq \R$. The discussion above suggests that we want $\phi(\cdot,\pmb{\eta}) \in \mathcal{Y}\emb L^2_\loc\tp{\R^2;\R^2}$  and the map $\mathcal{X} \supseteq U \ni \pmb{\eta} \mapsto \phi(\cdot,\pmb{\eta}) \in \mathcal{Y}$ to be real-analytic. This is relatively easy to do if we assume that $\phi$ has a special structure. Indeed, if we assume $\phi$ is continuous and a polynomial in its final variable, say $\phi(x,y) = \sum_{j=0}^\ell g_j(x) y^j$ for $g_j \in C^0\tp{\R^2;\R^2}\cap\mathcal{Y}$, then the pointwise evaluation $\Phi(x) = \phi(x,\pmb{\eta}(x))$ presents no trouble if we assume $\pmb{\eta} \in \mathcal{X}$ for $\mathcal{X}$ a unital Banach subalgebra of $(C^0\cap L^\infty)(\R^2;\R)$, and we can obtain analyticity of the map $\pmb{\eta} \mapsto \Phi$ by noting that $\mathcal{X} \ni \pmb{\eta} \mapsto \sum_{j=0}^\ell g_j\pmb{\eta}^j \in \mathcal{Y}$ is analytic when $\mathcal{X}\emb\mathcal{L}\tp{\mathcal{Y}}$ via the pointwise product.

This special polynomial form of $\phi$ is too restrictive, but the analysis presents a natural generalization. We first Curry the map $\phi$ in its second argument to think of $\phi(x,y) = \hat{\phi}(y)(x)$ for a given analytic map $\hat{\phi} : I \to \mathcal{Y}$.  There is now some technical hassle in making sense of the formal diagonal composition $\R^2\ni x \mapsto \hat{\phi}(\pmb{\eta}(x))(x) \in \R^2$ in the general cases in which we only have the embedding $\mathcal{Y}\emb L^2_{\loc}\tp{\R^2;\R^2}$ and pointwise evaluation is discontinuous.  We circumvent this issue and more in Appendix~\ref{appendix on analytic composition in unital Banach algebras} by developing a variant of the holomorphic functional calculus that allows us to \emph{uniquely} (subject to some mild technical assumptions) define a `Nemytskii map' $\Upomega_{\hat\phi}(\pmb{\eta}) \in \mathcal{Y}$ in such a way that the function $\mathcal{X}\supseteq U \ni \pmb{\eta} \mapsto \Upomega_{\hat{\phi}}(\pmb{\eta})\in\mathcal{Y}$ is analytic and also satisfies: $\Upomega_{\hat{\phi}}\tp{\pmb{\eta}} = \sum_{j=0}^\ell g_j\pmb{\eta}^j$ whenever $\hat{\phi}(y) = \sum_{j=0}^\ell g_jy^j$ is a polynomial with $g_j\in\mathcal{Y}$; and $\tsb{\Upomega_{\hat{\phi}}\tp{\pmb{\eta}}}\tp{x} = \hat{\phi}\tp{\pmb{\eta}\tp{x}}\tp{x}$ for $x\in\R^2$ when $\mathcal{Y}\emb C^0_\loc\tp{\R^2;\R^2}$.

The upshot of this discussion is that in the following theorem statements we will take $\Phi$ in~\eqref{stationary nondimensionalized equations}, whose formal expression we initially gave as~\eqref{formal form of the forcing phi}, to be rigorously defined as
\begin{equation}\label{percolate propane pulvarized paintbrush}
    \Phi = \Upomega_\phi\tp{\pmb{\eta}} + \grad\tp{\tp{1+\be+\pmb{\eta}}\Upomega_\psi\tp{\pmb{\eta}}} + \Upomega_\tau\tp{\pmb{\eta}}\grad\pmb{\eta}
\end{equation}
where $\Upomega$ is the Nemytskii map constructed by the holomorphic function calculus in Appendix~\ref{appendix on analytic composition in unital Banach algebras} as described above, and $\phi$, $\psi$, $\tau$ are generically selected real analytic maps defined on an interval subset of the real line with values in certain Banach spaces of locally square summable functions.
\subsubsection{Results for periodic functions}\label{subsubsection for results for periodic functions}

Our first two theorems consider solutions to system~\eqref{stationary nondimensionalized equations} in spaces of periodic functions. We fix a tuple of period lengths $L=\tp{L_1,L_2}\in\tp{\R^+}^2$ and let $\T^2_L=\tp{L_1\T}\times\tp{L_2\T}$ denote the $2$-torus of side lengths $L_1$ and $L_2$. Given any $s\in\N$ and a finite dimensional real vector space $V$ we let $H^s\tp{\T^2_L;V}$ denote the standard Sobolev space of $V$-valued $L$-periodic functions; also we let $\z{H}^s \tp{\T^2_L;V}\subset H^s\tp{\T^2_L; V}$ denote the closed subspace of functions with vanishing mean.

To properly state our first theorem on small periodic solutions, we now need to specify the appropriate Banach space for the data $\phi$, $\psi$, and $\tau$ of~\eqref{percolate propane pulvarized paintbrush} (whose role is discussed in Section~\ref{subsubsection for data discussion}). The precise definition is a bit technical (see Section~\ref{section on small solutions}, and in particular equation~\eqref{welcome to the}), so for now we shall content ourselves with denoting the data space
\begin{equation}\label{This is an annoying data space}
    \pmb{Y}(\T^2_L)=\pmb{A}\tp{L^2\tp{\T^2_L;\R^2}}\times\pmb{A}\tp{H^1\tp{\T^2_L}}\times\pmb{A}\tp{L^2\tp{\T^2_L;\R^{2\times 2}_{\m{sym}}}}
\end{equation}
and mentioning that $\pmb{A}(\mathcal{Y})$, for $\mathcal{Y}\in\tcb{L^2\tp{\T^2_L;\R^2},H^1\tp{\T^2_L},L^2\tp{\T^2_L;\R^{2\times 2}_{\m{sym}}}}$, is a Banach space of $\mathcal{Y}-$valued real analytic functions defined on a real interval, containing at least every entire function.

\begin{customthm}{1}[Proved in Theorem~\ref{thm on theory of small solutions}]\label{z_MAIN_THEOREM_1}
    Let $\be\in C^\infty\tp{\T^2_L}$ satisfy $\min\be=0$, $A>0$, and $G\ge0$. There exists positive radii $r_{\m{soln}},r_{\m{data}}\in\R^+$ and nonempty open balls $B_{\m{data}}=B(0,r_{\m{data}})\subset\pmb{Y}\tp{\T^2_L}$ and $B_{\m{soln}}=B(0,r_{\m{soln}})\subset H^2\tp{\T^2_L;\R^2}\times \z{H}^3 \tp{\T^2_L}$ with the property that for all $(\phi,\psi,\tau)\in B_{\m{data}}$ there exists a unique $\tp{\pmb{u},\pmb{\eta}}\in B_{\m{soln}}$ solving the system of equations~\eqref{stationary nondimensionalized equations}  with $\kappa=1$ and $\Phi$ determined by~\eqref{percolate propane pulvarized paintbrush}. Moreover, the mapping $(\phi,\psi,\tau)\mapsto\tp{\pmb{u},\pmb{\eta}}$ is real analytic. 
\end{customthm}
Theorem~\ref{z_MAIN_THEOREM_1} establishes the well-posedness of equations~\eqref{stationary nondimensionalized equations} in the class of small amplitude and periodic functions.  Our next main result studies large solutions in the periodic setting. In this case we make some further (mild) restrictions on the data: we require a bit more smoothness on the $\tau$ component and that the domain of analyticity is unbounded from above (so that the composition with tall free surfaces $\pmb{\eta}$ is well-defined). The hypotheses on the bathymetry $\be$ are the same as in Theorem~\ref{z_MAIN_THEOREM_1}.

\begin{customthm}{2}[Proved in Theorem~\ref{thm on theory of large solutions}]\label{z_MAIN_THEOREM_2}
    Let $\be\in C^\infty\tp{\T^2_L}$ satisfy $\min\be=0$, $A>0$, and $G\ge 0$. Set $a_\be=1+\max\be$, fix any choice of real analytic functions
    \begin{equation}
        (\phi,\psi,\tau):(-a_\be,\infty)\to L^2\tp{\T^2_L;\R^2}\times H^1\tp{\T^2_L}\times\tp{L^\infty\cap H^1}\tp{\T^2_L;\R^{2\times 2}_{\m{sym}}},
    \end{equation}
    and define the set of admissible solutions (corresponding to the above fixed data profile and variable forcing strength) as
    \begin{equation}\label{_PSS_}
        \pmb{S}=\tcb{(\pmb{u},\pmb{\eta},\kappa)\in H^2\tp{\T^2_L;\R^2}\times \z{H}^3 \tp{\T^2_L}\times\R\;:\;\min\tp{1+\be+\pmb{\eta}}>0,\;\text{system~\eqref{stationary nondimensionalized equations} is satisfied with~\eqref{percolate propane pulvarized paintbrush}}}.
    \end{equation}
    Then there exists a locally analytic curve $\mathscr{S}\subseteq\pmb{S}$, with $0\in\mathscr{S}$, that is parametrized by the continuous mapping $\R\ni s\mapsto\tp{\pmb{u}\tp{s},\pmb{\eta}\tp{s},\kappa\tp{s}}\in\mathscr{S}$ and satisfies $0\mapsto0$ as well as the limits
    \begin{equation}\label{__PERIODIC__BU}
        \lim_{|s|\to\infty}\bsb{\tnorm{\pmb{\eta}\tp{s}}_{L^\infty}+\f{1}{\min\tp{1+\be+\pmb{\eta}\tp{s}}}+|\kappa(s)|}=\infty.
    \end{equation}
\end{customthm}

Note that in Theorems~\ref{z_MAIN_THEOREM_1} and~\ref{z_MAIN_THEOREM_2} the periodicity allows us to set the nondimensionalized constant $G=0$; as one can see from the definition in~\eqref{nondimensional part 3}, this can occur due to vanishing gravitational field or other ratios vanishing. The limits~\eqref{__PERIODIC__BU} indicate that for a fixed forcing profile, either there are solutions in $\mathscr{S}$ for an unbounded set of forcing amplitudes $\kappa$, or else $\kappa$ remains bounded but the solutions degenerate by the free surface touching the bottom or diverging in supremum norm.   See Figures~\ref{the progression figure} and~\ref{the div-curl figure} for depictions of some of the periodic solutions considered by Theorem~\ref{z_MAIN_THEOREM_2}.  
\begin{figure}[!h]
    \centering
    \scalebox{1.056}{\includegraphics{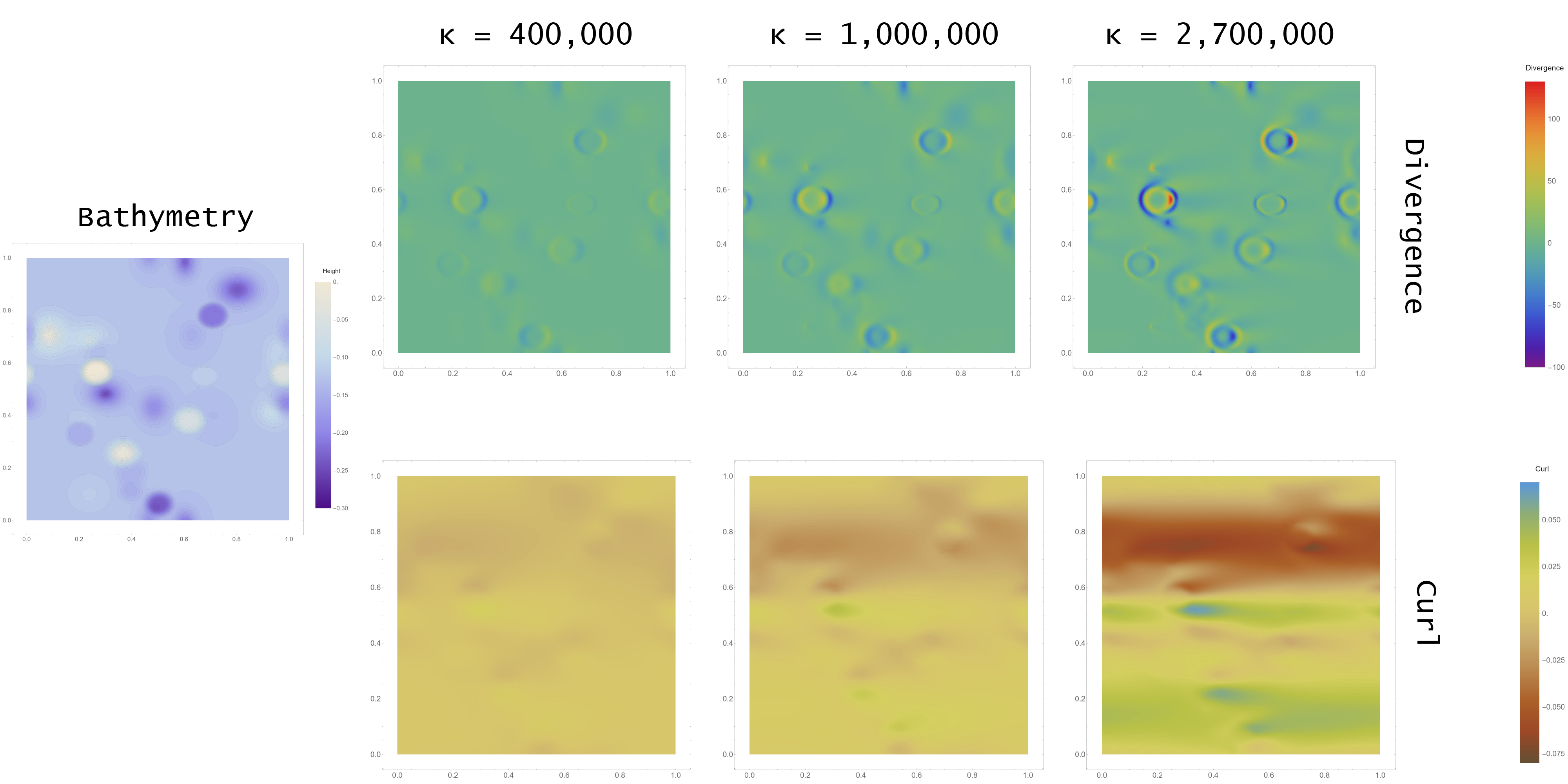}}
    \caption{Shown here are are numerical simulations of the divergence (top row) and curl (bottom row) of the tangential velocity $v$ of gravity driven shallow water flow over a fixed, but randomly generated bathymetry $b$ (shown on the left). These correspond to stationary solutions to system~\eqref{time-dependent variable-batheymetry shallow water equations} with forcing $F$ determined by $f=e_1$, $\varphi=0$, $\Xi=0$, and $\xi=0$. The forcing strength varies in the set ${\upkappa}\in\tcb{400000,1000000,2700000}$ while the other parameters $\al$, $g$, $\mu$, $\sig$, and equilibrium height are all set to $1$.
    }
    \label{the div-curl figure}
\end{figure}

\subsubsection{Results for solitary functions}\label{subsubsection on results for solitary functions}

Our final two main results are the analogs of Theorems~\ref{z_MAIN_THEOREM_1} and~\ref{z_MAIN_THEOREM_2} but in solitary (i.e. decaying at infinity) function spaces. The functional framework becomes significantly more complicated due to the fact that in this setting we must overcome a certain low-mode degeneracy in the estimates for the free surface. We achieve this by working with Sobolev spaces built on weighted Lebesgue spaces. The full details of these function spaces can be found in Sections~\ref{appendix on weighted sobolev spaces} and~\ref{appendix on weighted gradient sobolev spaces}, but for now we shall content ourselves with the following abbreviated definitions.  Given $1\le q\le\infty$, $0<\del<1$, and a finite dimensional real vector space $V$, we let $L^q_{\del}\tp{\R^2;V}$ denote the set of $f\in L^q\tp{\R^2;V}$ for which $\tbr{X}^{\del}f\in L^q\tp{\R^2;V}$, where $X:\R^2\to\R^2$ denotes the identity map and $\tbr{\cdot}$ is the  bracket~\eqref{NIHON BRAK}. Given $s\in\N$, we then define the weighted Sobolev space $W^{s,q}_{\del}\tp{\R^2;V}$ to be the collection of $f\in L^q_{\del}\tp{\R^2;V}$ such that for all multiindices $|\al|\le s$ we have $\pd^\al f\in L^q_{\del}\tp{\R^2;V}$. When $q=2$ we  write $H^s_{\del}\tp{\R^2;V}$ in place of $W^{s,2}_{\del}\tp{\R^2;V}$.

These weighted Sobolev spaces are appropriate container spaces for the velocity $\pmb{u}$ and  the data, but they are not quite adequate for the free surface function $\pmb{\eta}$.  For it we also need to introduce the weighted gradient Sobolev spaces, denoted by $\tilde{H}^s_{\del}\tp{\R^2}$ for $s\in\N^+$ and $0<\del<1$, and consisting  of $f\in L^{2/\del}\tp{\R^2}$ with $\grad f\in H^{s-1}_{\del}\tp{\R^2;\R^2}$, normed by the latter inclusion.  These spaces satisfy the embeddings $H^s_{\del}\tp{\R^2}\emb\tilde{H}^s_{\del}\tp{\R^2}\emb \tp{H^s_{\del}+W^{\infty,2/\del}}\tp{\R^2}$.

The solitary case necessitates imposing some admissibility conditions on the bathymetry profile $\be$.  Namely, we need its derivatives to be localized, which we impose by assuming that there exists $\del\in\tp{0,1}$ such that 
\begin{equation}\label{solitary assumptions on the bathymetry}
\be\in C^\infty\tp{\R^2},\;\grad\be\in H^\infty_\del\tp{\R^2;\R^2},\;\text{and}\;\inf\be=0.
\end{equation}

We now come to our third main theorem, which considers small solitary solutions.  In place of~\eqref{This is an annoying data space}, the solitary data space  is
\begin{equation}\label{data spaces for weighted small solutions}
    \pmb{Y}\tp{\R^2}=\pmb{A}\tp{L^2_{\scriptscriptstyle1/2}\tp{\R^2;\R^2}}\times\pmb{A}\tp{H^1_{\scriptscriptstyle1/2}\tp{\R^2}}\times\pmb{A}\tp{L^2_{\scriptscriptstyle1/2}\tp{\R^2;\R^{2\times 2}_{\m{sym}}}},
\end{equation}
whose precise definition is given at the start of Section~\ref{section on small solutions}.

\begin{customthm}{3}[Proved in Theorem~\ref{thm on theory of small solutions}]\label{z_MAIN_THEOREM_3}
    Let $\be$ be as in~\eqref{solitary assumptions on the bathymetry}, $A>0$, and $G>0$. There exists positive radii $\rho_{\m{soln}},\rho_{\m{data}}\in\R^+$ and nonempty open balls $B_{\m{data}}=B(0,\rho_{\m{data}})\subset\pmb{Y}\tp{\R^2}$ and $B_{\m{soln}}=B(0,\rho_{\m{soln}})\subset H^2_{\scriptscriptstyle1/2}\tp{\R^2;\R^2}\times\tilde{H}^3_{\scriptscriptstyle1/2}\tp{\R^2}$ with the property that for all $(\phi,\psi,\tau)\in B_{\m{data}}$ there exists a unique $\tp{\pmb{u},\pmb{\eta}}\in B_{\m{soln}}$ solving~\eqref{stationary nondimensionalized equations} with $\kappa =1$ and $\Phi$ determined by~\eqref{percolate propane pulvarized paintbrush}. Moreover, the mapping $\tp{\phi,\psi,\tau}\mapsto\tp{\pmb{u},\pmb{\eta}}$ is real analytic.
\end{customthm}
Theorem~\ref{z_MAIN_THEOREM_3} establishes the well-posedness of equations~\eqref{stationary nondimensionalized equations} in a class of small amplitude and solitary functions. Our final main theorem examines large solutions in the solitary case. As in the periodic case, we make some further minor restrictions on the class of data considered.

\begin{customthm}{4}[Proved in Theorem~\ref{thm on theory of large solutions}]\label{z_MAIN_THEOREM_4}
    Let $\be$ be as in~\eqref{solitary assumptions on the bathymetry}, $A>0$, and $G>0$. Set $a_\be=1+\sup\be$, fix any choice of real analytic functions
    \begin{equation}
        (\phi,\psi,\tau):\tp{-a_\be,\infty}\to L^2_{\scriptscriptstyle1/2}\tp{\R^2;\R^2}\times H^1_{\scriptscriptstyle1/2}\tp{\R^2}\times\tp{L^\infty_{\scriptscriptstyle1/2}\cap H^1_{\scriptscriptstyle1/2}}\tp{\R^2;\R^{2\times 2}_{\m{sym}}},
    \end{equation}
    and define the set of admissible solutions (corresponding to the above fixed data profile and variable forcing strength) as
    \begin{equation}\label{SSS}
        \pmb{S}=\tcb{\tp{\pmb{u},\pmb{\eta},\kappa}\in H^2_{\scriptscriptstyle1/2}\tp{\R^2;\R^2}\times\tilde{H}^3_{\scriptscriptstyle1/2}\tp{\R^2}\times\R\;:\;\inf\tp{1+\be+\pmb{\eta}}>0,\;\text{system~\eqref{stationary nondimensionalized equations} with~\eqref{percolate propane pulvarized paintbrush} is satisfied}}.
    \end{equation}
    There exists a locally analytic curve $\mathscr{S}\subseteq\pmb{S}$, with $0\in\mathscr{S}$, that is parametrized by the continuous mapping $\R\ni s\mapsto\tp{\pmb{u}(s),\pmb{\eta}(s),\kappa(s)}\in\mathscr{S}$ and satisfies $0\mapsto 0$ as well as the limits
    \begin{equation}\label{__SOLITARY__BU}
        \lim_{|s|\to\infty}\bsb{\tnorm{\pmb{\eta}\tp{s}}_{L^\infty}+\min\tcb{\tnorm{\pmb{u}(s)}_{L^q},\tnorm{\pmb{\eta}(s)}_{L^r}}+\f{1}{\inf\tp{1+\be+\pmb{\eta}(s)}}+|\kappa(s)|}=\infty
    \end{equation}
    for every choice of $q\in(4/3,2)$ and $r\in[4,\infty)$.
\end{customthm}

Note that in the stationary case we must set $G>0$ to produce solutions. Again, the limits~\eqref{__SOLITARY__BU} indicate that either solutions exist for an unbounded set of $\kappa$, or else they degenerate in a particular way.  The blowup scenarios from the periodic case remain but are augmented with a new option: the blowup of $\min\tcb{\tnorm{\pmb{u}(s)}_{L^q},\tnorm{\pmb{\eta}(s)}_{L^r}}\to\infty$ when $\limsup_{s\to\pm\infty}\tnorm{\pmb{\eta}\tp{s}}_{L^\infty}<\infty$, which indicates a sort of delocalization of both the velocity and free surface functions. We included this new option for purely technical reasons stemming from the a priori estimates developed in Proposition~\ref{prop on weighted a priori estimates version 1}; in principle, it might be possible to rule out this scenario by improving this proposition, but at the moment we lack the tools to do so.

\subsubsection{Discussion of techniques}\label{subsubsection on discussion of techniques}

We now turn to a discussion of our techniques for proving the main theorems.  At the highest level of abstraction, our general methods are the same in the periodic and solitary cases; what differs between the cases is the precise reformulation of \eqref{stationary nondimensionalized equations} and the specific functional framework. The main goal in reformulating the equations is to \emph{semilinearize} them, which we take to mean rewrite abstractly as
\begin{equation}\label{SeMiLiNeArIzAtIoN}
    P \Xi = N(\Xi) + \kappa F(\Xi)
\end{equation}
for an unknown $\Xi$, $P$ a linear differential operator, $N$ a nonlinear map acting at lower order than $P$, and $F$ given data also acting at lower order. We aim to do this in such a way that the operator $P$ is invertible, which then allows for the abstract fixed-point reformulation 
\begin{equation}\label{Xi is letter}
(I + \mathcal{K})(\Xi) = \Xi + \mathcal{K}(\Xi) = \kappa \mathcal{F}(\Xi) \text{ for } \mathcal{K} = P^{-1} \circ N \text{ and } \mathcal{F} = P^{-1}\circ F.
\end{equation}
The global implicit function theorems (GIFT) we are aware of all require some form of Fredholm condition on the derivative of the nonlinear map, and the $I + \mathcal{K}$ functional form above is convenient for verifying this by checking that $\mathcal{K}$ is a compact map. Moreover, this form lets us readily eliminate a number of the alternatives provided by the particular GIFT we use (see Theorem~\ref{thm on analytic global implicit function theorem}). The specific forms of $P$ and $N$ (and thus $\mathcal{K}$) differ in the periodic and solitary cases for a number of technical reasons, which we now discuss.

In both cases we assume that the fluid depth $1+\be+\pmb{\eta}$ is everywhere positive, which is convenient in that it permits us to arrive at systems equivalent to~\eqref{stationary nondimensionalized equations} by dividing each equation by the fluid depth. Once this is done for the momentum equation, one notices that the highest order term for $\pmb{u}$ is $\grad\cdot\mathbb{S}\pmb{u}$, while the highest order term for $\pmb{\eta}$ is $-\Delta\grad\pmb{\eta}$.  As these are both \emph{linear}, we have successfully semilinearized the momentum equation.  Unfortunately, this procedure does not so nicely semilinearize the continuity equation. The problem is that, while the top order contribution is indeed the linear term $\grad\cdot\pmb{u}$, the function space in which the equation must hold is a subset of the range of the divergence and hence is not closed under multiplication by smooth functions. Fortunately, in the periodic case there is a simple workaround for this issue:  in Lemma~\ref{lem on divergence equation reformulation} we establish that the continuity equation of~\eqref{stationary nondimensionalized equations} holds if and only if 
\begin{equation}\label{avg zero trick}
    \grad\cdot\pmb{u} + P_0\tp{\pmb{u}\cdot\grad\log\tp{1 + \be + \pmb{\eta}}}=0,
\end{equation}
where $P_0$ is the projection onto the space of mean-zero functions.  Equation~\eqref{avg zero trick} is a valid semilinear reformulation of the continuity equation that respects the target space.  Unfortunately, in the solitary case the projection operator $P_0$ is not even defined in our functional framework, so the above does not admit an analogous reformulation.

To semilinearize in the solitary case, we instead take an alternate route that does not start with dividing the equations by the depth, but rather changes the unknowns. Instead of considering the tangential velocity $\pmb{u}$ as the fundamental unknown we swap to the vector field  $u = \tp{1 + \be + \pmb{\eta}}\pmb{u}$, which is sometimes called the `discharge' in the shallow water literature.  Then the continuity equation greatly simplifies to $\grad\cdot u=0$, which we can encode into the domain function space for $u$ by restricting to the closed subspace of solenoidal fields. Upon turning our attention to the momentum equation in~\eqref{stationary nondimensionalized equations}, we see that the `depth dependent viscosity' nonlinearity in the highest order term for $\pmb{u}$, namely $\tp{1 + \be + \pmb{\eta}}\grad\cdot\mathbb{S}\pmb{u}$ crucially semilinearizes under the discharge swap to $\grad\cdot\mathbb{S}u$. The remaining issue is that the highest order term for $\pmb{\eta}$ is manifestly nonlinear: $-\tp{1 + \be + \pmb{\eta}}\Delta\grad\pmb{\eta}$. To semilinearize this capillary contribution, we now exploit a feature of the solitary case that is not reflected in the periodic setting: the function space for the free surface is a Banach algebra.  This permits us to make the \emph{nonlinear} but invertible change of unknown
\begin{equation}
    \eta = \tp{\tp{1 + \be + \pmb{\eta}}^2 - \tp{1 + \be}^2}/2
\end{equation}
and reformulate the free surface contributions in terms of the new unknown $\eta$. The upshot is that, up to lower order terms, the aforementioned capillary contribution transforms into the \emph{linear} expression $-\Delta\grad\eta$.   

The above reformulations provide the semilinear operator encoding for both cases of the system~\eqref{stationary nondimensionalized equations} as in~\eqref{SeMiLiNeArIzAtIoN}. We then select $P$ to be the entirety of the linear part so that $DN(0)=0$, meaning that $N$ is an at least quadratic and lower order nonlinearity.  By implementing Fourier multiplier and Fredholm techniques, we establish the invertibility of the operator $P$, which achieves the compact perturbation of the identity form as in~\eqref{Xi is letter}.  After selecting an appropriate functional framework, we are then in a position to read off the small data well-posedness results of Theorems~\ref{z_MAIN_THEOREM_1} and~\ref{z_MAIN_THEOREM_3}.  These follow from a direct application of the implicit function theorem, the hypotheses of which are easily verified thanks to the identity~\eqref{Xi is letter}.  Note that the small data theory does not require compactness of $\mathcal{K}$ or even that it is lower order; all we need is that $\mathcal{K}$ is at least quadratic.

We now turn our attention to the issue of how to select a functional framework appropriate for use in the abstract semilinear operator formulation.  The equations have a natural energy structure that suggests working in $L^2-$based Sobolev spaces. The problem with this is that the equations fail to provide good control of the low-Fourier-mode part of the free surface function; indeed, the best one can hope for is control over its gradient.  In the periodic setting this issue is readily dispatched with the Poincar\'e inequality by forcing $\pmb{\eta}$ to have vanishing average. On the other hand, in the solitary setting there is no such simple resolution, and the lack of control of the low-mode part of $\eta$ leads to disastrous problems: for example, the lack of $L^\infty$ bounds in the standard $L^2-$theory for $\eta$ would make the PDE coefficients unbounded or possibly degenerate, and the natural container space for $\eta$ fails to be complete, which negates the ability to use implicit function theorem techniques. We are thus forced to abandon the idea of working in $L^2\tp{\R^2}$-based spaces entirely. Instead, we turn to function spaces that provide stronger control at spatial infinity than the $L^2\tp{\R^2}$ norm. Our solution in this paper is to work in \emph{weighted} Sobolev spaces based on $L^2(\R^2:\m{m}_1)$ where $\m{m}_1$ is the absolutely continuous measure with density $\tbr{X}$. These weighted spaces still maintain a deep connection to the natural energy structure and make certain Fourier analytic techniques available. In principle, the low mode degeneracy could also be overcome by working in unweighted $L^p\tp{\R^2}$-based Sobolev spaces for $1<p<2$; however, the weighted $L^2$ spaces have additional technical advantages that we crucially exploit in the next component of our analysis, which makes them an indispensable tool for our use here.

As previously mentioned, we rely heavily on the compactness of $\mathcal{K}$ (and also $\mathcal{F}$) to construct large solutions to~\eqref{stationary nondimensionalized equations} in both the periodic and solitary cases.  In the periodic case compactness comes from the Rellich–Kondrachov theorem and the fact that $\mathcal{K}$ and $\mathcal{F}$ are lower order operators.  In the solitary case Rellich-Kondrachov compactness is unavailable due to the unboundedness of $\R^2$ and translation invariance of standard Sobolev norms.  However, weighted Sobolev spaces enjoy a weaker form of the compactness theorem that suffices for our purposes.  The full details of this result are found in Proposition~\ref{prop on embeddings of weighted Sobolev spaces},  but a rough summary is: the embedding $W^{k_1,r}_{\del_1}\tp{\R^2}\emb W^{k_0,r}_{\del_0}\tp{\R^2}$ is compact whenever the regularity and weight strength indices satisfy $k_1>k_0$ and $\del_1>\del_0$.  In other words, compactness is achieved by paying both in regularity and localization.  It turns out that the nonlinear structure of $N$ (and hence $\mathcal{K}$) actually helps make this payment in localization since the product of weighted Sobolev functions actually obeys a \emph{stronger} weighted condition.  For example, whenever $k\in\N$ and $r\in[1,\infty]$ satisfy $k>2/r$ and $\del,\rho\ge0$ the product map $W^{k,r}_{\del}\tp{\R^2}\times W^{k,r}_{\rho}\tp{\R^2}\to W^{k,r}_{\del+\rho}\tp{\R^2}$ is continuous (see Remark~\ref{remark on algebra property of weighted spaces}). Thus, the map $\mathcal{K}$, which is at least quadratic, is naturally  more localized than its argument.  Crucially, this is what makes it possible for us to invoke the modified compactness theorem and deduce that $\mathcal{K}$ is compact in the solitary setting, and the compactness of $\mathcal{F}$ follows similarly.

With the abstract form ~\eqref{Xi is letter} established with $\mathcal{K}$ and $\mathcal{F}$ compact, we are then in position to apply a GIFT.  Luckily, in our case the maps $\mathcal{K}$ and $\mathcal{F}$ turn out to be real analytic, which permits the application of a particularly strong form of the GIFT, Theorem~\ref{thm on analytic global implicit function theorem}, which is a minor modification of a GIFT proved by Chen, Walsh, and Wheeler~\cite{chen2023globalbifurcationmonotonefronts}.  This provides a maximal locally analytic curve of solutions, parametrized by the continuous map $\R\ni s\mapsto\tp{\Xi(s),\kappa(s)}$, to the fixed point reformulation~\eqref{Xi is letter} with $0\mapsto\tp{0,0}$. This curve is either a closed loop, or else for each of the limits $s\to\pm\infty$ one of the following terminal scenarios occurs: $(1)$ blowup, $(2)$ loss of compactness, or $(3)$ loss of Fredholm index zero.  The uniqueness of the trivial solution, which is essentially established via~\eqref{the power dissipation identity}, prohibits the closed loop scenario.  In turn, the compactness of $\mathcal{K}$ and $\mathcal{F}$ allows us to rule out alternatives $(2)$ and $(3)$, showing that blowup must occur.   More precisely, this means that as $s\to\pm\infty$, either the solution curve approaches the boundary of the domain of $\Xi$ or $|\kappa(s)|+\tnorm{\Xi(s)}\to\infty$.

The final thrust in completing the large solution analysis of Theorems~\ref{z_MAIN_THEOREM_2} and~\ref{z_MAIN_THEOREM_4} is transforming the solution curves $s\mapsto\tp{\Xi(s),\kappa(s)}$ to the abstract reformulation~\eqref{Xi is letter} back into solution curves $s\mapsto\tp{\pmb{u}\tp{s},\pmb{\eta}\tp{s},\kappa(s)}$ for the system~\eqref{stationary nondimensionalized equations} with~\eqref{percolate propane pulvarized paintbrush} and extracting refined information on their blowup.   The latter is accomplished with a host of a priori estimates for the full nonlinear system~\eqref{stationary nondimensionalized equations}, which allow us to conclude that blowup occurs due to growth in a coarser quantity (see~\eqref{__PERIODIC__BU} and~\eqref{__SOLITARY__BU}) than the norm used to define the function spaces.  The discrepancy in the blowup quantities between the periodic and stationary cases stems from the fact that a priori estimates in weighted Sobolev spaces are more delicate than the simple energy estimates available in the periodic setting.  

The blowup of each quantity in ~\eqref{__PERIODIC__BU} and~\eqref{__SOLITARY__BU} admits a distinct physical interpretation. The blowup $|\kappa(s)|\to\infty$ indicates that for the given forcing profile, solutions exist regardless of the forcing amplitude.  The blowup $1/\inf\tp{1 + \be + \pmb{\eta}\tp{s}}\to\infty$ means the free surface gets arbitrarily close to touching the bottom bathymetry profile.  The blowup of the remaining quantity indicates that specific integral norms (involving no derivatives) of the solution grow arbitrarily large. 

\subsubsection{Outline}

After providing our notational conventions in Section~\ref{section on conventions of notation}, the remainder of the paper is organized as follows. Section~\ref{section on analysis in weighted spaces} develops tools for using weighted Sobolev spaces (Section~\ref{appendix on weighted sobolev spaces}) and the weighted gradient Sobolev spaces (Section~\ref{appendix on weighted gradient sobolev spaces}).  Section~\ref{section on linear analysis} studies the linear operator $P$ in the semilinearization~\eqref{SeMiLiNeArIzAtIoN} for the periodic and solitary cases in tandem.  Section~\ref{section on semilinearization} works out the details of the semilinearization.  Section~\ref{section on large solutions and blow up criteria} applies the analytic global implicit function theorem (Section~\ref{subsection on analytic global implicit function theorem and solution curves}), produces large solutions, and refines the blowup scenarios in the periodic (Section~\ref{subsection on blow up criteria in the periodic case}) and solitary (Section~\ref{subsectrion on blow up criteria in the solitary case}) cases separately.   Section~\ref{section on conclusions} records the proofs of the main theorems.  The remainder of the paper is appendices.  In Appendix~\ref{appendix on tools from analysis} we record useful tools from analysis.  Appendix~\ref{appendix on derivation of SWE w/ bathymetry} contains the details of the derivation of the shallow water equations from the free boundary Navier-Stokes system over general  bathymetry.
   
\subsection{Conventions of notation}\label{section on conventions of notation}

We write $\N = \tcb{0,1,2,\dots}$ for the naturals, $\N^+ = \N\setminus\tcb{0}$, and $\R^+ = (0,\infty)$.  The notation $\mathfrak{a}\lesssim\mathfrak{b}$  indicates that there exists a constant $C\in\R^+$, depending only on the parameters mentioned in context, for which $\mathfrak{a}\le C\mathfrak{b}$; to add emphasis on the dependence of $C$ on one or more particular parameters $\mathfrak{X},\dots,\mathfrak{Y}$, we will write $\mathfrak{a}\lesssim_{\mathfrak{X},\dots,\mathfrak{Y}}\mathfrak{b}$. The expression of the equivalence of quantities $\mathfrak{a}$ and $\mathfrak{b}$ is written $\mathfrak{a}\asymp\mathfrak{b}$ and means that both $\mathfrak{a}\lesssim\mathfrak{b}$ and $\mathfrak{b}\lesssim\mathfrak{a}$ are true.  For $d\in\N^+$ the bracket is the function $\tbr{\cdot}:\C^d\to[1,\infty)$ with action given by
\begin{equation}\label{NIHON BRAK}
    \tbr{x}=\tp{1 + \tabs{x_1}^2 + \cdots + \tabs{x_d}^2}^{1/2}\text{ for }x = \tp{x_1,\dots,x_d}\in\C^d.
\end{equation}
If $\mathfrak{f}:\mathfrak{A}\to\mathfrak{B}$, we denote its image by $\m{img}\tp{\mathfrak{f}} = \tcb{\mathfrak{f}\tp{\mathfrak{a}}\in\mathfrak{B}\;:\;\mathfrak{a}\in\mathfrak{A}}$. If $E$ is a subset of a topological space, the closure of $E$ is written $\Bar{E}$.  The Fourier transform and its inverse on the space of tempered distributions $\mathscr{S}^\ast\tp{\R^d;\C}$, which are normalized to be unitary on $L^2\tp{\R^d;\C}$, are denoted by $\mathscr{F}$ and $\mathscr{F}^{-1}$, respectively. For functions $\chi:\R^d\to\C$ we write $\chi(D)$ to be the Fourier multiplication operator with symbol $\xi\mapsto\chi(\xi)$.

The $\R^d$-gradient is denoted by $\grad=\tp{\pd_1,\dots,\pd_d}$ and the divergence of a vector field $U:\R^d\to\R^d$ is $\grad\cdot U = \sum_{j=1}^d\pd_j\tp{U\cdot e_j}$. For $k\in\N\cup\tcb{\infty}$ we let $C^k(\mathfrak{A};\mathfrak{B})$ denote the set of $k$-times continuously differentiable functions on $\mathfrak{A}$ with values in $\mathfrak{B}$.

\section{Analysis in weighted spaces}\label{section on analysis in weighted spaces}

The purpose of this section is to introduce two types of weighted spaces that play a role in our analysis of system~\eqref{stationary nondimensionalized equations} in solitary function spaces.  These are the weighted Sobolev spaces and the weighted gradient spaces. The benefits of working in weighted spaces are two-fold. Firstly, weighted estimates on the gradient of the free surface overcome the low mode degeneracies present in the equations (as discussed in Section~\ref{subsubsection on discussion of techniques}) and provide global control in the supremum norm.  Secondly, the product of functions belonging to weighted spaces has better decay at infinity than either of the factors; this leads us to important compactness in the mapping properties of our nonlinearities.

We are including this concise development of material here for the readers' convenience. However, we are not asserting that any of the content in this section will be a surprise to experts in the area and hence such readers are encouraged to skip directly to the following section. 

\subsection{Weighted Sobolev spaces}\label{appendix on weighted sobolev spaces}

Let us begin with the definition of the weighted Sobolev spaces of functions on $\R^d$. Recall that our notation is that $X:\R^d\to\R^d$ denotes the identity map.
\begin{defn}[Weighted Sobolev spaces]\label{defn of weighted Sobolev spaces}
    Let $V$ be a finite dimensional real or complex vector space. Let $s\in\N$, $\del\in[0,\infty)$, and $p\in[1,\infty]$. We define
    \begin{equation}
        W^{s,p}_\del\tp{\R^d;V}=\tcb{f\in W^{s,p}\tp{\R^d;V}\;:\;\tbr{X}^\del f\in W^{s,p}\tp{\R^d;V}},\quad\tnorm{f}_{W^{s,p}_\del}=\tnorm{\tbr{X}^\del f}_{W^{s,p}}.
    \end{equation}
    When $s=0$ we shall write $L^p_\del\tp{\R^d;V}$ in place of $W^{0,p}_\del\tp{\R^d;V}$ and if $p=2$, we write $H^s_\del\tp{\R^d;V}$ in place of $W^{s,2}_\del\tp{\R^d;V}$.
\end{defn}

\begin{rmk}[Completeness]\label{rmk on completeness}
    The linear mapping $W^{s,p}_\del\tp{\R^d;V}\ni f\mapsto\tbr{X}^\del f\in W^{s,p}\tp{\R^d;V}$ is an isometric isomorphism and so the weighted Sobolev spaces inherit completeness. Moreover, by the observation that this mapping preserves the collection of smooth functions with compact support, we readily deduce density of this space in $W^{s,p}_\del\tp{\R^d;V}$ whenever $p<\infty$.
\end{rmk}

The main cases of interest in studying the weighted Sobolev spaces (for our application) are when $p\in\tcb{2,\infty}$, but we shall occasionally make use of the broader scale. Our first two results of this subsection study equivalent norms.

\begin{prop}[Equivalent norm on the weighted Sobolev spaces]\label{prop on equivalent norms in the weighted Sobolev spaces}
    Suppose that $s\in\N$, $\del\in[0,\infty)$, and $p\in[1,\infty]$. For all $f\in W^{s,p}\tp{\R^d;V}$ it holds that $f\in W^{s,p}_\del\tp{\R^d;V}$ if and only if for all $\al\in\N^d$ with $|\al|\le s$ we have $\pd^\al f\in L^p_\del\tp{\R^d;V}$; moreover, the following norm equivalence holds
    \begin{equation}\label{basic weighted norm equivalence}
        \tnorm{f}_{W^{s,p}_\del}\asymp\bp{\sum_{|\al|\le s}\tnorm{\pd^\al f}^p_{L^p_\del}}^{1/p}
    \end{equation}
    with the obvious modifications when $p=\infty$ and with implicit constants depending only on the dimension, $\del$, $p$, and $s$.  Moreover, if $\ep>0$ then we have the continuous embedding $W^{s,p}_{\del+\ep}\tp{\R^d;V} \hookrightarrow W^{s,p}_\del\tp{\R^d;V}$.
\end{prop}
\begin{proof}
    The continuous embedding $L^{p}_{\del+\ep}\tp{\R^d;V} \hookrightarrow L^{p}_\del\tp{\R^d;V}$ is trivial, which means that the second assertion follows immediately from the first.  We will now prove the first assertion.

    Suppose initially that $f\in W^{s,p}_\del\tp{\R^d;V}$.  We shall show that $\pd^\al f\in L^p_\del\tp{\R^d;V}$ and the `$\gtrsim$' bound in~\eqref{basic weighted norm equivalence} holds for all $|\al|\le s$ by employing a finite induction.  For $k\in\tcb{0,\dots,s}$ let $\mathbb{P}(k)$ denote the proposition that 
    \begin{equation}\label{induction hypothesis}
        \bp{\sum_{|\al|\le k}\tnorm{\pd^\al f}_{L^p_\del}^p}^{1/p}\lesssim\tnorm{f}_{W^{s,p}_\del}
    \end{equation}
    for an implicit constant depending only on $k$, $\del$, $p$, and the dimension.  The truth of the base case $\mathbb{P}(0)$ is straightforward since
    \begin{equation}
        \tnorm{f}_{L^p_\del} = \tnorm{\tbr{X}^\del f}_{L^p} \le \tnorm{\tbr{X}^\del f}_{W^{s,p}} = \tnorm{f}_{W^{s,p}_\del}.
    \end{equation}
    Now let us fix $k\in\tcb{0,\dots, s-1}$ and assume that $\mathbb{P}(k)$ holds.  As $s\ge k+1$, we have that
    \begin{equation}\label{[][]}
         \bp{\sum_{|\al|=k+1}\tnorm{\pd^\al\tp{\tbr{X}^\del f}}_{L^p}^p}^{1/p} \le \tnorm{f}_{W^{s,p}_\del}.
    \end{equation}
    For any $|\al|=k+1$ we have by the Leibniz rule and the triangle inequality
    \begin{equation}\label{uno reverse}
        \tnorm{\pd^\al\tp{\tbr{X}^\del f}}_{L^p}\ge\tnorm{\tbr{X}^\del\pd^\al f}_{L^p}-\sum_{\be<\al}\bn{\al}{\be}\tnorm{\pd^{\al-\be}\tp{\tbr{X}^\del}\pd^\be f}_{L^p}.
    \end{equation}
    The weight $\tbr{X}^\del$ has the property that $|\pd^\gamma\tp{\tbr{X}^\del}|\le C_\gamma\tbr{X}^{\del-|\gam|}\le C_\gamma\tbr{X}^\del$ and hence from~\eqref{uno reverse} we deduce that for some $C$ depending on $k$ and $d$
    \begin{equation}\label{[][][]}
        \tnorm{\pd^\al\tp{\tbr{X}^\del f}}_{L^p_\del}\ge\tnorm{\pd^\al f}_{L^p_\del}-C\sum_{\be<\al}\tnorm{\pd^\be f}_{L^p_\del}.
    \end{equation}
    We combine~\eqref{[][]} and~\eqref{[][][]} with the induction hypothesis to deduce 
    \begin{equation}
          c\bp{\sum_{|\al|=k+1}\tnorm{\pd^\al f}^p_{L^p_\del}}^{1/p}-C\tnorm{f}_{W^{s,p}} \le \tnorm{f}_{W^{s,p}}
    \end{equation}
    for some suitably universal $c,C\in\R^+$. This shows that $\mathbb{P}(k+1)$ is true, and so finite induction shows that $\mathbb{P}(s)$ is true, as desired.

    We now prove the opposite inequality, assuming that $f\in W^{s,p}\tp{\R^d;V}$ satisfies $\pd^\al f\in L^p_\del\tp{\R^d;V}$ for every $|\al|\le s$.  Then, thanks to the triangle inequality and the Leibniz rules we have
    \begin{equation}
        \tnorm{\tbr{X}^\del f}_{W^{s,p}}\lesssim\sum_{|\al|\le s}\sum_{\be\le\al}\tnorm{\pd^{\al-\be}\tp{\tbr{X}^\del}\pd^\be f}_{L^p}.
    \end{equation}
    Again we use the properties of the weights mentioned above to deduce $|\pd^{\al-\be}\tp{\tbr{X}^\del}|\lesssim\tbr{X}^\del$ and hence the `$\lesssim$' inequality in~\eqref{basic weighted norm equivalence} is established and so too is the inclusion $f\in W^{s,p}_\del\tp{\R^d;V}$.
\end{proof}

When $p=2$ the weighted Sobolev spaces of Definition~\ref{defn of weighted Sobolev spaces} are Hilbert and based on $L^2$. Therefore, Fourier  methods can be expected to play a significant role. Our next result gives a particularly useful choice of inner product and norm for Fourier analytic tasks. In order to state the following result, we need to recall a choice of inner product on non integer indexed Sobolev spaces.

\begin{defn}[A choice of Sobolev inner product]\label{defn of a choice of inner product}
    Let $V$ be a complex and finite dimensional inner product space and $\del\in[0,\infty)$. We define the inner product $\tbr{\cdot,\cdot}_{H^\del}^\star$ on $H^\del\tp{\R^d;V}$ via
    \begin{equation}
        \tbr{f,g}_{H^\del}^\star=\sum_{|\al|\le m}\int_{\R^d}\bp{\tbr{\pd^\al f(x),\pd^\be g(x)}_V + \sigma \int_{B(0,1)}\f{\tbr{\pd^\al f(x+h)-\pd^\al f(x),\pd^\al g(x+h)-\pd^\al g(x)}_V}{|h|^{d+2\sig}}\;\m{d}h}\;\m{d}x
    \end{equation}
    where $\del=m+\sig$ with $m\in\N$ and $\sig\in [0,1)$. For more information on Besov spaces and non-integer order Sobolev spaces (and their norms) we refer the reader to Leoni~\cite{MR3726909,MR4567945}, specifically Chapter 17 in the former and Part 2 in the latter. 
\end{defn}
\begin{rmk}[Pseudolocality]\label{rmk on pseudolocality}
    The inner product from Definition~\ref{defn of a choice of inner product} has the property that if $f,g\in H^\del\tp{\R^d;V}$ satisfy $\m{dist}\tp{\m{supp}\tp{f},\m{supp}\tp{g}} > 2$, then $\tbr{f,g}_{H^\del}^\star=0$. 
\end{rmk}

\begin{prop}[Equivalent norm and inner product on weighted square summable Sobolev spaces]\label{prop on equivalent norm on weighted square summable Sobolev spaces}
    Let $V$ be a finite dimensional complex inner product space and suppose that $s\in\N$, $\del\in[0,\infty)$. The following hold.
    \begin{enumerate}
        \item For $f\in H^s\tp{\R^d;V}$ we have that $f\in H^s_\del\tp{\R^d;V}$ if and only if for all $\al\in\N^d$ with $|\al|\le s$ we have $\mathscr{F}[\pd^\al f]\in H^\del\tp{\R^d;V}$; moreover we have the norm equivalence
        \begin{equation}\label{Fourier space norm equivalence}
            \tnorm{f}_{H^s_\del}\asymp\bp{\sum_{|\al|\le s}\tnorm{\mathscr{F}[\pd^\al f]}^2_{H^\del}}^{1/2}
        \end{equation}
        for implicit constants depending only on the dimension, $s$, and $\del$.
        \item $H^s_\del\tp{\R^d;V}$ is Hilbert and the following defines an inner product
        \begin{equation}
            \tbr{f,g}_{H^s_\del}^{\bigstar}=\sum_{|\al|\le s}\tbr{\mathscr{F}[\pd^\al f],\mathscr{F}[\pd^\al g]}^\star_{H^\del}
        \end{equation}
        where $\tbr{\cdot,\cdot}_{H^\del}^\star$ is given in Definition~\ref{defn of a choice of inner product}.
        \item If $f,g\in H^s_\del\tp{\R^d;V}$ satisfy $\m{dist}\tp{\m{supp}\mathscr{F}[f],\m{supp}\mathscr{F}[g]}>2$, then $\tbr{f,g}_{H^s_\del}^{\bigstar}=0$.
    \end{enumerate}
\end{prop}
\begin{proof}
    It suffices to justify the first item, since the second and third items will then follow by Definition~\ref{defn of a choice of inner product} and Remark~\ref{rmk on pseudolocality}. Thanks to Proposition~\ref{prop on equivalent norms in the weighted Sobolev spaces} we are assured that
    \begin{equation}
        \tnorm{f}_{H^s_\del}\asymp\bp{\sum_{|\al|\le s}\tnorm{\pd^\al f}^2_{L^2_\del}}^{1/2}.
    \end{equation}
    Now by the Fourier space characterization of Sobolev spaces, we know that
    \begin{equation}
        \mathscr{F}:L^2_\del\tp{\R^d;V}\to H^\del\tp{\R^d;V}
    \end{equation}
    is a bounded linear isomorphism and so the claimed norm equivalence in~\eqref{Fourier space norm equivalence} now follows.
\end{proof}

Next, we study products of weighted Sobolev functions.
\begin{prop}[Weighted space product estimates]\label{prop on product estimates}
    Let $\F=\R$ or $\F=\C$. Suppose that $s\in\N^+$, $\del,\rho\in[0,\infty)$, and $1\le p_0,p_1,q_0,q_1,t\le\infty$ satisfy~\eqref{exponent relationships}. Let $f_0\in\tp{L^{q_0}_\rho\cap W^{s,p_0}_\rho}\tp{\R^d;\F}$ and $f_1\in\tp{L^{q_1}_\del\cap W^{s,p_1}_\del}\tp{\R^d;\F}$. Then, the pointwise product $f_0f_1$ belongs to $W^{s,t}_{\rho+\del}\tp{\R^d;\F}$ and obeys the estimate
    \begin{equation}
        \tnorm{f_0f_1}_{W^{s,t}_{\rho+\del}}\lesssim\tnorm{f_0}_{L^{q_0}_\rho}\tnorm{f_1}_{W^{s,p_1}_\del}+\tnorm{f_0}_{W^{s,p_0}_\rho}\tnorm{f_1}_{L^{q_1}_\del}.
    \end{equation}
    for an implicit constant depending only on $d$, $s$, $p_0$, $p_1$, $q_0$, $q_1$, $t$, $\rho$, and $\del$.
\end{prop}
\begin{proof}
This is a simple application of Theorem~\ref{thm on high low product estimates} on the product $\tbr{X}^{\rho+\del}f_0f_1=\tp{\tbr{X}^\rho f_0}\tp{\tbr{X}^\del f_1}$.
\end{proof}

To conclude this subsection, we record a number of useful embedding results.

\begin{prop}[Embeddings of weighted Sobolev spaces]\label{prop on embeddings of weighted Sobolev spaces}
    Let $V$ be a finite dimensional real or complex vector space, $p_0,p_1\in[1,\infty]$, $s_0,s_1\in\N$, and $\del_0,\del_1\in[0,\infty)$ satisfy $p_1\le p_0$, $s_0\le s_1$, and $\del_0\le\del_1$.  Further suppose that 
\begin{equation}\label{numerology}
    \frac{1}{p_1} - \frac{s_1-s_0}{d} \le  \frac{1}{p_0}, \text{ and } 
    \frac{1}{p_1} - \frac{s_1-s_0}{d} < 0 \text{ if } \tp{p_0,p_1}\in\tcb{\infty}\times[1,\infty).     
\end{equation}
Then the following hold.
    \begin{enumerate}
        \item We have the continuous embedding $W^{s_1,p_1}_{\del_1}\tp{\R^d;V}\emb W^{s_0,p_0}_{\del_0}\tp{\R^d;V}$.
        \item If $\del_0<\del_1$ and the left hand inequality in~\eqref{numerology} is strict, then the embedding from the previous item is compact.
    \end{enumerate}
\end{prop}
\begin{proof}
    If $p_0 = p_1$ then the embedding $W^{s_1,p_1}_{\del_1}\tp{\R^d;V}\emb W^{s_0,p_0}_{\del_0}\tp{\R^d;V}$ is a trivial consequence of Proposition~\ref{prop on equivalent norms in the weighted Sobolev spaces}. Note that if $s_0=s_1$ and the hypotheses inequalities are satisfied then necesarily we have $p_0=p_1$. So it remains to consider the case that $p_1 < p_0$ and $s_0 < s_1$.  In this case we observe that if $Y$ is a Banach space such that $W^{s_1,p_1}\tp{\R^d;V}\emb Y$, then the injective map
    \begin{equation}\label{observation embedding}
        W^{s_1,p_1}_\del\tp{\R^d;V}\ni f\mapsto\tbr{X}^\del f\in Y
    \end{equation}
    is continuous. 
    
    The embedding from the first item then follows by combining observation~\eqref{observation embedding}, interpolation, Proposition~\ref{prop on equivalent norms in the weighted Sobolev spaces}, the Morrey inequality when $1/p_1 -(s_1-s_0)/d < 0$ and $p_0\le\infty$, the Gagliardo-Nirenberg-Sobolev inequality when $p_0<\infty$ and $1/p_1 - (s_1 - s_0)/d>0$, and the critical Sobolev embedding when $1/p_1-(s_1 - s_0)/d = 0$ and $p_0<\infty$. This completes the proof of the first item.
    
    We now turn to the proof of the second item, which requires strict inequality on the left of~\eqref{numerology}.  Note that this requires $s_0 < s_1$, as $s_0 =s_1$ in the inequality implies that $p_0 < p_1$, which is not allowed.  
        
    To prove the second item it suffices to show that the closed unit ball $B \subset W^{s_1,p_1}_{\del_1}\tp{\R^d;V}$ is totally bounded in $W^{s_0,p_0}_{\del_0}\tp{\R^d;V}$, i.e. that for every $\ep \in \R^+$ we have that 
    \begin{equation}\label{compactness}
        B \subset\bigcup_{i=1}^m B_{W^{s_0,p_0}_{\del_0}}(f_i,\ep) \text{ for some } m \in \N^+ \text{ and }  f_1,\dotsc,f_m \in W^{s_0,p_0}_{\del_0}\tp{\R^d;V}.
    \end{equation}
    To this end, fix $\ep \in \R^+$ and let $\varphi\in C^\infty_{\m{c}}\tp{B(0,2)}$ satisfy $0 \le \varphi \le 1$ and $\varphi=1$ on $\Bar{B(0,1)}$.   For $R\in[1,\infty)$ we write $\varphi_R(x)=\varphi(x/R)$.
        
    If $f\in B$  then we may use the first item to estimate
    \begin{multline}\label{tail estimates}
        \tnorm{\tp{1-\varphi_R}f}_{W^{s_0,p_0}_{\del_0}} \lesssim 
        \tnorm{\tbr{X}^{\del_0-\del_1}(1-\varphi_R)}_{W^{s_0,\infty}} \tnorm{f}_{W^{s_0,p_0}_{\del_1}}
        \\
        \lesssim 
        \tnorm{\tbr{X}^{\del_0-\del_1}(1-\varphi_R)}_{W^{s_0,\infty}} \tnorm{f}_{W^{s_1,p_1}_{\del_1}}
        \lesssim
        R^{\del_0-\del_1} \le\ep/3,
    \end{multline}
    provided that  $R\gtrsim 1+\tp{1/\ep}^{1/(\del_1-\del_0)}$.  Assume that $R$ is fixed and satisfies this condition, and define the set
    \begin{equation}
        F=\tcb{\varphi_{R} f\vert_{B(0,2R)}  \;:\;f\in B}\subset W^{s_1,p_1}_{[0]}\tp{B(0,2R);V} = \overline{C^\infty_{\m{c}}(B(0,2R);V)},
    \end{equation}
    where closure on the right is with respect to the $W^{s_1,p_1}\tp{B(0,2R);V}$ norm.  Note that the nonstandard subscript $[0]$ on the right hand side is meant to distinguish the notation for this standard space from our weighted space notation.  We have the trivial bound
    \begin{equation}
        \sup_{f \in F} \norm{f}_{W^{s_1,p_1}_{[0]}} \lesssim 1,
    \end{equation}
    and the Rellich-Kondrachov theorem assures that the embedding $W^{s_1,p_1}_{[0]}\tp{B(0,2R);V}\emb W^{s_0,p_0}_{[0]}\tp{B(0,2R);V}$ is compact, which together imply that the set $F$ is totally bounded in the latter space.  Therefore, if we let $C\in\R^+$ denote the embedding constant for $W_{[0]}^{s_0,p_0}\tp{B(0,2R);V}\emb W^{s_0,p_0}_{\del_0}\tp{\R^d;V}$ (where the subspace identification is made via zero extension, which by minor abuse of notation we do not indicate with an operator) then we have the existence of $m\in\N^+$ and functions $\tcb{g_i}_{i=1}^m\subset F$ such that
    \begin{equation}\label{almost compactness}
        F\subset \bigcup_{i=1}^m B_{W^{s_0,p_0}_{[0]}}\tp{g_i,\ep/3C} \emb \bigcup_{i=1}^m B_{W^{s_0,p_0}_{\delta_0}}\tp{g_i,\ep/3}. 
    \end{equation}
    By definition, for each $i\in\tcb{1,\dots,m}$ there exists $f_i\in B$ such that $g_i=\varphi_R f_i \vert_{B(0,2R)}$, which in turn means the zero extension of $g_i$ is $\varphi_R f_i$.  

    Now let $f \in B$.  By~\eqref{almost compactness} there exists $i\in\tcb{1,\dots,m}$ such that $\varphi_R f \in B_{W^{s_0,p_0}_{\del_0}}\tp{\varphi_R f_i,\ep/3}$. Using this and the tail estimates of~\eqref{tail estimates}, we may then bound
    \begin{equation}
        \tnorm{f-f_i}_{W^{s_0,p_0}_{\del_0}}\le\tnorm{\varphi_Rf-\varphi_{R}f_i}_{W^{s_0,p_0}_{\del_0}}+\tnorm{\tp{1-\varphi_R}f}_{W^{s_0,p_0}_{\del_0}}+\tnorm{\tp{1-\varphi_R}f_i}_{W^{s_0,p_0}_{\del_0}}<\ep/3+\ep/3+\ep/3=\ep.
    \end{equation}
    Thus, the covering~\eqref{compactness} is established, completing the proof.    
\end{proof}

\begin{rmk}[Algebra property of weighted spaces]\label{remark on algebra property of weighted spaces}
    By combining Propositions~\ref{prop on product estimates} and~\ref{prop on embeddings of weighted Sobolev spaces} we find that the pointwise product map
    \begin{equation}
        W^{s,p}_\del\tp{\R^d;\F}\times W^{s,p}_\rho\tp{\R^d;\F} \ni (f,g) \mapsto fg \in W^{s,p}_{\del+\rho}\tp{\R^d;\F}
    \end{equation}
    is continuous whenever $s>d/p$, $\del,\rho\ge0$.
\end{rmk}

\subsection{Weighted gradient Sobolev spaces}\label{appendix on weighted gradient sobolev spaces}

This subsection analyzes the functional framework for the free surface unknown in the solitary case of~\eqref{stationary nondimensionalized equations}. Note that in contrast with Section~\ref{appendix on weighted sobolev spaces}, we are specializing here to dimension $d=2$ and integrability exponent $p=2$.

\begin{defn}[Weighted gradient Sobolev spaces]\label{defn of weighted gradient sobolev spaces}
    Let $\F=\R$ or $\F=\C$. Given $s\in\N^+$ and $\del\in(0,1)$ we define
    \begin{equation}
        \tilde{H}^s_\del\tp{\R^2;\F}=\tcb{f\in L^{2/\del}\tp{\R^2;\F}\;:\;\grad f\in H^{s-1}_\del\tp{\R^2;\F^2}},\quad\tnorm{f}_{\tilde{H}^s_\del}=\tnorm{\grad f}_{H^{s-1}_\del}.
    \end{equation}
\end{defn}
Our first task is to establish completeness of the weighted gradient Sobolev spaces. The main tool that helps us here is the following weighted estimate for the Riesz potential operator.

\begin{lem}[Weighted estimates for Riesz potential operators]\label{lem weighted estimates for riesz potential operator}
    Let
    \begin{equation}
        \mathcal{I}_1f(x)=\int_{\R^2}\f{f(x-y)}{2\pi |y|}\;\m{d}y
    \end{equation}
    denote the Riesz potential operator of order $1$ and let $\del\in(0,1)$, $0\le\tilde{\del}<\del$, and $q=2(\del-\tilde{\del})^{-1}$. Then
    \begin{equation}
        \mathcal{I}_1:L^2_\del\tp{\R^2;\F}\to L^q_{\tilde{\del}}\tp{\R^2;\F}.
    \end{equation}
    is a bounded linear operator.
\end{lem}
\begin{proof}
    The claimed operator bound is a special consequence of the more general estimates of Proposition 2.5 in Duarte and Silva~\cite{MR4593125}. There the authors establish that the fractional integration operators $\mathcal{I}_\al:L^p(\R^d:v)\to L^q(\R^d:w)$ (where these are Lebesgue spaces over $\R^d$ equipped with the absolutely continuous measures $v\;\m{d}x$ and $w\;\m{d}x$ respectively) are bounded whenever $1<p\le q<\infty$, $0<\al<d$, $(v,w)\in A^\al_{p,q}$, and $v^{1/(p-1)},\;w\in A_\infty$ (with $A^\al_{p,q}$ and $A_\infty$ denoting certain classes of measures satisfying Muckenhoupt-type conditions).

    We invoke this result in the special case of $d=2$, $\al=1$, $p=2$, $q=2\tp{\del-\tilde{\del}}^{-1}$, $v=\tbr{X}^{2\del}$, and $w=\tbr{X}^{q\tilde{\del}}$ (so that $L^2(\R^2:v)=L^2_\del\tp{\R^2}$ and $L^q\tp{\R^2:w}=L^q_{\tilde{\del}}\tp{\R^d}$). We are assured that the hypothesis $(v,w)\in A^{1}_{2,q}$ is satisfied thanks to Lemma 3.3 in~\cite{MR4593125}, where Duarte and Silva characterize which inhomogeneous power type measures belong to $A^\al_{p,q}$.
\end{proof}

\begin{rmk}\label{remark on the I1 operator}
    Thanks to the computations in Chapter 4 Section B in Folland~\cite{MR1357411} we may identify the convolution operator $\mathcal{I}_1$ from Lemma~\ref{lem weighted estimates for riesz potential operator} with the Fourier multiplication operator $|\grad|^{-1}$. As a result we have the operator identity $\mathcal{I}_1\grad = \mathcal{R}$ where $\mathcal{R} = \tp{\mathcal{R}_1,\mathcal{R}_2}$ is the vector of Riesz transforms.
\end{rmk}

With the weighted estimates for Riesz potentials in hand, we now prove completeness.

\begin{prop}[Completeness of the weighted gradient spaces]\label{prop on completeness of the weighted gradient spaces}
    Let $s\in\N^+$ and $\del\in(0,1)$. The following hold.
    \begin{enumerate}
        \item We have the embedding $\tilde{H}^s_\del\tp{\R^2;\F}\emb L^{2/\del}\tp{\R^2;\F}$.
        \item $\tilde{H}^s_\del\tp{\R^2;\F}$ is complete.
    \end{enumerate}
\end{prop}
\begin{proof}
    If $f\in\tilde{H}^s_\del\tp{\R^2;\F}$, then $\tnorm{\grad f}_{L^2_\del}\lesssim\tnorm{f}_{\tilde{H}^s_\del}$. Now we apply the Riesz potential operator to $\grad f$ and use Lemma~\ref{lem weighted estimates for riesz potential operator} (with $\tilde{\del}=0$) and the fact that $\mathcal{I}_1\grad=\mathcal{R}$ (see Remark~\ref{remark on the I1 operator}). This gives $\tnorm{\mathcal{R}f}_{L^{2/\del}}\lesssim\tnorm{\grad f}_{L^2_\del}$. Now we use the identity $\mathcal{R}\cdot\mathcal{R}=-I$ and apply $\mathcal{R}\cdot$ to $\mathcal{R}f$ and use the boundedness of Riesz transforms on $L^{2/\del}\tp{\R^d;\F}$ to conclude the proof of the first item.

    Now suppose that $\tcb{f_n}_{n\in\N}\subset\tilde{H}^s_\del\tp{\R^2;\F}$ is Cauchy. Thanks to the embedding of the first item and the completeness of $L^{2/\del}\tp{\R^2;\F}$ we are assured of the existence of $f\in L^{2/\del}\tp{\R^2;\F}$ for which $f_n\to f$ as $n\to\infty$ in the norm on $L^{2/\del}\tp{\R^2;\F}$. On the other hand, thanks to the definition of the norm on $\tilde{H}^s_\del\tp{\R^2;\F}$, we also know that the sequence $\tcb{\grad f_n}_{n\in\N}\subset H^{s-1}_\del\tp{\R^2;\F^2}$ is Cauchy. Remark~\ref{rmk on completeness} establishes completeness of this weighted Sobolev space and so there exists $g\in H^{s-1}_\del\tp{\R^2;\F^2}$ such that $\grad f_n\to g$ as $n\to\infty$ in this norm. By taking the limits in the sense of distributions, we find that  $\grad f=g$ and so $f\in\tilde{H}^{s}_\del\tp{\R^2;\F}$. Finally, we conclude by noting that $\tnorm{f-f_n}_{\tilde{H}^s_\del}=\tnorm{g-\grad f_n}_{H^{s-1}_\del}\to0$ as $n\to\infty$.
\end{proof}

Our next tasks are to develop an important equivalent norm on the weighted gradient spaces and analyze refined low mode embeddings.
\begin{prop}[Equivalent norms and low mode embeddings]\label{prop on equivalent norms and low mode embeddings}
    Let $\F=\R$ or $\F=\C$, $s\in\N^+$, $\del\in(0,1)$, and $\varphi\in C^\infty_{\m{c}}\tp{B(0,2)}$ be a radial function satisfying $\varphi=1$ on $\Bar{B(0,1)}$. The following hold for $f\in\tilde{H}^s_\del\tp{\R^2;\F}$.
    \begin{enumerate}
        \item $\grad\varphi(D)f\in L^2_\del\tp{\R^2;\F^2}$ and $(1-\varphi(D))f\in H^s_\del\tp{\R^2;\F}$; moreover, we have the norm equivalence
        \begin{equation}\label{the low high equivalent norm}
            \tnorm{f}_{\tilde{H}^s_\del}\asymp\tp{\tnorm{\varphi(D)\grad f}_{L^2_\del}^2+\tnorm{(1-\varphi(D))f}^2_{H^s_\del}}^{1/2}
        \end{equation}
        for implicit constants depending only on $s$, $\del$, and $\varphi$.
        \item Given $\tilde{\del}\in[0,\del)$ and $q=2\tp{\del-\tilde{\del}}^{-1}$ we have that $\varphi(D)f\in W^{\infty,q}_{\tilde{\del}}\tp{\R^2;\F}$; moreover, for any $t\in\N$ we have the estimate
        \begin{equation}\label{refined information on the low modes}
            \tnorm{\varphi(D)f}_{W^{t,q}_{\tilde{\del}}}\lesssim\tnorm{f}_{\tilde{H}^1_\del},
        \end{equation}
        for an implicit constant depending only on $\del$, $\tilde{\del}$, $t$, and $\varphi$.
    \end{enumerate}
\end{prop}
\begin{proof}
    Let us begin by proving `$\gtrsim$' in~\eqref{the low high equivalent norm}. For $j\in\tcb{1,2}$ we define the multipliers
    \begin{equation}
        m_j(\xi)=\tp{1-\varphi(\xi)}\f{\xi_j}{2\pi\ii|\xi|^2}\quad\text{which satisfy}\quad\sum_{j=1}^2m_j(D)\pd_j=(1-\varphi(D)).
    \end{equation}
    We use Proposition~\ref{prop on improved multiplier bounds}, the fact that As $[m_j]_{1,-1}<\infty$ (see~\eqref{the hypotheses on the multiplier}), and Remark~\ref{rmk on gen = reg on natty}) to deduce the estimate
    \begin{equation}
        \tnorm{(1-\varphi(D))f}_{H^s_\del}\le\sum_{j=1}^2\tnorm{m_j(D)\pd_jf}_{H^s_\del}\lesssim\tnorm{\grad f}_{H^{s-1}_\del}=\tnorm{f}_{\tilde{H}^{s}_\del}.
    \end{equation}
    On the other hand, the low mode estimate is even easier, since we need only use the first item of Proposition~\ref{prop on a multiplier bound an weighted Bernstein inequalities} followed by Proposition~\ref{prop on embeddings of weighted Sobolev spaces} to deduce that
    \begin{equation}
        \tnorm{\varphi(D)\grad f}_{L^2_\del}\lesssim\tnorm{\grad f}_{L^2_\del}\lesssim\tnorm{\grad f}_{H^{s-1}_\del}=\tnorm{f}_{\tilde{H}^s_\del}.
    \end{equation}

    Next, we turn our attention to the `$\lesssim$' inequality of~\eqref{the low high equivalent norm}. We begin with the triangle inequality
    \begin{equation}\label{{}-{}}
        \tnorm{f}_{\tilde{H}^s_\del}\le\tnorm{\grad\varphi(D)f}_{H^{s-1}_\del}+\tnorm{(1-\varphi(D))\grad f}_{H^{s-1}_\del}.
    \end{equation}
    For the low mode piece above we need only look to the second item of Proposition~\ref{appendix on harmonic analysis in weighted square summable Sobolev spaces}, which permits the estimate $\tnorm{\grad\varphi(D)f}_{H^{s-1}_\del}\lesssim\tnorm{\grad\varphi(D)f}_{L^2_\del}$. For the high mode piece in~\eqref{{}-{}} we instead quote the boundedness of $\grad: H^s_\del\tp{\R^2;\F}\to H^{s-1}_\del\tp{\R^2;\F^2}$ from Proposition~\ref{prop on improved multiplier bounds} and Remark~\ref{rmk on gen = reg on natty}. This completes the proof of the first item.

    We now prove the second item. We first establish the estimate~\eqref{refined information on the low modes} with $t=0$. Combining the weighted bounds for Riesz potential operators from Lemma~\ref{lem weighted estimates for riesz potential operator} with the first item, we find that
    \begin{equation}\label{:::}
        \tnorm{\mathcal{I}_1\grad\varphi(D)f}_{L^q_{\tilde{\del}}}\lesssim\tnorm{\grad\varphi(D)f}_{L^2_\del}\lesssim\tnorm{f}_{\tilde{H}^1_\del}.
    \end{equation}
    Now, by arguing as in the proof of Proposition~\ref{prop on completeness of the weighted gradient spaces} we know that $\mathcal{I}_1\grad=\mathcal{R}$. We again would like to use $\mathcal{R}\cdot\mathcal{R}+I=0$ to then relate the above estimate to one on $\varphi(D)f$ itself. To achieve this we require an auxiliary result on the boundedness of the Riesz transform operators on certain weighted Lebesgue spaces.  Namely, we claim that
    \begin{equation}\label{Reisz potential boundedness}
        \mathcal{R}_j:L^q_{\tilde{\del}}\tp{\R^2;\F}\to L^q_{\tilde{\del}}\tp{\R^2;\F} \text{ is a bounded linear operator for } j\in\tcb{1,2}.
    \end{equation}
    To establish the claim we first note that $L^q_{\tilde{\del}}\tp{\R^2}=L^q(\R^2:\tbr{X}^{q\tilde{\del}})$, where the space on the right is the Lebesgue space equipped with the measure $\tbr{X}^{q\tilde{\del}}\;\m{d}x$.  As a consequence of the first corollary in Section 4.2 in Chapter 5 of Stein~\cite{MR1232192}, we have that~\eqref{Reisz potential boundedness} holds provided that $\tbr{X}^{q\tilde{\del}}\in A_q$, where $A_q$ denotes the space of Muckenhoupt measures. Thanks to the first equation in Section 3.2 in Duarte and Silva~\cite{MR4593125}, we deduce that this condition holds as soon as $\tilde{\del}<2(1-1/q)$, but this latter condition is indeed satisfied since $\tilde{\del}<1$ and $q>2$ by hypothesis. Therefore, the claim~\eqref{Reisz potential boundedness} holds, and together with~\eqref{:::} provides the estimate
    \begin{equation}\label{case s = 0 bound}
        \tnorm{\varphi(D)f}_{L^q_{\tilde{\del}}}=\tnorm{\mathcal{R}\cdot\mathcal{R}\varphi(D)f}_{L^q_{\tilde{\del}}}\lesssim\tnorm{\mathcal{I}_1\grad\varphi(D)f}_{L^q_{\tilde{\del}}}\lesssim\tnorm{f}_{\tilde{H}^1_\del}.
    \end{equation}

     We now establish estimate~\eqref{refined information on the low modes} for $t>0$. Let $\tilde{\varphi}\in C^\infty_{\m{c}}(B(0,4))$ be another radial function such that $\tilde{\varphi}=1$ on $\Bar{B(0,2)}$. Notice that $\varphi(D)f=\tilde{\varphi}(D)\varphi(D)f=K\ast\tp{\varphi(D)f}$ where $K=\mathscr{F}^{-1}[\tilde{\varphi}]$. Since $\tilde{\varphi}$ is a Schwartz function, we have that $K\in\mathscr{S}\tp{\R^2;\F}$ (and all of its derivatives are Schwartz as well).  We then deduce that $\varphi(D)f$ is smooth and for all $\al\in\N^2$ we have $\pd^\al\varphi(D)f=(\pd^\al K)\ast\tp{\varphi(D)f}$. For any $\psi\in\mathscr{S}\tp{\R^d;\F}$ the linear map $L^q_{\tilde{\del}}\tp{\R^2;\F}\ni g\mapsto g\ast\psi\in L^q_{\tilde{\del}}\tp{\R^2;\F}$ is easily seen to be continuous, and so we deduce (using also~\eqref{case s = 0 bound}) that $\tnorm{\pd^\al\varphi(D)f}_{L^q_{\tilde{\del}}}\lesssim\tnorm{\varphi(D)f}_{L^q_{\tilde{\del}}}\lesssim\tnorm{f}_{\tilde{H}^1_\del}$ for an implicit constant depending on $K$, $\del$, $\tilde{\del}$, $q$, and $\al$. Estimate~\eqref{refined information on the low modes} now follows after heeding to Proposition~\ref{prop on equivalent norms in the weighted Sobolev spaces}.
\end{proof}

\begin{rmk}[Embedding of the weighted Sobolev spaces]\label{rmk on embedding of the weighted Sobolev spaces}
    By using the equivalent norm from the first item of Proposition~\ref{prop on equivalent norms and low mode embeddings} and the multiplier estimate from the first item of Proposition~\ref{prop on a multiplier bound an weighted Bernstein inequalities} it is straightforward to deduce the embedding $H^s_\del\tp{\R^2;\F}\emb\tilde{H}^s_\del\tp{\R^2;\F}$.
\end{rmk}

We now provide a weighted gradient space analog of Proposition~\ref{prop on embeddings of weighted Sobolev spaces}.
\begin{prop}[Inclusion relations]\label{prop on inclusion relations}
    Let $s,s_0,s_1\in\N^+$ and $\del,\del_0,\del_1\in\tp{0,1}$. The following hold.
    \begin{enumerate}
        \item If $s_0\le s_1$ and $\del_0\le\del_1$ then $\tilde{H}^{s_1}_{\del_1}\tp{\R^2;\F}\emb\tilde{H}^{s_0}_{\del_0}\tp{\R^2;\F}$.
        \item If $s_0<s_1$ and $\del_0<\del_1$ then the embedding from the first item is compact.
        \item If $s\ge 2$ and $r\in[2,\infty)$ the embedding $\tilde{H}^s_\del\tp{\R^2;\F}\emb\tp{C^0\cap L^\infty\cap\dot{W}^{1,r}}\tp{\R^2;\F}$ is compact.
    \end{enumerate}
\end{prop}
\begin{proof}
    The first item follows from the inclusions established in Proposition \ref{prop on equivalent norms in the weighted Sobolev spaces}.  We now prove the second item.  Suppose that $\tcb{f_n}_{n\in\N}\subset\tilde{H}^{s_1}_{\del_1}\tp{\R^2;\F}$ is a bounded sequence. Thanks to the definition of the norm on this space we have that the sequence of gradients $\tcb{\grad f_n}_{n\in\N}\subset H^{s_1-1}_{\del_1}\tp{\R^2;\F^2}$ is also bounded; moreover, from the first item and Proposition~\ref{prop on completeness of the weighted gradient spaces} we have that $\tcb{f_n}_{n\in\N}\subset L^{2/\del_0}\tp{\R^2;\F}$ is also a bounded sequence. Due to the reflexivity of this Lebesgue space and Proposition~\ref{prop on embeddings of weighted Sobolev spaces} we are assured the existence of $f\in L^{2/\del_0}\tp{\R^2;\F}$ and $g\in H^{s_0-1}_{\del_0}\tp{\R^2;\F}$ with the property that (along a subsequence we neglect to relabel)
    \begin{equation}\label{___do you hear the street sweeper it is a very nice street sweeper___}
        f_n\rightharpoonup f\;\text{in }L^{2/\del_0}\tp{\R^2;\F}\quad\text{and}\quad\grad f_n\to g\;\text{in }H^{s_0-1}_{\del_0}\tp{\R^2;\F^2}
    \end{equation}
    as $n\to\infty$, where the former convergence is weak and the latter is strong. By considering distributional limits, we deduce that $g=\grad f$ and so $f\in\tilde{H}^{s_0}_{\del_0}\tp{\R^2;\F}$. The strong convergence in~\eqref{___do you hear the street sweeper it is a very nice street sweeper___} paired with the definition of the norm on $\tilde{H}^{s_0}_{\del_0}\tp{\R^2;\F}$ implies that $f_n\to f$ strongly in this space as $n\to\infty$.

    We now prove the third item. Let $\tcb{f_n}_{n\in\N}\subset\tilde{H}^s_\del\tp{\R^2;\F}$ be a bounded sequence. We split into a low mode and low mode compliment sequence with $\varphi$ as in Proposition~\ref{prop on equivalent norms and low mode embeddings}. By this result (and also the second item of Proposition~\ref{prop on a multiplier bound an weighted Bernstein inequalities}) we find then that the sequences $\tcb{\varphi(D)f_n}_{n\in\N}\subset W^{1,\infty}_{\del/2}\tp{\R^2;\F}$, $\tcb{\grad\varphi(D)f_n}_{n\in\N}\subset H^{2}_\del\tp{\R^2;\F^2}$, and $\tcb{\tp{1-\varphi(D)}f_{n}}_{n\in\N}\subset H^2_\del\tp{\R^2;\F}$ are all bounded.  We have that the embeddings $W^{1,\infty}_{\del/2}\tp{\R^2;\F}\emb \tp{C^0\cap L^\infty}\tp{\R^2;\F}$, $H^2_\del\tp{\R^2;\F}\emb \tp{C^0\cap L^\infty\cap W^{1,r}}\tp{\R^2;\F}$ are all compact by Proposition~\ref{prop on embeddings of weighted Sobolev spaces}; therefore, we can extract a (non-labeled) subsequence with $\tcb{\varphi(D)f_n}_{n\in\N}$ Cauchy in $\tp{C^0\cap L^\infty\cap\dot{W}^{1,r}}\tp{\R^2;\F}$ and $\tcb{\tp{1-\varphi\tp{D}}f_n}_{n\in\N}$ Cauchy in $\tp{C^0\cap L^\infty\cap W^{1,r}}\tp{\R^2;\F}$. By summing and passing to the limit, we obtain the sought after compact embedding.
    \end{proof}

We next record an important density result.

\begin{prop}[Density of functions of bounded support]\label{prop on density of functions of bounded support}
    For $\F=\R$ or $\F=\C$, $s\in\N^+$, and $\del\in\tp{0,1}$ we have that $C^\infty_{\m{c}}\tp{\R^2;\F}\subset\tilde{H}^s_\del\tp{\R^2;\F}$ is dense.
\end{prop}
\begin{proof}
    Fix $\varphi\in C^\infty_{\m{c}}\tp{\R^2}$ radial with $\m{supp}\tp{\varphi}\subset B(0,2)$ and $\varphi=1$ on $\Bar{B(0,1)}$. We begin by arguing that the subspace of functions of bounded support is dense. To this end, fix $f\in\tilde{H}^s_\del\tp{\R^2;\F}$. For any $\psi\in C^\infty_{\m{c}}\tp{\R^2;\F}$ we may use Propositions~\ref{prop on product estimates} and~\ref{prop on equivalent norms and low mode embeddings} to deduce that the pointwise product $\psi f$ belongs to $\tilde{H}^s_\del\tp{\R^2;\F}$, as
    \begin{multline}
        \tnorm{\psi f}_{\tilde{H}^s_{\del}}\le\tnorm{\psi\grad f}_{H^{s-1}_{\del}}+\tnorm{\varphi(D)f\grad\psi}_{H^{s-1}_\del}+\tnorm{\tp{1-\varphi(D)}f\grad\psi}_{H^{s-1}_\del}
        \lesssim\tnorm{\psi}_{W^{s-1,\infty}}\tnorm{\grad f}_{H^{s-1}_\del}\\+\tnorm{\varphi(D)f}_{W^{s-1,2/\del}}\tnorm{\grad\psi}_{W^{s-1,2/(1-\del)}_\del}+\tnorm{\tp{1-\varphi(D)}f}_{H^{s-1}_\del}\tnorm{\grad\psi}_{W^{s-1,\infty}}\lesssim_\psi\tnorm{f}_{\tilde{H}^{s}_\del}.
    \end{multline}
    Now, given $R\in[1,\infty)$ we take $\psi=\varphi_R$, where $\varphi_R=\varphi(x/R)$. The above shows that $\varphi_Rf\in\tilde{H}^s_{\del}\tp{\R^2;\F}$ and is a function of bounded support for every $R$.  Now we shall study the convergence of $\varphi_Rf$ to the function $f$ as $R\to\infty$. Initially, we have
    \begin{equation}
        \tnorm{f-\varphi_R f}_{\tilde{H}^s_\del}\le\tnorm{\tp{1-\varphi_R}\grad f}_{H^{s-1}_\del}+\tnorm{\varphi(D)f\grad\varphi_R}_{H^{s-1}_\del}+\tnorm{\tp{1-\varphi(D)}f\grad\varphi_R}_{H^{s-1}_\del}.
    \end{equation}
    From the support conditions on $\varphi$, we have $\grad\varphi_R=(1-\varphi_{R/2})\grad\varphi_R$ and so we can give this additional factor to the $f$ contributions in the penultimate term above. We then use the product estimates of Proposition~\ref{prop on product estimates} with the norm equivalence of the first item of Proposition~\ref{prop on equivalent norms and low mode embeddings} and the bounds
    \begin{equation}
        \tnorm{\grad\varphi_{R}}_{W^{s-1,2/(1-\del)}_\del}\lesssim_\varphi 1 
        \text{ and  }
        \tnorm{\grad\varphi_R}_{W^{s-1,\infty}}\lesssim_\varphi R^{-1}
    \end{equation}
    to deduce that
    \begin{equation}
        \tnorm{f-\varphi_Rf}_{\tilde{H}^s_\del}\lesssim\tnorm{\tp{1-\varphi_R}\grad f}_{H^{s-1}_\del}+\tnorm{\tp{1-\varphi_{R/2}}\varphi(D)f}_{W^{s-1,2/\del}}+R^{-1}\tnorm{f}_{\tilde{H}^s_\del}.
    \end{equation}
    Now we have an expression in which the vanishing limit as $R\to\infty$ is clear, since $f\in\tilde{H}^{s}_\del\tp{\R^2;\F}$ implies that $\grad f\in H^{s-1}_\del\tp{\R^2;\F^2}$ and (by the second item of Proposition~\ref{prop on equivalent norms and low mode embeddings}) $\varphi(D)f\in W^{s-1,2/\del}\tp{\R^2;\F}$. The density of functions of bounded support now follows.

    We conclude by showing the density of smooth and compactly supported functions. Given any $f\in\tilde{H}^s_\del\tp{\R^2;\F}$ and an $\ep\in\R^+$, the above argument grants us the existence of $g\in\tilde{H}^s_\del\tp{\R^2;\F}$ such that $g$ has bounded support and $\tnorm{f-g}_{\tilde{H}^s_\del}<\ep/2$. For any $\mu>0$ we then let $g_\mu=g\ast\upvarphi_{\mu}$, with $\upvarphi_\mu(x)=\tp{\mu^2\int_{\R^2}\varphi}^{-1}\varphi(x/\mu)$ denoting the standard sequence of mollifiers generated by $\varphi$. Since $g\in L^{2/\del}(\R^2;\F)$ by the first item of Proposition~\ref{prop on completeness of the weighted gradient spaces}, we know that $g_\mu\in L^{2/\del}\tp{\R^2;\F}$ as well.  Moreover, the compactness of the supports of $\varphi$ and $g$ implies that $g_\mu$ also has compact support.  Finally, since $\grad g_\mu=\tp{\grad g}\ast\upvarphi_\mu$ and $\grad g\in H^{s-1}_\del\tp{\R^2;\F^2}$, standard properties of mollification operators imply that $\grad g_\mu\to\grad g$ in $H^{s-1}_\del\tp{\R^2;\F^2}$ as $\mu\to0$. In particular, $g_\mu\in\tilde{H}^s_\del\tp{\R^2;\F}$ and there exists $\upmu>0$ such that $\tnorm{g-g_{\upmu}}_{\tilde{H}^{s}_\del}<\ep/2$. It then follows that $\tnorm{f-g_\upmu}_{\tilde{H}^{s}_\del}<\ep$.
\end{proof}

The remainder of this subsection is devoted to nonlinear analysis in the weighted gradient spaces.

\begin{prop}[Product estimates in weighted gradient spaces]\label{prop on product estimates in weighted gradient spaces}
    Let $\F=\R$ or $\F=\C$, $\N\ni s\ge 2$, and $\del\in(0,1)$. The following hold.
    \begin{enumerate}
        \item Given any $\tilde{\del}\in[0,\del)$ we have the embedding $\tilde{H}^s_\del\tp{\R^2;\F}\emb L^\infty_{\tilde{\del}}\tp{\R^2;\F}$.
        \item $\tilde{H}^s_\del\tp{\R^2;\F}$ is a Banach algebra under pointwise multiplication. In fact, for any $\rho\in\tp{0,1}$ and $\max\tcb{\del,\rho}\le\mu<\min\tcb{1,\rho+\del}$ the pointwise product map
        \begin{equation}
            \tilde{H}^s_\del\tp{\R^2;\F}\times\tilde{H}^s_\rho\tp{\R^2;\F}\ni\tp{f,g}\mapsto fg\in\tilde{H}^s_\mu\tp{\R^2;\F}
        \end{equation}
        is continuous.
        \item Given any $\tilde{\del}\in[0,\del)$ and $\rho\in[0,\infty)$ the pointwise product map
        \begin{equation}\label{a wonderful product map}
            \tilde{H}^s_\del\tp{\R^2;\F}\times H^s_\rho\tp{\R^2;\F}\ni(f,g)\mapsto fg\in H^{s}_{\rho+\tilde{\del}}\tp{\R^2;\F}
        \end{equation}
        is continuous.
    \end{enumerate}
\end{prop}
\begin{proof}
    Let $\varphi\in C^\infty_{\m{c}}\tp{B(0,2)}$ be a radial function with $\varphi=1$ on $\Bar{B(0,1)}$. The first item follows from Propositions~\ref{prop on embeddings of weighted Sobolev spaces} and~\ref{prop on equivalent norms and low mode embeddings}:
    \begin{equation}
        \tnorm{f}_{L^\infty_{\tilde{\del}}}\le\tnorm{\varphi(D)f}_{L^\infty_{\tilde{\del}}}+\tnorm{\tp{1-\varphi(D)}f}_{L^\infty_{\tilde{\del}}}\lesssim\tnorm{\varphi(D)f}_{W^{2,q}_{\tilde{\del}}}+\tnorm{(1-\varphi(D))f}_{H^2_\del}\lesssim\tnorm{f}_{\tilde{H}^2_\del},
    \end{equation}
    where we have taken $q=2\tp{\del-\tilde{\del}}^{-1}$. This gives the first item.

    Next, we shall prove the third item - so fix $f\in\tilde{H}^s_\del\tp{\R^2;\F}$ and $g\in H^s_\rho\tp{\R^2;\F}$. We split $f$ into low and high modes and invoke Proposition~\ref{prop on product estimates} (and Remark~\ref{remark on algebra property of weighted spaces}) to bound
    \begin{equation}\label{combo_0}
        \tnorm{fg}_{H^s_{\rho+\tilde{\del}}}\lesssim\tnorm{\varphi(D)f}_{W^{s,\infty}_{\tilde{\del}}}\tnorm{g}_{H^s_\rho}+\tnorm{\tp{1-\varphi(D)}f}_{H^s_{\tilde{\del}}}\tnorm{g}_{H^s_\rho}.
    \end{equation}
    Then, by Propositions~\ref{prop on equivalent norms and low mode embeddings} and~\ref{prop on embeddings of weighted Sobolev spaces} we have
    \begin{equation}\label{combo_1}
        \tnorm{\varphi(D)f}_{W^{s,\infty}_{\tilde{\del}}}\lesssim\tnorm{f}_{\tilde{H}^1_\del},\quad\tnorm{(1-\varphi(D))f}_{H^s_{\tilde{\del}}}\lesssim\tnorm{f}_{\tilde{H}^s_\del}.
    \end{equation}
    Synthesize~\eqref{combo_0} and~\eqref{combo_1} to get~\eqref{a wonderful product map}.

    Finally, we prove the second item. We shall frequently use Remark~\ref{rmk on embedding of the weighted Sobolev spaces}. For $f\in\tilde{H}^s_\del\tp{\R^2;\F}$ and $g\in\tilde{H}^s_\rho\tp{\R^2;\F}$ it holds
    \begin{equation}\label{._.}
        \tnorm{fg}_{\tilde{H}^s_\mu}\lesssim\tnorm{\varphi(D)fg}_{\tilde{H}^s_\mu}+\tnorm{\tp{1-\varphi(D)}fg}_{H^s_\mu}.
    \end{equation}
    The final term in~\eqref{._.} can be given a suitable upper bound by using the continuity of~\eqref{a wonderful product map} combined with Proposition~\ref{prop on embeddings of weighted Sobolev spaces} and the first item of Proposition~\ref{prop on equivalent norms and low mode embeddings}:
    \begin{equation}\label{UWU}
        \tnorm{(1-\varphi(D))fg}_{H^s_\mu}\lesssim\tnorm{\tp{1-\varphi(D)}f}_{H^s_\del}\tnorm{g}_{\tilde{H}^s_\rho}\lesssim\tnorm{f}_{\tilde{H}^s_\del}\tnorm{g}_{\tilde{H}^s_\rho}.
    \end{equation}
    On the other hand, for the penultimate term of~\eqref{._.} we split $g$ into high and low modes and apply the same argument as~\eqref{UWU} to the high mode piece
    \begin{equation}
        \tnorm{\varphi(D)fg}_{\tilde{H}^s_\mu}\lesssim\tnorm{\varphi(D)f\varphi(D)g}_{\tilde{H}^s_\mu}+\tnorm{f}_{\tilde{H}^s_\del}\tnorm{g}_{\tilde{H}^s_\rho}.
    \end{equation}
    We are thus left considering the product of the low mode pieces. Thanks to the second item of Proposition~\ref{prop on a multiplier bound an weighted Bernstein inequalities}, the fact that $\m{supp}\mathscr{F}[\varphi(D)f\varphi(D)g]\subset B(0,4)$, H\"older's inequality, and the embedding from the first item we have
    \begin{multline}
        \tnorm{\varphi(D)f\varphi(D)g}_{\tilde{H}^s_\mu}\lesssim\tnorm{\varphi(D)f\grad\varphi(D)g}_{L^2_\mu}+\tnorm{\grad\varphi(D)f\varphi(D)g}_{L^2_\mu}\\\lesssim\tnorm{\varphi(D)f,\varphi(D)g}_{L^\infty_{\tilde{\del}}\times L^\infty_{\tilde{\rho}}}\tnorm{\varphi(D)f,\varphi(D)g}_{\tilde{H}^1_\del\times\tilde{H}^1_\rho}\lesssim\tnorm{f}_{\tilde{H}^s_\del}\tnorm{g}_{\tilde{H}^s_\rho}
    \end{multline}
    where $\tilde{\del}\in(0,\del)$ and $\tilde{\rho}\in(0,\rho)$ satisfy $\mu\le\min\tcb{\tilde{\del}+\rho,\del+\tilde{\rho}}$. This completes the proof.
\end{proof}

For technical convenience when dealing with analytic composition, we would like to make a minimal addition to the weighted gradient Sobolev spaces in order to obtain a unital Banach algebra. Therefore, we introduce the following variation on Definition~\ref{defn of weighted gradient sobolev spaces}.

\begin{defn}[Extended weighted gradient Sobolev spaces]\label{defn on extended weighted gradient spaces}
Given $\F=\R$ or $\F=\C$, $s\in\N^+$, and $\del\in\tp{0,1}$ we define
\begin{equation}
    {_{\m{e}}}\tilde{H}^s_\del\tp{\R^2;\F}=\tcb{f\in\F+L^{2/\del}\tp{\R^2;\F}\;:\;\grad f\in H^{s-1}_\del\tp{\R^2;\F^2}},\quad\tnorm{f}_{{_{\m{e}}}\tilde{H}^s_\del}=\tp{|\bf{c}(f)|^2+\tnorm{\grad f}_{H^{s-1}_\del}^2}^{1/2}
\end{equation}
where $\bf{c}:\F+L^{2/\del}\tp{\R^2;\F}\to\F$ is the unique function satisfying $f-\bf{c}(f)\in L^{2/\del}\tp{\R^2;\F}$ for all $f\in\F+L^{2/\del}\tp{\R^2;\F}$. When $\F=\R$ we shall write ${_{\m{e}}}\tilde{H}^s_\del\tp{\R^2}$ in place of ${_{\m{e}}}\tilde{H}^s_\del\tp{\R^2;\R}$.
\end{defn}

Basic properties of the extended weighted gradient spaces are enumerated in the next result.

\begin{prop}[Properties of extended weighted gradient spaces]\label{prop on properties of extended weighted gradient spaces}
    Let $\F=\R$ or $\F=\C$, $s\in\N^+$ and $\del\in(0,1)$. The following hold.
    \begin{enumerate}
        \item The map
        \begin{equation}
            {_{\m{e}}}\tilde{H}^s_\del\tp{\R^2;\F}\ni f\mapsto\tp{\bf{c}(f),f-\bf{c}(f)}\in\F\times\tilde{H}^s_\del\tp{\R^2;\F}
        \end{equation}
        is a linear isometric isomorphism.
        \item ${_{\m{e}}}\tilde{H}^s_\del\tp{\R^2;\F}$ is a Banach space and $\tilde{H}^s_{\del}\tp{\R^2;\F}\emb{_{\m{e}}}\tilde{H}^s_\del\tp{\R^2;\F}$ as a codimension 1 subspace.
        \item If $s\ge 2$, then ${_{\m{e}}}\tilde{H}^s_\del\tp{\R^2;\F}\emb \tp{C^0\cap L^\infty}\tp{\R^2;\F}$.
        \item If $s\ge 2$, then ${_{\m{e}}}\tilde{H}^s_\del\tp{\R^2;\F}$ is a unital Banach algebra where the constant function $1$ is the unit; moreover for all $f,g\in{_{\m{e}}}\tilde{H}^s_\del\tp{\R^2;\F}$ we have $\bf{c}\tp{fg}=\bf{c}\tp{f}\bf{c}\tp{g}$.
    \end{enumerate}
\end{prop}
\begin{proof}
    The first and second items are trivial given Definition~\ref{defn of weighted gradient sobolev spaces} and Proposition~\ref{prop on completeness of the weighted gradient spaces}. Similarly, the third and fourth items follow immediately from Proposition~\ref{prop on product estimates in weighted gradient spaces}.
\end{proof}

The final result of this subsection computes the spectrum of elements of the extended weighted gradient spaces and records mapping properties of the resolvent.

\begin{prop}[Admissibility of the extended weighted gradient spaces]\label{prop on admissability of the extended weighted gradient spaces}
    Let $\N\ni s\ge 2$, $\del\in\tp{0,1}$. The unital Banach algebra ${_{\m{e}}}\tilde{H}^s_\del\tp{\R^2;\C}$ is admissible in the sense of Definition~\ref{defn of admissable Banach algebras}.
\end{prop}
\begin{proof}
    The third and fourth items of Proposition~\ref{prop on properties of extended weighted gradient spaces} show that the first item of Definition~\ref{defn of admissable Banach algebras} is satisfied. Now it is elementary using these embeddings to deduce that for all $f\in{_{\m{e}}}\tilde{H}^s_\del\tp{\R^2;\C}$ the inclusion $\C\setminus\m{spec}(f)\subseteq\C\setminus\Bar{\m{img}(f)}$ we omit the details for the sake of brevity. 
    
    We shall fix $f\in{_{\m{e}}}\tilde{H}^s_\del\tp{\R^2;\C}$ and $z\in\C\setminus\Bar{\m{img}\tp{f}}$ and prove (via induction on $s$) that the pointwise inverse $(z-f)^{-1}$ belongs to $\tilde{H}^s_\del\tp{\R^2;\C}$ and there exists a locally bounded function $\bf{C}_s:[0,\infty)^2\to\R^+$ increasing in both arguments such that we may estimate
    \begin{equation}\label{IH resolvent estimate}
        \tnorm{\tp{z-f}^{-1}}_{{_{\m{e}}}\tilde{H}^s_\del}\le\bf{C}_s\tp{\m{dist}_{\C}\tp{\Bar{\m{img}\tp{f}},z}^{-1},\tnorm{f}_{{_{\m{e}}}\tilde{H}^s_\del}}.
    \end{equation}
    The base case is $s=2$. Set $\ep=\m{dist}_{\C}\tp{\Bar{\m{img}\tp{f}},z}\in\R^+$. Firstly we calculate that $\tnorm{\tp{z-f}^{-1}}_{L^\infty}\le\ep^{-1}$ and $\tnorm{(z-f)^{-1}-(z-\bf{c}\tp{f})^{-1}}_{L^{2/\del}}\le\ep^{-2}\tnorm{f-\bf{c}(f)}_{L^{2/\del}}\lesssim\ep^{-2}\tnorm{f-\bf{c}\tp{f}}_{\tilde{H}^2_\del}$ (thanks to the first item of Proposition~\ref{prop on completeness of the weighted gradient spaces}). So to prove that $(z-f)^{-1}\in{_{\m{e}}}\tilde{H}^2_\del\tp{\R^2;\C}$ (with the correct form of the estimates), we need only estimate $\grad\tp{\tp{z-f}^{-1}}$ in the space $H^1_\del\tp{\R^2;\C^2}$ - which reduces to estimating the following three inclusions for all $i,j\in\tcb{1,2}$:
    \begin{equation}\label{three base case inclusions}
        (z-f)^{-2}\pd_if,\;\tp{z-f}^{-2}\pd_i\pd_j f,\;\tp{z-f}^{-3}\pd_if\pd_jf\in L^2_\del\tp{\R^2;\C}.
    \end{equation}
    The first two of these are very simple to estimate (using implicitly Proposition~\ref{prop on equivalent norms in the weighted Sobolev spaces}):
    \begin{equation}
        \tnorm{(z-f)^{-1}\pd_if}_{L^2_\del}+\tnorm{\tp{z-f}^{-2}\pd_i\pd_jf}_{L^2_\del}\lesssim\ep^{-2}\tnorm{\grad f}_{H^1_\del}=\ep^{-2}\tnorm{f}_{{_{\m{e}}}\tilde{H}^2_\del}.
    \end{equation}
    The final inclusion of~\eqref{three base case inclusions} is a tad more complicated to estimate. We shall use the identity 
    \begin{equation}\label{",,,"}
        \pd_i f\pd_j f=\pd_i\pd_j\tp{f^2/2}-f\pd_i\pd_j f
    \end{equation}
    and reduce to estimating the resulting two terms. For the former we use again Proposition~\ref{prop on equivalent norms in the weighted Sobolev spaces} and also the fourth item of Proposition~\ref{prop on properties of extended weighted gradient spaces}, 
    \begin{equation}
        \tnorm{(z-f)^{-3}\pd_i\pd_j\tp{f^2/2}}_{L^2_\del}\lesssim\ep^{-3}\tnorm{f^2}_{{_{\m{e}}}\tilde{H}^2_\del}\lesssim\ep^{-3}\tnorm{f}^2_{{_{\m{e}}}\tilde{H}^2_\del}.
    \end{equation}
    For the latter inclusion and estimate of~\eqref{",,,"} we also invoke the third item of Proposition~\ref{prop on properties of extended weighted gradient spaces}:
    \begin{equation}
        \tnorm{(z-f)^{-3}f\pd_i\pd_j f}_{L^2_\del}\lesssim\ep^{-3}\tnorm{f}_{L^\infty}\tnorm{f}_{{_{\m{e}}}\tilde{H}^s_\del}\lesssim\ep^{-3}\tnorm{f}^2_{{_{\m{e}}}\tilde{H}^s_\del}.
    \end{equation}
    By synthesizing the above information and estimates we indeed find that $(z-f)^{-1}\in{_{\m{e}}}\tilde{H}^2_\del\tp{\R^2;\C}$ with $\bf{c}\tp{\tp{z-f}^{-1}}=\tp{z-\bf{c}(f)}^{-1}$ and with the resolvent estimate
    \begin{equation}
        \tnorm{\tp{z-f}^{-1}}_{{_{\m{e}}}\tilde{H}^2_\del}\lesssim\ep^{-1}+\ep^{-2}\tnorm{f}_{{_{\m{e}}}\tilde{H}^2_\del}+\ep^{-3}\tnorm{f}^2_{{_{\m{e}}}\tilde{H}^2_\del}
    \end{equation}
    with implicit constants depending only on $\del$. This obeys the necessary structure of~\eqref{IH resolvent estimate} and hence the base case is established.

    Now let us fix $\N\ni s\ge 2$ and suppose the induction hypotheses and estimate~\eqref{IH resolvent estimate} hold at regularity index $s$; we endeavor to prove the same for $s+1$. So let $f\in{_{\m{e}}}\tilde{H}^{s+1}_\del\tp{\R^2;\C}$ and $z\not\in\Bar{\m{img}\tp{f}}$. The previous argument has already established that $(z-f)^{-1}\in \C+L^{2/\del}\tp{\R^2;\C}$ and therefore we need only verify that $\grad\tp{\tp{z-f}^{-1}}\in H^s_\del\tp{\R^2;\C^2}$ with a structured estimate. Given $i\in\tcb{1,2}$ we check that
    \begin{equation}\label{-_..._-}
        \pd_i\tp{\tp{z-f}^{-1}}=\tp{\tp{z-f}^{-2}-\tp{z-\bf{c}\tp{f}}^{-2}}\pd_if+\tp{z-\bf{c}\tp{f}}^{-2}\pd_i f.
    \end{equation}
    By the induction hypothesis, the fourth item of Proposition~\ref{prop on properties of extended weighted gradient spaces}, and Definition~\ref{defn on extended weighted gradient spaces} we know that $\tp{z-f}^{-2}-\tp{z-\bf{c}\tp{f}}^{-2}\in\tilde{H}^s_\del\tp{\R^d;\C}$. Hence for the first term on the right hand side of~\eqref{-_..._-} we may invoke the third item of Proposition~\ref{prop on product estimates in weighted gradient spaces}, the fourth item of Proposition~\ref{prop on properties of extended weighted gradient spaces}, and the induction hypothesis to estimate
    \begin{equation}
        \tnorm{\tp{\tp{z-f}^{-2}-\tp{z-\bf{c}\tp{f}}^{-2}}\pd_if}_{H^s_\del}\lesssim\tnorm{\tp{z-f}^{-1}}_{{_{\m{e}}}\tilde{H}^s_\del}^2\tnorm{f}_{\tilde{H}^{s+1}_\del}\le\bf{C}_s\tp{\m{dist}_{\C}\tp{\Bar{\m{img}\tp{f}},z}^{-1},\tnorm{f}_{{_{\m{e}}}\tilde{H}^s_\del}}^2\tnorm{f}_{\tilde{H}^{s+1}_\del}.
    \end{equation}
    For the final term in~\eqref{-_..._-} there is nothing to do since $(z-\bf{c}(f))^{-2}\in\C$ and so a correctly structured estimate is clearly obtainable. Synthesizing the above estimates shows that the structured bound~\eqref{IH resolvent estimate} with $s$ replaced by $s+1$ holds for $\bf{C}_{s+1}$ given by
    \begin{equation}
        \bf{C}_{s+1}(\theta,\phi)=\theta+c_0\tp{\bf{C}_s\tp{\theta,c_1\phi}^2+\theta^2}\phi
    \end{equation}
    for a constants $c_0,c_1\in\R^+$ depending only on $s$ and $\del$. It is easily verified that $\bf{C}_{s+1}$ is locally bounded and increasing in both arguments. The induction, and with it the proof, is complete.
\end{proof}

\section{Linear analysis}\label{section on linear analysis}

This section analyzes a trivial forcing linearization of system~\eqref{stationary nondimensionalized equations} in both periodic and solitary function spaces. Our strategy is to decouple the bathymetric contributions as a compact remainder from the main translation invariant linear part. We then establish well-posedness of this latter part through  energy arguments and multiplier theorems.  The contribution of the bathymetric part is then accounted for with Fredholm theory.

\subsection{Principal part analysis}\label{subsection on principal part analysis}

In this subsection we study the linear system of equations 
\begin{equation}\label{principal part linear system}
    \begin{cases}
        \grad\cdot u=\psi,\\
        Au-\grad\cdot\mathbb{S}u+\tp{G-\Delta}\grad\eta=\phi,
    \end{cases}
\end{equation}
for the unknown vector and scalar fields $u:\R^2\to\R^2$ and $\eta:\R^2\to\R$ with the data scalar $\psi:\R^2\to\R$ and vector $\phi:\R^2\to\R^2$. Here $A,G\ge0$ are parameters, and we shall always assume $A>0$ and we take $G>0$ as well in the solitary case. The analysis of this subsection corresponds in a sense to the linearization of system~\eqref{stationary nondimensionalized equations} with trivial bathymetry.

Our first result establishes the well-posedness of system~\eqref{principal part linear system} in spaces of periodic functions.

\begin{prop}[Principal part analysis - periodic case]\label{prop on principal part linear analysis periodic case}
    Assume that $A>0$, $G\ge0$, and $s\in\N$. For each $\tp{\psi,\phi}\in \z{H}^{1+s}\tp{\T_L^2}\times H^s\tp{\T_L^2;\R^2}$ there exists a unique $\tp{u,\eta}\in H^{2+s}\tp{\T_L^2;\R^2}\times \z{H}^{3+s}\tp{\T_L^2}$ such that system~\eqref{principal part linear system} is satisfied.
\end{prop}
\begin{proof}
    Let us first establish uniqueness. Due to the linearity of the equations, it is sufficient to prove that if $(\psi,\phi)=0$ then $(u,\eta)=0$. We achieve this through a simple integration by parts argument. By testing the second equation in~\eqref{principal part linear system} with $u$ in the $L^2\tp{\T_L^2;\R^2}$ inner product and integrating by parts, we see that
    \begin{equation}\label{the uniqueness identity}
        0=\int_{\T^2_L}\tp{Au-\grad\cdot\mathbb{S}u+\tp{G-\Delta}\grad\eta}\cdot u = \int_{\T^2_L} A |u|^2 +  2^{-1}\tabs{\grad u+\grad u^{\m{t}}}^2+2\tabs{\grad\cdot u}^2  \ge   A\int_{\T^2_L}|u|^2.
    \end{equation}
    As $A>0$, we learn that $u=0$. Now the second equation in~\eqref{principal part linear system} reduces to $\tp{G-\Delta}\grad\eta=0$. Since $G\ge0$ and $\mathscr{F}[\eta](0)=0$, we deduce that $\eta=0$.  

    Now we verify the existence of solutions. Given $\tp{\psi,\phi}\in\z{H}^{1+s}\tp{\T^2_L}\times H^s\tp{\T^2_L;\R^2}$ we define $\eta\in \z{H}^{3+s}\tp{\T^2_L}$ and $u\in H^{2+s}\tp{\T^2_L;\R^2}$ according to 
    \begin{equation}\label{linear eta}
        \eta=\tp{G-\Delta}^{-1}\Delta^{-1}\tp{\grad\cdot\phi-\tp{A-4\Delta}\psi}
    \end{equation}
    and
    \begin{equation}\label{linear u}
        u=\tp{A-\Delta}^{-1}\tp{\phi+3\grad\psi-\tp{G-\Delta}\grad\eta}.
    \end{equation}
    We compute
    \begin{equation}
        \grad\cdot u=\tp{A-\Delta}^{-1}\tp{3\Delta\psi+\tp{A-4\Delta}\psi}=\psi,
    \end{equation}
    which verifies that the first equation in~\eqref{principal part linear system} is satisfied.  Next, we compute 
    \begin{equation}\label{divergence of linear stress}
        \grad\cdot\mathbb{S}u=\Delta u+3\grad\grad\cdot u=\Delta u+3\grad\psi.
    \end{equation}
    Therefore,
    \begin{equation}
        Au-\grad\cdot\mathbb{S}u=\tp{A-\Delta}u-3\grad\psi=\phi-\tp{G-\Delta}\grad\eta
    \end{equation}
    so the second equation in~\eqref{principal part linear system} is satisfied as well. Existence has been verified, and with that the proof is complete.
\end{proof}

Our next result establishes the well-posedness of system~\eqref{principal part linear system} when $\psi=0$ in spaces of functions whose members vanish at infinity. The first equation in~\eqref{principal part linear system} is just the linear constraint that the velocity vector has vanishing divergence or, equivalently, is \emph{solenoidal}. It is therefore convenient (both here and in what follows) to encode this condition into the function spaces via the introduction of the notation
\begin{equation}\label{the definition of the solenoidal space}
    {_{\m{sol}}}H^s_\del\tp{\R^2;\R^2} = \tcb{u\in H^s_\del\tp{\R^2;\R^2}\;:\;\grad\cdot u=0},\quad s\in\N,\quad\del\ge0
\end{equation}
and focus just on the second equation in~\eqref{principal part linear system}.

The proof, while in spirit is the same as that of Proposition~\ref{prop on principal part linear analysis periodic case}, is complicated by the fact that we are essentially barred by the subsequent nonlinear analysis from using the standard $L^2$-based Sobolev spaces. We instead establish well-posedness in weighted $L^2$-based Sobolev spaces as to ensure enough control is gained over the free surface unknown. The reader is referred to Sections~\ref{appendix on weighted sobolev spaces} and~\ref{appendix on weighted gradient sobolev spaces} for the definitions and properties of the weighed Sobolev and weighted gradient Sobolev spaces, respectively.

\begin{prop}[Principal part analysis - solitary case]\label{prop on principal part analysis solitary case}
    Assume that $A,G>0$, $s\in\N$, and $\del\in(0,1)$. For each $\phi\in H^s_\del\tp{\R^2;\R^2}$ there exists a unique $(u,\eta)\in{_{\m{sol}}}H^{2+s}_\del\tp{\R^2;\R^2}\times\tilde{H}^{3+s}_\del\tp{\R^2}$ such that system~\eqref{principal part linear system} is satisfied with $\psi=0$.
\end{prop}
\begin{proof}
    As in the proof of Proposition~\ref{prop on principal part linear analysis periodic case}, we begin by establishing uniqueness of solutions; assume that system~\eqref{principal part linear system} is satisfied by $(u,\eta)\in{_{\m{sol}}}H^{2+s}_\del\tp{\R^2;\R^2}\times\tilde{H}^{3+s}\tp{\R^2}$ with $\phi=0$ and $\psi=0$. Our plan is to establish the analog of~\eqref{the uniqueness identity}, but now we have to be slightly more careful as $\eta \notin L^2(\R^2)$ in general. By testing the equation with $u$ and integrating by parts we again derive the identity
    \begin{equation}\label{and he waddled away}
        \int_{\R^2}A|u|^2+2^{-1}|\grad u+\grad u^{\m{t}}|^2+2\tabs{\grad\cdot u}^2=\babs{\int_{\R^2}\tp{G-\Delta}\grad\eta\cdot u}.
    \end{equation}
    Next, we argue that the right hand side of the above vanishes identically. To achieve this we fix $\ep\in\R^+$ and employ  
    Proposition~\ref{prop on density of functions of bounded support} to obtain $\tilde{\eta}\in C^\infty_{\m{c}}\tp{\R^2}$ such that $\tnorm{\eta-\tilde{\eta}}_{\tilde{H}^{3+s}_\del}\le\ep$. Then we are free write $\eta=\tilde{\eta}+\tp{\eta-\tilde{\eta}}$ in the right hand side of~\eqref{and he waddled away} and integrate by parts in the $\tilde{\eta}$ term to exploit $\grad\cdot u=0$. Hence,
    \begin{equation}
        \babs{\int_{\R^2}\tp{G-\Delta}\grad\eta\cdot u}\le\babs{\int_{\R^2}\tp{G-\Delta}\grad\tp{\eta-\tilde{\eta}}\cdot u}\lesssim\ep\tnorm{u}_{L^2},
    \end{equation}
    and since $\ep>0$ was arbitrary we find that the expression on the left must be zero. In turn, thanks to~\eqref{and he waddled away} we deduce that $u=0$ and then $\tp{G-\Delta}\grad\eta=0$. As $G>0$ and $\grad\eta\in L^2_\del\tp{\R^2;\R^2}\emb L^2\tp{\R^2;\R^2}$, it follows from Plancherel's theorem that $\grad\eta=0$ and hence $\eta$ is a constant. The first item of Proposition~\ref{prop on completeness of the weighted gradient spaces} shows that $\tilde{H}^{3+s}_\del\tp{\R^2}$ embeds into $L^{2/\del}\tp{\R^2}$, so it must follow that $\eta=0$.

    We now prove the existence of solutions. Fix $\phi\in H^s_\del\tp{\R^2;\R^2}$. Our goal is to make sense of the formulas~\eqref{linear eta} and~\eqref{linear u} (with $\psi=0$) but in the context of the weighted spaces. The former is a bit trickier and requires a more complicated approach than in the periodic case.  We first define the auxiliary function 
    \begin{equation}
        q=\tp{G-\Delta}^{-1}\tp{\Delta^{-1}\grad\grad\cdot\phi}=m_G(D)\phi\quad\text{with}\quad m_G(\xi)=\f{1}{G+4\pi^2|\xi|^2}\f{\xi\otimes\xi}{|\xi|^2}.
    \end{equation}
    Since $G>0$ the (components of the) multiplier $m_{G}$ satisfy the hypotheses of Proposition~\ref{prop on improved multiplier bounds} with $k=1$ and $r=-2$; therefore, we can invoke this result (taking heed of Remark~\ref{rmk on gen = reg on natty}) to deduce that $q\in H^{2+s}_\del\tp{\R^2;\R^2}$,  noting that $q$ is real-valued since $\Bar{m_{G}(-\xi)}=m_G\tp{\xi}$.

    We now aim to define $\eta$ to be the function in the weighted gradient Sobolev space whose gradient agrees with $q$.  Indeed, upon noting $\mathcal{I}_1\mathcal{R}\cdot=\Delta^{-1}\grad\cdot$ (see Remark~\ref{remark on the I1 operator}) we set
    \begin{equation}\label{eta in terms of q}
        \eta=\mathcal{I}_1\mathcal{R}\cdot q.
    \end{equation}
    Due to the embedding $H^{2+s}_\del\tp{\R^2;\R^2}\emb L^2_\del\tp{\R^2;\R^2}$ (proved in Proposition~\ref{prop on embeddings of weighted Sobolev spaces}), the boundedness of $\mathcal{R}\cdot:L^2_\del\tp{\R^2;\R^2}\to L^2_\del\tp{\R^2}$ (proved in Proposition~\ref{prop on improved multiplier bounds} and Remark~\ref{rmk on gen = reg on natty}), and the boundedness of $\mathcal{I}_1:L^2_\del\tp{\R^2}\to L^{2/\del}\tp{\R^2}$ (proved in Lemma~\ref{lem weighted estimates for riesz potential operator}),  we see that~\eqref{eta in terms of q} defines $\eta\in L^{2/\del}\tp{\R^2}$. Since we already know that $q\in H^{2+s}_\del\tp{\R^2;\R^2}$, we obtain $\eta\in\tilde{H}^{3+s}_\del\tp{\R^2}$ as soon as we verify that 
    \begin{equation}
        \grad\eta=\Delta^{-1}\grad\grad\cdot q=(G-\Delta)^{-1}\tp{\Delta^{-1}\grad\grad\cdot}^2\phi=q.
    \end{equation}
    
    Next, we define the velocity according to
    \begin{equation}
        u=\tp{A-\Delta}^{-1}\tp{\phi-\tp{G-\Delta}\grad\eta}.
    \end{equation}
    By repeated applications of Proposition~\ref{prop on improved multiplier bounds} and Remark~\ref{rmk on gen = reg on natty}, together with an observation as above that the relevant symbols preserve realness,   we deduce from the inclusions $\phi\in H^{s}_\del\tp{\R^2;\R^2}$ and $\grad\eta\in H^{2+s}_\del\tp{\R^2;\R^2}$ that $u\in H^{2+s}_{\del}\tp{\R^2;\R^2}$. It is now a simple calculation to verify that $\grad\cdot u=0$, $(A-\Delta)u=Au-\grad\cdot\mathbb{S}u$, and the second equation in~\eqref{principal part linear system} is satisfied.  This completes the existence proof.
\end{proof}

\subsection{Inclusion of bathymetric remainders}\label{subsection on bathymetric remainder}

The goal of this subsection is again to study the linearization of system~\eqref{stationary nondimensionalized equations} but, in contrast with Section~\ref{subsection on principal part analysis}, we shall allow for general bathymetry. The form of this linearization is slightly different in the periodic and solitary cases due to a change of unknowns made in the latter case.

The following lemma is crucial for our analysis of the periodic case. We use the notation $P_0f=f-\tp{L_1L_2}^{-1}\int_{\T^2_L}f$ to denote the projection onto the subspace of functions of vanishing average.
\begin{lem}[Divergence equation reformulation]\label{lem on divergence equation reformulation}
    Suppose that $v\in H^1\tp{\T^2_{L};\R^2}$, $f\in\tp{C^0\cap H^1}\tp{\T^2_{L}}$, and $g\in \z{L}^2\tp{\T^2_L}$ solve the equation
    \begin{equation}\label{broken divergence equation}
        \grad\cdot v+P_0\tp{v\cdot\grad f}=g.
    \end{equation}
    Then
    \begin{equation}\label{the claimed identity in the divergence equation lemma}
        \f{1}{L_1L_2}\int_{\T^2_L}v\cdot\grad f=\bp{\int_{\T^2_L}e^f}^{-1}\int_{\T^2_L}ge^f
    \end{equation}
\end{lem}
\begin{proof}
    We compute 
    \begin{equation}
        \grad\cdot\tp{e^fv}=e^f\tp{\grad\cdot v+v\cdot\grad f}=e^f\bp{g-\f{1}{L_1L_2}\int_{\T^2_L}v\cdot\grad f}.
    \end{equation}
    The result then follows by integrating over $\T^2_L$, applying the divergence theorem, and rearranging.
\end{proof}

We now arrive at the main linear analysis result for the  periodic case.

\begin{prop}[Linear analysis with bathymetry - periodic case]\label{prop on linear analysis with bathymetry - periodic case}
    Let $s\in\N$, $A>0$, $G\ge0$, and $\be\in C^\infty\tp{\T^2_L}$ with $\min\be=0$. The following hold.
    \begin{enumerate}
        \item The linear map $\mathcal{P}:H^{2+s}\tp{\T^2_L;\R^2}\times \z{H}^{3+s}\tp{\T^2_L}\to \z{H}^{1+s}\tp{\T^2_L}\times H^s\tp{\T^2_L;\R^2}$ given via
        \begin{equation}\label{periodic principal part}
            \mathcal{P}\tp{u,\eta}=\tp{\grad\cdot u,Au-\grad\cdot\mathbb{S}u+\tp{G-\Delta}\grad\eta}
        \end{equation}
        is a bounded linear isomorphism.
        \item The linear map $\mathcal{L}_\be: H^{1+s}\tp{\T^2_L;\R^2}\times \z{H}^{2+s}\tp{\T^2_L;\R^2}\to \z{H}^{1+s}\tp{\T^2_L}\times H^s\tp{\T^2_L;\R^2}$ given via
        \begin{equation}\label{periodic bathymetry remainder}
            \mathcal{L}_\be\tp{u,\eta}=\tp{P_0\tp{u\cdot\grad\log\tp{1+\be}},-A\be\tp{1+\be}^{-1}u-\mathbb{S}u\grad\log\tp{1+\be}}
        \end{equation}
        is bounded.
        \item The map $(\mathcal{P}+\mathcal{L}_\be):H^{2+s}\tp{\T^2_L;\R^2}\times \z{H}^{3+s}\tp{\T^2_L}\to \z{H}^{1+s}\tp{\T^2_L}\times H^s\tp{\T^2_L;\R^2}$ is a bounded linear isomorphism.
    \end{enumerate}
\end{prop}
\begin{proof}
    The first item is an operator-theoretic restatement of Proposition~\ref{prop on principal part linear analysis periodic case}. 

    To prove the second item, we first note that $\be\ge 0$ is smooth, and hence $(1+\be)^{-1}$ and $\log\tp{1+\be}$ are smooth functions as well.  Using this and simple applications of Theorem~\ref{thm on high low product estimates} shows that each term boundedly maps into the stated space.

    We now turn to the proof of the third item.   The Rellich–Kondrachov theorem implies that the inclusion map $H^{2+s}\tp{\T^2_L;\R^2}\times \z{H}^{3+s}\tp{\T^2_L}\emb H^{1+s}\tp{\T^2_L;\R^2}\times \z{H}^{2+s}\tp{\T^2_L}$ is compact; from this and the mapping properties established in the second item, we find that $\mathcal{L}_\be$ restricted to the former domain is a compact operator. We combine this fact with the first item to deduce that
    \begin{equation}
        \mathcal{P}+\mathcal{L}_\be:H^{2+s}\tp{\T^2_L;\R^2}\times \z{H}^{3+s}\tp{\T^2_L}\to \z{H}^{1+s}\tp{\T^2_L}\times H^s\tp{\T^2_L;\R^2}
    \end{equation}
    is a Fredholm operator with index zero; hence, it is a linear isomorphism if and only if $\m{ker}\tp{\mathcal{P}+\mathcal{L}_\be}=\tcb{0}$. 
    
    To complete the proof of the third item, we will now show that $\mathcal{P}+\mathcal{L}_\be$ is injective.  Suppose that $(u,\eta)\in\m{ker}\tp{\mathcal{P}+\mathcal{L}_\be}$.  Simple  algebraic manipulations then reveal that $(u,\eta)$ satisfy
    \begin{equation}\label{linear equations with bathymetry}
        \begin{cases}
            \grad\cdot u+P_0\tp{u\cdot\grad\log\tp{1+\be}}=0,\\
            A\tp{1+\be}^{-1}u+(1+\be)^{-1}\grad\cdot\tp{\tp{1+\be}\mathbb{S}u}+\tp{G-\Delta}\grad\eta=0.
        \end{cases}
    \end{equation}
    We can use Lemma~\ref{lem on divergence equation reformulation} to simplify the first equation in system~\eqref{linear equations with bathymetry}: we find that $P_0\tp{u\cdot\grad\log\tp{1+\be}}=u\cdot\grad\log\tp{1+\be}$, and hence the first equation is equivalent to
    \begin{equation}
        \grad\cdot\tp{\tp{1+\be}u}=0.
    \end{equation}
    Therefore, if we test the second equation in~\eqref{linear equations with bathymetry} with $(1+\be)u$ and integrate by parts, we learn that
    \begin{equation}
        \int_{\T^2_L}A|u|^2+(1+\be)\tp{2^{-1}|\grad u+\grad u^{\m{t}}|^2+2|\grad\cdot u|^2}=0,
    \end{equation}
    and thus $u =0$.     In turn, this implies that $(G-\Delta)\grad\eta=0$, and since $G\ge0$ and $\mathscr{F}[\eta]\tp{0}=0$ we must have $\eta=0$ as well, and injectivity is proved.  
\end{proof}

To conclude this section we give the main result on the linear analysis in the solitary case. Recall the definition of the extended weighted gradient spaces (and the function $\bf{c}$) from Definition~\ref{defn on extended weighted gradient spaces} and the notation for the space of solenoidal vector fields introduced in~\eqref{the definition of the solenoidal space}. We shall also use the following quadratic form notation for $m\in\tcb{1,2}$:
\begin{equation}\label{quadratic form notation}
    Q_m:\R^2\times\R^2\to\R^{2\times 2}_{\m{sym}},\quad Q_m(v,w)= m\tp{v\cdot w}I+v\otimes w+w\otimes v,\text{ and } Q_m(v)=Q_m(v,v).
\end{equation}

\begin{prop}[Linear analysis with bathymetry - solitary case]\label{prop on linear analysis with bathymetry solitary case}
    Let $s\in\N$, $A,G>0$, $\del,\rho,\mu\in(0,1)$ satisfy $\rho-\del<\mu<\rho$, and let $\be\in{_{\m{e}}}\tilde{H}^\infty_\del\tp{\R^2}$ with $\inf\be=0$. The following hold.
    \begin{enumerate}
        \item The linear map $\mathcal{P}_\be:{_{\m{sol}}}H^{2+s}_\rho\tp{\R^2;\R^2}\times\tilde{H}^{3+s}_\rho\tp{\R^2}\to H^s_\rho\tp{\R^2;\R^2}$ given via
        \begin{equation}\label{solitary principal part map}
            \mathcal{P}_\be\tp{u,\eta}= A_\be u-\grad\cdot\mathbb{S}u+\tp{G-\Delta}\grad\eta, \text{ for } A_\be=\f{A}{1+\bf{c}\tp{\be}},
        \end{equation}
        is a bounded linear isomorphism.
        \item The linear map $\mathcal{L}_\be:H^{1+s}_\mu\tp{\R^2;\R^2}\times\tilde{H}^{2+s}_\mu\tp{\R^2}\to H^s_\rho\tp{\R^2;\R^2}$ given via
        \begin{multline}\label{solitary remainder map}
            \mathcal{L}_\be\tp{u,\eta}=-\f{A\tp{\be-\bf{c}\tp{\be}}}{\tp{1+\be}\tp{1+\bf{c}\tp{\be}}}u+\grad\cdot Q_2\tp{\grad\log\tp{1+\be},u}\\-\tp{(1+\be)^{-1}\tp{G-\Delta}\grad\tp{1+\be}}\eta+\grad\cdot Q_1\tp{\grad\tp{1+\be},\grad\tp{\tp{1+\be}^{-1}\eta}}
        \end{multline}
        is well-defined and continuous.
        \item The map $\tp{\mathcal{P}_\be+\mathcal{L}_\be}:{_{\m{sol}}}H^{2+s}_\rho\tp{\R^2;\R^2}\times\tilde{H}^{3+s}_\rho\tp{\R^2}\to H^s_\rho\tp{\R^2;\R^2}$ is a bounded linear isomorphism.
    \end{enumerate}
\end{prop}
\begin{proof}
    A simple application of Proposition~\ref{prop on equivalent norms in the weighted Sobolev spaces} reveals that the mapping~\eqref{solitary principal part map} is well-defined and continuous. This map is invertible thanks to Proposition~\ref{prop on principal part analysis solitary case}, so the first item is proved.

    To prove the second item we begin by recording some facts. By combining the first and second items of Proposition~\ref{prop on equivalent norms and low mode embeddings} with Proposition~\ref{prop on embeddings of weighted Sobolev spaces}, we deduce the continuous embedding  $\tilde{H}^\infty_\del\tp{\R^2}\emb W^{\infty,\infty}_{\tilde{\del}}\tp{\R^2}$ for any $\tilde{\del}\in[0,\del)$. On the other hand, since $\be\ge0$ we combine Proposition~\ref{prop on admissability of the extended weighted gradient spaces} with Corollary~\ref{coro on analytic superposition real case} to deduce the inclusions $\log\tp{1+\be},\tp{1+\be}^{-1},\be\tp{1+\be}^{-1}\in{_{\m{e}}}\tilde{H}^\infty_\del\tp{\R^2}$. The first item of Proposition~\ref{prop on properties of extended weighted gradient spaces} implies that $\be-\bf{c}\tp{\be}\in\tilde{H}^\infty_\del\tp{\R^2}$, while its fourth item tells us that $\tilde{H}_\del^\infty\tp{\R^2}\subset{_{\m{e}}}\tilde{H}^\infty_\del\tp{\R^2}$ is an ideal.
    
    There are four terms in the expression~\eqref{solitary remainder map} for $\mathcal{L}_\be$ that we need to estimate. The first two of these can be felled with the aforementioned facts and Proposition~\ref{prop on product estimates}. Indeed (taking $\tilde{\del}=\rho-\mu<\del$):
    \begin{equation}
        \bnorm{\f{A\tp{\be-\bf{c}\tp{\be}}}{\tp{1+\be}\tp{1+\bf{c}\tp{\be}}}u}_{H^s_\rho}\lesssim\bnorm{\f{A\tp{\be-\bf{c}\tp{\be}}}{\tp{1+\be}\tp{1+\bf{c}\tp{\be}}}}_{W^{s,\infty}_{\tilde{\del}}}\tnorm{u}_{H^s_\mu}\lesssim\tnorm{u}_{H^s_\mu}
    \end{equation}
    and
    \begin{equation}
        \tnorm{\grad\cdot Q_2\tp{\grad\log\tp{1+\be},u}}_{H^s_\rho}\lesssim\tnorm{\grad\log\tp{1+\be}}_{W^{1+s,\infty}_{\tilde{\del}}}\tnorm{u}_{H^{1+s}_\mu}\lesssim\tnorm{u}_{H^{1+s}_\mu}.
    \end{equation}
    To estimate the third term in the expression for $\mathcal{L}_\be$ we split $\eta=\varphi(D)\eta+\tp{1-\varphi(D)}\eta$ with a frequency cut-off function as in Proposition~\ref{prop on equivalent norms and low mode embeddings}. For the low mode contribution, we use that $\varphi(D)\eta\in W^{\infty,\infty}_{\tilde{\mu}}\tp{\R^2}$ (with $\max\tcb{0,\rho-\del}<\tilde{\mu}<\mu$) and $\grad\be\in H^\infty_\del\tp{\R^2;\R^2}$:
    \begin{equation}
        \tnorm{\tp{\tp{1+\be}^{-1}\tp{G-\Delta}\grad\tp{1+\be}}\varphi\tp{D}\eta}_{H^s_\rho}\lesssim\tnorm{\tp{1+\be}^{-1}}_{W^{s,\infty}}\tnorm{\tp{G-\Delta}\grad\be}_{H^s_\del}\tnorm{\varphi(D)\eta}_{W^{s,\infty}_{\tilde{\mu}}}\lesssim\tnorm{\eta}_{\tilde{H}^{2+s}_\mu}.
    \end{equation}
    For the high mode contribution we instead take $\grad\be\in W^{s,\infty}_\del\tp{\R^2;\R^2}$ and $\tp{1-\varphi\tp{D}}\eta\in H^{2+s}_\mu\tp{\R^2}$, which grants us the estimate
    \begin{equation}
        \tnorm{\tp{\tp{1+\be}^{-1}\tp{G-\Delta}\grad\tp{1+\be}}\tp{1-\varphi\tp{D}}\eta}_{H^s_\rho}\lesssim\tnorm{\tp{1+\be}^{-1}}_{W^{s,\infty}}\tnorm{\tp{G-\Delta}\grad\be}_{W^{s,\infty}_\del}\tnorm{\tp{1-\varphi(D)}\eta}_{H^s_\mu}\lesssim\tnorm{\eta}_{\tilde{H}^{2+s}_\mu}.
    \end{equation}
    In order to estimate the final term in~\eqref{solitary remainder map} we argue as follows. By the second item of Proposition~\ref{prop on product estimates in weighted gradient spaces} we get
    \begin{equation}
        \tnorm{(1+\be)^{-1}\eta}_{\tilde{H}^{2+s}_\mu}\lesssim\tp{|\tp{1+\bf{c}\tp{\be}}^{-1}|+\tnorm{\tp{1+\be}^{-1}-\tp{1+\bf{c}\tp{\be}}^{-1}}_{\tilde{H}^{2+s}_{\del}}}\tnorm{\eta}_{\tilde{H}^{2+s}_\mu}\lesssim\tnorm{\eta}_{\tilde{H}^{2+s}_\mu}.
    \end{equation}
    In turn, we may then bound
    \begin{equation}
        \tnorm{\grad\cdot Q_1\tp{\grad\tp{1+\be},\grad\tp{\tp{1+\be}^{-1}\eta}}}_{H^s_\rho}\lesssim\tnorm{\grad\be}_{W^{1+s,\infty}_\del}\tnorm{\grad\tp{(1+\be)^{-1}\eta}}_{H^{1+s}_\mu}\lesssim\tnorm{\eta}_{\tilde{H}^{2+s}_\mu}.
    \end{equation}
    By combining the above estimates, we then find that $\mathcal{L}_\be$ obeys the mapping properties stated in the second item.

    Finally, to prove the third item we will use Fredholm theory as in the proof of Proposition~\ref{prop on linear analysis with bathymetry - periodic case}.  Due to Propositions~\ref{prop on embeddings of weighted Sobolev spaces} and~\ref{prop on inclusion relations}, the inclusion map
    \begin{equation}\label{the compact embedding solitary case}
        H^{2+s}_\rho\tp{\R^2;\R^2}\times\tilde{H}^{3+s}_{\rho}\tp{\R^2}\emb H^{1+s}_{\mu}\tp{\R^2;\R^2}\times\tilde{H}^{2+s}_\mu\tp{\R^2}
    \end{equation}
    is compact and so from the mapping properties of $\mathcal{L}_\be$ established in the second item we find that the restriction of $\mathcal{L}_\be$ to the space on the left of~\eqref{the compact embedding solitary case} is a compact linear operator. Therefore, the sum $\mathcal{P}_\be+\mathcal{L}_\be$ is a Fredholm operator of index zero (as $\mathcal{P}_\be$ is invertible by the first item),  which means that $\mathcal{P}_\be+\mathcal{L}_\be$ is invertible if and only if it has trivial kernel.  Suppose then that $(u,\eta)\in\m{ker}\tp{\mathcal{P}_\be+\mathcal{L}_\be}$. The following auxiliary identities hold:
    \begin{equation}
        A_\be-\f{A\tp{\be-\bf{c}\tp{\be}}}{\tp{1+\be}\tp{1+\bf{c}\tp{\be}}}=\f{A}{1+\be},
    \end{equation}
    \begin{equation}
        -\mathbb{S}u+Q_2\tp{\grad\log\tp{1+\be},u}=\tp{1+\be}\mathbb{S}\tp{\tp{1+\be}^{-1}u},
    \end{equation}
    and for $\al=1+\be$ and $\zeta=\tp{1+\be}^{-1}\eta$:
    \begin{equation}
        \Delta\grad\eta=\Delta\grad\tp{\al\zeta}=\al\Delta\grad\zeta+\zeta\Delta\grad\al+\grad\cdot Q_1\tp{\grad\al,\grad\zeta}.
    \end{equation}
    Thus, $\tp{\mathcal{P}_\be+\mathcal{L}_\be}(u,\eta)=0$ if and only if
    \begin{equation}
        A\tp{1+\be}^{-1}u-\grad\cdot\tp{\tp{1+\be}\mathbb{S}\tp{\tp{1+\be}^{-1}u}}+\tp{1+\be}\tp{G-\Delta}\grad\tp{\tp{1+\be}^{-1}\eta}=0,
    \end{equation}
    but the identity $\grad\cdot u=0$ is built into the domain of $u$.  We then test the above equation with $(1+\be)^{-1}u$ and integrate by parts to deduce that
    \begin{equation}\label{___+++___}
        \int_{\R^2}A|\tp{1+\be}^{-1}u|^2\le\babs{\int_{\R^2}\tp{G-\Delta}\grad\tp{\tp{1+\be}^{-1}\eta}\cdot u}.
    \end{equation}
    We next use the inclusion $(1+\be)^{-1}\eta\in\tilde{H}_\rho^{3+s}\tp{\R^2}$ and the same density argument employed in the proof of Proposition~\ref{prop on principal part analysis solitary case} to deduce that the right hand side of~\eqref{___+++___} vanishes.  In turn, since $A>0$ we learn that $u=0$ and  $(G-\Delta)\grad\tp{(1+\be)^{-1}\eta}=0$.  Since $G>0$ and $\grad\tp{\tp{1+\be}^{-1}\eta}\in L^2\tp{\R^2;\R^2}$, we then learn that $(1+\beta)^{-1}\eta$ is a constant function that also belongs to $L^{2/\rho}\tp{\R^2}$ and so must be trivial.  This completes the proof that $\mathcal{P}_\be+\mathcal{L}_\be$ is injective, and hence the proof of the third item.  
\end{proof}

\section{Semilinearization}\label{section on semilinearization}

The goal of this section is to show that the system~\eqref{stationary nondimensionalized equations} with~\eqref{percolate propane pulvarized paintbrush} is equivalent to a modified system of equations that are semilinear in the sense that the equations are the sum of a linear part and a lower order nonlinear remainder. As it turns out, this linear part is the one we have already studied in Section~\ref{subsection on bathymetric remainder}. Once this is established, we make yet another reformulation of the equations that will be useful in our subsequent analysis.  By inverting the linear part of the semilinearization, we obtain a representation of solutions as zeros of a compact perturbation of the identity map on an appropriate Banach space. 

\subsection{Domains and forcing}\label{subsection on functional framework and forcing}

The purpose of this subsection is to introduce open sets for the collection of free surfaces in both the periodic and solitary cases. We then encode the forcing function~\eqref{percolate propane pulvarized paintbrush} within these spaces and verify basic mapping properties.

\begin{lem}[Free surface domain - periodic case]\label{lem on free surface domain, periodic case}
    Let $\be\in C^\infty\tp{\T^2_L}$ with $\min\be=0$. The following hold.
    \begin{enumerate}
        \item The sets $\bf{O}_\be$ and, for $\ep\in\tp{0,1}$, $\bf{O}_{\be}^\ep$ given by
        \begin{equation}
            \bf{O}_\be=\tcb{\eta\in C^0\tp{\T^2_L}\;:\;\min\tp{1+\be+\eta}>0},\quad\bf{O}_\be^\ep=\tcb{\eta\in\bf{O}_\be\;:\;\min\tp{1+\be+\eta}>\ep}
        \end{equation}
        are open in $C^0\tp{\T^2_L}$ and satisfy $\bigcup_{0<\ep<1}\bf{O}_\be^\ep=\bf{O}_\be$.
        \item Given any $\N\ni s\ge 2$, the sets $\bf{O}_\be\cap\z{H}^{s}\tp{\T^2_L}$ and $\bf{O}_\be^\ep\cap\z{H}^{s}\tp{\T^2_L}$ are open in $\z{H}^s\tp{\T^2_L}$.
        \item The sets $\bf{O}_\be\cap \z{H}^1\tp{\T^2_L}$ and $\bf{O}_\be^\ep\cap\z{H}^1\tp{\T^2_L}$ are open in $\tp{C^0\cap\z{H}^1}\tp{\T^2_L}$.
    \end{enumerate}
\end{lem}
\begin{proof}
    The first item and third items are elementary.  The second item follows from the fact that $s>1$ allows for the supercritical Sobolev embedding $H^{s}\tp{\T^2_L}\emb C^0\tp{\T^2_L}$. 
\end{proof}

In the following result we let $\Upomega$ denote the composition operator of Corollary~\ref{coro on compound analytic superposition real case}.

\begin{prop}[Data map - periodic case]\label{prop on data map, periodic case}
    Let $\be\in C^\infty\tp{\T^2_L}$ with $\min\be=0$ and set $a_\be=1+\max\be$.  Fix a trio of (vector valued) real analytic mappings
    \begin{equation}
        \phi:(-a_\be,\infty)\to L^2\tp{\T^2_L;\R^2},\quad\psi:(-a_\be,\infty)\to H^1\tp{\T^2_L},\quad\tau:(-a_\be,\infty)\to \tp{L^\infty\cap H^1}\tp{\T^2_L;\R^{2\times 2}_{\m{sym}}}.
    \end{equation}
    The following hold.
    \begin{enumerate}
        \item The forcing map $\mathcal{F}_\be:\bf{O}_\be\cap \z{H}^1\tp{\T^2_L}\to L^2\tp{\T^2_L;\R^2}$ given via
        \begin{equation}\label{the forcing function periodic case}
            \mathcal{F}_\be\tp{\eta}=\tp{1+\be+\eta}^{-1}\tp{\Upomega_\phi\tp{\eta}+\grad\tp{\tp{1+\be+\eta}\Upomega_\psi\tp{\eta}}+\Upomega_\tau\tp{\eta}\grad\eta}
        \end{equation}
        is well-defined and real analytic.
        \item For any $\ep\in\tp{0,1}$ the restriction $\mathcal{F}_\be:\bf{O}_\be^\ep\cap \z{H}^1\tp{\T^2_L}\to L^2\tp{\T^2_L;\R^2}$ maps bounded sets to bounded sets.
        \item For any $\ep\in\tp{0,1}$ and $\N\ni s\ge 2$ the restriction $\mathcal{F}_\be:\bf{O}^\ep_\be\cap \z{H}^s\tp{\T^2_L}\to L^2\tp{\T^2_L;\R^2}$ maps bounded sets to precompact sets.
    \end{enumerate}
\end{prop}
\begin{proof}
    Proposition~\ref{prop on admissability in the periodic setting} guarantees that the unital Banach algebra $\tp{C^0\cap H^1}\tp{\T^2_L;\C}$ is admissible in the sense of Definition \ref{defn of admissable Banach algebras}; consequently, we may apply Corollary~\ref{coro on compound analytic superposition real case} to learn that that the maps
    \begin{multline}\label{been the mountain}
        \Upomega_\phi:\bf{O}_\be\cap \z{H}^1\tp{\T^2_L}\to L^2\tp{\T^2_L;\R^2},\quad\Upomega_\psi:\bf{O}_\be\cap \z{H}^1\tp{\T^2_L}\to H^1\tp{\T^2_L},\\\Upomega_\tau:\bf{O}_\be\cap \z{H}^1\tp{\T^2_L}\to \tp{L^\infty\cap H^1}\tp{\T^2_L;\R^{2\times 2}_{\m{sym}}}
    \end{multline}
    are all well-defined and real analytic, and that their restrictions to $\bf{O}_\be^\ep\cap \z{H}^1\tp{\T^2_L}$ all map bounded sets to bounded sets.  On the other hand, due to Corollary~\ref{coro on analytic superposition real case}, the mappings
    \begin{equation}\label{going to china}
        \bf{O}_\be\cap \z{H}^1\tp{\T^2_L}\ni\eta\mapsto(1+\be+\eta),\;\tp{1+\be+\eta}^{-1}\in\tp{C^0\cap H^1}\tp{\T^2_L;\C}
    \end{equation}
    are also well-defined and real analytic, with their restrictions to $\bf{O}_\be^\ep\cap\z{H}^1\tp{\T^2_L}$ mapping bounded sets to bounded sets.

    The map $\mathcal{F}_\be$ of~\eqref{the forcing function periodic case} combines the maps of~\eqref{been the mountain} and~\eqref{going to china} (and $\grad\eta$) via simple products, which can been handled with Theorem~\ref{thm on high low product estimates}.  Thus,  $\mathcal{F}_\be$ is well-defined and real analytic. The same argument also shows that the second item holds.

    Finally, we prove the third item.  Thanks to the Rellich–Kondrachov theorem, the inclusion map $\z{H}^s\tp{\T^2_L}\emb\tp{C^0\cap H^1}\tp{\T^2_L}$ is compact, from which it follows that the inclusion map $\bf{O}^\ep_\be\cap \z{H}^s\tp{\T^2_L}\emb\bf{O}^{\ep/2}_\be\cap \z{H}^1\tp{\T^2_L}$ maps bounded sets to precompact sets.  The second item requires $\mathcal{F}_\be$ to be at least continuous, so it maps these precompact sets to precompact sets.
\end{proof}

We now turn our attention to the solitary case. In contrast with the periodic case, we additionally study an analytic change of unknowns. We remind the reader that the definition of the extended weighted gradient spaces is given in Definition~\ref{defn on extended weighted gradient spaces}.

\begin{lem}[Free surface domains - solitary case]\label{lem on free surface domsains, solitary case}
    Let $\del,\rho\in(0,1)$ and $\be\in{_{\m{e}}}\tilde{H}^\infty_\del\tp{\R^2}$ with $\inf\be=0$. The following hold.
    \begin{enumerate}
        \item The sets $\bf{O}_\be$, $\bf{U}_\be$, and for $\ep\in\tp{0,1}$, $\bf{O}_\be^\ep$ and $\bf{U}_\be^\ep$ given by
        \begin{equation}
            \bf{O}_\be=\tcb{\eta\in\tp{C^0\cap L^\infty}\tp{\R^2}\;:\;\inf\tp{1+\be+\eta}>0},\quad\bf{O}_\be^\ep=\tcb{\eta\in\bf{O}_\be\;:\;\inf\tp{1+\be+\eta}>\ep},
        \end{equation}
        and
        \begin{equation}
            \bf{U}_\be=\tcb{\eta\in\tp{C^0\cap L^\infty}(\R^2) \;:\; \inf\tp{\tp{1+\be}^2+2\eta}>0},\quad\bf{U}_\be^\ep=\tcb{\eta\in\bf{U}_\be\;:\;\inf\tp{\tp{1+\be}^2+2\eta}>\ep^2}
        \end{equation}
        are open in $\tp{C^0\cap L^\infty}\tp{\R^2}$ and satisfy $\bigcup_{0<\ep<1}\bf{O}_\be^\ep=\bf{O}_\be$ and $\bigcup_{0<\ep<1}\bf{U}_\be^\ep=\bf{U}_\be$.
        \item Let $\N\ni s\ge 2$, $r\in[2,\infty]$, and $\mathcal{X}\in\tcb{\tilde{H}^s_\rho\tp{\R^2},\tp{C^0\cap L^\infty\cap\dot{W}^{1,r}}\tp{\R^2}}$. The sets $\bf{O}\cap\mathcal{X}$, $\bf{O}_\ep\cap\mathcal{X}$, $\bf{U}\cap\mathcal{X}$, and $\bf{U}_\ep\cap\mathcal{X}$ are all open in $\mathcal{X}$.
        \item The functions $\mathcal{T}_\be:\bf{O}_\be\to\bf{U}_\be$ and $\mathcal{T}_\be^{-1}:\bf{U}_\be\to\bf{O}_\be$ given by
        \begin{equation}
            \mathcal{T}_\be\tp{\eta}=\f{1}{2}\tp{\tp{1+\be+\eta}^2-\tp{1+\be}^2},\quad\mathcal{T}_\be^{-1}\tp{\eta}=\sqrt{2\eta+\tp{1+\be}^2}-\tp{1+\be}
        \end{equation}
        are well-defined, real analytic, and mutually inverse; moreover, for every $\ep\in\tp{0,1}$ we have that the restrictions satisfy $\mathcal{T}_\be:\bf{O}_\be^\ep\to\bf{U}_\be^\ep$, $\mathcal{T}_\be^{-1}:\bf{U}_\be^\ep\to\bf{O}_\be^\ep$ and map bounded sets to bounded sets.
        \item Let $\mathcal{X}$ be one of the spaces from the second item. Then $\mathcal{T}_\be:\bf{O}_\be\cap\mathcal{X}\to\bf{U}_\be\cap\mathcal{X}$ and $\mathcal{T}_\be^{-1}:\bf{U}_\be\cap\mathcal{X}\to\bf{O}_\be\cap\mathcal{X}$ are well-defined and real analytic.
        \item For any $\ep\in\tp{0,1}$ and $\mathcal{X}$ as in the second item the maps $\mathcal{T}_\be:\bf{O}^\ep_\be\cap\mathcal{X}\to\bf{U}^\ep_\be\cap\mathcal{X}$ and $\mathcal{T}_\be^{-1}:\bf{U}^\ep_\be\cap\mathcal{X}\to\bf{O}_\be^\ep\cap\mathcal{X}$ map bounded sets to bounded sets.
    \end{enumerate}
\end{lem}
\begin{proof}
    The first item is trivial. The second item follows from the first and the fact that each choice of $\mathcal{X}$ embeds into $\tp{C^0\cap L^\infty}\tp{\R^2}$ (see the third item of Proposition~\ref{prop on properties of extended weighted gradient spaces}). The third item is an easy consequence of Proposition~\ref{prop on admissability in the periodic setting} (with $s=0$) and Corollary~\ref{coro on analytic superposition real case}.

    For the forth and fifth items we consider the various cases of $\mathcal{X}$ separately. The case $\mathcal{X}=\tp{C^0\cap L^\infty\cap\dot{W}^{1,r}}\tp{\R^2}$ is the easiest. Thanks to Proposition~\ref{prop on preliminary admissibility in the solitary case} we see that $\mathcal{X}$ is the real-valued subspace of an admissible unital Banach algebra in the sense of Definition~\ref{defn of admissable Banach algebras}; moreover, the hypotheses on $\be$ and $r$ ensure that $\be\in\mathcal{X}$ in this case. Therefore, we readily read off the mapping properties of $\mathcal{T}_\be$ and $\mathcal{T}_\be^{-1}$ claimed in the fourth and fifth items for this specific $\mathcal{X}$ as a consequence of Corollary~\ref{coro on analytic superposition real case}.

    Next, we handle the case $\mathcal{X}=\tilde{H}^s_\rho\tp{\R^2}$. We will rewrite $\mathcal{T}_\be$ and its inverse $\mathcal{T}_\be^{-1}$ in certain ways to read off the correct mapping properties. For the forward mapping note that we may write
    \begin{equation}\label{who is that and why are they doing those strange things}
        \mathcal{T}_\be\tp{\eta}=\f{1}{2}\eta^2+\tp{\be-\bf{c}\tp{\be}}\eta+\tp{1+\bf{c}\tp{\be}}\eta
    \end{equation}
    where $\bf{c}$ is the function in Definition~\ref{defn on extended weighted gradient spaces}. As $\mathcal{T}_\be$ is a polynomial, real analyticity and the mapping of bounded sets to bounded sets follow as soon as we verify the continuity of each constituent product. The first and second terms on the right hand side of~\eqref{who is that and why are they doing those strange things} are handled with the second item of Proposition~\ref{prop on product estimates in weighted gradient spaces}, while the final term is linear.

    To study $\mathcal{T}_\be^{-1}$ we write
    \begin{equation}
        \mathcal{T}_\be^{-1}\tp{\eta}=2\tp{\sqrt{2\eta+\tp{1+\be}^2}+\tp{1+\be}}^{-1}\eta.
    \end{equation}
    Let $\mu=\min\tcb{\del,\rho}\in\tp{0,1}$. By noting the continuity of the product map ${_{\m{e}}}\tilde{H}^s_\mu\tp{\R^2}\times\tilde{H}^s_\rho\tp{\R^2}\to\tilde{H}^s_\rho\tp{\R^2}$ (which is a consequence of the second item of Proposition~\ref{prop on product estimates in weighted gradient spaces} and the fourth item of Proposition~\ref{prop on properties of extended weighted gradient spaces}) we reduce to proving that
    \begin{equation}
        \bf{U}_\be\cap{_{\m{e}}}\tilde{H}^s_\mu\tp{\R^2}\ni\eta\mapsto\tp{\sqrt{2\eta+\tp{1+\be}^2}+\tp{1+\be}}^{-1}\in{_{\m{e}}}\tilde{H}^{s}_\mu\tp{\R^2}
    \end{equation}
    is well-defined, real analytic, and its restriction to $\bf{U}_\be^\ep\cap{_{\m{e}}}\tilde{H}^s_\mu\tp{\R^2}$, for any $\ep\in\tp{0,1}$, maps bounded sets to bounded sets. Now thanks to Proposition~\ref{prop on admissability of the extended weighted gradient spaces} we are assured that ${_{\m{e}}}\tilde{H}^s_\del\tp{\R^2}$ is the real-valued subspace of an admissible unital Banach algebra. This permits repeated applications of Corollary~\ref{coro on analytic superposition real case} to yield the sought after properties.
\end{proof}

We now come to the solitary case analog of Proposition~\ref{prop on data map, periodic case} where again $\Upomega$ is the composition operator from Corollary~\ref{coro on compound analytic superposition real case}.

\begin{prop}[Data map - solitary case]\label{prop on data map - solitary case}
    Let $\del,\rho\in\tp{0,1}$, $\be\in{_{\m{e}}}\tilde{H}^\infty_\del\tp{\R^2}$ with $\inf\be=0$, and set $a_\be=1+\sup\be$. Fix a trio of (vector valued) real analytic mappings
    \begin{equation}
        \phi:(-a_\be,\infty)\to L^2_\rho\tp{\R^2;\R^2},\quad\psi:(-a_\be,\infty)\to H^1_\rho\tp{\R^2},\quad\tau:(-a_\be,\infty)\to\tp{L^\infty_\rho\cap H^1_\rho}\tp{\R^2;\R^{2\times 2}_{\m{sym}}}.
    \end{equation}
    The following hold.
    \begin{enumerate}
        \item For any $r\in[2,\infty]$ the forcing map $\mathcal{F}_\be:\bf{U}_\be\cap\tp{C^0\cap L^\infty\cap\dot{W}^{1,r}}\tp{\R^2}\to L^2_\rho\tp{\R^2;\R^2}$ given via
        \begin{equation}\label{what even is helicity}
            \mathcal{F}_\be(\eta)=\Upomega_\phi\circ\mathcal{T}_\be^{-1}\tp{\eta}+\grad\tp{\sqrt{2\eta+\tp{1+\be}^2}\Upomega_\psi\circ\mathcal{T}_\be^{-1}\tp{\eta}}+\Upomega_\tau\circ\mathcal{T}_\be^{-1}\tp{\eta}\grad\tp{\mathcal{T}_\be^{-1}\tp{\eta}}
        \end{equation}
        is well-defined and real analytic.
        \item For any $\ep\in\tp{0,1}$ the restriction $\mathcal{F}_\be:\bf{U}_\be^\ep\cap\tp{C^0\cap L^\infty\cap\dot{W}^{1,r}}\tp{\R^2}\to L^2_\rho\tp{\R^2;\R^2}$ maps bounded sets to bounded sets.
        \item For any $\ep\in(0,1)$ and $\N\ni s\ge 2$ the restriction $\mathcal{F}_\be:\bf{U}_\be^\ep\cap\tilde{H}^s_\rho\tp{\R^2}\to L^2_\rho\tp{\R^2;\R^2}$ maps bounded sets to precompact sets.
    \end{enumerate}
\end{prop}
\begin{proof}
    Due to the properties of the bi-analytic change of unknowns $\mathcal{T}_\be$ from the third, fourth, and fifth items of Lemma~\ref{lem on free surface domsains, solitary case} we may reduce to studying the map $\mathcal{G}_\be:\bf{O}_\be\cap\tp{C^0\cap L^\infty\cap\dot{W}^{1,r}}\tp{\R^2}\to L^2_\rho\tp{\R^2;\R^2}$ given by
    \begin{equation}
        \mathcal{G}_\be\tp{\eta}=\mathcal{F}_\be\circ\mathcal{T}_\be\tp{\eta}=\Upomega_{\phi}\tp{\eta}+\grad\tp{\tp{1+\beta+\eta}\Upomega_\psi\tp{\eta}}+\Upomega_\tau\tp{\eta}\grad\eta.
    \end{equation}
    Our first claim is that $\mathcal{G}_\be$ is well-defined, real analytic, and obeys the restriction bounded sets mapping property. We may proceed with the same strategy as that of Proposition~\ref{prop on data map, periodic case}. The unital Banach algebra $\tp{C^0\cap L^\infty\cap\dot{W}^{1,r}}\tp{\R^2;\C}$ is admissible thanks to Proposition~\ref{prop on preliminary admissibility in the solitary case} and so Corollary~\ref{coro on compound analytic superposition real case} is in play and tells us that each of the maps $\Upomega_\phi$, $\Upomega_\psi$, and $\Upomega_\tau$ are real analytic and for every $\ep\in\tp{0,1}$ they map bounded subsets of $\bf{U}^\ep_\be\cap\tp{C^0\cap L^\infty\cap\dot{W}^{1,r}}\tp{\R^2}$ to bounded subsets of their respective image spaces. The map $\mathcal{G}_\be$ then combines these compound Nemytskii operators via simple bounded products. 

    By using these established properties of $\mathcal{G}_\be$, the identity $\mathcal{F}_\be=\mathcal{G}_\be\circ\mathcal{T}_\be^{-1}$, and the mapping properties of $\mathcal{T}_\be^{-1}$ from Lemma~\ref{lem on free surface domsains, solitary case}, we readily deduce the first and second items.

    The third item is a consequence of the compact embedding $\bf{U}^\ep_\be\cap\tilde{H}^s_\rho\tp{\R^2}\emb\bf{U}^{\ep/2}_\be\cap\tp{C^0\cap L^\infty\cap\dot{W}^{1,2}}\tp{\R^2}$ (which holds due to the third item of Proposition~\ref{prop on inclusion relations}) and the first and second items.
\end{proof}

\subsection{Operator decomposition}\label{subsection on operator decomposition}

Our goal now is to reformulate the stationary, variable bathymetry, viscous shallow water system~\eqref{stationary nondimensionalized equations} with~\eqref{percolate propane pulvarized paintbrush} as an equation for an operator acting between open subsets of Banach spaces.  We then decompose this operator in a semilinear fashion - writing it as the sum of a top order linear part and a properly nonlinear and lower order remainder. For technical reasons, we do not do this in a unified way; rather, the periodic and solitary settings are given slightly different operators and decompositions.

We first discuss the periodic setting. Recall from Lemma~\ref{lem on divergence equation reformulation} that the operator $P_0$ is the projection onto the space of mean zero functions and the periodic data maps $\mathcal{F}_\be$ are given in Proposition~\ref{prop on data map, periodic case}.
\begin{prop}[Operator formulation - periodic case]\label{prop on operator formulation periodic case}
    Let $\be\in C^\infty\tp{\T^2_L}$ satisfy $\min\be=0$, $A>0$, $G\ge0$, and $\kappa\in[0,\infty)$. Consider the mapping $\mathcal{Q}_\be:H^2\tp{\T^2_L;\R^2}\times [\bf{O}_\be\cap \z{H}^3\tp{\T^2_L}] \to\z{H}^1\tp{\T^2_L}\times L^2\tp{\T^2_L;\R^2}$ defined via
    \begin{equation}\label{operator encoding periodic case}
        \mathcal{Q}_\be\tp{u,\eta}=\tp{\grad\cdot u+P_0\tp{u\cdot\grad\log\tp{1+\be+\eta}},u\cdot\grad u+A\tp{1+\be+\eta}^{-1}u-\grad\cdot\mathbb{S}u-\mathbb{S}u\grad\log\tp{1+\be+\eta}+\tp{G-\Delta}\grad\eta}.
    \end{equation}
    The following hold.
    \begin{enumerate}
        \item $\mathcal{Q}_\be$ is well-defined and real-analytic.
        \item For all $\tp{u,\eta}\in H^2\tp{\T^2_L;\R^2}\times [\bf{O}_\be\cap \z{H}^3\tp{\T^2_L}]$ the identity $\mathcal{Q}_\be\tp{u,\eta}=(0,\kappa\mathcal{F}_\be\tp{\eta})$ holds if and only if system~\eqref{stationary nondimensionalized equations} is satisfied with velocity $\pmb{u}=u$, free surface $\pmb{\eta}=\eta$, and forcing $\Phi=\tp{1+\be+\eta}\mathcal{F}_\be(\eta)$.
    \end{enumerate}
\end{prop}
\begin{proof}
    The operator $\mathcal{Q}_\be$ is composed of linear terms and of products of the functions $u$, $\eta$, $\log\tp{1+\be+\eta}$, and $\tp{1+\be+\eta}^{-1}$,  so the first item follows as soon as we verify boundedness of these products and analyticity of the latter two nonlinearities. The combination of Corollary~\ref{coro on analytic superposition real case} and Proposition~\ref{prop on admissability in the periodic setting} shows that the mappings
    \begin{equation}
        \z{H}^3\tp{\T^2_L}\in\eta\mapsto\tp{1+\be+\eta}^{-1},\log\tp{1+\be+\eta}\in H^3\tp{\T^2_L}
    \end{equation}
    are real analytic. For the boundedness of the products of $\mathcal{Q}_\be$, we make repeated use Theorem~\ref{thm on high low product estimates} to see that for any $\zeta\in H^3\tp{\T^2_L}$ we have the estimates:
    \begin{equation}
        \tnorm{P_0\tp{u\cdot\grad\zeta}}_{H^1}\lesssim\tnorm{u}_{L^\infty}\tnorm{\grad\zeta}_{H^1}+\tnorm{u}_{H^1}\tnorm{\grad\zeta}_{L^\infty}\lesssim\tnorm{u}_{H^2}\tnorm{\zeta}_{H^3},
    \end{equation}
    \begin{equation}
        \tnorm{u\cdot\grad u}_{L^2}\lesssim\tnorm{u}_{L^\infty}\tnorm{\grad u}_{L^2}\lesssim\tnorm{u}_{H^2}^2,\quad\tnorm{\zeta u}_{L^2}\lesssim\tnorm{\zeta}_{H^3}\tnorm{u}_{H^2},
    \end{equation}
    \begin{equation}
        \tnorm{\mathbb{S}u\grad\zeta}_{L^2}\lesssim\tnorm{u}_{H^1}\tnorm{\grad\zeta}_{L^\infty}\lesssim\tnorm{u}_{H^2}\tnorm{\zeta}_{H^3}.
    \end{equation}
    This completes the proof of the first item.

    We now prove the second item.  By using Lemma~\ref{lem on divergence equation reformulation} and the bound $\min\tp{1+\be+\eta}>0$, we see that $\grad\cdot u+P_0\tp{u\cdot\grad\log\tp{1+\be+\eta}}=0$ if and only if $\grad\cdot u+u\cdot\grad\log\tp{1+\be+\eta}=0$ if and only if $\grad\cdot\tp{\tp{1+\be+\eta}u}=0$. To check the equivalence in the momentum equation we merely multiply the second component of $\mathcal{Q}_\be(u,\eta)$ by $(1+\be+\eta)$ and then perform algebraic manipulations until the left hand side of system~\eqref{stationary nondimensionalized equations} is revealed.
\end{proof}

In the next result we take the operator $\mathcal{Q}_\be$ of~\eqref{operator encoding periodic case} and write is as the sum of the linear operators $\mathcal{P}$ and $\mathcal{L}_\be$ from Proposition~\ref{prop on linear analysis with bathymetry - periodic case} and a compact nonlinear remainder.

\begin{prop}[Operator decomposition - periodic case]\label{prop on operator decomposition periodic case}
    Let $\be$, $A$, $G$, and $\mathcal{Q}_\be$ be as in Proposition~\ref{prop on operator formulation periodic case}. The following hold.
    \begin{enumerate}
        \item The map $\mathcal{N}_\be:W^{1,4}\tp{\T^2_L;\R^2}\times [\bf{O}_\be\cap\z{H}^2\tp{\T^2_L}] \to \z{H}^1\tp{\T^2_L}\times L^2\tp{\T^2_L;\R^2}$ given via
        \begin{equation}
            \mathcal{N}_\be\tp{u,\eta}=\bp{P_0\bp{u\cdot\grad\log\bp{1+\f{\eta}{1+\be}}},u\cdot\grad u-\f{Au\eta}{\tp{1+\be}\tp{1+\be+\eta}}-\mathbb{S}u\grad\log\bp{1+\f{\eta}{1+\be}}}
        \end{equation}
        is well-defined and real analytic; moreover, $\mathcal{N}_\be\tp{0,0}=0$ and $D\mathcal{N}_\be\tp{0,0}=0$.
        \item For every $\ep\in\tp{0,1}$ the restriction of $\mathcal{N}_\be$ to $W^{1,4}\tp{\T^2_L;\R^2}\times [\bf{O}_\be^\ep\cap \z{H}^2\tp{\T^2_L}]$ maps bounded sets to bounded sets.
        \item For every $\ep\in\tp{0,1}$ the restriction of $\mathcal{N}_\be$ to $H^2\tp{\T^2_L;\R^2}\times [\bf{O}_\be^\ep\cap \z{H}^3\tp{\T^2_L}]$ maps bounded sets to precompact sets.
        \item For all $(u,\eta)\in H^2\tp{\T^2_L;\R^2}\times [\bf{O}_\be\cap \z{H}^3\tp{\T^2_L}]$ we have the operator decomposition identity
        \begin{equation}\label{operator decomposition identity, periodic case}
            \mathcal{Q}_\be\tp{u,\eta}=\mathcal{P}\tp{u,\eta}+\mathcal{L}_\be\tp{u,\eta}+\mathcal{N}_\be\tp{u,\eta},
        \end{equation}
        where $\mathcal{Q}_\be$, $\mathcal{P}$, and $\mathcal{L}_\be$ are the operators given by~\eqref{operator encoding periodic case}, \eqref{periodic principal part}, and~\eqref{periodic bathymetry remainder}, respectively.
    \end{enumerate}
\end{prop}
\begin{proof}
    We prove the first and second items in tandem. The nonlinear maps that comprise $\mathcal{N}_\be$ are simple products of the functions $u$, $\eta$, $\tp{1+\be}^{-1}$, $\tp{1+\eta+\be}^{-1}$, and $\log\tp{1+\tp{1+\be}^{-1}\eta}$. The expressions $\tp{1 + \eta + \be}^{-1}$ and $\log\tp{1 + \tp{1 + \be}^{-1}\eta}$ are real analytic functions of $\eta$ valued in $H^2\tp{\T^2_L}$, as can be deduced from Proposition~\ref{prop on admissability in the periodic setting} and Corollary~\ref{coro on analytic superposition real case}, which also map bounded subsets of $\bf{O}_\be^\ep\cap\z{H}^2\tp{\T^2_L}$ to bounded subsets of $H^2\tp{\T^2_L}$. So given $\zeta\in \z{H}^2\tp{\T^2_L}$ and $u\in W^{1,4}\tp{\T^2_L;\R^2}$ it remains only to observe the following product bounds (which follow from Theorem~\ref{thm on high low product estimates} and the Sobolev embedding)
    \begin{equation}
        \tnorm{P_0\tp{u\cdot\grad\zeta}}_{H^1}\lesssim\tnorm{u}_{L^\infty}\tnorm{\zeta}_{H^2}+\tnorm{u}_{W^{1,4}}\tnorm{\zeta}_{W^{1,4}}\lesssim\tnorm{u}_{W^{1,4}}\tnorm{\zeta}_{H^2},
    \end{equation}
    \begin{equation}
        \tnorm{u\cdot\grad u}_{L^2}\lesssim\tnorm{u}_{L^4}\tnorm{u}_{W^{1,4}}\le\tnorm{u}^2_{W^{1,4}},\quad\tnorm{\zeta u}_{L^2}\le\tnorm{\zeta}_{L^\infty}\tnorm{u}_{L^2}\lesssim\tnorm{\zeta}_{H^2}\tnorm{u}_{W^{1,4}},
    \end{equation}
    \begin{equation}
        \tnorm{\mathbb{S}u\grad\zeta}_{L^2}\lesssim\tnorm{u}_{W^{1,4}}\tnorm{\zeta}_{W^{1,4}}\lesssim\tnorm{u}_{W^{1,4}}\tnorm{\zeta}_{H^2}.
    \end{equation}
    To complete the proof of the first and second items we note that that the vanishing of $\mathcal{N}_\be$ and its derivative at the origin follow from elementary calculations.

    The third item is a consequence of the previous items and of the compactness of the embedding
    \begin{equation}
        H^2\tp{\T^2_L;\R^2}\times [\bf{O}_\be^\ep\cap \z{H}^3\tp{\T^2_L}] \emb W^{1,4}\tp{\T^2_L;\R^2} \times [\bf{O}_\be^{\ep/2}\cap \z{H}^2\tp{\T^2_L}].
    \end{equation}
    The fourth item is another elementary calculation which follows from the identities
    \begin{equation}
        \f{1}{1+\be}-\f{\eta}{\tp{1+\be}\tp{1+\be+\eta}}=\f{1}{1+\be+\eta},\quad\log\tp{1+\be}+\log\tp{1+\tp{1+\be}^{-1}\eta}=\log\tp{1+\be+\eta}.
    \end{equation}
\end{proof}

We next turn our attention to the solitary setting and derive analogs of Proposition~\ref{prop on operator formulation periodic case} and~\ref{prop on operator decomposition periodic case}. Recall that the solitary case's data map $\mathcal{F}_\be$ is given in Proposition~\ref{prop on data map - solitary case}. Note also that we are again using the notation for solenoidal vector fields introduced in~\eqref{the definition of the solenoidal space}.

\begin{prop}[Operator formulation - solitary case]\label{prop on operator formulation solitary case}
    Let $\del,\rho\in\tp{0,1}$, $\be\in{_{\m{e}}}\tilde{H}^\infty_\del\tp{\R^2}$ with $\inf\be=0$, $A,G>0$, and $\kappa\in[0,\infty)$. Consider the mapping $\mathcal{Q}_\be:{_{\m{sol}}}H^2_{\rho}\tp{\R^2;\R^2}\times [\bf{U}_\be\cap\tilde{H}^3_\rho\tp{\R^2}] \to L^2_\rho\tp{\R^2;\R^2}$ defined via
    \begin{multline}\label{the Quarterback}
        \mathcal{Q}_\be(u,\eta)=\grad\cdot\bp{\f{u\otimes u}{\sqrt{2\eta+\tp{1+\be}^2}}}+\f{Au}{\sqrt{2\eta+\tp{1+\be}^2}}-\grad\cdot\bp{\sqrt{2\eta+\tp{1+\be}^2}\mathbb{S}\bp{\f{u}{\sqrt{2\eta+\tp{1+\be}^2}}}}\\+\sqrt{2\eta+\tp{1+\be}^2}\tp{G-\Delta}\grad\tp{\sqrt{2\eta+\tp{1+\be}^2}-\tp{1+\be}}.
    \end{multline}
    The following hold.
    \begin{enumerate}
        \item $\mathcal{Q}_\be$ is well-defined and real analytic.
        \item For all $\tp{u,\eta}\in H^2_\rho\tp{\R^2;\R^2}\times [\bf{U}_\be\cap\tilde{H}^3_\rho\tp{\R^2}]$ the identity $\mathcal{Q}_\be\tp{u,\eta}=\kappa\mathcal{F}_\be\tp{\eta}$ holds if and only if system~\eqref{stationary nondimensionalized equations} is satisfied with velocity $\pmb{u}=\f{u}{\sqrt{2\eta+\tp{1+\be}^2}}$, free surface $\pmb{\eta}=\sqrt{2\eta+\tp{1+\be}^2}-\tp{1+\be}$, and forcing $\Phi=\mathcal{F}_\be\tp{\eta}$.
    \end{enumerate}
\end{prop}
\begin{proof}
    We observe that $\mathcal{Q}_\be(u,\eta)=\mathcal{R}_\be(u,\mathcal{T}_\be^{-1}\tp{\eta})$ where $\mathcal{T}_\be$ is the analytic change of unknowns mapping from Proposition~\ref{lem on free surface domsains, solitary case} and $\mathcal{R}_\be:{_{\m{sol}}}H^2_\rho\tp{\R^2;\R^2}\times [\bf{O}_\be\cap\tilde{H}^3_\rho\tp{\R^2}] \to L^2_\rho\tp{\R^2;\R^2}$ is the map
    \begin{equation}
        \mathcal{R}_\be\tp{u,\eta}=\grad\cdot\bp{\f{u\otimes u}{1+\be+\eta}}+\f{Au}{1+\be+\eta}-\grad\cdot\bp{\tp{1+\be+\eta}\mathbb{S}\bp{\f{u}{1+\be+\eta}}}+\tp{1+\be+\eta}\tp{G-\Delta}\grad\eta.
    \end{equation}
    Thus, to prove the first item it suffices to establish that $\mathcal{R}_\be$ is well-defined and analytic.  The map $\mathcal{R}_\be$ is comprised of linear terms and simple products interacting with the Nemytskii operator
    \begin{equation}
        \bf{O}_\be\cap{_{\m{e}}}\tilde{H}^3_\mu\tp{\R^2}\ni\eta\mapsto\tp{1+\be+\eta}^{-1}\in{_{\m{e}}}\tilde{H}^3_\mu\tp{\R^2},\quad\mu=\min\tcb{\del,\rho}\in\R^+,
    \end{equation}
    which is well-defined and analytic thanks to Proposition~\ref{prop on admissability of the extended weighted gradient spaces} and Corollary~\ref{coro on analytic superposition real case}. Therefore, it remains to study the boundedness of these simple products. We note that as a consequence of Propositions~\ref{prop on product estimates}, \ref{prop on equivalent norms and low mode embeddings}, and~\ref{prop on product estimates in weighted gradient spaces} we have the bounds for $\zeta,\theta\in{_{\m{e}}}\tilde{H}^3_\mu\tp{\R^2}$:
    \begin{equation}
        \tnorm{\grad\cdot\tp{\zeta u\otimes u}}_{L^2_\rho}\lesssim\tnorm{\zeta}_{{_{\m{e}}}\tilde{H}^3_\mu}\tnorm{u}_{H^2_\rho}^2,\quad\tnorm{\zeta u}_{L^2_\rho}\lesssim\tnorm{\zeta}_{{_{\m{e}}}\tilde{H}^3_\mu}\tnorm{u}_{H^2_\rho},
    \end{equation}
    \begin{equation}
        \tnorm{\grad\cdot\tp{\theta\mathbb{S}\tp{\zeta u}}}_{L_\rho^2}\lesssim\tnorm{\theta}_{W^{1,\infty}}\tnorm{\mathbb{S}\tp{\zeta u}}_{H^1_\rho}\lesssim\tnorm{\theta}_{{_{\m{e}}}\tilde{H}^3_\mu}\tnorm{\zeta}_{{_{\m{e}}}\tilde{H}^3_\mu}\tnorm{u}_{H^2_\rho},
    \end{equation}
    \begin{equation}
        \tnorm{\theta\tp{G-\Delta}\grad\eta}_{L^2_\rho}\lesssim\tnorm{\theta}_{L^\infty}\tnorm{\grad\eta}_{H^2}\lesssim\tnorm{\theta}_{{_{\m{e}}}\tilde{H}^3_\mu}\tnorm{\eta}_{\tilde{H}^3_\rho}.
    \end{equation}
    These bounds then conclude the proof of the first item.

    The second item follows directly from the simple calculation $\mathcal{Q}_\be\tp{u,\eta}=\mathcal{R}_\be\tp{\tp{1+\be+\pmb{\eta}}\pmb{u},\pmb{\eta}}$.
\end{proof}

\begin{rmk}[Differences between the solitary and periodic cases]\label{remark on codomain disparity}
    Upon comparison of the operator formulation in the periodic case (Proposition~\ref{prop on operator formulation periodic case}) with that of the solitary case (Proposition~\ref{prop on operator formulation solitary case}), one notices that there are the following important distinctions. First, in the solitary case the domain of $\mathcal{Q}_\be$ hard codes the velocity variable to be divergence free; this is not a feature of the periodic case. Second, the codomain of $\mathcal{Q}_\be$ in the periodic case includes the `extra' factor $\z{H}^1\tp{\T^2_L}$ in the Cartesian product of function spaces to tabulate the source term of the continuity equation; the solitary case does not require such an addition to the codomain (due to the built in restriction to solenoidal velocity fields).
\end{rmk}

We remind the reader that the linear operators $\mathcal{P}_\be$ and $\mathcal{L}_\be$ of the solitary case are studied in Proposition~\ref{prop on linear analysis with bathymetry solitary case}, where the quadratic forms $Q_m$ for $m\in\tcb{1,2}$ also make their first appearance.  Our next result relates these linear maps to $\mathcal{Q}_\be$ and a compact nonlinear remainder. The corresponding result in the periodic case is Proposition~\ref{prop on operator decomposition periodic case}.

\begin{prop}[Operator decomposition - solitary case]\label{prop on operator decomposition - solitary case}
    Let $\del$, $\rho$, $\be$, $A$, $G$, and $\mathcal{Q}_\be$ be as in Proposition~\ref{prop on operator formulation solitary case}. The following hold.
    \begin{enumerate}
        \item For $\rho/2<\mu<\rho$ the map $\mathcal{N}_\be:\tp{{_{\m{sol}}}H^1_\mu\cap L^\infty_\mu}\tp{\R^2;\R^2}\times [\bf{U}_\be\cap\tp{\tilde{H}_\mu^2\cap W^{1,\infty}_\mu}\tp{\R^2}] \to L^2_\rho\tp{\R^2;\R^2}$ given by
        \begin{multline}\label{nonlinearity in the solitary case}
            \mathcal{N}_\be\tp{u,\eta}=\grad\cdot\bp{\f{u\otimes u}{\sqrt{2\eta+\tp{1+\be}^2}}}-\f{2Au\eta}{\tp{1+\be}\sqrt{2\eta+\tp{1+\be}^2}\tp{\tp{1+\be}+\sqrt{2\eta+\tp{1+\be}^2}}}\\
            +\f12\grad\cdot Q_2\bp{u,\grad\log\bp{1+\f{2\eta}{\tp{1+\be}^2}}}+\f{\eta^2\tp{G-\Delta}\grad\tp{1+\be}}{\tp{1+\be}^2\sqrt{2\eta+\tp{1+\be}^2}+\tp{1+\be}\tp{\tp{1+\be}^2+\eta}}\\
            \f{1}{2}\grad\cdot Q_1\bp{\grad\bp{\f{2\eta}{\sqrt{2\eta+\tp{1+\be}^2}+\tp{1+\be}}}}\\-\grad\cdot Q_1\bp{\grad\tp{1+\be},\grad\bp{\f{\eta^2}{\tp{1+\be}^2\sqrt{2\eta+\tp{1+\be}^2}+\tp{1+\be}\tp{\tp{1+\be}^2+\eta}}}}.
        \end{multline}
        is well-defined and real analytic and satisfies $\mathcal{N}_\be(0,0)=0$ and $D\mathcal{N}_\be\tp{0,0}=0$.
        \item For every $\ep\in\tp{0,1}$ the restriction of $\mathcal{N}_\be$ to $\tp{{_{\m{sol}}}H^1_\mu\cap L^\infty_\mu}\tp{\R^2;\R^2}\times [\bf{U}^\ep_\be\cap\tp{\tilde{H}^2_\mu\cap W^{1,\infty}_\mu}\tp{\R^2}]$ maps bounded sets to bounded sets.
        \item For every $\ep\in\tp{0,1}$ the restriction of $\mathcal{N}_\be$ to ${_{\m{sol}}}H^2_\rho\tp{\R^2;\R^2}\times [\bf{U}_\be^\ep\cap\tilde{H}^3_\rho\tp{\R^2}]$ maps bounded sets to precompact sets.
        \item For all $(u,\eta)\in{_{\m{sol}}}H^{2}_\rho\tp{\R^2;\R^2}\times [\bf{U}_\be\cap\tilde{H}^3_\rho\tp{\R^2}]$ we have the operator decomposition identity
        \begin{equation}
            \mathcal{Q}_\be\tp{u,\eta}=\mathcal{P}_\be\tp{u,\eta}+\mathcal{L}_\be\tp{u,\eta}+\mathcal{N}_\be\tp{u,\eta},
        \end{equation}
        where $\mathcal{Q}_\be$, $\mathcal{P}_\be$, and $\mathcal{L}_\be$ are the operators given by~\eqref{the Quarterback}, \eqref{solitary principal part map}, and~\eqref{solitary remainder map}, respectively.
    \end{enumerate}
\end{prop}
\begin{proof}
    We again prove the first and second items in tandem. The nonlinearity $\mathcal{N}_\be$ contains simple products of the velocity, free surface, and analytic nonlinear expressions in the free surface. We shall first study these latter expressions. Set $\chi=\min\tcb{\mu,\del}\in\tp{0,1}$ and consider the auxiliary mappings
    \begin{equation}
        \bf{U}_\be\cap\tp{{_{\m{e}}}\tilde{H}^2_\chi\cap W^{1,\infty}_\chi}\tp{\R^2}\ni\eta\mapsto\begin{cases}
            f_0(\eta)=\sqrt{2\eta+\tp{1+\be}^2},\\
            f_1(\eta)=1/f_0(\eta),\\
            f_2(\eta)=\tp{1+\be+f_0\tp{\eta}}^{-1},\\
            f_3(\eta)=\tp{\tp{1+\be}f_0\tp{\eta}+f_0(\eta)^2+\tp{1+\be}^2/2}^{-1}
        \end{cases}
        \in\tp{{_{\m{e}}}\tilde{H}^2_\chi\cap W^{1,\infty}_\chi}\tp{\R^2}.
    \end{equation}
    The space $\tp{{_{\m{e}}}\tilde{H}^2_\chi\cap W^{1,\infty}}\tp{\R^2;\C}$ is easily seen to be an admissible unital Banach algebra in the sense of Definition~\ref{defn of admissable Banach algebras} by combining Propositions~\ref{prop on preliminary admissibility in the solitary case} and~\ref{prop on admissability of the extended weighted gradient spaces}; consequently, an application of Corollary~\ref{coro on analytic superposition real case} shows that $f_0$ and $f_1$ are real analytic and, when restricted to their domain intersected with $\bf{U}_\be^\ep$, map bounded sets to bounded sets.  Since $\inf\be=0$ and $f_0>0$ on its domain, we compute the lower bounds $1/f_2> 1$ and $1/f_3>1/2$,  which mean that the denominators in $f_2$ and $f_3$ stay well away from the singularity of the pointwise inversion map. Hence, we can use Corollary~\ref{coro on analytic superposition real case} again along with the preservation of real analyticity under composition to deduce that $f_2$ and $f_3$ are well-defined, real analytic, and when restricted to their domain intersected with $\bf{U}_\be^\ep$, map bounded sets to bounded sets.  This Nemytskii operator analysis is sufficient to handle all but one term in~\eqref{nonlinearity in the solitary case}, namely the third one on the right hand side. In this instance we consider the mapping
    \begin{equation}\label{the log nonlinearity}
        \bf{U}_\be\cap\tp{\tilde{H}^2_\mu\cap W^{1,\infty}_\mu}\tp{\R^2}\ni\eta\mapsto\log\bp{1+\f{2\eta}{\tp{1+\be}^2}}\in\tp{\tilde{H}^2_\mu\cap W^{1,\infty}_\mu}\tp{\R^2}.
    \end{equation}
    By quoting the same results as previously mentioned, the boundedness of the product map 
    \begin{equation}
        \tp{{_{\m{e}}}\tilde{H}^2_\del\cap W^{1,\infty}_\del}\tp{\R^2}\times\tp{\tilde{H}^2_\mu\cap W^{1,\infty}_\mu}\tp{\R^2}\to\tp{\tilde{H}^2_\mu\cap W^{1,\infty}_\mu}\tp{\R^2},
    \end{equation}
    which is a consequence of Proposition~\ref{prop on product estimates in weighted gradient spaces}, along with the fact that $\log(1)=0$ we deduce that the mapping~\eqref{the log nonlinearity} is well-defined, analytic, and maps bounded sets to bounded sets when restricted to the intersection of its domain and $\bf{U}^\ep_\be$.

    The previous analysis then allows us to reduce the proof of the first and second items to a series of simple product bounds. Suppose that $u\in\tp{{_{\m{sol}}}H^1_\mu\cap L^\infty_\mu}\tp{\R^2;\R^2}$, $\eta\in\tp{\tilde{H}^2_\mu\cap W^{1,\infty}_\mu}\tp{\R^2}$, and $\zeta\in\tp{{_{\m{e}}}\tilde{H}^2_\chi\cap W^{1,\infty}_\chi}$. Through repeated applications of the product bounds of Propositions~\ref{prop on product estimates} and~\ref{prop on product estimates in weighted gradient spaces} we deduce that
    \begin{equation}\label{saturn eating his son 1}
        \tnorm{\grad\cdot\tp{\zeta u\otimes u}}_{L^2_\rho}\lesssim\tnorm{\zeta}_{W^{1,\infty}}\tnorm{u}_{L^\infty_{\mu}}\tnorm{u}_{H^1_\mu},\quad\tnorm{\zeta u\eta}_{L^2_\rho}\lesssim\tnorm{\zeta}_{L^\infty}\tnorm{u}_{L^2_\mu}\tnorm{\eta}_{L^\infty_\mu},
    \end{equation}
    \begin{equation}\label{saturn eating his son 2}
        \tnorm{\grad\cdot Q_2\tp{u,\grad\eta}}_{L^2_\rho}\lesssim\tnorm{u}_{L^\infty_\mu}\tnorm{\grad\eta}_{H^1_\mu}+\tnorm{u}_{H^1_\mu}\tnorm{\grad\eta}_{L^\infty_\mu},
    \end{equation}
    \begin{equation}\label{saturn eating his son 3}
        \tnorm{\zeta\eta^2\tp{G-\Delta}\grad\tp{1+\be}}_{L^2_\rho}\lesssim\tnorm{\zeta}_{L^\infty}\tnorm{\eta}_{L^\infty_\mu}^2,
    \end{equation}
    \begin{equation}\label{saturn eating his son 4}
        \tnorm{\grad\cdot Q_1\tp{\grad\tp{\zeta\eta}}}_{L^2_\rho}\lesssim\tnorm{\grad\tp{\zeta\eta}}_{L^\infty_\mu}\tnorm{\grad\tp{\zeta\eta}}_{H^1_\mu}\lesssim\tnorm{\zeta}_{W^{1,\infty}}\tnorm{\eta}_{W^{1,\infty}_\mu}\tnorm{\zeta}_{{_{\m{e}}}\tilde{H}^2_\chi}\tnorm{\eta}_{\tilde{H}^2_\mu},
    \end{equation}
    \begin{equation}\label{saturn eating his son 5}
        \tnorm{\grad\cdot Q_1\tp{\grad\tp{1+\be},\grad\tp{\zeta\eta^2}}}_{L^2_\rho}\lesssim\tnorm{\grad\tp{\zeta\eta^2}}_{H^1_\rho}\lesssim\tnorm{\zeta}_{{_{\m{e}}}\tilde{H}^2_\chi}\tnorm{\eta}^2_{\tilde{H}^2_\mu}.
    \end{equation}
    We then read off by the embeddings of the weighted spaces (Propositions~\ref{prop on equivalent norms and low mode embeddings} and~\ref{prop on inclusion relations}) that each of the polynomial nonlinearities above (given within the norms on the left hand sides of~\eqref{saturn eating his son 1}, \eqref{saturn eating his son 2}, \eqref{saturn eating his son 3}, \eqref{saturn eating his son 4}, and~\eqref{saturn eating his son 5}) are a continuous function of $u$, $\eta$, $\zeta$ for the norms on their respective spaces. By synthesizing the previous work we complete the proof of the first and second items.

    The third item follows from the first two items and the compactness of the embedding
    \begin{equation}
        {_{\m{sol}}}H^2_\rho\tp{\R^2;\R^2}\times\bf{U}^\ep_\be\cap\tilde{H}^3_\rho\tp{\R^2}\emb\tp{{_{\m{sol}}}\tilde{H}^1_\mu\cap L^\infty_\mu}\tp{\R^2;\R^2}\times\bf{U}^{\ep/2}_\be\cap\tp{\tilde{H}^2_\mu\cap W^{1,\infty}_\mu}\tp{\R^2}.
    \end{equation}

    Finally, we prove the fourth item.      Recall that in the proof of Proposition~\ref{prop on linear analysis with bathymetry solitary case} we computed 
    \begin{equation}
        \tp{\mathcal{P}_\be+\mathcal{L}_\be}\tp{u,\eta}=\f{Au}{1+\be}-\grad\cdot\mathbb{S}u+\grad\cdot Q_2\tp{\grad\log\tp{1+\be},u}+\tp{1+\be}\tp{G-\Delta}\grad\bp{\f{\eta}{1+\be}}.
    \end{equation}
    To add $\mathcal{N}_\be$ to the above expression and result in $\mathcal{Q}_\be$, we observe the following identities:
    \begin{equation}
        \f{1}{1+\be}-\f{2\eta}{\tp{1+\be}\sqrt{2\eta+\tp{1+\be}^2}\tp{\tp{1+\be}+\sqrt{2\eta+\tp{1+\be}^2}}}=\f{1}{\sqrt{2\eta+\tp{1+\be}^2}},
    \end{equation}
    \begin{equation}
        \log\tp{1+\be}+\f12\log\bp{1+\f{2\eta}{\tp{1+\be}^2}}=\log\sqrt{2\eta+\tp{1+\be}^2},
    \end{equation}
    \begin{equation}
        -\grad\cdot\mathbb{S}u+\grad\cdot Q_2\tp{u,\grad\log\sqrt{2\eta+\tp{1+\be}^2}}=-\grad\cdot\bp{\sqrt{2\eta+\tp{1+\be}^2}\mathbb{S}\bp{\f{u}{\sqrt{2\eta+\tp{1+\be}^2}}}}.
    \end{equation}
    These are sufficient to match the drag term, the viscous stress tensor, and the advective derivative pieces of the expression for $\mathcal{Q}_\be$~\eqref{the Quarterback}.  It remains to study the gravity-capillary term (which is the final summand in~\eqref{the Quarterback}). For this we crucially use the following divergence identity which holds for scalars $\phi$:
    \begin{equation}\label{the divergence of the gravity capillary stress tensor}
        \phi\tp{G-\Delta}\grad\phi=\f12\tp{G-\Delta}\grad\tp{\phi^2}+\f12\grad\cdot Q_1\tp{\grad\phi,\grad\phi}.
    \end{equation}
    This allows us to write
    \begin{multline}\label{decomposition of the gravity capillary stress tensor}
        \sqrt{2\eta+\tp{1+\be}^2}\tp{G-\Delta}\grad\tp{\sqrt{2\eta+\tp{1+\be}^2}-\tp{1+\be}}=\tp{G-\Delta}\grad\eta\\
        +\grad\cdot Q_1\tp{\grad\tp{1+\be},\grad\tp{\sqrt{2\eta+\tp{1+\be}^2}-\tp{1+\be}}}+\f12\grad\cdot Q_1\tp{\grad\tp{\sqrt{2\eta+\tp{1+\be}^2}-\tp{1+\be}}}\\
        -\tp{\sqrt{2\eta+\tp{1+\be}^2}-\tp{1+\be}}\tp{G-\Delta}\grad\tp{1+\be}.
    \end{multline}
    Notice that the first term on the right hand side of~\eqref{decomposition of the gravity capillary stress tensor} is the principal linear part's gravity-capillary contribution. To remove the lower order bathymetric linear remainder we rewrite the second and fourth terms on the right side of~\eqref{decomposition of the gravity capillary stress tensor}, employing the identity 
    \begin{multline}
        \sqrt{2\eta+\tp{1+\be}^2}-\tp{1+\be}=\f{2\eta}{\sqrt{2\eta+\tp{1+\be}^2}+\tp{1+\be}}\\
        =\f{\eta}{1+\be}-\f{\eta^2}{\tp{1+\be}^2\sqrt{2\eta+\tp{1+\be}^2}+\tp{1+\be}\tp{\tp{1+\be}^2+\eta}}
    \end{multline}
    to leave behind only terms that are linear in $\eta$ or at least quadratic in $\eta$.  Once this is done, the identity $\mathcal{P}_\be+\mathcal{L}_\be+\mathcal{N}_\be=\mathcal{Q}_\be$ follows readily.   
\end{proof}

\subsection{Fixed point reformulation and consequences}\label{subsection on fixed point reformulation}

Our goal now is to use the analysis developed in the previous sections to recast the system \eqref{stationary nondimensionalized equations} with~\eqref{percolate propane pulvarized paintbrush} as a fixed point equation for a compact perturbation of the identity. Once this is done, we will be able to read off some important consequences. We begin with a definition that serves to unify the periodic and solitary cases at an abstract level.  In it we will use the invertibility of the linear maps $\mathcal{P}+\mathcal{L}_\be$ and $\mathcal{P}_\be+\mathcal{L}_\be$, which are consequences of Propositions~\ref{prop on linear analysis with bathymetry - periodic case} and~\ref{prop on linear analysis with bathymetry solitary case}, the data map notation $\mathcal{F}_\be$ form Propositions~\ref{prop on data map, periodic case} and~\ref{prop on data map - solitary case}, and the nonlinear map $\mathcal{N}_\be$ notation from Propositions~\ref{prop on operator decomposition periodic case} and~\ref{prop on operator decomposition - solitary case}.

\begin{defn}[Periodic and solitary unification]\label{defn of periodic and solitary unification}
    We make the following definitions which have an implicit dependence on either the periodic or solitary cases.
    \begin{enumerate}
        \item The set of admissible drag, gravity, and bathymetry parameters is given by
        \begin{equation}
            \pmb{\mathcal{A}}=\begin{cases}
                \tcb{(A,G,\be)\;:\;A>0,\;G\ge0,\;\be\in C^\infty(\T^2_L),\;\min\be=0}&\text{in the }\T^2_L\text{ case},\\
                \tcb{(A,G,\be)\;:\;A,G>0,\;\exists\;\del\in(0,1),\;\be\in{_{\m{e}}}\tilde{H}^\infty_\del\tp{\R^2},\;\inf\be=0}&\text{in the }\R^2\text{ case}.
            \end{cases}
        \end{equation}
        Henceforth we shall assume that the above inclusions are satisfied in their respective cases.
        \item We define the domain Banach spaces (recall~\eqref{the definition of the solenoidal space})
        \begin{equation}\label{the spaces for the solutions in a unified way}
            \pmb{X}=\begin{cases}
                H^2\tp{\T^2_L;\R^2}\times \z{H}^3\tp{\T^2_L}&\text{in the }\T^2_L\text{ case},\\
                {_{\m{sol}}}H^2_{\scriptscriptstyle 1/2}\tp{\R^2;\R^2}\times\tilde{H}^3_{\scriptscriptstyle 1/2}\tp{\R^2}&\text{in the }\R^2\text{ case}.
            \end{cases}
        \end{equation}
        \item We endow the domain Banach spaces with a (Lipschitz) continuous functional $\varrho:\pmb{X}\to\R$ given by  
        \begin{equation}
             \varrho(u,\eta) = \begin{cases}
                \min(1+\be+\eta)&\text{in the }\T^2_L\text{ case},\\
                \inf((1+\be)^2+2\eta)&\text{in the }\R^2\text{ case}.
            \end{cases}
        \end{equation}
        We then define the open sets $U=\varrho^{-1}\tp{0,\infty}$ and, for $\ep\in\tp{0,1}$, $U_\ep=\varrho^{-1}\tp{\ep,\infty}$.
        
        \item Finally, we define the perturbation operator $\mathcal{K}:U\times\R\to\pmb{X}$ via
        \begin{equation}\label{Nymphs}
            \mathcal{K}(u,\eta,\kappa) = \begin{cases}
                \tp{\mathcal{P}+\mathcal{L}_\be}^{-1}\tp{\mathcal{N}_\be\tp{u,\eta}-\kappa\tp{0,\mathcal{F}_\be\tp{\eta}}}&\text{in the }\T^2_L\text{ case},\\
                \tp{\mathcal{P}_\be+\mathcal{L}_\be}^{-1}\tp{\mathcal{N}_\be\tp{u,\eta}-\kappa\mathcal{F}_\be\tp{\eta}}&\text{in the }\R^2\text{ case}.
            \end{cases}
        \end{equation}
    \end{enumerate}
\end{defn}

\begin{rmk}
In equation~\eqref{Nymphs} (and later in equation~\eqref{Persephone}) we observe that in the periodic case the form of the forcing is $\kappa\tp{0,\mathcal{F}_\be\tp{\eta}}$ whereas in the solitary case we only have $\kappa\mathcal{F}_\be\tp{\eta}$. This discrepancy stems from our analysis of the periodic case requiring bookkeeping of source terms for the continuity equation. See also Remark~\ref{remark on codomain disparity}.
\end{rmk}

We may now summarize our findings of Sections~\ref{section on linear analysis} and~\ref{section on semilinearization} with the following result. Recall that the operators $\mathcal{Q}_\be$ are studied in the periodic and solitary cases in Propositions~\ref{prop on operator formulation periodic case} and~\ref{prop on operator formulation solitary case}, respectively.

\begin{thm}[Unified fixed point reformulation]\label{thm on unified fixed point reformulation}
    The following hold.
    \begin{enumerate}
        \item The operator $\mathcal{K}:U\times\R\to\pmb{X}$ from the fourth item of Definition~\ref{defn of periodic and solitary unification} is well-defined and real analytic, and satisfies $\mathcal{K}(0,0)=0$ and $D_1\mathcal{K}(0,0)=0$, where $D_1$ is the derivative with respect to the $\pmb{X}$-factor in the domain.
        \item For every $\ep\in(0,1)$ the restricted operator $\mathcal{K}:U_\ep\times\R\to\pmb{X}$ maps bounded sets to precompact sets; moreover, at every $(x,\kappa)\in U_\ep\times \R$ the linear map $D\mathcal{K}\tp{x,\kappa}:\pmb{X}\to\pmb{X}$ is compact.
        \item For all $\tp{x,\kappa}\in U\times\R$ we have $x+\mathcal{K}\tp{x,\kappa}=0$ if and only if $x=\tp{u,\eta}\in\pmb{X}$ with
        \begin{equation}\label{Persephone}
            \begin{cases}
                \mathcal{Q}_\be\tp{u,\eta}=\kappa\tp{0,\mathcal{F}_\be\tp{\eta}}&\text{in the }\T^2_L\text{ case},\\
                \mathcal{Q}_\be\tp{u,\eta}=\kappa\mathcal{F}_\be\tp{\eta}&\text{in the }\R^2\text{ case}.
            \end{cases}
        \end{equation}
        
        \item For all $(x,\kappa)\in  U\times\R$, the linear operator $I_{\pmb{X}}+D_1\mathcal{K}(x,\kappa):\pmb{X}\to\pmb{X}$ is Fredholm with index zero.

        \item  Denote the set of solutions by
    \begin{equation}\label{the set of solutions}
        S=\tcb{\tp{x,\kappa}\in U\times\R\;:\;x+\mathcal{K}\tp{x,\kappa}=0}.
    \end{equation}
    If $E \subseteq S$ satisfies
    \begin{equation}\label{the alternative control quantity}
       \sup\{\tnorm{x}_{\pmb{X}}+|\kappa|+1/\varrho\tp{x} \;:\: (x,\kappa) \in E \} <\infty,
    \end{equation}
        then $\Bar{E} \subset \pmb{X}\times\R$ is a compact subset of $U\times\R$.
    \end{enumerate}
\end{thm}
\begin{proof}
    Theorem 1.40 in Schwartz~\cite{MR433481} establishes that the derivatives of compact maps are compact linear operators. Hence, in the periodic case the first and second items are a synthesis of Propositions~\ref{prop on linear analysis with bathymetry - periodic case}, \ref{prop on data map, periodic case}, and~\ref{prop on operator decomposition periodic case}. On the other hand, in the solitary case we combine Propositions~\ref{prop on linear analysis with bathymetry solitary case}, \ref{prop on data map - solitary case}, and~\ref{prop on operator decomposition - solitary case}.

    The third item is a consequence of the operator formulation Propositions~\ref{prop on operator formulation periodic case} and~\ref{prop on operator formulation solitary case} in addition to the final items of the operator decomposition Propositions~\ref{prop on operator decomposition periodic case} and~\ref{prop on operator decomposition - solitary case}.

    To prove the fourth item fix $(x,\kappa)\in U\times\R$ and note that there exists $\epsilon\in\tp{0,1}$ such that $(x,\kappa)\in U_\epsilon\times\R$. The second item then implies that that $D_1\mathcal{K}(x,\kappa)$ is a compact operator on $\pmb{X}$, and hence the sum $I_{\pmb{X}}+D_1\mathcal{K}\tp{x,\kappa}$ is a Fredholm operator with index zero.
    
     We now prove the fifth item.  The bound~\eqref{the alternative control quantity} ensures the existence of $\ep\in\tp{0,1/2}$ such that $\tcb{x : (x,\kappa) \in E} \subset U_{2\ep}$, and hence $\Bar{E} \subset U_\ep\times\R$.  The compactness of $\mathcal{K}:U_\ep\times\R\to\pmb{X}$ then guarantees that $\Bar{\mathcal{K}(E)}$  is a compact set.  By definition, $x=-\mathcal{K}\tp{x,\kappa}$ for every $(x,\kappa) \in E$, so $\Bar{\tcb{x \;:\; (x,\kappa) \in E}}\subset\pmb{X}$ is  compact as well.  Finally, the set $\Bar{\{\kappa \;:\; (x,\kappa) \in E\}}$ is compact by Heine-Borel.  The fifth item follows.
    
\end{proof}

\section{Large solutions and blowup scenarios}\label{section on large solutions and blow up criteria}

Our aim now is to combine the semilinearization analysis of Section~\ref{section on semilinearization} with an analytic global implicit function theorem to produce curves of solutions in~\eqref{the set of solutions}, which will ultimately lead to the existence of large solutions.  In addition, we prove a host of a priori estimates for solutions to~\eqref{stationary nondimensionalized equations} that provide a refined understanding of the large solutions.

\subsection{Analytic global implicit function theorem and solution curves}\label{subsection on analytic global implicit function theorem and solution curves}

One of our main tools to study the solution set~\eqref{the set of solutions} from Theorem~\ref{thm on unified fixed point reformulation} is the analytic global implicit function theorem of Chen, Walsh, and Wheeler, Theorem B.1 in~\cite{chen2023globalbifurcationmonotonefronts}. We state a minor retooling of this result here, which allows for more customization in the `blowup' scenario.

\begin{thm}[Analytic global implicit function theorem]\label{thm on analytic global implicit function theorem}
    Let $\mathscr{W}$ and $\mathscr{Z}$ be Banach spaces, $\rho:\mathscr{W}\times\R\to\R$ a continuous function, $\mathcal{W}=\rho^{-1}\tp{0,\infty}\subset\mathscr{W}\times\R$, $\tp{w_0,\lambda_0}\in\mathcal{W}$, and $\mathscr{G}:\mathcal{W}\to\mathscr{Z}$ a real analytic function satisfying
    \begin{equation}\label{AGIFT h1}
        \mathscr{G}\tp{w_0,\lambda_0}=0  \text{ with }   D_1\mathscr{G}(w_0,\lambda_0)\in \mathcal{L}(\mathscr{W};\mathscr{Z}) \text{ an isomorphism}.
    \end{equation}
    Then there exists a zero solution curve $\mathscr{K}\subset\mathcal{W}\cap\mathscr{G}^{-1}\tcb{0}$, admitting the global continuous parametrization $\R\ni s\mapsto\tp{w(s),\lambda(s)}\in\mathscr{K}$ with $\tp{w(0),\lambda(0)}=\tp{w_0,\lambda_0}$, for which the following hold.
    \begin{enumerate}
        \item At each $s\in\R$  the linear operator $D_1\mathscr{G}\tp{w(s),\lambda(s)}$ is Fredholm with index zero.
        \item  The curve $\mathscr{K}$ is locally real analytic: for each $(w,\lambda) \in \mathscr{K}$ there exists an open set $(w,\lambda) \in U \subseteq \mathcal{W}$ such that $\mathscr{K} \cap U$ admits a real analytic reparameterization.
        \item The curve $\mathscr{K}$ is maximal in the sense that if $\mathscr{J}\subset\mathcal{W}\cap\mathscr{G}^{-1}\tcb{0}$ is a locally real analytic curve containing $\tp{w_0,\lambda_0}$ and along which $D_1\mathscr{G}$ is a Fredholm operator with index equal to zero, then $\mathscr{J}\subseteq\mathscr{K}$.
        \item Either the curve $\mathscr{K}$ is a \emph{closed loop} in the sense that there exists a period $T\in\R^+$ such that for all $s\in\R$ we have $(w(s),\lambda(s))=(w(s+T),\lambda(s+T))$ or, for each of the limits as $s\to\pm\infty$, one of the following alternatives holds:
        \begin{enumerate}[(A)]
            \item \emph{Blowup} - The quantity 
            \begin{equation}\label{the blowup quantity}
                \pmb{N}(s)=\tnorm{w(s)}_\mathscr{W}+|\lambda(s)|+1/\rho\tp{w(s),\lambda(s)}
            \end{equation}
            tends to $\infty$ as $s\to\pm\infty$.
            \item \emph{Loss of compactness} - There exists a sequence $\tcb{s_n}_{n\in\N}\subset\R$ with $s_n\to\pm\infty$ as $n\to\infty$ and $\sup_{n\in\N}\pmb{N}(s_n)<\infty$ but $\tcb{\tp{w(s_n),\lambda\tp{s_n}}}_{n\in\N}$ has no convergent subsequence in $\mathcal{W}$.
            \item \emph{Loss of Fredholm index zero} - There exists $(w_\star,\lambda_\star)\in\mathcal{W}$ and a sequence $\tcb{s_n}_{n\in\N}\subset\R$ with $s_n\to\pm\infty$ as $n\to\infty$ satisfying $\sup_{n\in\N}\pmb{N}(s_n)<\infty$ and $\tp{w(s_n),\lambda\tp{s_n}}\to\tp{w_\star,\lambda_\star}$ as $n\to\infty$ but $D_1\mathscr{G}\tp{w_\star,\lambda_\star}$ is not a Fredholm operator with index equal to zero.
        \end{enumerate}
    \end{enumerate}
\end{thm}
\begin{proof}
    The proof follows from a very minor adaptation of that of Chen, Walsh, and Wheeler's Theorem B.1 in~\cite{chen2023globalbifurcationmonotonefronts}, which itself is based on Theorem 6.1 in Chen, Walsh, Wheeler~\cite{MR3765551} and Theorem 9.1.1 in Buffoni and Toland~\cite{MR1956130}.  As such, we only point out the singular place where there is a difference.

    The theorem assertions here are identical to those of B.1 in~\cite{chen2023globalbifurcationmonotonefronts} except for the definition of the blowup quantity $\pmb{N}$ in alternative $(A)$ of the fourth item.  In the last term of~\eqref{the blowup quantity} we have $1/\rho$, where $\rho^{-1}\tp{0,\infty}=\mathcal{W}$, whereas Chen, Walsh, and Wheeler use the pointwise inversion of the distance to the boundary of $\mathcal{W}$. In their proof this term serves to guarantee that if a sequence $\tcb{(w(s_n),\lambda(s_n))}_{n\in\N}\subset\mathcal{W}$ satisfies 
    \begin{equation}\label{CWW condition}
        \sup_{n\in\N}\tsb{\tnorm{w(s_n)}_{\mathscr{W}}+|\lambda(s_n)|+1/\m{dist}\tp{(w(s_n),\lambda(s_n)),\pd\mathcal{W}]}}<\infty,
    \end{equation}
    then every subsequential limit lies within the open set $\mathcal{W}$. The same conclusion holds if we replace~\eqref{CWW condition} with the condition $\sup_{n\in\N}\pmb{N}\tp{s_n}<\infty$;  indeed, if this holds, then there exists $\ep>0$ such that $\tcb{w(s_n),\lambda(s_n)}_{n\in\N}\subset \mathcal{W}_\ep$, where $\mathcal{W}_\ep=\rho^{-1}\tp{\ep,\infty}$, and $\Bar{\mathcal{W}_\ep}\subseteq\mathcal{W}_{\ep/2}\subset\mathcal{W}$. Hence, any subsequential limit belongs to $\mathcal{W}$, as claimed.
\end{proof}

Our next task is to apply Theorem~\ref{thm on analytic global implicit function theorem} within the framework established in Definition~\ref{defn of periodic and solitary unification} and Theorem~\ref{thm on unified fixed point reformulation}. Recall that the set of solutions in our fixed point reformulation is given by $S$ in~\eqref{the set of solutions}.   We will now obtain a solution curve within $S$ and  dispense with the possibility of the \emph{closed loop}, \emph{loss of compactness}, and \emph{loss of Fredholm index zero} alternatives, which means that the function $\pmb{N}$ of~\eqref{the blowup quantity} must blowup at the endpoints of this curve.

\begin{prop}[Maximal locally analytic curve of solutions]\label{prop on maximally locally analytic curve of solutions}
    There exists a maximal locally analytic curve $\mathscr{S}\subseteq S$ with $(0,0)\in\mathscr{S}$, parametrized by the continuous map $\R\ni s\mapsto\tp{x(s),\kappa(s)}\in\mathscr{S}$ sending $0\mapsto (0,0)$ and satisfying
    \begin{equation}\label{the blowup quantity____}
        \lim_{|s|\to\infty}\tsb{\tnorm{x(s)}_{\pmb{X}}+|\kappa\tp{s}|+1/\varrho\tp{x(s)}}=\infty.
    \end{equation}
\end{prop}
\begin{proof}
    According to Theorem~\ref{thm on unified fixed point reformulation}, the hypotheses of Theorem~\ref{thm on analytic global implicit function theorem} are satisfied with  $\mathscr{W}=\mathscr{Z}=\pmb{X}$, $\rho=\varrho$, $\mathcal{W}=U\times\R$ and $\mathscr{G}=I_{\pmb{X}}+\mathcal{K}$. The theorem then grants the solution curve $\mathscr{S}\subseteq S$ with continuous parametrization $\R\ni s\mapsto\tp{x(s),\kappa(s)}\in \mathscr{S}$, and it guarantees that the curve $\mathscr{S}$ is either a closed loop or else one of the three alternatives of the fourth item of Theorem~\ref{thm on analytic global implicit function theorem} holds.  We claim that $\mathscr{S}$ is not a closed loop.  Once this is established, we may invoke Theorem~\ref{thm on unified fixed point reformulation} to exclude the possibility of the loss of compactness (alternative B) or loss of Fredholm index zero (alternative C) scenarios from Theorem~\ref{thm on analytic global implicit function theorem}.  This leaves blowup (alternative A) as the only possibility, which will then complete the proof.

    To prove the claim it suffices to show that $\mathscr{S}\setminus\tcb{\tp{0,0}}$ is disconnected.  Consider the open sets $W_{\pm}=\tcb{(x,\kappa)\in U\times\R\;:\;\pm\kappa>0}$, which by definition satisfy $W_+\cap W_-=\es$.  We will prove that $\mathscr{S}\setminus\tcb{\tp{0,0}}\subset W_+\cup W_-$, from which the claimed disconnectedness follows.  For this we may further reduce to proving that if $\tp{x,0}\in\mathscr{S}$, then $x=0$.  Fix such an $(x,0)\in\mathscr{S}$ and write $x=(u,\eta)\in\pmb{X}$. We then know from the third item of Theorem~\ref{thm on unified fixed point reformulation} that $\mathcal{Q}_\be\tp{u,\eta}=0$.  We now break to cases based on the periodic or solitary settings.

    Consider first the periodic setting.  Since $\mathcal{Q}_\be\tp{u,\eta}=0$,  the second item of Proposition~\ref{prop on operator formulation periodic case} implies that system~\eqref{stationary nondimensionalized equations} is satisfied by $(\pmb{u},\pmb{\eta})=(u,\eta)$ with $\Phi=0$. By testing the second equation of \eqref{stationary nondimensionalized equations} with $\pmb{u}$, integrating by parts, and using the first equation to annihilate the gravity-capillary contribution, we derive the identity
    \begin{equation}
        \int_{\T^2_L}A|\pmb{u}|^2+\tp{1+\be+\pmb{\eta}}\tp{2^{-1}|\grad\pmb{u}+\grad\pmb{u}^{\m{t}}|^2+2|\grad\cdot\pmb{u}|^2}=0,
    \end{equation}
    which implies that $u=0$.  In turn, we find that $(1+\be+\pmb{\eta})\tp{G-\Delta}\grad\pmb{\eta}=0$, but since $(1+\be+\pmb{\eta})>0$ and $\mathscr{F}[\pmb{\eta}]\tp{0}=0$ we must have $\pmb{\eta}=0$ as well.  This completes the proof of the claim in the periodic case.

    Next, we consider the solitary setting, in which case almost the same argument works.  We use the second item of Proposition~\ref{prop on operator formulation solitary case} to again find that~\eqref{stationary nondimensionalized equations} holds, but this time by the velocity $\pmb{u}=\tp{2\eta+\tp{1+\be}^2}^{-1/2}u$ and free surface $\pmb{\eta}=\sqrt{2\eta+\tp{1+\be}^2}-\tp{1+\be}$. Lemma~\ref{lem on free surface domsains, solitary case} and Proposition~\ref{prop on product estimates in weighted gradient spaces} imply that $(\pmb{u},\pmb{\eta})\in H^2_{\scriptscriptstyle1/2}\tp{\R^2;\R^2}\times\bf{O}_\be\cap\tilde{H}^3_{\scriptscriptstyle1/2}\tp{\R^2}$, so the same integration by parts argument as in the periodic case reveals that 
    \begin{equation}\label{__+_+_+_+_+_+_++_+_+_+++__)_}
        \int_{\R^2}A|\pmb{u}|^2+\tp{1+\be+\pmb{\eta}}\tp{2^{-1}|\grad\pmb{u}+\grad\pmb{u}^{\m{t}}|^2+2|\grad\cdot\pmb{u}|^2}=\babs{\int_{\R^2}\tp{1+\be+\pmb{\eta}}\tp{G-\Delta}\grad\pmb{\eta}\cdot\pmb{u}}.
    \end{equation}
    By using the density of $C^\infty_{\m{c}}\tp{\R^2}\subset\tilde{H}^3_{\scriptscriptstyle1/2}\tp{\R^2}$ (see Proposition~\ref{prop on density of functions of bounded support}) as in the proof of Proposition~\ref{prop on principal part analysis solitary case} we find that the right hand side of~\eqref{__+_+_+_+_+_+_++_+_+_+++__)_} vanishes, and so $\pmb{u}=0$.  Next we deduce from $1+\be+\pmb{\eta}>0$ and $(1+\be+\pmb{\eta})\tp{G-\Delta}\grad\pmb{\eta}$ that $\grad\pmb{\eta}=0$ and so $\tnorm{\pmb{\eta}}_{\tilde{H}^3_{1/2}}=0$. Finally, we deduce from the expressions $\pmb{u}$ and $\pmb{\eta}$ that necessarily $u=0$ and $\eta=0$ as well.  This proves the claim in the solitary case.
\end{proof}

\subsection{Blowup scenarios in the periodic case}\label{subsection on blow up criteria in the periodic case}

We now aim to more carefully analyze what happens along the solution curve $\mathscr{S}$ from Proposition~\ref{prop on maximally locally analytic curve of solutions} as the parametrization parameter $s$ tends to $\pm\infty$.  In this subsection we restrict our attention to the periodic case, which is much simpler, and delay the solitary case to the next subsection.  Our main strategy is to derive a priori estimates for system~\ref{stationary nondimensionalized equations}, starting with control of relatively weak quantities.  These quantities can then replace the strong norm in the blowup quantity~\eqref{the blowup quantity____}.

\begin{prop}[A priori estimates - periodic case]\label{prop on periodic a priori estimates}
    Suppose that for some $\ep\in\tp{0,1}$ we have $(u,\eta)\in H^2\tp{\T^2_L;\R^2}\times [\bf{O}_\be^\ep\cap \z{H}^3\tp{\T^2_L}]$ with $\tnorm{\eta}_{L^\infty}\le\ep^{-1}$ a solution to system~\eqref{stationary nondimensionalized equations} for data $\kappa=1$ and a generic forcing $\Phi\in L^2\tp{\T^2_L;\R^2}$. Then, for implicit constants depending only on $\ep>0$ and the physical parameters we have the a priori bounds
    \begin{equation}\label{a priori estimates in the periodic case}
        \tnorm{u,\eta}_{H^1\times H^2}\lesssim\tnorm{\Phi}_{H^{-1}}\tbr{\tnorm{\Phi}_{H^{-1}}}\quad\text{and}\quad\tnorm{u,\eta}_{H^2\times H^3}\lesssim\tnorm{\Phi}_{L^2}\tbr{\tnorm{\Phi}_{L^2}}^{16/3}.
    \end{equation}
\end{prop}
\begin{proof}
    We begin by testing the second equation in \eqref{stationary nondimensionalized equations} with $u$ and integrating by parts to see that
    \begin{equation}
        \tnorm{u}^2_{H^1}\lesssim\int_{\T^2_L}A|u|^2+\tp{1+\be+\eta}\tp{2^{-1}|\grad u+\grad u^{\m{t}}|^2+2|\grad\cdot u|^2}=\int_{\T^2_L}\Phi\cdot u\le\tnorm{\Phi}_{H^{-1}}\tnorm{u}_{H^1},
    \end{equation}
    from which it follows that $\tnorm{u}_{H^1}\lesssim\tnorm{\Phi}_{H^{-1}}$.

    Next, we test the same equation with the vector field $\tp{1+\be+\eta}^{-1}\grad\eta\in H^2\tp{\T^2_L;\R^2}$, isolate the gravity-capillary contribution, and recall that $\mathscr{F}[\eta](0)=0$:
    \begin{multline}\label{the long lost eta identity}
        \norm{\eta}_{H^2}^2 \lesssim 
        \int_{\T^2_L}G|\grad\eta|^2+|\grad^2\eta|^2 \\
        =\int_{\T^2_L}\tp{1+\be+\eta}\tp{u\otimes u-\mathbb{S}u}:\grad\tp{\tp{1+\be+\eta}^{-1}\grad\eta}+\tp{\Phi-Au}\cdot\tp{1+\be+\eta}^{-1}\grad\eta.
    \end{multline}
    To estimate the right hand side, we claim the following preliminary bound:
    \begin{equation}\label{the important bound on eta}
        \tnorm{\tp{1+\be+\eta}^{-1}\grad\eta}_{H^1}\lesssim\tnorm{\eta}_{H^2}
    \end{equation}
    with an implicit constant dependent on $\ep$.  To see why~\eqref{the important bound on eta} holds, we first note that $\tnorm{(1+\be+\eta)^{-1}\grad\eta}_{L^2}\lesssim\tnorm{\eta}_{H^1}$ and then use the product rule to compute
    \begin{equation}\label{the derivative identity for this thing}
        \grad\tp{\tp{1+\be+\eta}^{-1}\grad\eta}=(1+\be+\eta)^{-1}\grad^2\eta-\tp{1+\be+\eta}^{-2}\grad\eta\otimes\grad\eta+\tp{1+\be+\eta}^{-2}\grad\eta\otimes\grad\be.
    \end{equation}
    By taking the norm in $L^2$ of the above expression, we then derive the bound
    \begin{equation}\label{Sphynx}
        \tnorm{\grad\tp{\tp{1+\be+\eta}^{-1}\grad\eta}}_{L^2}\lesssim\tnorm{\eta}_{H^2}+\tnorm{\eta}_{W^{1,4}}^2.
    \end{equation}
    To handle the final term here we employ Gagliardo-Nirenberg interpolation (Theorem~\ref{thm on gagliardo nirenberg}):  
    \begin{equation}\label{not good enough for the solitary case}
        \tnorm{\eta}_{W^{1,4}}^2\lesssim\tnorm{\eta}_{L^\infty}\tnorm{\eta}_{H^2}\lesssim\tnorm{\eta}_{H^2}.
    \end{equation}
    Synthesizing these estimates proves the claimed estimate~\eqref{the important bound on eta}, which then allows us to  bound the terms on the right side of~\eqref{the long lost eta identity}:
    \begin{equation}\label{uwu_0}
        \babs{\int_{\T^2_L}\tp{1+\be+\eta}\tp{u\otimes u-\mathbb{S}u}:\grad\tp{\tp{1+\be+\eta}^{-1}\grad\eta}}\lesssim\tp{\tnorm{u}_{L^4}^2+\tnorm{u}_{H^1}}\tnorm{\eta}_{H^2}\lesssim\tnorm{\Phi}_{H^{-1}}\tbr{\tnorm{\Phi}_{H^{-1}}}\tnorm{\eta}_{H^2}
    \end{equation}
    and
    \begin{equation}\label{uwu_1}
        \babs{\int_{\T^2_L}\tp{\Phi-Au}\cdot\tp{1+\be+\eta}^{-1}\grad\eta}\lesssim\tp{\tnorm{\Phi}_{H^{-1}}+\tnorm{u}_{L^2}}\tnorm{\eta}_{H^2}\lesssim\tnorm{\Phi}_{H^{-1}}\tnorm{\eta}_{H^2}.
    \end{equation}
    We then combine~\eqref{uwu_0} and~\eqref{uwu_1} with~\eqref{the long lost eta identity} to deduce the remaining $\eta$ estimate on the left of~\eqref{a priori estimates in the periodic case}.

    We next prove the higher regularity estimate.  The system~\eqref{stationary nondimensionalized equations} is equivalent to the semilinear form
    \begin{equation}\label{a useful choice of semilinear equations in the periodic case}
        \begin{cases}
            \grad\cdot u=-P_0\tp{u\cdot\grad\log\tp{1+\be+\eta}},\\
            Au-\grad\cdot\mathbb{S}u+\tp{G-\Delta}\grad\eta=\tp{1+\be+\eta}^{-1}\tp{\Phi +A\tp{\be+\eta}u}-u\cdot\grad u+\mathbb{S}u\grad\log\tp{1+\be+\eta}.
        \end{cases}
    \end{equation}
    The linear operator on the left hand side of~\eqref{a useful choice of semilinear equations in the periodic case} is $\mathcal{P}$, which is an isomorphism $H^2\tp{\T^2_L;\R^2}\times \z{H}^3\tp{\T^2_L}\to\z{H}^1\tp{\T^2_L}\times L^2\tp{\T^2_L;\R^2}$ thanks to the first item of Proposition~\ref{prop on linear analysis with bathymetry - periodic case}. Therefore, we may deduce the right bound in \eqref{a priori estimates in the periodic case} by bounding the $H^1$ and $L^2$ norms of the right hand sides of the first and second equations in~\eqref{a useful choice of semilinear equations in the periodic case}, respectively.     We begin with the first term, using Theorem~\ref{thm on high low product estimates} (and~\eqref{the important bound on eta}) to estimate
    \begin{equation}\label{T__(0)}
        \tnorm{P_0\tp{u\cdot\grad\log\tp{1+\be+\eta}}}_{H^1}\lesssim\tnorm{u}_{W^{1,4}}\tbr{\tnorm{\eta}_{W^{1,4}}}+\tnorm{u}_{L^\infty}\tbr{\tnorm{\eta}_{H^2}}\lesssim\tnorm{u}_{W^{1,4}}\tbr{\tnorm{\eta}_{H^2}}\lesssim\tnorm{u}_{W^{1,4}}\tbr{\tnorm{\Phi}_{L^2}}^2.
    \end{equation}
    We then bound the second term via
    \begin{equation}\label{T__(1)}
        \tnorm{\tp{1+\be+\eta}^{-1}\tp{\Phi +A\tp{\be+\eta}u}}_{L^2}\lesssim\tnorm{\Phi}_{L^2}+\tnorm{u}_{L^2}\lesssim\tnorm{\Phi}_{L^2},
    \end{equation}
    \begin{equation}\label{T__(2)}
        \tnorm{u\cdot\grad u}_{L^2}\lesssim\tnorm{u}_{L^4}\tnorm{u}_{W^{1,4}}\lesssim\tnorm{u}_{W^{1,4}}\tnorm{\Phi}_{L^2},
    \end{equation}
    and
    \begin{equation}\label{T__(3)}
        \tnorm{\mathbb{S}u\grad\log\tp{1+\be+\eta}}_{L^2}\lesssim\tnorm{u}_{W^{1,4}}\tbr{\tnorm{\eta}_{W^{1,4}}}\lesssim\tnorm{u}_{W^{1,4}}\tbr{\tnorm{\Phi}_{L^2}}^2.
    \end{equation}
    Using the sum of~\eqref{T__(0)}, \eqref{T__(1)}, \eqref{T__(2)}, and~\eqref{T__(3)} with the $\mathcal{P}$ isomorphism then shows that 
    \begin{equation}
        \tnorm{u,\eta}_{H^2\times H^3}\lesssim\tnorm{\Phi}_{L^2}+\tnorm{u}_{W^{1,4}}\tbr{\tnorm{\Phi}_{L^2}}^2.
    \end{equation}
    Again using Theorem~\ref{thm on gagliardo nirenberg}, Sobolev embeddings, and interpolation, we find that 
    \begin{equation}
        \tnorm{u}_{W^{1,4}}\lesssim\tnorm{u}_{L^\infty}^{1/2}\tnorm{u}_{H^2}^{1/2},\quad\tnorm{u}_{L^\infty}\lesssim\tnorm{u}_{H^{5/4}}\lesssim\tnorm{u}_{H^1}^{3/4}\tnorm{u}_{H^2}^{1/4}\imp\tnorm{u}_{W^{1/4}}\lesssim\tnorm{u}_{H^1}^{3/8}\tnorm{u}_{H^2}^{5/8},
    \end{equation}
    so Young's inequality implies that for any $\mu\in\R^+$ there exists a $C_\mu<\infty$ such that
    \begin{equation}
        \tnorm{u,\eta}_{H^2\times H^3}\lesssim\tnorm{\Phi}_{L^2}+C_\mu\tnorm{\Phi}_{L^2}\tbr{\tnorm{\Phi}_{L^2}}^{16/3}+\mu\tnorm{u}_{H^2}.
    \end{equation}
    Taking $\mu$ sufficiently small (depending on $\ep$ and the physical parameters), we can absorb and obtain the right hand side of~\eqref{a priori estimates in the periodic case}. 
\end{proof}

The a priori estimates of Proposition~\ref{prop on periodic a priori estimates} lead us to the following blowup refinement.

\begin{coro}[Blowup refinement - periodic case]\label{coro on blowup refinement - periodic case}
    Let $\mathscr{S}\subset S$ be the maximal locally analytic curve of solutions granted by Proposition~\ref{prop on maximally locally analytic curve of solutions} in the periodic case, parametrized by the continuous map $\R\ni s\mapsto(u(s),\eta(s),\kappa(s))\in \mathscr{S}\subset H^2\tp{\T^2_L;\R^2}\times [\bf{O}_\be\cap \z{H}^3\tp{\T^2_L}]$. Then
    \begin{equation}\label{improved blowup quantity in the periodic case}
        \lim_{|s|\to\infty}\tsb{\tnorm{\eta(s)}_{L^\infty}+1/\inf\tp{1+\be+\eta(s)}+|\kappa(s)|}=\infty.
    \end{equation}
\end{coro}
\begin{proof}
    The argument for the limits $s\to\infty$ and $s\to-\infty$ are identical, so we will only address the former case here.  We argue  by contradiction, supposing that~\eqref{improved blowup quantity in the periodic case} is false as $s\to\infty$.  Then we can find a sequence $\tcb{s_n}_{n\in\N}\subset\R^+$ with $s_n\to\infty$ as $n\to\infty$ and a number $1<C<\infty$ such that upon setting $\eta_n=\eta(s_n)$ and $\kappa_n=\kappa(s_n)$ we have
    \begin{equation}\label{the contradition hypothesis}
        \sup_{n\in\N}\tsb{\tnorm{\eta_n}_{L^\infty}+1/\inf\tp{1+\be+\eta_n}+|\kappa_n|}\le C.
    \end{equation}
    Let us also set $u_n=u(s_n)$. According to Proposition~\ref{prop on maximally locally analytic curve of solutions}, the solution curve $\mathscr{S}$ is contained in the set $S$ from~\eqref{the set of solutions}; in turn, thanks to the third item of Theorem~\ref{thm on unified fixed point reformulation} and the definition of the forcing map $\mathcal{F}_\be$ in the periodic case (see Proposition~\ref{prop on data map, periodic case}), for every $n\in\N$ we have that
    \begin{equation}\label{the sequence of equations in the periodic case}
        \begin{cases}
            \grad\cdot\tp{\tp{1+\be+\eta_n}u_n}=0,\\
            \grad\cdot\tp{\tp{1+\be+\eta_n}u_n\otimes u_n}+Au_n-\grad\cdot\tp{\tp{1+\be+\eta_n}\mathbb{S}u_n}+\tp{1+\be+\eta_n}\tp{G-\Delta}\grad\eta_n=\Phi_n,\\
            \Phi_n=\kappa_n\tsb{\Upomega_\phi\tp{\eta_n}+\grad\tp{\tp{1+\be+\eta_n}\Upomega_\psi\tp{\eta_n}}+\Upomega_\tau\tp{\eta_n}\grad\eta_n}.
        \end{cases}
    \end{equation}
    
    Our first goal is to invoke the left a priori estimate of~\eqref{a priori estimates in the periodic case}, and for this we require uniform $H^{-1}$ bounds on the forcing sequence $\tcb{\Phi_n}_{n\in\N}$.  The hypothesis~\eqref{the contradition hypothesis} implies that for $\ep=C^{-1}\in\tp{0,1}$ we have $\tcb{\eta_n}_{n\in\N}\subset\bf{O}_\be^\ep\cap B_{C^0\cap L^\infty}(0,\ep^{-1})$.   Corollary~\ref{coro on compound analytic superposition real case} implies that the functions $\Upomega_\phi:\bf{O}_\be^\ep\to L^2\tp{\T^2_L;\R^2}$ and $\Upomega_\psi:\bf{O}_\be^\ep\to L^2\tp{\T^2_L}$ map bounded sets to bounded sets, so we may estimate
    \begin{equation}\label{PART+1}
        \tnorm{\Phi_n}_{H^{-1}}\lesssim\tnorm{\Upomega_\phi\tp{\eta_n}}_{L^2}+\tnorm{\tp{1+\be+\eta_n}\Upomega_\psi\tp{\eta_n}}_{L^2}+\tnorm{\Upomega_\tau\tp{\eta_n}\grad\eta_n}_{H^{-1}}\lesssim 1+\tnorm{\Upomega_\tau\tp{\eta_n}\grad\eta_n}_{H^{-1}}
    \end{equation}
    for implicit constants depending only on $C$.  To handle the final term using only the $L^\infty$ bounds on $\eta_n$, we aim to pull out a derivative and obtain a nice remainder.  To this end, we define the real analytic primitive
    \begin{equation}
        \tilde{\tau}:\tp{-a_\be,\infty}\to\tp{L^\infty\cap H^1}\tp{\T^2_L;\R^{2\times 2}_{\m{sym}}} \text{ via }\tilde{\tau}\tp{z}=\int_0^z\tau(\mathfrak{z})\;\m{d}\mathfrak{z}
    \end{equation}
    and claim the identity:
    \begin{equation}\label{the weirdo identity}
    \Upomega_\tau\tp{\eta_n}\grad\eta_n=\grad\cdot\tp{\Upomega_{\tilde{\tau}}\tp{\eta_n}}-\Upomega_{\grad\cdot\tilde{\tau}}\tp{\eta_n}.
    \end{equation}

    To establish~\eqref{the weirdo identity}, we begin by noting that Corollary~\ref{coro on compound analytic superposition real case} ensures that the mapping $\Upomega_{\tilde{\tau}}:\bf{O}^\ep_\be\cap C^1\tp{\T^2_L}\to H^1\tp{\T^2_L;\R^{2\times 2}_{\m{sym}}}$ maps bounded sets to bounded sets while the proof of Corollary~\ref{coro on compound analytic superposition real case} and Theorem~\ref{thm on holomorphic functional calculus} (specifically~\eqref{Dunford integral}) provides the Cauchy integral formula
    \begin{equation}\label{this is an equation for us to later cite, mmmmkkkay buh byee}
        \Upomega_{\tilde{\tau}}\tp{\eta} = \f{1}{2\pi\ii}\int_{\gam_\eta}\tp{z - \eta}^{-1}\tilde{\tau}(z)\;\m{d}z \text{ for } \eta\in\bf{O}^\ep_\be\cap C^1\tp{\T^2_L},
    \end{equation}
    where $\tilde{\tau}$ is analytically continued into an open and simply connected domain in $\C$ containing the compact interval $\Bar{\m{img}\tp{\eta}}$ and $\gam_\eta$ is a simple closed contour that loops around this image once within the domain of $\tilde{\tau}$.  Due to the integrand of~\eqref{this is an equation for us to later cite, mmmmkkkay buh byee} being a holomorphic function of $z$ belonging to the compact image of the contour $\gamma_\eta$ with values in $H^1\tp{\T^2_L;\C^{2\times 2}}$, we are free to take the divergence, commute with the integral, and use the standard calculus of weak derivatives on the integrand to  derive the initial identity
    \begin{equation}\label{Mozart Symphony No. 39}
        \grad\cdot\tp{\Upomega_{\tilde{\tau}}\tp{\eta}} = \f{1}{2\pi\ii}\int_{\gam_\eta}\tp{z - \eta}^{-1}\grad\cdot\tilde{\tau}\tp{z}\;\m{d}z + \bsb{\f{1}{2\pi\ii}\int_{\gam_\eta}\tp{z - \eta}^{-2}\tilde{\tau}\tp{z}\;\m{d}z}\grad\eta.
    \end{equation}
    Thanks to equation~\eqref{Dunford integral}, the first term on the left of~\eqref{Mozart Symphony No. 39} is none other than $\Upomega_{\grad\cdot\tilde{\tau}}\tp{\eta}$. To handle the second term we note that 
    \begin{equation}\label{Bellerophon}
        \f{\m{d}}{\m{d}z}\tp{\tp{z-\eta}^{-1}\tilde{\tau}\tp{z}} = \tp{z - \eta}^{-1}\tau\tp{z} + \tp{z - \eta}^{-2}\tilde{\tau}\tp{z},\quad z\in\gam_\eta
    \end{equation}
    and the contour integral over $\gamma_\eta$ of the left hand side of~\eqref{Bellerophon} vanishes. Therefore (after again using~\eqref{Dunford integral})
    \begin{equation}\label{Hecatoncheires}
        \f{1}{2\pi\ii}\int_{\gam_\eta}\tp{z-\eta}^{-2}\tilde{\tau}\tp{z}\;\m{d}z = \f{1}{2\pi\ii}\int_{\gam_\eta}\tp{z - \eta}^{-1}\tau(z)\;\m{d}z = \Upomega_{\tau}\tp{\eta}.
    \end{equation}
    By synthesizing~\eqref{Mozart Symphony No. 39} with~\eqref{Hecatoncheires} the claimed identity of~\eqref{the weirdo identity} is established.

    Now by using Corollary~\ref{coro on compound analytic superposition real case} again, we find that the functions $\Upomega_{\tilde{\tau}}:\bf{O}^\ep_\be\to  L^\infty\tp{\T^2_L;\R^{2\times2}_{\m{sym}}}$ and $\Upomega_{\grad\cdot\tilde{\tau}}:\bf{O}_\be^\ep\to L^2\tp{\T^2_L;\R^{2\times 2}_{\m{sym}}}$  map bounded sets to bounded sets, and hence from~\eqref{the weirdo identity}
    \begin{equation}\label{PART+2}
        \tnorm{\Upomega_\tau\tp{\eta_n}\grad{\eta_n}}_{H^{-1}}\lesssim\tnorm{\Upomega_{\tilde{\tau}}\tp{\eta_n}}_{L^2}+\tnorm{\Upomega_{\grad\cdot\tilde{\tau}}\tp{\eta_n}}_{L^2}\lesssim 1.
    \end{equation}
    By combining~\eqref{PART+1} and~\eqref{PART+2} we then find that $\sup_{n\in\N}\tnorm{\Phi_n}_{H^{-1}}\lesssim 1$.  Pairing this fact with the low regularity estimate of Proposition~\ref{prop on periodic a priori estimates} then yields the bound 
    \begin{equation}\label{is that a snore or a bore}
        \sup_{n\in\N}\tnorm{u_n,\eta_n}_{H^1\times H^2}\lesssim 1.
    \end{equation}

    The next step is to use the improved bounds~\eqref{is that a snore or a bore} on $\tcb{\eta_n}_{n\in\N}$ to deduce that the forcing sequence is uniformly bounded in $L^2$.  This is a direct consequence of the second item of Proposition~\ref{prop on data map, periodic case} and the identity $\Phi_n=\kappa_n(1+\be+\eta_n)\mathcal{F}_\be\tp{\eta_n}$. Therefore, $\sup_{n\in\N}\tnorm{\Phi_n}_{L^2}\lesssim 1$ and so the high regularity estimate of Proposition~\ref{prop on periodic a priori estimates} tells us that
    \begin{equation}\label{here lies the contradiction}
        \sup_{n\in\N}\tnorm{u_n,\eta_n}_{H^2\times H^3}\lesssim 1.
    \end{equation}
    Together, the bounds~\eqref{the contradition hypothesis} and~\eqref{here lies the contradiction} contradict the limit~\eqref{the blowup quantity____} from Proposition~\ref{prop on maximally locally analytic curve of solutions},  so it must be the case that the limit~\eqref{improved blowup quantity in the periodic case} holds as $s\to\infty$.
\end{proof}

\subsection{Blowup scenarios in the solitary case}\label{subsectrion on blow up criteria in the solitary case}

In this subsection we derive results analogous to those of Section~\ref{subsection on blow up criteria in the periodic case}, but in the solitary case. The situation here is significantly complicated by the fact that the blowup norm in~\eqref{the blowup quantity____} comes from weighted Sobolev spaces, but the natural `energy' structure of the equations themselves is framed in unweighted spaces.  As a result, our estimates in this section follow a multi-step procedure. We begin by deriving an unweighted $L^2$-based result analogous to Proposition~\ref{prop on periodic a priori estimates}. The following lemma is an important ingredient.

\begin{lem}[Unweighted estimates for the principal part]\label{lem on unweighted estimates for the principal part}
    Assume that $A,G>0$. The following hold.
    \begin{enumerate}
        \item If $u\in H^2\tp{\R^2;\R^2}$ and $\eta\in\mathscr{S}^\ast\tp{\R^2}$ with $\grad\eta\in H^2\tp{\R^2;\R^2}$ are a solution to the linear system~\eqref{principal part linear system} with data $\psi\in\tp{\dot{H}^{-1}\cap H^1}\tp{\R^2}$ and $\phi\in L^2\tp{\R^2;\R^2}$ then we have the bound
    \begin{equation}
        \tnorm{u,\grad\eta}_{H^2\times H^2}\lesssim\tnorm{\psi,\phi}_{\tp{\dot{H}^{-1}\cap H^1}\times L^2}
    \end{equation}
    for implicit constants depending only on the physical parameters.
        \item If $1<q<2$, $u\in W^{2,q}\tp{\R^2;\R^2}$, $\eta\in L^{2q/(2-q)}\tp{\R^2}$ with $\grad\eta\in W^{2,q}\tp{\R^2}$ and $\grad\cdot u=0$ are a solution to the linear system~\eqref{principal part linear system} with data $\psi=0$ and $\phi\in L^q\tp{\R^2;\R^2}$ then we have the estimate
        \begin{equation}
            \tnorm{u,\eta,\grad\eta}_{W^{2,q}\times L^{2q/(2-q)}\times W^{2,q}}\lesssim\tnorm{\phi}_{L^q},
        \end{equation}
        for implicit constants depending only on $q$ and the physical parameters.
    \end{enumerate}
\end{lem}
\begin{proof}
    We begin with the first item. The proof strategy here is in the same spirit as that of Proposition~\ref{prop on principal part analysis solitary case}. We begin by defining the auxiliary function $w=\Delta^{-1}\grad\psi\in H^2\tp{\R^2;\R^2}$. Notice that 
    \begin{equation}
        \tnorm{w}_{H^2}\lesssim\tnorm{\psi}_{\dot{H}^{-1}\cap H^1} \text{ and } \grad\cdot w=\psi,
    \end{equation}
    and hence the function $v=u-w$ is divergence free and satisfies the PDE
    \begin{equation}\label{amazing equations}
        \begin{cases}
            \grad\cdot v=0,\\
            \tp{A-\Delta}v+\tp{G-\Delta}\grad\eta=\phi-Aw+\grad\cdot\mathbb{S}w.
        \end{cases}
    \end{equation}
    To the second equation we apply the operator $(G-\Delta)^{-1}\Delta^{-1}\grad\grad\cdot$, which annihilates the $v$ contribution, and gives us
    \begin{equation}\label{formula for the gradient of free surface}
        \grad\eta=\tp{G-\Delta}^{-1}\Delta^{-1}\grad\grad\cdot\tp{\phi-Aw+\grad\cdot\mathbb{S}w}\quad\text{and so}\quad\tnorm{\grad\eta}_{H^2}\lesssim\tnorm{\phi}_{L^2}+\tnorm{w}_{H^2}\lesssim\tnorm{\psi,\phi}_{\tp{\dot{H}^{-1}\cap H^1}\times L^2}.
    \end{equation}
    On the other hand, we can use the second equation in~\eqref{amazing equations} to write
    \begin{equation}\label{formula for the velocity}
        v=\tp{A-\Delta}^{-1}\tp{\phi-Aw+\grad\cdot\mathbb{S}w-\tp{G-\Delta}\grad\eta}\quad\text{from which the bound}\quad\tnorm{v}_{H^2}\lesssim\tnorm{\psi,\phi}_{\tp{\dot{H}^{-1}\cap H^1}\times L^2}
    \end{equation}
    follows.

    The second item is also simple. We derive the $\psi=0$ analogs of formulas~\eqref{formula for the gradient of free surface} and~\eqref{formula for the velocity} which are
    \begin{equation}
        \grad\eta=\tp{G-\Delta}^{-1}\Delta^{-1}\grad\grad\cdot\phi,\quad u=\tp{A-\Delta}^{-1}\tp{\phi-\tp{G-\Delta}\grad\eta}.
    \end{equation}
    From the boundedness of $(G-\Delta)^{-1}:L^q\tp{\R^2}\to W^{2,q}\tp{\R^2}$ and $\Delta^{-1}\grad\grad\cdot:L^q\tp{\R^2}\to L^q\tp{\R^2}$ (which are consequences of the Mikhlin multiplier theorem) we deduce the bound
    \begin{equation}
        \tnorm{\grad\eta}_{W^{2,q}}\lesssim\tnorm{\phi}_{L^q}\quad\text{from which it follows}\quad\tnorm{u}_{W^{2,q}}\lesssim\tnorm{\phi}_{L^q}.
    \end{equation}
    Finally, the subcritical Sobolev embedding provides the estimate $\tnorm{\eta}_{L^{2q/(2-q)}}\lesssim\tnorm{\grad\eta}_{L^q}$. This completes the proof.
\end{proof}

We now prove the unweighted estimate.

\begin{prop}[Unweighted a priori estimates]\label{prop on solitary case unweighted a priori estimates}
    Suppose that for some $\ep\in\tp{0,1}$ we have $(u,\eta)\in H^2_{\scriptscriptstyle1/2}\tp{\R^2;\R^2}\times\bf{O}_\be^\ep\cap\tilde{H}^3_{\scriptscriptstyle1/2}\tp{\R^2}$ with $\tnorm{\eta}_{L^\infty}\le\ep^{-1}$ a solution to system~\eqref{stationary nondimensionalized equations} for data $\kappa=1$ and a generic forcing $\Phi\in L^2_{\scriptscriptstyle1/2}\tp{\R^2;\R^2}$. Then, for implicit constants depending only on $\ep>0$ and the physical parameters, we have the a priori bounds
    \begin{equation}\label{a priori bounds, solitary case}
        \tnorm{u,\grad\eta}_{H^1\times H^1}\lesssim\tnorm{\Phi}_{H^{-1}}\tbr{\tnorm{\Phi}_{H^{-1}}},\quad\tnorm{u,\grad\eta}_{H^2\times H^2}\lesssim\tnorm{\Phi}_{L^2}\tbr{\tnorm{\Phi}_{L^2}}^{16/3}.
    \end{equation}
\end{prop}
\begin{proof}
    By arguing via density of smooth and compactly supported functions in the weighted gradient spaces as in the proof of Proposition~\ref{prop on maximally locally analytic curve of solutions} (and elsewhere) to annihilate the gravity-capillary contribution, we derive the energy bound:
    \begin{equation}
        \tnorm{u}_{H^1}^2\lesssim\int_{\R^2}A|u|^2+\tp{1+\be+\eta}\tp{2^{-1}|\grad u+\grad u^{\m{t}}|^2+2|\grad\cdot u|^2}=\int_{\R^2}\Phi\cdot u\le\tnorm{\Phi}_{H^{-1}}\tnorm{u}_{H^1}.
    \end{equation}
Hence $\tnorm{u}_{H^1}\lesssim\tnorm{\Phi}_{H^{-1}}$. The next step is again to test the equation with $(1+\be+\eta)^{-1}\grad\eta\in H^2_{\scriptscriptstyle1/2}\tp{\R^2;\R^2}$ as in~\eqref{the long lost eta identity}. We require a solitary analog of the preliminary bound~\eqref{the important bound on eta}. We still have the estimate $\tnorm{(1+\be+\eta)^{-1}\grad\eta}_{L^2}\lesssim\tnorm{\grad\eta}_{L^2}$ and we still have the expression~\ref{the derivative identity for this thing} for the derivative. By taking the $L^2$ norm of this expression, we obtain the following bound (which happens to be slightly shaper than~\eqref{Sphynx})
    \begin{equation}
        \tnorm{\grad\tp{\tp{1+\be+\eta}^{-1}\grad\eta}}_{L^2}\lesssim\tnorm{\grad\eta}_{H^1}+\tnorm{\grad\eta}_{L^4}^2.
    \end{equation}
    Now instead of using~\eqref{not good enough for the solitary case}, we instead use the homogeneous form of the Gagliardo-Nirenberg interpolation estimates (see~\eqref{strong form in Rd GNI} of Theorem~\ref{thm on gagliardo nirenberg}) to see that
    \begin{equation}
        \tnorm{\grad\eta}^2_{L^4}\lesssim\tnorm{\eta}_{L^\infty}\tnorm{\grad^2\eta}_{L^2}\lesssim\tnorm{\grad\eta}_{H^1}.
    \end{equation}
    Therefore, we deduce the estimate
    \begin{equation}\label{good enough for the solitary case}
        \tnorm{\tp{1+\be+\eta}^{-1}\grad\eta}_{H^1}\lesssim\tnorm{\grad\eta}_{H^1}.
    \end{equation}
    We combine~\eqref{good enough for the solitary case} with the integral identity
    \begin{equation}
        \int_{\R^2}G|\grad\eta|^2+|\grad^2\eta|^2=\int_{\R^2}\tp{1+\be+\eta}\tp{u\otimes u-\mathbb{S}u}:\grad\tp{\tp{1+\be+\eta}^{-1}\grad\eta}+\tp{\Phi-Au}\cdot\tp{1+\be+\eta}^{-1}\grad\eta,
    \end{equation}
    and the bounds
    \begin{equation}
        \tnorm{\tp{1+\be+\eta}\tp{u\otimes u-\mathbb{S}u}}_{L^2}\lesssim\tnorm{u}_{L^4}^2+\tnorm{u}_{H^1}\lesssim\tnorm{\Phi}_{H^{-1}}\tbr{\tnorm{\Phi}_{H^{-1}}},
    \end{equation}
    \begin{equation}
        \tnorm{\Phi-Au}_{H^{-1}}\lesssim\tnorm{\Phi}_{H^{-1}},\quad\tnorm{\grad\eta}_{H^1}^2\lesssim\int_{\R^2}G|\grad\eta|^2+|\grad^2\eta|^2
    \end{equation}
    to deduce that $\tnorm{\grad\eta}_{H^1}\lesssim\tnorm{\Phi}_{H^{-1}}\tbr{\tnorm{\Phi}_{H^{-1}}}$. This proves the left hand inequality in~\eqref{a priori bounds, solitary case}.

    To prove the higher regularity estimate, we will again move to a semilinear form of the equations. Fortunately, we will not be using the semilinear system suggested by the decomposition $\mathcal{P}_\be+\mathcal{L}_\be+\mathcal{N}_\be=\mathcal{Q}_\be$ from Proposition~\ref{prop on operator decomposition - solitary case}; rather, we rewrite system~\eqref{stationary nondimensionalized equations} in the following more pleasant form (which is analogous to~\eqref{a useful choice of semilinear equations in the periodic case}):
    \begin{equation}
        \begin{cases}
            \grad\cdot u=-u\cdot\grad\log\tp{1+\be+\eta},\\
            Au-\grad\cdot\mathbb{S}u+\tp{G-\Delta}\grad\eta=\tp{1+\be+\eta}^{-1}\tp{\Phi+A\tp{\be+\eta}u}-u\cdot\grad u+\mathbb{S}u\grad\log\tp{1+\be+\eta}.
        \end{cases}
    \end{equation}
    The goal now is to use the linear estimate from Lemma~\ref{lem on unweighted estimates for the principal part}, and for this we are tasked with bounding the first equation's right hand side in $\tp{\dot{H}^{-1}\cap H^1}\tp{\R^2}$ and the second equation's in $L^2\tp{\R^2;\R^2}$. Estimating the negative Sobolev norm is trivial:
    \begin{equation}\label{F_0}
        \tnorm{u\cdot\grad\log\tp{1+\be+\eta}}_{\dot{H}^{-1}}=\tnorm{\grad\cdot u}_{\dot{H}^{-1}}\lesssim\tnorm{u}_{L^2}\lesssim\tnorm{\Phi}_{L^2}.
    \end{equation}
    To estimate the $H^1$ norm we use Theorem~\ref{thm on high low product estimates} (and~\eqref{good enough for the solitary case}):
    \begin{equation}\label{F_1}
        \tnorm{u\cdot\grad\log\tp{1+\be+\eta}}_{H^1}\lesssim\tnorm{u}_{W^{1,4}}\tbr{\tnorm{\grad\eta}_{L^4}}+\tnorm{u}_{L^\infty}\tbr{\tnorm{\grad\eta}_{H^1}}\lesssim\tnorm{u}_{W^{1,4}}\tbr{\tnorm{\Phi}_{L^2}}^2.
    \end{equation}
    We then make bounds analogous to~\eqref{T__(1)}, \eqref{T__(2)}, and~\eqref{T__(3)}:
    \begin{equation}\label{F_2}
        \tnorm{\tp{1+\be+\eta}^{-1}\tp{\Phi+A\tp{\be+\eta}u}}_{L^2}\lesssim\tnorm{\Phi}_{L^2},\quad\tnorm{u\cdot\grad u}_{L^2}\lesssim\tnorm{u}_{W^{1,4}}\tnorm{\Phi}_{L^2},
    \end{equation}
    and
    \begin{equation}\label{F_3}
        \tnorm{\mathbb{S}u\grad\log\tp{1+\be+\eta}}_{L^2}\lesssim\tnorm{u}_{W^{1,4}}\tbr{\tnorm{\Phi}_{L^2}}^2.
    \end{equation}
    By combining~\eqref{F_0}, \eqref{F_1}, \eqref{F_2}, and~\eqref{F_3} with the bounds of Lemma~\ref{lem on unweighted estimates for the principal part} we find that
    \begin{equation}
        \tnorm{u,\grad\eta}_{H^2\times H^2}\lesssim\tnorm{\Phi}_{L^2}+\tnorm{u}_{W^{1,4}}\tbr{\tnorm{\Phi}_{L^2}}^2.
    \end{equation}
    The very same interpolation and absorption argument at the end of the proof of Proposition~\ref{prop on periodic a priori estimates} applies here as well and we can close the right hand estimate in~\eqref{a priori bounds, solitary case}.
\end{proof}

Our next result shows that if we know some improved integrability on either the velocity or the free surface then we may promote all the up to full weighted control. Recall the notation for solenoidal vector fields introduced in~\eqref{the definition of the solenoidal space}.

\begin{prop}[Weighted a priori estimates]\label{prop on weighted a priori estimates version 1}
    Suppose that for some $\ep\in(0,1)$, $4/3<q<2$, and $4\le r<\infty$ we have $(u,\eta)\in{_{\m{sol}}} H^{2}_{1/2}\tp{\R^2;\R^2}\times\bf{U}_\be^\ep\cap\tilde{H}^3_{1/2}\tp{\R^2}$ with $\tnorm{\eta}_{L^\infty}\le\ep^{-1}$ and $\min\tcb{\tnorm{u}_{L^q},\tnorm{\eta}_{L^r}}\le\ep^{-1}$ a solution to identity $\mathcal{Q}_\be\tp{u,\eta}=\Phi$ for a generic $\Phi\in L^2_{\scriptscriptstyle1/2}\tp{\R^2;\R^2}$. Then, for implicit constants depending only on $\tnorm{\Phi}_{L^2_{\scriptscriptstyle{1/2}}}$, $\ep$, $q$, $r$, and the physical parameters, we have the a priori bounds
    \begin{equation}\label{a priori weighted estimate in the solitary case}
        \tnorm{u,\eta}_{H^2_{1/2}\times\tilde{H}^3_{1/2}}\lesssim 1.
    \end{equation}
\end{prop}
\begin{proof}
    The first step is to unpack the unweighted a priori bounds on the solution that we get from Proposition~\ref{prop on solitary case unweighted a priori estimates}. To do this we require Propositions~\ref{prop on operator formulation solitary case} and~\ref{prop on operator decomposition - solitary case} to deduce that the functions $\pmb{u}\in H^2_{\scriptscriptstyle1/2}\tp{\R^2;\R^2}$ and $\pmb{\eta}\in\bf{O}_\be^{\ep}\cap\tilde{H}^3_{\scriptscriptstyle1/2}\tp{\R^2}$ defined via
    \begin{equation}\label{these are the variables to which we shall later change back}
        \pmb{u}=\tp{2\eta+\tp{1+\be}^2}^{-1/2}u \text{ and } \pmb{\eta}=\sqrt{2\eta+\tp{1+\be}^2}-\tp{1+\be}
    \end{equation}
    are a solution to system~\eqref{stationary nondimensionalized equations} for $\kappa=1$ and forcing $\Phi\in L^2_{\scriptscriptstyle1/2}\tp{\R^2;\R^2}$. It is clear that $\tnorm{\pmb{\eta}}_{L^\infty}\lesssim 1$; therefore, we can invoke Proposition~\ref{prop on solitary case unweighted a priori estimates} to deduce the bounds $\tnorm{\pmb{u},\grad\pmb{\eta}}_{H^2\times H^2}\lesssim 1$. We claim that the same estimate holds for $u$ and $\grad\eta$. We compute that
    \begin{equation}
        \grad\eta=\tp{1+\be+\pmb{\eta}}\grad\pmb{\eta}+\pmb{\eta}\grad\be,\quad\grad^2\eta=\tp{1+\be+\pmb{\eta}}\grad^2\pmb{\eta}+\grad\pmb{\eta}\otimes\grad\pmb{\eta}+\grad\be\otimes\grad\pmb{\eta}+\grad\pmb{\eta}\otimes\grad\be+\pmb{\eta}\grad^2\be.
    \end{equation}
    Now we use Theorem~\ref{thm on high low product estimates} and the bounds $\tnorm{\eta,\grad\eta}_{L^\infty\times H^2}\lesssim 1$ to estimate
    \begin{equation}
        \tnorm{\grad\eta}_{L^2}\lesssim\tnorm{\grad\pmb{\eta},\pmb{\eta}}_{L^2\times L^\infty}\lesssim1,\quad\tnorm{\grad^2\eta}_{H^1}\lesssim\tbr{\tnorm{\pmb{\eta}}_{W^{1,\infty}}}\tnorm{\grad^2\pmb{\eta}}_{H^1}+\tnorm{\grad\pmb{\eta}}^2_{W^{1,4}}+\tnorm{\grad\pmb{\eta}}_{H^1}+\tnorm{\pmb{\eta}}_{W^{1,\infty}}\lesssim1.
    \end{equation}
    It now follows that $\tnorm{\grad\eta}_{H^2}\lesssim1$. To obtain the sought after estimates on $u$ we use the identity $u=(1+\be+\pmb{\eta})\pmb{u}$: 
    \begin{equation}
        \tnorm{u}_{H^2}\lesssim\tnorm{u}_{L^2}+\tnorm{\grad\tp{\be+\pmb{\eta}}\pmb{u}}_{H^1}+\tnorm{\tp{1+\be+\pmb{\eta}}\grad\pmb{u}}_{H^1}\lesssim 1+\tbr{\tnorm{\grad\pmb{\eta}}_{W^{1,4}}}\tnorm{\pmb{u}}_{W^{1,4}}+\tbr{\tnorm{\pmb{\eta}}_{W^{1,\infty}}}\tnorm{\grad\pmb{u}}_{H^1}\lesssim 1.
    \end{equation}
    Thus, we can indeed transfer the unweighted bounds $\tnorm{\pmb{u},\grad\pmb{\eta}}_{H^2\times H^2}\lesssim 1$ to get $\tnorm{u,\grad\eta}_{H^2\times H^2}\lesssim 1$. Notice that we have thus far not used the assumption that $\min\tcb{\tnorm{u}_{L^q},\tnorm{\eta}_{L^r}}\le\ep^{-1}$.

    The next step of the proof is to show that if $\tnorm{u}_{L^q}\le\ep^{-1}$, then $\tnorm{\eta}_{L^{2q/(2-q)}}\lesssim 1$. In other words, knowledge of improved velocity integrability implies finite integrability for the free surface. So let us now assume that $\tnorm{u}_{L^q}\le\ep^{-1}$.  We look to the semilinearized version of the equations granted by Proposition~\ref{prop on operator decomposition - solitary case}, i.e
    \begin{equation}
        \mathcal{P}_\be\tp{u,\eta}=\Phi-\mathcal{L}_\be\tp{u,\eta}-\mathcal{N}_\be\tp{u,\eta},
    \end{equation}
    where we recall that $\mathcal{P}_\be$ and $\mathcal{L}_\be$ are defined in Proposition~\ref{prop on linear analysis with bathymetry solitary case}, while $\mathcal{N}_\be$ is defined in~\eqref{nonlinearity in the solitary case}.  We  take the norm in $L^q\tp{\R^2;\R^2}$ on both sides of this equation. The $\Phi$ contribution is estimated through the embedding $L^2_{\scriptscriptstyle1/2}\tp{\R^2;\R^2}\emb L^q\tp{\R^2;\R^2}$ (which holds since $4/3<q<2$). For the left hand side it is the second item of Lemma~\ref{lem on unweighted estimates for the principal part} that gives us a coercive bound. This leaves us with
    \begin{equation}\label{this too is going to be cited later}
        \tnorm{u,\eta,\grad\eta}_{W^{2,q}\times L^{2q/(2-q)}\times W^{2,q}}\lesssim 1+\tnorm{\mathcal{L}_\be\tp{u,\eta}}_{L^q}+\tnorm{\mathcal{N}_\be\tp{u,\eta}}_{L^q}.
    \end{equation}
    The idea is now that each of the expressions on the right consists of products, either with $\be$ or the solution itself, and so by using the assumed $u\in L^q\tp{\R^2;\R^2}$ control paired with the unweighted $L^2$ control on $u$ and $\grad\eta$ each term is then promoted to the improved $L^q$ integrability.

    The expression for $\mathcal{L}_\be$ consists of four terms. The bound of the first of these is one of two places where the $L^q$-integrability of $u$ is used (the other is within the expression for $\mathcal{N}_\be$). We estimate this term simply as
    \begin{equation}
        \bnorm{\f{A\tp{\be-\bf{c}\tp{\be}}}{\tp{1+\be}\tp{1+\bf{c}\tp{\be}}}u}_{L^q}\lesssim\tnorm{u}_{L^q}\lesssim 1
    \end{equation}
    while the remaining terms in $\mathcal{L}_\be$ are treated to either H\"older's inequality or Theorem~\ref{thm on high low product estimates}. We have
    \begin{equation}
        \tnorm{\grad\cdot Q_2\tp{\grad\log\tp{1+\be},u}}_{L^q}\lesssim\tnorm{\grad\log\tp{1+\be}}_{W^{1,2q/(2-q)}}\tnorm{u}_{H^1}\lesssim 1.
    \end{equation}
    We also use $\be\in{_{\m{e}}}H^{\infty}_\del\tp{\R^2}$ ($\del>0$) with the embedding $L^2_\updelta\tp{\R^2}\emb L^{4/(2+\updelta)}\tp{\R^2}$ and interpolation to bound (for $\updelta<\del$ sufficiently small as to satisfy $0<\updelta<2/q-1$)
    \begin{equation}\label{unfortunately i need to cite this equation}
        \tnorm{\tp{\tp{1+\be}^{-1}\tp{G-\Delta}\grad\tp{1+\be}}\eta}_{L^q}\lesssim\tnorm{\grad\be}_{W^{2,\f{4}{2+\updelta}}}\tnorm{\eta}_{L^{\f{4q}{4-(2+\updelta)q}}}\lesssim\tnorm{\eta}_{L^{2q/(2-q)}}^{1-\f{\updelta q}{4-2q}}\tnorm{\eta}_{L^\infty}^{\f{\updelta q}{4-2q}}\lesssim\tnorm{\eta}_{L^{2q/(2-q)}}^\theta,
    \end{equation}
    where $\theta=\updelta q\tp{4-2q}^{-1}\in(0,1)$. We bound the final term in $\mathcal{L}_\be$ according to
    \begin{equation}
        \tnorm{\grad\cdot Q_1\tp{\grad\tp{1+\be},\grad\tp{\tp{1+\be}^{-1}\eta}}}_{L^q}\lesssim\tnorm{\grad\be}_{W^{1,2q/(2-q)}}\tnorm{\grad\tp{(1+\be)^{-1}\eta}}_{L^2}\lesssim\tnorm{\eta,\grad\eta}_{L^\infty\times L^2}\lesssim 1.
    \end{equation}

    Next, we bound the six expressions within $\mathcal{N}_\be$ by implementing similar tricks. We have
    \begin{equation}
        \bnorm{\grad\cdot\bp{\f{u\otimes u}{\sqrt{2\eta+\tp{1+\be}^2}}}}_{L^q}\lesssim\tnorm{\tp{2\eta+\tp{1+\be}^2}^{-1/2}}_{W^{1,\infty}}\tnorm{u}_{W^{1,2q}}^2\lesssim 1.
    \end{equation}
    The second term in $\mathcal{N}_\be$ is where we shall again use $\tnorm{u}_{L^q}\lesssim 1$:
    \begin{equation}
        \norm{\f{2Au\eta}{\tp{1+\be}\sqrt{2\eta+\tp{1+\be}^2\tp{\tp{1+\be}+\sqrt{2\eta+\tp{1+\be}^2}}}}}_{L^q}\lesssim\tnorm{u}_{L^q}\tnorm{\eta}_{L^\infty}\lesssim 1.
    \end{equation}
    For the third, fifth, and sixth terms, we just use Theorem~\ref{thm on high low product estimates}:
    \begin{equation}
        \bnorm{\grad\cdot Q_2\bp{u,\grad\log\bp{1+\f{2\eta}{\tp{1+\be}^2}}}}_{L^q}\lesssim\tnorm{u}_{H^1}\bnorm{\grad\log\bp{1+\f{2\eta}{\tp{1+\be}^2}}}_{W^{1,2q/(2-q)}}\lesssim 1,
    \end{equation}
    \begin{multline}
        \bnorm{\grad\cdot Q_1\bp{\grad\bp{\f{2\eta}{\sqrt{2\eta+\tp{1+\be}^2}+(1+\be)}}}}_{L^q}\\\lesssim\bnorm{\grad\bp{\f{2\eta}{\sqrt{2\eta+\tp{1+\be}^2}+\tp{1+\be}}}}_{L^{2q/(2-q)}}\bnorm{\grad\bp{\f{2\eta}{\sqrt{2\eta+\tp{1+\be}^2}+\tp{1+\be}}}}_{H^1}\lesssim1,
    \end{multline}
    and
    \begin{multline}
        \norm{\grad\cdot Q_1\bp{\grad\tp{1+\be},\grad\bp{\f{\eta^2}{\tp{1+\be}^2\sqrt{2\eta+\tp{1+\be}^2}+\tp{1+\be}\tp{\tp{1+\be}^2+\eta}}}}}_{L^q}\\\lesssim\tnorm{\grad\beta}_{W^{1,2q/(2-q)}}\norm{\grad\bp{\f{\eta^2}{\tp{1+\be}^2\sqrt{2\eta+\tp{1+\be}^2}+\tp{1+\be}\tp{\tp{1+\be}^2+\eta}}}}_{H^1}\lesssim 1.
    \end{multline}
    For the fourth term we instead argue as in~\eqref{unfortunately i need to cite this equation}:
    \begin{equation}
        \norm{\f{\eta^2\tp{G-\Delta}\grad\tp{1+\be}}{\tp{1+\be}^2\sqrt{2\eta+\tp{1+\be}^2}+\tp{1+\be}\tp{\tp{1+\be}^2+\eta}}}_{L^q}\lesssim\tnorm{\eta}_{L^\infty}\tnorm{\eta}_{L^{2q/(2-q)}}^\theta\lesssim\tnorm{\eta}^\theta_{L^{2q/(2-q)}}.
    \end{equation}
    
    By synthesizing the above estimates on $\mathcal{L}_\be$ and $\mathcal{N}_\be$ and returning to~\eqref{this too is going to be cited later} we obtain 
    \begin{equation}
        \tnorm{\eta}_{L^{2q/(2-q)}}\lesssim 1+\tnorm{\eta}^\theta_{L^{2q/(2-q)}}\quad\text{and so we can absorb in the standard way to get}\quad\tnorm{\eta}_{L^{2q/(2-q)}}\lesssim 1.
    \end{equation}
    This completes the part of the proof in which we show that improved integrability on $u$ leads to finite integrability on $\eta$.

    We now come to the final step of the proof. In light of the previous step, it now suffices to show that if $\tnorm{\eta}_{L^\infty}+\tnorm{\eta}_{L^r}\lesssim 1$ for some $4\le r<\infty$ then we can obtain the weighted estimate~\eqref{a priori weighted estimate in the solitary case}. Initially, we need to transform back into the variables $\pmb{u}$ and $\pmb{\eta}$ from~\eqref{these are the variables to which we shall later change back} as these solve system~\eqref{stationary nondimensionalized equations} which has a more exploitable energy structure. Note that by writing
    \begin{equation}
        \pmb{\eta}=2\tp{\sqrt{2\eta+\tp{1+\be}^2}+\tp{1+\be}}^{-1}\eta
    \end{equation}
    we can transfer the $L^r$ bound on $\eta$ to $\pmb{\eta}$: $\tnorm{\pmb{\eta}}_{L^r}\lesssim 1$. Let us now define the weight $\rho=1/r\in(0,1/4]$. By utilizing the energy structure of~\eqref{stationary nondimensionalized equations} we show initially that $\pmb{u}\in H^1_\rho\tp{\R^2;\R^2}$ and $\pmb{\eta}\in\tilde{H}^2_{\rho}\tp{\R^2}$ are controlled inclusions. We test the second equation in~\eqref{stationary nondimensionalized equations} with $\tbr{X}^{2\rho}\pmb{u}\in H^2\tp{\R^2;\R^2}$ and integrate by parts to deduce the following identity.
    \begin{multline}\label{you can keep em}
        \int_{\R^2}\tbr{X}^{\rho}\Phi\cdot\tbr{X}^\rho \pmb{u}=\int_{\R^2}\tbr{X}^{2\rho}\ssb{A|\pmb{u}|^2+\tp{1+\be+\pmb{\eta}}\mathbb{S}\pmb{u}:\grad\pmb{u}}\\+\rho\int_{\R^2}-\tp{1+\be+\pmb{\eta}}\tbr{X}^{2\rho-2}X\cdot\pmb{u}|\pmb{u}|^2+2\tp{1+\be+\pmb{\eta}}\mathbb{S}\pmb{u}:\pmb{u}\otimes\tbr{X}^{2\rho-2}X
        -2\tp{1+\be+\pmb{\eta}}\tp{G-\Delta}\pmb{\eta}\pmb{u}\cdot\tbr{X}^{2\rho-2}X.
    \end{multline}
    The first line's right hand side gives us our weighted control on $\pmb{u}$. We shall need to estimate the terms appearing in the second line. Since $\rho<1/2$ we have $\tbr{X}^{2\rho-2}X\in L^\infty$ and so thanks to H\"older we have
    \begin{equation}\label{my mized pinabbple}
        \babs{\int_{\R^2}-\tp{1+\be+\pmb{\eta}}\tbr{X}^{2\rho-2}X\cdot\pmb{u}|\pmb{u}|^2+2\tp{1+\be+\pmb{\eta}}\mathbb{S}\pmb{u}:\pmb{u}\otimes\tbr{X}^{2\rho-2}X}\lesssim\tnorm{\pmb{u}}_{L^2}\tp{\tnorm{\pmb{u}}_{L^4}^2+\tnorm{\pmb{u}}_{H^1}}\lesssim 1.
    \end{equation}
    We may also estimate the capillary contribution with ease:
    \begin{equation}\label{fo shixzle}
        \babs{\int_{\R^2}2(1+\be+\pmb{\eta})\Delta\pmb{\eta}\pmb{u}\cdot\tbr{X}^{2\rho-2}X}\lesssim\tnorm{\grad\pmb{\eta}}_{H^1}\tnorm{\pmb{u}}_{L^2}\lesssim1.
    \end{equation}
    The gravity contribution of~\eqref{you can keep em} is where the finite integrability on $\pmb{\eta}$ and the exponent relation $\f{2r}{r-2}\tp{1-\rho}>2$ are used:
    \begin{equation}\label{try tpo express}
        \babs{\int_{\R^2}-2G\tp{1+\be+\pmb{\eta}}\pmb{\eta}\pmb{u}\cdot\tbr{X}^{2\rho-2}X}\lesssim\tnorm{\pmb{u}}_{L^2_\rho}\tnorm{\pmb{\eta}}_{L^r}\tnorm{\tbr{X}^{\rho-1}}_{L^{2r/(r-2)}}\lesssim\tnorm{\pmb{u}}_{L^2_\rho}.
    \end{equation}
    We combine~\eqref{you can keep em}, \eqref{my mized pinabbple}, \eqref{fo shixzle}, and~\eqref{try tpo express} to gain
    \begin{equation}
        \tnorm{\pmb{u}}^2_{H^1_\rho}\lesssim1+\tbr{\tnorm{\Phi}_{L^2_\rho}}\tnorm{\pmb{u}}_{L^2_\rho}\quad\text{which leads us to}\quad\tnorm{\pmb{u}}_{H^1_\rho}\lesssim 1.
    \end{equation}
    The next thing to do is use this new weighted bound on $\pmb{u}$ to prove a $\rho$-weighted bound on $\grad\pmb{\eta}$; to achieve this we test the equation~\eqref{stationary nondimensionalized equations} with $\tbr{X}^{2\rho}\tp{1+\be+\pmb{\eta}}^{-1}\grad\pmb{\eta}\in H^2\tp{\R^2;\R^2}$ and integrate by parts as appropriate to derive the following identity.
    \begin{multline}\label{_the_first_weight_}
        \int_{\R^2}\tbr{X}^{2\rho}\tp{G\tabs{\grad\pmb{\eta}}^2+\tabs{\grad^2\pmb{\eta}}^2}=\int_{\R^2}\tp{1+\be+\pmb{\eta}}^{-1}\tbr{X}^\rho\tp{\Phi-A\pmb{u}}\cdot\tbr{X}^\rho\grad\pmb{\eta}-2\rho\grad^2\pmb{\eta}:\grad\pmb{\eta}\otimes\tbr{X}^{2\rho-2}X\\+\int_{\R^2}\tbr{X}^{\rho}\tp{1+\be+\pmb{\eta}}\tp{\pmb{u}\otimes\pmb{u}-\mathbb{S}\pmb{u}}:\tbr{X}^{\rho}\grad\tp{\tp{1+\be+\pmb{\eta}}^{-1}\grad\pmb{\eta}}+2\rho\tp{\pmb{u}\otimes\pmb{u}-\mathbb{S}\pmb{u}}:\grad\pmb{\eta}\otimes\tbr{X}^{2\rho-2}X.
    \end{multline}
    By implementing H\"older's inequality with the weighted bounds on $\pmb{u}$ and $\Phi$, the unweighted bounds on $\pmb{u}$ and $\pmb{\eta}$, and
    \begin{equation}
        \tnorm{\grad\tp{\tp{1+\be+\pmb{\eta}}^{-1}\grad\pmb{\eta}}}_{L^2_\rho}\lesssim\tnorm{\grad^2\pmb{\eta}}_{L^2_\rho}+\tbr{\tnorm{\grad\pmb{\eta}}_{L^\infty}}\tnorm{\grad\pmb{\eta}}_{L^2_\rho}\lesssim\tnorm{\grad\pmb{\eta}}_{H^1_\rho}
    \end{equation}
    we derive from~\eqref{_the_first_weight_} the estimate
    \begin{equation}
        \tnorm{\grad\pmb{\eta}}_{H^1_\rho}^2\lesssim1+\tnorm{\grad\pmb{\eta}}_{H^1_\rho}\quad\text{and hence}\quad\tnorm{\grad\pmb{\eta}}_{H^1_\rho}\lesssim 1.
    \end{equation}

    We now have established the control $\tnorm{\pmb{u},\pmb{\eta}}_{H^1_\rho\times\tilde{H}^2_\rho}\lesssim1$ and it is time to again swap back to the variables $u$ and $\eta$ determined from the inverse transformation in~\eqref{these are the variables to which we shall later change back}. From the identity
    \begin{equation}
        \eta=2^{-1}\tp{\tp{1+\be+\pmb{\eta}}^2-\tp{1+\be}^2}=2^{-1}\pmb{\eta}^2+\tp{1+\be}\pmb{\eta}
    \end{equation}
    and the product estimates of Proposition~\ref{prop on product estimates in weighted gradient spaces} we deduce immediately that $\tnorm{\eta}_{\tilde{H}^2_\rho}\lesssim 1$. Also, by using $u=\tp{1+\be+\pmb{\eta}}\pmb{u}$ and the boundedness of the product map $W^{1,\infty}\times H^1_\rho\to H^1_\rho$ we find $\tnorm{u}_{H^1_\rho}\lesssim 1$.

    Next, we write the equation $\mathcal{Q}_\be\tp{u,\eta}=\Phi$ satisfied by $u$ and $\eta$ into the form (which is permissible by Propositions~\ref{prop on linear analysis with bathymetry solitary case} and~\ref{prop on operator decomposition - solitary case})
    \begin{equation}
        (u,\eta)=\tp{\mathcal{P}_\be+\mathcal{L}_\be}^{-1}\tp{\Phi-\mathcal{N}_\be\tp{u,\eta}}
    \end{equation}
    and so for any weight $\rho\le\tilde{\rho}\le1/2$ we have the estimate
    \begin{equation}\label{setting up notation for a later induction argument}
        \tnorm{u,\eta}_{H^2_{\tilde{\rho}}\times\tilde{H}^3_{\tilde{\rho}}}\lesssim 1+\tnorm{\mathcal{N}_\be\tp{u,\eta}}_{L^2_{\tilde{\rho}}}.
    \end{equation}
    We first shall examine~\eqref{setting up notation for a later induction argument} with $\tilde{\rho}=\rho$ so that we may promote our initial weighted estimate up a derivative. Thanks to repeated applications of the known bounds with the product estimates of Proposition~\ref{prop on product estimates} we may derive an estimate in the spirit of the proof of Proposition~\ref{prop on operator decomposition - solitary case} on the nonlinearity of the form
    \begin{equation}\label{im so sorry the details are missing this is getting really long of a proof at this point}
        \tnorm{\mathcal{N}_\be\tp{u,\eta}}_{L^2_\rho}\lesssim1+\tnorm{u}_{W^{1,4}_\rho}+\tnorm{\grad\eta}_{W^{1,4}_\rho}\lesssim 1+\tnorm{\tbr{X}^\rho u}_{H^2}^{1/2}+\tnorm{\tbr{X}^\rho\grad\eta}^{1/2}_{H^2}.
    \end{equation}
    Then upon combining~\eqref{setting up notation for a later induction argument} and~\eqref{im so sorry the details are missing this is getting really long of a proof at this point} and implementing the standard absorption argument we arrive at $\tnorm{u,\eta}_{H^2_\rho\times\tilde{H}^3_\rho}\lesssim 1$.

    We are finally in a position where we can iterate the conclusion from the section item of Proposition~\ref{prop on operator decomposition - solitary case}. By the $\rho$-weighted estimate we have just derived we find that for every $\tilde{\rho}<2\rho$ we now learn that $\tnorm{\mathcal{N}_\be\tp{u,\eta}}_{L^2_{\tilde{\rho}}}\lesssim 1$ and so through~\eqref{setting up notation for a later induction argument} we gain the control $\tnorm{u,\eta}_{H^2_{\tilde{\rho}}\times\tilde{H}^2_{\tilde{\rho}}}\lesssim 1$. In fact, if we let $n\in\N^+$ be the smallest positive integer such that $(3/2)^{-n}\tp{1/2}<2\rho$ and set for $j\in\tcb{1,\dots,n}$ the weight strength $\m{w}_j=(3/2)^{-j}(1/2)$. By what we have just argued we know that $\tnorm{u,\eta}_{H^2_{\m{w}_n}\times\tilde{H}^3_{\m{w}_n}}\lesssim 1$. Furthermore, the combination of~\eqref{setting up notation for a later induction argument} with the aforementioned bounds on $\mathcal{N}_\be$ show that if for some $j\in\tcb{2,\dots,n}$ we know that
    \begin{equation}
        \tnorm{u,\eta}_{H^2_{\m{w}_j}\times\tilde{H}^3_{\m{w}_j}}\lesssim 1\quad\text{then we actually find that}\quad\tnorm{u,\eta}_{H^2_{\m{w}_{j-1}}\times\tilde{H}^3_{\m{w}_{j-1}}}\lesssim 1.
    \end{equation}
    Therefore the obvious induction argument leads us at last to the bound~\eqref{a priori weighted estimate in the solitary case}.    
\end{proof}

We now come to the solitary analog of Corollary~\ref{coro on blowup refinement - periodic case}.

\begin{coro}[Blowup refinement - solitary case]\label{coro on blowup refinement solitary case}
    Let $\mathscr{S}\subset S$ be the maximal locally analytic curve of solutions granted by Proposition~\ref{prop on maximally locally analytic curve of solutions} in the solitary case, parametrized by the continuous map $\R\ni s\mapsto\tp{u(s),\eta\tp{\eta},\kappa\tp{s}}\in\mathscr{S}\subset H^2_{\scriptscriptstyle1/2}\tp{\R^2;\R^2}\times\bf{U}_\be\cap\tilde{H}^3_{\scriptscriptstyle1/2}\tp{\R^2}$. Then for every choice of $q\in(4/3,2)$ and $r\in[4,\infty)$ we have
    \begin{equation}\label{the blowup quantity in the solitary case}
        \lim_{|s|\to\infty}\tsb{\tnorm{\eta\tp{s}}_{L^\infty}+\min\tcb{\tnorm{u(s)}_{L^q},\tnorm{\eta(s)}_{L^r}}+1/\inf\tp{\tp{1+\be}^2+2\eta\tp{s}}+|\kappa(s)|}=\infty.
    \end{equation}
\end{coro}
\begin{proof}
    The arguments for the limits $s\to\infty$ and $s\to-\infty$ are identical, so we only handle the former via the following contradiction argument. If the limit~\ref{the blowup quantity in the solitary case} were not true, then we could find a sequence $\tcb{s_n}_{n\in\N}\subset\R^+$ with $s_n\to\infty$ as $n\to\infty$ and a number $1<C<\infty$ such that upon setting $u_n=u(s_n)$, $\eta_n=\eta(s_n)$, $\kappa_n=\kappa\tp{s_n}$, and $\ep=C^{-2}\in\tp{0,1}$ we have the bound
    \begin{equation}\label{solitary contradition bounds}
        \sup_{n\in\N}\tsb{\tnorm{\eta_n}_{L^\infty}+\min\tcb{\tnorm{u_n}_{L^q},\tnorm{\eta_n}_{L^r}}+1/\inf\tp{\tp{1+\be}^2+2\eta_n}+|\kappa_n|}\le C.
    \end{equation}

    Due to Proposition~\ref{prop on maximally locally analytic curve of solutions} the solution curve $\mathscr{S}$ is a subset of $S$ from~\eqref{the set of solutions}; then, by the third item of Theorem~\ref{thm on unified fixed point reformulation} we find that the identity $\mathcal{Q}_\be\tp{u_n,\eta_n}=\kappa_n\mathcal{F}_\be\tp{\eta_n}$ is satisfied for every $n\in\N$. We are then led to study in conjunction the auxiliary variables $\tcb{\tp{\pmb{u}_n,\pmb{\eta}_n}}_{n\in\N}\subset H^2_{\scriptscriptstyle1/2}\tp{\R^2;\R^2}\times\bf{O}_\be^\ep\cap\tilde{H}^3_{\scriptscriptstyle1/2}\tp{\R^2}$
    \begin{equation}
        \pmb{u}_n=\tp{2\eta_n+\tp{1+\be}^2}^{-1/2}u_n,\quad\pmb{\eta}_n=\sqrt{2\eta_n+\tp{1+\be}^2}-\tp{1+\be},\quad\Phi_n=\mathcal{F}_\be\tp{\eta_n}
    \end{equation}
    which, thanks to Propositions~\ref{prop on data map - solitary case} and~\ref{prop on operator formulation solitary case} solve the system of equations
    \begin{equation}
        \begin{cases}
            \grad\cdot\tp{\tp{1+\be+\pmb{\eta}_n}\pmb{u}_n}=0,\\
            \grad\cdot\tp{\tp{1+\be+\pmb{\eta}_n}\pmb{u}_n\otimes\pmb{u}_n}+A\pmb{u}_n-\grad\cdot\tp{\tp{1+\be+\pmb{\eta}_n}\mathbb{S}\pmb{u}_n}+\tp{1+\be+\pmb{\eta}_n}\tp{G-\Delta}\grad\pmb{\eta}_n=\Phi_n,\\
            \Phi_n=\kappa_n\tsb{\Upomega_\phi\tp{\pmb{\eta}_n}+\grad\tp{\tp{1+\be+\pmb{\eta}_n}\Upomega_\psi\tp{\pmb{\eta}_n}}+\Upomega_\tau\tp{\pmb{\eta}_n}\grad\pmb{\eta}_n}.
        \end{cases}
    \end{equation}
    
    Our first task is to consider the consequences of the unweighted a priori estimates of Proposition~\ref{prop on solitary case unweighted a priori estimates} on the auxiliary variables. We then introduce the real analytic primitive 
    \begin{equation}
        \tilde{\tau}:(-a_\be,\infty)\to\tp{L^\infty_{\scriptscriptstyle1/2}\cap H^1_{\scriptscriptstyle1/2}}\tp{\R^2;\R^{2\times 2}_{\m{sym}}},\quad\tilde{\tau}(z)=\int_0^z\tau(\mathfrak{z})\;\m{d}\mathfrak{z}
    \end{equation}
    and argue as in Corollary~\ref{coro on blowup refinement - periodic case} (see~\eqref{the weirdo identity} and what follows) to see that the maps 
    \begin{equation}
        \Upomega_{\phi}:\bf{O}_\be^\ep\to L^2_{\scriptscriptstyle1/2}\tp{\R^2;\R^2},\quad\Upomega_\psi:\bf{O}^\ep_\be\to L^2_{\scriptscriptstyle1/2}\tp{\R^2},\quad\Upomega_{\tilde{\tau}}:\bf{O}_\be^\ep\to\tp{L^2_{\scriptscriptstyle1/2}\cap L^\infty_{\scriptscriptstyle1/2}}\tp{\R^2;\R^{2\times 2}_{\m{sym}}}
    \end{equation}
    send bounded sets to bounded sets and obey the identity
    \begin{equation}
        \Phi_n=\kappa_n\tsb{\Upomega_\phi(\pmb{\eta}_n)+\grad\tp{\tp{1+\be+\pmb{\eta}_n}\Upomega_\psi\tp{\pmb{\eta}_n}}+\grad\cdot\tp{\Upomega_{\tilde{\tau}}\tp{\pmb{\eta}_n}}-\Upomega_{\grad\cdot\tilde{\tau}}\tp{\pmb{\eta}_n}}.
    \end{equation}
    It then follows that $\sup_{n\in\N}\tnorm{\Phi_n}_{H^{-1}}\lesssim 1$ for an implicit constant depending only on $C$. We are then permitted to invoke the left a priori estimate of equation~\eqref{a priori bounds, solitary case} to deduce that $\sup_{n\in\N}\tnorm{\pmb{u}_n,\grad\pmb{\eta}_n}_{H^1\times H^1}\lesssim 1$. Now we combine the third item of Lemma~\ref{lem on free surface domsains, solitary case} with the second item of Proposition~\ref{prop on data map - solitary case} and these new bounds to deduce further that $\sup_{n\in\N}\tnorm{\Phi_n}_{L^2_{1/2}}\lesssim 1$.

    With these uniform forcing bounds in the weighted space, we have reached a position where Proposition~\ref{prop on weighted a priori estimates version 1} is in play. Since we also know that $\tcb{(u_n,\eta_n)}_{n\in\N}\subset{_{\m{sol}}}H^2_{\scriptscriptstyle1/2}\tp{\R^2;\R^2}\times\bf{U}_\be^\ep\cap\tilde{H}^{3}_{\scriptscriptstyle1/2}\tp{\R^2}$ and~\eqref{solitary contradition bounds} are true we can indeed invoke this result to learn that $\sup_{n\in\N}\tnorm{u_n,\eta_n}_{H^2_{1/2}\times\tilde{H}^3_{1/2}}\lesssim 1$. But now this is a contradiction of the limit~\eqref{the blowup quantity____} from Proposition~\ref{prop on maximally locally analytic curve of solutions}. Thus, our assumption that~\eqref{the blowup quantity in the solitary case} does not hold as $s\to\infty$ must be false, and the proof is complete.
\end{proof}

\section{Conclusions}\label{section on conclusions}

The purpose of this section is to complete the proofs of the main results, Theorems~\ref{z_MAIN_THEOREM_1}, \ref{z_MAIN_THEOREM_2}, \ref{z_MAIN_THEOREM_3}, and~\ref{z_MAIN_THEOREM_4}, from Section~\ref{section on main results and discussion}.

\subsection{Small solutions}\label{section on small solutions}

Our analysis thus far has been geared toward understanding the behavior of large solutions to system~\eqref{stationary nondimensionalized equations} with fixed forcing profile constituents $\phi$, $\psi$, and $\tau$ determining $\Phi$ as in~\eqref{percolate propane pulvarized paintbrush} while increasing the variable $\kappa$. It is also natural to study small solutions with small data and, as it turns out, our previous work is also capable of quickly felling the problem of developing such a theory. In this subsection we will consider $\kappa=1$ and study the data to solution map near zero.

We need to introduce some notation for the spaces of forcing data.  Let us fix an open and bounded rectangle $R\subset\C$ with sides parallel to the real and imaginary axes and center $0\in R$. Given a real-valued function space $\mathcal{Y}$, e.g. $\mathcal{Y}=L^2(\T^2_L;\R^2)$ or $\mathcal{Y}=L^2_{\scriptscriptstyle1/2}\tp{\R^2;\R^2}$, with complexification $\mathcal{Y}_{\C}$, we define the function space and norm
\begin{equation}\label{welcome to the}
    \pmb{A}\tp{\mathcal{Y}}=\tcb{\phi\in C^0\tp{\Bar{R};\mathcal{Y}_{\C}}\;:\;\phi\text{ holomorphic in }R \text{ and }\phi(\R\cap\Bar{R})\subset\mathcal{Y}},\quad\tnorm{\phi}_{\pmb{A}(\mathcal{Y})}=\sup_{z\in\Bar{R}}\tnorm{\phi(z)}_{\mathcal{Y}_{\C}}.
\end{equation}
A standard argument involving the Cauchy integral formula shows that space $\pmb{A}\tp{\mathcal{Y}}$ is Banach. We now make the following expansion of Definition~\ref{defn of periodic and solitary unification}. Let us denote the spaces $\pmb{Y}$, $\pmb{W}$, and $\pmb{Z}$ (in both the periodic and solitary cases) by
\begin{equation}\label{new kind of house cat}
    \pmb{Y}=\begin{cases}
        \pmb{A}\tp{L^2\tp{\T^2_L;\R^2}}\times\pmb{A}\tp{H^1\tp{\T^2_L}}\times\pmb{A}\tp{L^2\tp{\T^2_L;\R^{2\times 2}_{\m{sym}}}}&\text{in the }\T^2_L\text{ case},\\
        \pmb{A}\tp{L^2_{\scriptscriptstyle1/2}\tp{\R^2;\R^2}}\times\pmb{A}\tp{H^1_{\scriptscriptstyle1/2}\tp{\R^2}}\times\pmb{A}\tp{L^2_{\scriptscriptstyle1/2}\tp{\R^2;\R^{2\times 2}_{\m{sym}}}}&\text{in the }\R^2\text{ case},
        
    \end{cases}
\end{equation}
and
\begin{equation}
    \pmb{W}=\begin{cases}
        H^2\tp{\T^2_L;\R^2}\times \z{H}^3\tp{\T^2_L}&\text{in the }\T^2_L\text{ case},\\
        H^2_{\scriptscriptstyle1/2}\tp{\R^2;\R^2}\times\tilde{H}^3_{\scriptscriptstyle1/2}\tp{\R^2}&\text{in the }\R^2\text{ case},
    \end{cases}\quad\pmb{Z}=\begin{cases}
        L^2\tp{\T^2_L;\R^2}&\text{in the }\T^2_L\text{ case},\\
        L^2_{\scriptscriptstyle1/2}\tp{\R^2;\R^2}&\text{in the }\R^2\text{ case}.
    \end{cases}
\end{equation}
Note the similarity between the spaces $\pmb{W}$ and $\pmb{X}$ (where the latter are defined in~\eqref{the spaces for the solutions in a unified way}). Let us also fix $\iota\in\R^+$ sufficiently small so that the open set $B_{\pmb{W}}\tp{0,\iota}\subset\pmb{W}$ has the property that for all $\tp{u,\eta}\in B_{\pmb{W}}\tp{0,\iota}$ we have $1+\be+\eta>1/2$ and $\m{diam}\tp{\Bar{\m{img}\tp{\eta}}}<\m{diam}\tp{\Bar{R}\cap\R}/2$. Our last piece of notation we introduce is the following variable data map, which is a generalization of fixed data maps of Propositions~\ref{prop on data map, periodic case} and~\ref{prop on data map - solitary case}.

\begin{lem}[Variable data map]\label{lem on variable data map}
    The map $\mathcal{D}: B_{\pmb{W}}\tp{0,\iota} \times\pmb{Y}\to\pmb{Z}$ given via
    \begin{equation}\label{the formula for the data map function}
        \mathcal{D}\tp{\eta,\tp{\phi,\psi,\tau}}=\Upomega_\phi\tp{\eta}+\grad\tp{\tp{1+\be+\eta}\Upomega_\psi\tp{\eta}}+\Upomega_\tau\tp{\eta}\grad\eta
    \end{equation}
    is well-defined and real analytic.
\end{lem}
\begin{proof}
    This is a straightforward consequence of simple product estimates and Corollary~\ref{corollary on meta analyticity for HFC}.
\end{proof}

At last we arrive at our main result regarding small data well-posedness of system~\eqref{stationary nondimensionalized equations} with~\eqref{percolate propane pulvarized paintbrush} in both the periodic and solitary cases.

\begin{thm}[Theory of small solutions]\label{thm on theory of small solutions}
    There exists radii $r_{\pmb{W}},r_{\pmb{Y}}\in\R^+$ such that $r_{\pmb{W}}<\iota$ and $B_{\pmb{Y}}\tp{0,r_{\pmb{Y}}}\subset\pmb{Y}$ with the property that there exists a unique function $\Sampi:B_{\pmb{Y}}\tp{0,r_{\pmb{Y}}}\to B_{\pmb{W}}\tp{0,r_{\pmb{W}}}$ that satisfies $\Sampi\tp{0}=0$ and for all $(\phi,\psi,\tau)\in B_{\pmb{Y}}\tp{0,r_{\pmb{Y}}}$ the pair $(\pmb{u},\pmb{\eta})=\Sampi\tp{\phi,\psi,\tau}$ are a solution to the system of equations~\eqref{stationary nondimensionalized equations} with $\kappa=1$ and $\Phi=\mathcal{D}\tp{\eta,\tp{\phi,\psi,\tau}}$. Moreover, the mapping $\Sampi$ is real analytic.
\end{thm}
\begin{proof}
     We will verify these assertions via the implicit function theorem; however, the arguments in the periodic and solitary cases are slightly different, so we must break to cases.  For the periodic case we consider the mapping $\digamma:B_{\pmb{W}}\tp{0,\iota}\times\pmb{Y}\to \z{H}^1\tp{\T^2_L}\times\pmb{Z}$ given via
     \begin{equation}
         \digamma\tp{\tp{u,\eta},\tp{\phi,\psi,\tau}}=\mathcal{Q}_\be\tp{u,\eta}-\tp{0,\tp{1+\be+\eta}^{-1}\mathcal{D}\tp{\eta,\tp{\phi,\psi,\tau}}},
     \end{equation}
     where $\mathcal{Q}_\be$ is the operator from Proposition~\ref{prop on operator formulation periodic case}. That $\digamma$ is well-defined and real analytic is a consequence of the aforementioned result and Lemma~\ref{lem on variable data map}. Due to the decomposition of $\mathcal{Q}_\be$ from Proposition~\ref{prop on operator decomposition periodic case} we compute that $D_1\digamma(0,0)=\mathcal{P}+\mathcal{L}_\be$, where the latter two linear maps are studied in Proposition~\ref{prop on linear analysis with bathymetry - periodic case}. In particular, the final item of that result shows that $\mathcal{P}+\mathcal{L}_\be:\pmb{W}\to \z{H}^1\tp{\T^2_L}\times\pmb{Z}$ is an isomorphism. Therefore, the implicit function theorem applies for the map $\digamma$ and we obtain the claimed radii and the unique and analytic map $\Sampi$ satisfying $\digamma(\Sampi\tp{\phi,\psi,\tau},\tp{\phi,\psi,\tau})=0$ for all $\tp{\phi,\psi,\tau}\in B_{\pmb{Y}}\tp{0,r_{\pmb{Y}}}$. Thanks to $\mathcal{Q}_\be$ being an operator encoding of the left hand side of the PDE~\eqref{stationary nondimensionalized equations} we easily deduce that $\Sampi$ has the claimed solution operator properties.

     Next, we prove the result in the solitary case.  First recall that the space $\pmb{X}$ defined in~\eqref{the spaces for the solutions in a unified way} has the same norm as $\pmb{W}$ but enforces a diverge free condition on its first component; in particular, $B_{\pmb{W}}\tp{0,r}\cap\pmb X = B_{\pmb{X}}\tp{0,r}$ for $r >0$.  In the solitary case we use $\pmb{X}$ to initially define the operator  $\digamma: B_{\pmb{X}}\tp{0,\iota} \times\pmb{Y}\to\pmb{Z}$ given via
     \begin{equation}
         \digamma\tp{\tp{u,\eta},\tp{\phi,\psi,\tau}}=\mathcal{Q}_\be\tp{u,\mathcal{T}_\be\tp{\eta}}-\mathcal{D}\tp{\eta,\tp{\phi,\psi,\tau}}
     \end{equation}
     where $\mathcal{Q}_\be$ is the solitary encoding operator from Proposition~\ref{prop on operator formulation solitary case} and $\mathcal{T}_\be$ is the analytic change of unknowns from Lemma~\ref{lem on free surface domsains, solitary case}. These two results combine with Lemma~\ref{lem on variable data map} to show that $\digamma$ is well-defined and real analytic. We then compute (using the final item of Proposition~\ref{prop on operator decomposition - solitary case}) that $D_1\digamma\tp{0,0}(u,\eta)=\tp{\mathcal{P}_\be+\mathcal{L}_\be}\tp{u,(1+\be)\eta}$, so by Proposition~\ref{prop on linear analysis with bathymetry solitary case} we see that $D_1\digamma\tp{0,0}$ is the composition of invertible maps.  We have now verified the hypotheses of the implicit function theorem, so it grants us  $\tilde{r}_{\pmb{X}},\tilde{r}_{\pmb{Y}}\in\R^+$ with $\tilde{r}_{\pmb{X}} \le \iota$  and a unique and analytic mapping $\sampi:B_{\pmb{Y}}\tp{0,\tilde{r}_{\pmb{Y}}}\to B_{\pmb{X}}\tp{0,\tilde{r}_{\pmb{X}}}$ with the property that for all $(\phi,\psi,\tau)\in B_{\pmb{Y}}\tp{0,\tilde{r}_{\pmb{Y}}}$ we have $\digamma\tp{\sampi\tp{\phi,\psi,\tau},\tp{\phi,\psi,\tau}}=0$. We now define the map $\Sampi:B_{\pmb{Y}}\tp{0,\tilde{r}_{\pmb{Y}}}\to B_{\pmb{W}}\tp{0,\iota}$ via $\Sampi\tp{\phi,\psi,\tau}=\tp{\pmb{u},\pmb{\eta}}$ for $\pmb{u}=\tp{1+\be+\eta}^{-1}u$ and $\pmb{\eta}=\eta$ where $(u,\eta)=\sampi\tp{\phi,\psi,\tau}$. By the properties of $\mathcal{Q}_\be$ and $\mathcal{T}_\be$ we deduce that $\Sampi$ is indeed a solution operator to the PDE~\eqref{stationary nondimensionalized equations}; moreover, by restricting its domain and codomain as appropriate balls $B_{\pmb{Y}}\tp{0,r_{\pmb{Y}}}$ and $B_{\pmb{W}}\tp{0,r_{\pmb{W}}}\subset B_{\pmb{W}}\tp{0,\iota}$ it is seen to inherit the uniqueness assertion from $\sampi$. Finally, $\Sampi$ is real analytic as it is the composition of real analytic maps.
\end{proof}

\subsection{Large solutions}\label{section on large solutions}

Our goal here is to summarize and recontextualize the results of Section~\ref{section on large solutions and blow up criteria}. Previously, we have handled the solitary case of system~\eqref{stationary nondimensionalized equations} through the change of unknowns described in, e.g., equation~\eqref{these are the variables to which we shall later change back}. It is important to restate our results in terms of the natural variables of the parent system. We do this in a unified fashion for both the periodic and solitary cases.

Given any choice of real analytic data map constituents (as in Propositions~\ref{prop on data map, periodic case} and~\ref{prop on data map - solitary case})
\begin{equation}
   \kkappa=(\phi,\psi,\tau):(-a_\be,\infty)\to\begin{cases}
        L^2\tp{\T^2_L;\R^2}\times H^1\tp{\T^2_L}\times\tp{L^\infty\cap H^1}\tp{\T^2_L;\R^{2\times 2}_{\m{sym}}},\\
        L^2_{\scriptscriptstyle1/2}\tp{\R^2;\R^2}\times H^1_{\scriptscriptstyle1/2}\tp{\R^2}\times\tp{L^\infty_{\scriptscriptstyle1/2}\cap H^1_{\scriptscriptstyle1/2}}\tp{\R^2;\R^{2\times 2}_{\m{sym}}}
    \end{cases}
\end{equation}
we form the corresponding solution space (analogous to~\eqref{the set of solutions})
\begin{equation}\label{another set of solutions}
    S_{\kkappa}=\tcb{(\pmb{u},\pmb{\eta},\kappa)\in\pmb{W}\times\R\;:\;\inf\tp{1+\be+\pmb{\eta}}>0,\;\text{system~\eqref{stationary nondimensionalized equations} is satisfied by the triple with }\Phi=\mathcal{D}\tp{\pmb{\eta},\kkappa}},
\end{equation}
where we have extended the definition of $\mathcal{D}$ as in~\eqref{the formula for the data map function} to act on such $\pmb{\eta}$ and $\kkappa$. Notice that there exists some rectangle $Q$ as in the definitions~\eqref{welcome to the} and~\eqref{new kind of house cat} such that the restriction of the unique maximal analytic extension satisfies $\kkappa\in\pmb{Y}$. It is in this sense in which the following result discusses the nature of the solution operator from Theorem~\ref{thm on theory of small solutions} past its local definition. We shall show that the set~\eqref{another set of solutions} either contains solutions with arbitrarily large forcing strength, free surfaces arbitrarily close to touching the bottom, or arbitrarily `large' solutions. In the periodic case we can take `large' to mean the maximum height of the free surface while in the solitary case `large' refers to the sum of this aforementioned quantity with a quantification of the solution's locality.  

\begin{thm}[Theory of large solutions]\label{thm on theory of large solutions}
    There exists a locally analytic curve $\mathscr{S}_{\kkappa}\subseteq S_{\kkappa}$ with $0\in\mathscr{S}_\kkappa$ that is parametrized by a continuous mapping $\R\ni s\mapsto\tp{\pmb{u}(s),\pmb{\eta}(s),\kappa(s)}\in\mathscr{S}_\kkappa$ and satisfies $0\mapsto 0$ and for every choice of $q\in(4/3,2)$ and $r\in[4,\infty)$ 
    \begin{equation}\label{limit options periodic and solitary cases}
        \begin{cases}
            \lim_{|s|\to\infty}\tsb{\tnorm{\pmb{\eta}(s)}_{L^\infty}+\tp{\inf\tp{1+\be+\pmb{\eta}(s)}}^{-1}+|\kappa(s)|}=\infty&\text{in the }\T^2_L\text{ case},\\
            \lim_{|s|\to\infty}\tsb{\tnorm{\pmb{\eta}(s)}_{L^\infty}+\min\tcb{\tnorm{\pmb{u}(s)}_{L^q},\tnorm{\pmb{\eta}\tp{s}}_{L^r}}+\tp{\inf\tp{1+\be+\pmb{\eta}(s)}}^{-1}+|\kappa(s)|}=\infty&\text{in the }\R^2\text{ case}.
        \end{cases}
    \end{equation}
\end{thm}
\begin{proof}
    In the periodic case the result follows immediately from combining Proposition~\ref{prop on maximally locally analytic curve of solutions} and Corollary~\ref{coro on blowup refinement - periodic case}.  It remains to handle the solitary case. 
    
    Proposition~\ref{prop on maximally locally analytic curve of solutions} grants us a continuous and locally analytic solution curve $\R\ni s\mapsto(u(s),\eta(s))\in{_{\m{sol}}}H^2_{\scriptscriptstyle1/2}\tp{\R^2;\R^2}\times\bf{U}_\be\cap\tilde{H}^3_{\scriptscriptstyle1/2}\tp{\R^2}$ which (according to Theorem~\ref{thm on unified fixed point reformulation}) satisfies $\mathcal{Q}_\be\tp{u(s),\eta(s)}=\kappa(s)\mathcal{F}_\be\tp{\eta(s)}$ (where $\mathcal{F}_\be$ and $\mathcal{Q}_\be$ are defined in Propositions~\ref{prop on data map - solitary case} and~\ref{prop on operator formulation solitary case}) and (thanks to Corollary~\ref{coro on blowup refinement solitary case}) obeys the limits~\eqref{the blowup quantity in the solitary case}. We define the new curve $\R\ni s\mapsto\tp{\pmb{u}(s),\pmb{\eta}(s)}\in H^2_{\scriptscriptstyle1/2}\tp{\R^2;\R^2}\times\bf{O}_\be\cap\tilde{H}^3_{\scriptscriptstyle1/2}\tp{\R^2}$ via
    \begin{equation}\label{we've seen this before}
        \pmb{u}(s)=\tp{2\eta(s)+\tp{1+\be}^2}^{-1/2}u(s),\quad\pmb{\eta}(s)=\sqrt{2\eta(s)+\tp{1+\be}^2}-\tp{1+\be}
    \end{equation}
    and note that this is also manifestly continuous and locally analytic. Due to the third item of Proposition~\ref{prop on operator formulation solitary case} and the definition of the map $\mathcal{F}_\be$ from~\eqref{what even is helicity} it is evident that for each $s\in\R$ we have $(\pmb{u}(s),\pmb{\eta}(s))\in S_{\kkappa}$ where we recall that the latter set is defined in~\eqref{another set of solutions}.

    To complete the proof we only need to verify that the second limit in~\eqref{limit options periodic and solitary cases} holds along this curve.  We examine the inverse identities of~\eqref{we've seen this before}, namely:
    \begin{equation}\label{inverse transformation}
        u(s)=\tp{1+\be+\pmb{\eta}(s)}\pmb{u}(s),\quad\eta(s)=\tp{\tp{1+\be+\pmb{\eta}(s)}^2-\tp{1+\be}^2}/2.
    \end{equation}
    Now, if the claimed limit were false then we could find $\tcb{s_n}_{n\in\N}\subset\R$  with $s_n\to\pm\infty$ as $n\to\infty$ with
    \begin{equation}
        \sup_{n\in\N}\tsb{\tnorm{\pmb{\eta}(s_n)}_{L^\infty}+\min\tcb{\tnorm{\pmb{u}(s_n)}_{L^q},\tnorm{\pmb{\eta}(s_n)}_{L^r}}+\tp{\inf\tp{1+\be+\pmb{\eta}\tp{s_n}}}^{-1}+|\kappa(s_n)|}<\infty.
    \end{equation}
    We may then insert this information into~\eqref{inverse transformation} to contradict the limit from Corollary~\ref{coro on blowup refinement solitary case}.  Thus, our assumption that the claimed limit does not hold is false.  
\end{proof}

\appendix
\section{Tools from analysis}\label{appendix on tools from analysis}
We record here important results from analysis that are used throughout the paper.
\subsection{Analytic composition in unital Banach algebras}\label{appendix on analytic composition in unital Banach algebras}
We begin by setting some notation. Recall that if $\mathcal{X}$ is a unital Banach algebra over $\C$ (with unit $1$) then the spectrum of an element $x\in\mathcal{X}$ is the set
\begin{equation}\label{Definition of the spectrum}
    \m{spec}(x)=\tcb{\lambda\in\C\;:\;\lambda-x\text{ is not invertible}}.
\end{equation}
The set $\m{spec}\tp{x}$ is nonempty and compact, and on the complement the resolvent map
\begin{equation}\label{The resolvent map}
   \C\setminus\m{spec}\tp{x}\ni\lambda\mapsto\tp{\lambda - x}^{-1}\in\mathcal{X}
\end{equation}
is holomorphic (see, for example, Theorem 1.4, Lemma 1.5, and Proposition 1.6 in Folland~\cite{MR3444405}). Given an open set $U\subseteq\C$ we let
\begin{equation}\label{The domain in the BA}
    \mathcal{X}\tp{U} = \tcb{x\in\mathcal{X}\;:\;\m{spec}\tp{x}\subset U}.
\end{equation}
The set $\mathcal{X}\tp{U}\subseteq\mathcal{X}$ is open, as shown, for instance, in Theorem 10.20 in Rudin~\cite{MR1157815}.

The main abstract tool of this appendix is the holomorphic functional calculus, which is otherwise known as the Dunford operational calculus. We record the form of these ideas that are most useful to us in the following result. We write $C^\omega\tp{\mathfrak{A};\mathfrak{B}}$ to denote the space of complex analytic mappings defined on $\mathfrak{A}$ with values in $\mathfrak{B}$.

\begin{thm}[Holomorphic functional calculus]\label{thm on holomorphic functional calculus}
    Suppose that $\mathcal{X}$ is a unital Banach algebra and $\mathcal{Y}$ is a Banach space, both over $\C$. Suppose further that we have an embedding of unital Banach algebras $\mathcal{X}\emb\mathcal{L}\tp{\mathcal{Y}}$. Let $\es\neq U\subset\C$ be open and simply connected. Then there exists a unique linear map $\Upomega:C^\omega\tp{U;\mathcal{Y}}\to C^\omega\tp{\mathcal{X}\tp{U};\mathcal{Y}}$, written $\Phi\mapsto\Upomega_\Phi$, obeying the following properties.
    \begin{enumerate}
        \item If $P(z) = \sum_{n=0}^kY_n z^n$ for $z\in U$ and $\tcb{Y_n}_{n=0}^k\subset\mathcal{Y}$, then
        \begin{equation}\label{The agreement with polynomials property of the HFC}
            \Upomega_P(x) = \sum_{n=0}^k x^n Y_n \text{ for } x\in\mathcal{X}\tp{U}.
        \end{equation}

        \item For all $x\in\mathcal{X}\tp{U}$ there exists a constant $C_x<\infty$ and a compact set $\es\neq K_x\subset U$ such that for all $\Phi\in C^\omega\tp{U;\mathcal{Y}}$ we have the estimate
        \begin{equation}\label{The continuity property of the HFC}
            \tnorm{\Upomega_\Phi\tp{x}}_{\mathcal{Y}}\le C_x\max_{z\in K_x}\tnorm{\Phi(z)}_{\mathcal{Y}}.
        \end{equation}
    \end{enumerate}
\end{thm}
\begin{proof}
    Let us begin by proving uniqueness. Suppose that $\Upomega^i$, for $i\in\tcb{1,2}$, are linear maps $C^\omega\tp{U;\mathcal{Y}}\to C^\omega\tp{\mathcal{X}\tp{U};\mathcal{Y}}$ satisfying the first and second items above. We shall establish that for all $\Phi\in C^\omega\tp{U;\mathcal{Y}}$ and all $x\in\mathcal{X}\tp{U}$ we have the equality $\Upomega^1_\Phi\tp{x} = \Upomega^2_\Phi\tp{x}$. To this end, fix such $x$ and $\Phi$. Due to the second item's estimate~\eqref{The continuity property of the HFC} we may select a constant $C_x<\infty$ and a compact set $K_x\subset U$ which work for both of the $\Upomega^i$.

    The domain $U$ is simply connected, so by the $\mathcal{Y}$-valued variant of Runge's theorem (the proof of the standard $\C$-valued version in Chapter 12 of Remmert~\cite{MR1483074} also works for $\mathcal{Y}$-valued maps once one establishes the Cauchy integral formula for these) there exists a sequence of polynomials $\tcb{P_\ell}_{\ell=0}^\infty$ mapping $\C\to\mathcal{Y}$ that satisfy $\lim_{\ell\to\infty}\max_{z\in K_x}\tnorm{P_\ell(z) - \Phi\tp{z}}_{\mathcal{Y}}=0$. Due to the first item's equation~\eqref{The agreement with polynomials property of the HFC} we must have that $\Upomega_{P_\ell}^1\tp{x} = \Upomega_{P_\ell}^2\tp{x}$ for every $\ell\in\N$. In turn, we may use the second item and linearity of $\Upomega$ to estimate
    \begin{equation}
        \tnorm{\Upomega^1_\Phi\tp{x} - \Upomega^2_\Phi\tp{x}}_{\mathcal{Y}}\le\tnorm{\Upomega^1_\Phi\tp{x} - \Upomega_{P_\ell}^1\tp{x}}_{\mathcal{Y}} + \tnorm{\Upomega^2_{P_\ell}\tp{x} - \Upomega^2_{\Phi}\tp{x}}_{\mathcal{Y}}\le 2 C_x\max_{z\in K_x}\tnorm{P_\ell\tp{z} - \Phi\tp{z}}_{\mathcal{Y}}.
    \end{equation}
    Sending $\ell\to\infty$ completes the proof of uniqueness.

    Next, we turn our attention to the existence of the map $\Upomega$. The idea is based on the Cauchy integral formula, so we first need a family of compact contours. By the Riemann mapping theorem there exists a biholomorphic function $\Psi:B(0,1)\to U$. Given any $\mu\in\tp{0,1}$ we then let $\Gamma_\mu\subset U$ be the compact, counterclockwise oriented contour whose image agrees with $\Psi\tp{\pd B(0,\mu)}$. If $x\in\mathcal{X}\tp{U}$ then there exists $\mu_x\in(0,1)$ such that $\m{spec}\tp{x}\subset \Bar{\Psi\tp{B(0,\mu_x)}}$; we then define $\Upomega$ via the Bochner contour integral
    \begin{equation}\label{Dunford integral}
        \Upomega_\Phi\tp{x} = \f{1}{2\pi\ii} \int_{\Gamma_{\mu_x}}\tp{z - x}^{-1}\Phi\tp{z}\;\m{d}z,\quad\Phi\in C^\omega\tp{U;\mathcal{Y}}.
    \end{equation}
    By~\eqref{The resolvent map} the integrand of~\eqref{Dunford integral} is a $\mathcal{Y}$-valued holomorphic function on the domain $U\setminus\m{spec}\tp{x}$. Therefore the choice of the contour $\Gamma_{\mu_x}$ is not so important in~\eqref{Dunford integral}, as the contour integral has the same value when taken over any simple closed curve winding once around $\m{spec}\tp{x}$ within $U$ (see, e.g., Chapter VII of Dunford and Schwartz~\cite{MR117523} for arguments in a similar context that may establish this fact).

    Armed with the definition~\eqref{Dunford integral}, the estimate of the second item holds by simply taking the norm of the integral with
    \begin{equation}\label{the stuffy mcDuffy}
        K_x = \Gamma_{\mu_x} \text{ and } C_x = \tilde{C}\cdot \m{length}\tp{\Gamma_{\mu_x}}\cdot\max_{z\in\Gamma_{\mu_x}}\tnorm{\tp{z - x}^{-1}}_{\mathcal{X}},
    \end{equation}
    where $\tilde{C}$ is the embedding constant for $\mathcal{X}\emb\mathcal{L}\tp{Y}$.

    To prove that $\Upomega$ takes values in $C^\omega\tp{\mathcal{X}\tp{U};\mathcal{Y}}$ we may argue in a standard way by expanding into power series (see, e.g., Proposition 10 in Section 4 of Chapter 1 in Bourbaki~\cite{MR4301385}) to prove that for each $\Phi$, the map $x\mapsto \Upomega_\Phi\tp{x}$ is analytic.

    To complete the proof, we need establish~\eqref{The agreement with polynomials property of the HFC} for any polynomial $P:\C\to\mathcal{Y}$. This fact also follows from classical arguments; we refer the reader to, for example, the proof of Theorem 10.25 in Rudin~\cite{MR1157815}.
\end{proof}

It will occasionally be useful to examine the holomorphic functional calculus $\Upomega$ when drawing the holomorphic maps $\Phi$ from a Banach space. Given an open and simply connected set $\es\neq U\subset\C$ and a Banach space $\mathcal{Y}$ we note that as a consequence of the Cauchy integral formula the normed function space
\begin{equation}\label{normed space of analytic}
    \pmb{A}\tp{\Bar{U};\mathcal{Y}} = C^0\tp{\Bar{U};\mathcal{Y}}\cap C^\omega\tp{U;\mathcal{Y}} \text{ with norm } \tnorm{\Phi}_{\pmb{A}} =  \sup_{z\in\Bar{U}}\tnorm{\Phi(z)}_{\mathcal{Y}}
\end{equation}
is complete.

\begin{coro}[Joint analyticity]\label{corollary on meta analyticity for HFC}
    Under the hypotheses of Theorem~\ref{thm on holomorphic functional calculus} the induced map $\pmb{A}\tp{\Bar{U};\mathcal{Y}}\times\mathcal{X}\tp{U} \ni \tp{\Phi,x}\mapsto\Upomega_{\Phi}\tp{x} \in \mathcal{Y}$ is analytic.
\end{coro}
\begin{proof}
    Analyticity will be verified by showing the agreement of the stated map with a convergent power series. Let us fix $x\in\mathcal{X}\tp{U}$ and let $\Gamma_{\mu_x}$ be the contour used in the Cauchy integral formula~\eqref{Dunford integral}. There exists a $\del>0$ such that for all $y\in\Bar{B_{\mathcal{X}}\tp{0,\del_0}}$ we have that  $x + y\in \mathcal{X}(U)$, the contour $\Gamma_{\mu_x}$ winds once around $\m{spec}\tp{x + y}$,
    \begin{equation}\label{esimtaes dfot SW}
        \m{dist}\tp{\Gamma_{\mu_x},\m{spec}\tp{x + y}}\ge\f12\m{dist}\tp{\Gamma_{\mu_x},\m{spec}\tp{x}}, \text{ and }\tnorm{y}_{\mathcal{L}\tp{\mathcal{Y}}}\le\f{1}{2}\max_{z\in\Gamma_{\mu_x}}\tnorm{\tp{z - x}^{-1}}_{\mathcal{L}\tp{\mathcal{Y}}}.
    \end{equation}
    Now let $\Phi,\phi\in\pmb{A}\tp{\Bar{U};\mathcal{Y}}$.  When paired with the above estimates, the contour invariance of~\eqref{Dunford integral} yields the power series representation 
    \begin{multline}
        \Upomega_{\Phi + \phi}\tp{x + y} = \f{1}{2\pi\ii} \sum_{n=0}^\infty y^n\int_{\Gamma_{\mu_x}}\tp{z - x }^{-(n + 1)}\tp{\Phi + \phi}\tp{z}\;\m{d}z \\
        = \f{1}{2\pi\ii}\int_{\Gamma_{\mu_x}}\tp{z - x }^{-1}\Phi\tp{z}\;\m{d}z
        + \f{1}{2\pi\ii}\sum_{n=1}^\infty \int_{\Gamma_{\mu_x}}\tp{z - x}^{-n}\tp{y^n\tp{z - x}^{-1}\Phi(z) + y^{n-1}\phi(z)}\;\m{d}z,
    \end{multline}
    where in the last equality we have re-indexed the sum to make evident the multilinear maps acting on $(\phi,y)$. In fact, by taking the norm in $\mathcal{Y}$ and using estimate~\eqref{esimtaes dfot SW} and the definition of the norm~\eqref{normed space of analytic}, we readily deduce that above series is uniformly absolutely convergent for all $y\in\Bar{B_{\mathcal{X}}\tp{0,\del_0}}$ and for all $\phi\in\Bar{B_{\pmb{A}}\tp{0,1}}$. 
\end{proof}

We now aim to specialize to the case in which the Banach algebras $\mathcal{X}$ consists of continuous functions.  For this we will need to impose the auxiliary conditions of the following definition.

\begin{defn}[Admissibility]\label{defn of admissable Banach algebras}
    We say that a unital Banach algebra $\mathcal{X}$ over $\C$ is admissible if
    \begin{enumerate}
        \item $\mathcal{X}\emb\tp{L^\infty\cap C^0}\tp{\R^d;\C}$
        \item For all $x\in\mathcal{X}$ it holds that $\m{spec}(x)=\Bar{\m{img}\tp{x}}$ (where the former set is the spectrum as in~\eqref{Definition of the spectrum} and the set $\m{img}\tp{x}$ denotes the forward image of $x$), and for $z\in\C\setminus\Bar{\m{img}\tp{x}}$ the resolvent $(z-x)^{-1}\in\mathcal{X}$ agrees with the function $\R^d\ni\xi\mapsto\tp{z-x\tp{\xi}}^{-1}\in\C$.
        \item There exists a locally bounded function $\bf{C}:[0,\infty)^2\to\R^+$ increasing in both arguments such that for all $x\in\mathcal{X}$ and all $z\in\C\setminus\m{spec}\tp{x}$ we have the resolvent estimate
        \begin{equation}\label{resolvent estimates}
            \tnorm{\tp{z-x}^{-1}}_{\mathcal{X}}\le \bf{C}\tp{\m{dist}_{\C}\tp{\m{img}(x),z}^{-1},\tnorm{x}_{\mathcal{X}}}.
        \end{equation}
    \end{enumerate}
\end{defn}

To get a sense of the utility of this definition, suppose that $\mathcal{X}$ is admissible and let $\mathcal{Y} = \mathcal{X}$, which allows us to apply Theorem \ref{thm on holomorphic functional calculus} with the natural embedding $\mathcal{X} \hookrightarrow \mathcal{L}(\mathcal{X})$.  Then for a polynomial map $P(z) = \sum_{n=0}^k Y_n z^n$ as in the first item of the theorem, we have that $\Upomega_P : \mathcal{X}(U) \to \mathcal{X}$ is analytic, and is in fact also a polynomial map: $\Upomega_P(x) = \sum_{n=0}^k x^n Y_n   \in \mathcal{X}$.  In particular, $\tsb{\Upomega_P(x)}(\xi) = \sum_{n=0}^k (x(\xi))^n Y_n(\xi)$ for $\xi \in \R^d$.  On the other hand, the composition $P\circ x : \R^d \to \mathcal{X}$ satisfies $\tsb{P\circ x}(\xi) = \sum_{n=0}^k (x(\xi))^n Y_n$.  We thus arrive at the identity 
\begin{equation}
    \tsb{\Upomega_P(x)}(\xi) = P(x(\xi))(\xi),
\end{equation}
which shows that $x\mapsto \Upomega_P(x)$ is a Nemytskii operator in this special case.  Our goal now is to show that this realization of the abstract holomorphic functional calculus of Theorem~\ref{thm on holomorphic functional calculus} as Nemytskii operators carries over to a more general context.
 
In what follows we say that $\tcb{S_\epsilon}_{0<\epsilon<1}$  are standard regularizing operators if there exists a continuous and compactly supported function $\Qoppa:\R^d\to\R$ with $\int_{\R^d}\Qoppa=1$ such that for all $\epsilon\in\tp{0,1}$ and $f\in L^1_\loc\tp{\R^d;V}$, for a finite dimensional complex vector space $V$, we have the formula
\begin{equation}\label{standard smoothers}
    \tp{S_\epsilon f}\tp{\xi} = \int_{\R^d}\epsilon^{-d}\Qoppa\tp{\epsilon^{-1}\tp{\xi-\eta}}f(\xi)\;\m{d}\xi.
\end{equation}
We note that there exists positive constants $C$ and $c$ such that we have pointwise bound
\begin{equation}\label{maxmial nonsense}
    \sup_{0<\epsilon<1}\tabs{S_\epsilon f\tp{\xi}}\le C\mathcal{M}_{\le c}\tsb{f}\tp{\xi}
\end{equation}
for all $\xi\in\R^d$ where $\mathcal{M}_{\le c}$ denotes the maximal function taken over balls of radius at most $c$.

\begin{prop}[Complex analytic Nemytskii operators]\label{prop on compound analytic superposition complex case}
    Let $V$ be a finite dimensional complex vector space, let $\mathcal{Y}$ be a complex Banach space for which we have the continuous embedding $\mathcal{Y} \hookrightarrow L^1_{\loc}\tp{\R^d;V}$, and let $\mathcal{X}$ be an admissible unital Banach algebra in the sense of Definition~\ref{defn of admissable Banach algebras}.  Assume that the pointwise product mapping $\mathcal{Y}\times\mathcal{X}\to\mathcal{Y}$ is continuous. Let $U\subset\C$ be nonempty, open, and simply connected with $\Phi:U\to\mathcal{Y}$ holomorphic. For $\ep\in\tp{0,1}$ we denote $U_\ep = \tcb{z\in U\;:\;\m{dist}\tp{z,\pd U}>\ep}$. The following hold.
    \begin{enumerate}
        \item For all $x\in\mathcal{X}\tp{U}$ there exists $\Upomega_\Phi(x)\in \mathcal{Y}$ such that for any choice of standard regularizing operators $\tcb{S_\epsilon}_{0<\epsilon<1}$ and any sequence $\tcb{\epsilon_n}_{n=0}^\infty\subset(0,1)$ with $\epsilon_n\to0$ as $n\to\infty$ we have that for almost every $\xi\in\R^d$
        \begin{equation}\label{the limit computation of the weirdo composition}
            \lim_{n\to\infty}\tsb{S_{\epsilon_n}\Phi}\tp{x\tp{\xi}}\tp{\xi} = \tsb{\Upomega_\Phi\tp{x}}\tp{\xi}.
        \end{equation}
        \item The induced Nemytskii map $\Upomega_\Phi:\mathcal{X}\tp{U}\to\mathcal{Y}$ is complex analytic.  
        \item For any $\ep\in\tp{0,1}$ the restricted mapping $\Upomega_\Phi:\mathcal{X}\tp{U_\ep}\to\mathcal{Y}$ maps bounded sets to bounded sets.
        \item For each $x\in\mathcal{X}\tp{U}$ and $\xi_0\in\R^d$ there exits $\del>0$ such that in the set $B(\xi_0,\del) \subset \R^d$ we have the a.e. equality
        \begin{equation}
            \Upomega_\Phi\tp{x} = \sum_{n\in\N}\f{1}{n!}\tp{x - x(\xi_0)}^n\Phi^{\tp{n}}\tp{x\tp{\xi_0}},
        \end{equation}
        where $\Phi^{(n)}(z)\in \mathcal{Y}$ denotes the $n^{\m{th}}$ derivative of $\Phi$ at $z \in U$, and the series on the right is absolutely convergent in $L^1\tp{B(\xi_0,\del);V}$ and almost everywhere pointwise absolutely convergent.
        \item Suppose we have the continuous embedding $\mathcal{Y} \emb C^0_{\loc}(\R^d;V)$, where the latter space denotes $C^0(\R^d;V)$ endowed with the Fr\'{e}chet structure of local uniform convergence.  Then 
        \begin{equation}
            \tsb{\Upomega_\Phi\tp{x}}\tp{\xi} = \Phi\tp{x\tp{\xi}}\tp{\xi} \text{ for all } \xi \in \R^d\text{ and }x\in\mathcal{X}\tp{U}.
        \end{equation}
    \end{enumerate}
\end{prop}
\begin{proof}
The boundedness of the product map $\mathcal{Y}\times\mathcal{X}\to\mathcal{Y}$ yields the embedding $\mathcal{X}\emb\mathcal{L}\tp{\mathcal{Y}}$. Theorem~\ref{thm on holomorphic functional calculus} then  grants us a unique linear Dunford calculus map $\Upomega:C^\omega\tp{U;\mathcal{Y}}\to C^\omega\tp{\mathcal{X}\tp{U};\mathcal{Y}}$ which satisfies the first and second items of that aforementioned result; from this map we define $\Upomega_\Phi\in C^\omega\tp{\mathcal{X}\tp{U};\mathcal{Y}}$, which automatically satisfies the analyticity assertion of the second item.  

We now prove the third item.   The contour invariance of~\eqref{Dunford integral} and equation~\eqref{the stuffy mcDuffy}, when combined with the resolvent estimates~\eqref{resolvent estimates} from Definition~\ref{defn of admissable Banach algebras}, yield the bound for 
\begin{equation}\label{the bounds for the bounds one day I'll show you a picture of what it is}
    \tnorm{\Upomega_\Phi\tp{x}}_{\mathcal{Y}}\le\tilde{C}\cdot\m{length}\tp{\Gamma_{\mu(\ep)}}\cdot\bf{C}\tp{2/\ep,\tnorm{x}_{\mathcal{X}}} \text{ for } \ep\in(0,1) \text{ and } x\in\mathcal{X}\tp{U_\ep},
\end{equation}
where $\Gamma_{\mu\tp{\ep}}$ is the contour constructed in the proof of Theorem~\ref{thm on holomorphic functional calculus} corresponding to $\mu = \mu\tp{\ep}$ with $\mu\tp{\ep}\in(0,1)$ chosen sufficiently close to $1$ so that $\Gamma_{\mu\tp{\ep}}\cap U_{\ep/2}=\es$. Armed with~\eqref{the bounds for the bounds one day I'll show you a picture of what it is}, the third item now follows.

We next prove the first item. Fix a sequence $\tcb{S_\ep}_{0<\ep<1}$ of standard regularizing operators and also fix a closed cube $Q\subset\R^d$. Due to the embedding $\mathcal{Y}\emb L^1_\loc\tp{\R^d; V}$ we find that the maps $S_\ep\Phi\vert_Q:U\to C^0\tp{Q; V}$ are analytic, where for $\xi\in Q$ and $z\in U$ we have denoted
\begin{equation}
    \tsb{S_\ep\Phi\vert_Q\tp{z}}\tp{\xi} = \tsb{S_\ep\tp{\Phi\tp{z}}}\tp{\xi}.
\end{equation}
As the evaluation maps are bounded linear operators $C^0\tp{Q; V}\to V$ we find that by composition for each $\xi\in Q$ the map
\begin{equation}
     U\ni z\mapsto\tsb{S_\ep\Phi\vert_Q\tp{z}}\tp{\xi}\in V
\end{equation}
is also analytic. Therefore, the Cauchy integral formula applies and we learn that for all $\mu\in\tp{0,1}$ and all $w\in U$ that are surrounded by the contour $\Gamma_\mu$, the identity
\begin{equation}\label{cucumberbycycle}
    \tsb{S_\ep\Phi\vert_{Q}\tp{w}}\tp{\xi} = \f{1}{2\pi\ii}\int_{\Gamma_\mu}\tsb{S_\ep\Phi\tp{z}}\tp{\xi}\tp{z - w}^{-1}\;\m{d}z,\quad \xi\in Q
\end{equation}
holds. We let $x\in\mathcal{X}\tp{U}$, suppose that $\mu = \mu_x$ (as defined in the proof of Theorem~\ref{thm on holomorphic functional calculus}), take $w = x(\xi)$, and reinterpret the result~\eqref{cucumberbycycle} as a $C^0\tp{Q; V}$-valued Bochner integral; this procedure yields the identity
\begin{equation}\label{the local integral identity}
    \m{diag}\tsb{\tp{S_\ep\Phi}\circ x\vert_Q} = \f{1}{2\pi\ii}\int_{\Gamma_\mu}S_\ep\Phi\tp{z}\tp{z - x}^{-1}\;\m{d}z\in C^0\tp{Q; V}
\end{equation}
where $\m{diag}\tsb{\tp{S_\ep\Phi}\circ x\vert_{Q}}\tp{\xi} = \tp{S_\ep\Phi}\tp{x\tp{\xi}}\tp{\xi}$ for $\xi\in Q$.

We now consider the limit of~\eqref{the local integral identity} as $\ep\to 0$. Due to the embedding $\mathcal{Y}\emb L^1_\loc\tp{\R^d; V}$ we find that the right hand side integrand converges pointwise to $z\mapsto \Phi\tp{z}\tp{z - x}^{-1}$ in the space $L^1\tp{Q; V}$ while also remaining bounded in norm by a constant function. The dominated convergence theorem now applies and we thus acquire that both sides of~\eqref{the local integral identity} have a limit in the space $L^1\tp{Q;V}$, namely:
\begin{equation}\label{the limit that is here}
    \lim_{\ep\to 0}\m{diag}\tsb{\tp{S_\ep\Phi}\circ x\vert_Q} = \f{1}{2\pi\ii}\int_{\Gamma_\mu}\Phi\tp{z}\tp{z-x}^{-1}\;\m{d}z = \Upomega_{\Phi}\tp{x}\vert_Q\in L^1\tp{Q; V}.
\end{equation}
The final equality above follows from~\eqref{Dunford integral}. This establishes that the limit in~\eqref{the limit computation of the weirdo composition} holds in the space $L^1_\loc\tp{\R^d; V}$. We now wish to establish pointwise almost everywhere convergence as well.

Let us fix again a closed cube $Q$ and sequence $\tcb{\ep_n}_{n\in\N}\subset{0,1}$ with $\ep_n\to 0$ as $n\to\infty$. We consider the measurable sets indexed by $m\in\N$
\begin{equation}\label{this is a place of business}
    F_m = \tcb{\xi\in Q\;:\;\limsup_{n,k\to\infty}\tabs{S_{\ep_n}\Phi\tp{x\tp{\xi}}\tp{\xi} - S_{\ep_k}\Phi\tp{x\tp{\xi}}\tp{\xi}}>2^{-m}}.
\end{equation}
Thanks to identity~\eqref{the local integral identity} we are free to estimate for any $\xi\in Q$ and $n,k\in\N$:
\begin{equation}\label{C!1}
    \tabs{S_{\ep_n}\Phi\tp{x\tp{\xi}}\tp{\xi} - S_{\ep_k}\Phi\tp{x\tp{\xi}}\tp{\xi}}\le\f{\m{length}\tp{\Gamma_\mu}}{2\pi\m{dist}\tp{\Gamma_\mu,\m{img}\tp{x}}}\int_0^1\tabs{\tp{S_{\ep_n} - S_{\ep_k}}\Phi\tp{\gam\tp{\tau}}\tp{\xi}}\;\m{d}\tau
\end{equation}
where $\gam$ is a $\m{length}\tp{\Gamma_\mu}$-speed parametrization of $\Gamma_\mu$. Now for any $N\in\N$ we make the further estimates:
\begin{multline}\label{C!2}
    \int_0^1\tabs{\tp{S_{\ep_n} - S_{\ep_k}}\Phi\tp{\gam\tp{\tau}}\tp{\xi}}\;\m{d}\tau\le\int_{0}^1\tabs{\tp{S_{\ep_n} - S_{\ep_k}}S_{\ep_N}\Phi\tp{\gam\tp{\tau}}\tp{\xi}}\;\m{d}\tau\\
    +\int_0^1\tabs{S_{\ep_n}\tp{1 - S_{\ep_N}}\Phi\tp{\gam\tp{\tau}}\tp{\xi}}\;\m{d}\tau+\int_0^1\tabs{S_{\ep_k}\tp{1 - S_{\ep_N}}\Phi\tp{\gam\tp{\tau}}\tp{\xi}}\;\m{d}\tau.
\end{multline}
Now we combine~\eqref{C!1} and~\eqref{C!2}, take the limit superior as $n$ and $k$ tend to infinity, and use continuity of $S_{\ep_N}\Phi$ to deduce
\begin{multline}
    \limsup_{n,k\to\infty}\tabs{S_{\ep_n}\Phi\tp{x\tp{\xi}}\tp{\xi} - S_{\ep_k}\Phi\tp{x\tp{\xi}}\tp{\xi}}\le \f{\m{length}\tp{\Gamma_\mu}}{\pi\m{dist}\tp{\Gamma_\mu,\m{img}\tp{x}}}\sup_{n\in\N}\int_{0}^1\tabs{S_{\ep_n}\tp{1 - S_{\ep_N}}\Phi\tp{\gam\tp{\tau}}\tp{\xi}}\;\m{d}\tau\\
    \le C\mathcal{M}_{\le c}\bsb{\int_0^1\tabs{\tp{1-S_{\ep_N}}\Phi\tp{\gam\tp{\tau}}\tp{\cdot}}\;\m{d}\tau}\tp{\xi},
\end{multline}
where $\mathcal{M}_{\le c}$ is the maximal function over balls of radius at most $c$ - see~\eqref{maxmial nonsense} - and $C$ is a constant depending on $V$, $Q$, $x$, $\Gamma_\mu$, and $\tcb{S_\ep}_{0<\ep<1}$. From the above we may return to $F_m$ from~\eqref{this is a place of business} and invoke the Hardy–Littlewood maximal inequality and gain the estimate
\begin{equation}
    |F_m|\le C\cdot 2^m\int_{Q + B(0,c)}\int_0^1\tabs{\tp{1 - S_{\ep_N}}\Phi\tp{\gam\tp{\tau}}\tp{\xi}}\;\m{d}\tau\;\m{d}\xi.
\end{equation}
This constant above is independent of $N$ and so we may send $N\to\infty$ and deduce that $|F_m|=0$. As this holds for every $m\in\N$ we then deduce that for almost every $\xi\in Q$ the sequence $\tcb{S_{\ep_n}\Phi\tp{x\tp{\xi}}\tp{\xi}}_{n\in\N}\subset V$ is Cauchy and hence convergent. The $L^1\tp{Q;V}$-convergence in equation~\eqref{the limit that is here} forces the a.e. limit of the sequence $\tcb{\xi\mapsto S_{\ep_n}\tp{x\tp{\xi}}\tp{\xi}}_{n\in\N}$ of $L^1\tp{Q;V}$-functions to agree with $\Upomega_\Phi\tp{x}$. This completes the proof of the first item.

We now prove the fourth item. Let us fix an $x\in\mathcal{O}$ and $\xi_0\in\R^d$. The function $\Phi:U\to\mathcal{Y}$ is analytic near $x\tp{\xi_0}\in U$ while $\mathcal{Y}\emb L^1_{\loc}\tp{\R^d; V}$ and so there exists $\del_1>0$ such that $\Bar{B(x\tp{\xi_0},2\del_1)}\Subset U$ and for all $z\in\Bar{B(x(\xi_0),2\del_1)}$ we have the identity
\begin{equation}\label{the local series expansion}
    \Phi\tp{z} = \sum_{n=0}^\infty\tp{z - x\tp{\xi_0}}^n\f{\Phi^{(n)}\tp{x\tp{\xi_0}}}{n!}\in L^1\tp{B(x\tp{\xi_0},2\del_1); V}
\end{equation}
where the series is uniformly absolutely convergent in the space $L^1\tp{B(x\tp{\xi_0},2\del_1); V}$. Thanks to the embedding $\mathcal{X}\emb\tp{C^0\cap L^\infty}\tp{\R^d;\C}$ we have that $x$ is continuous at $\xi_0$ and so there exists $0<\del<\del_1$ with the property that if $\xi\in B(\xi_0,\del)$, then $|x\tp{\xi} - x\tp{\xi_0}|<\del_1/2$.

Now we consider the formula~\eqref{Dunford integral} for $\Upomega_\Phi\tp{x}$ and let $E\subseteq B(\xi_0,\del)$ be any measurable subset; upon integrating over $E$, invoking Fubini's theorem, and using invariance of choice of contour we find
\begin{equation}\label{STEP_1}
    \int_{E}\Upomega_\Phi\tp{x} = \f{1}{2\pi\ii}\int_{\Gamma_{\mu_x}}\int_{E}\Phi\tp{z}\tp{\xi}\tp{z - x\tp{\xi}}^{-1}\;\m{d}\xi\;\m{d}z = \f{1}{2\pi\ii}\int_{\pd B(x(\xi_0),\del_1)}\int_{E}\Phi\tp{z}\tp{\xi}\tp{z - x\tp{\xi}}^{-1}\;\m{d}\xi\;\m{d}z.
\end{equation}
Implicitly we have used that the function
\begin{equation}
    U\setminus\Bar{B(x(\xi_0),\del_1/2)}\ni z\mapsto\int_{E}\Phi\tp{z}\tp{\xi}\tp{z - x\tp{\xi}}^{-1}\;\m{d}\xi\in V
\end{equation}
is holomorphic thanks to $E\subseteq B(\xi_0,\del)$ and so the change of contour in~\eqref{STEP_1} is valid. Now in the final integral of~\eqref{STEP_1} due to $z\in\pd B(x\tp{\xi_0},\del_1)\Subset B(x\tp{\xi_0},2\del_1)$ we can replace $\Phi(z)$ by its local expansion~\eqref{the local series expansion}. Once this is done, we get
\begin{equation}\label{STEP_2}
    \int_{E}\Upomega_\Phi\tp{x} = \int_{E}\sum_{n=0}^\infty\f{1}{n!}\Phi^{(n)}\tp{x\tp{\xi_0}}\tp{\xi}\f{1}{2\pi\ii}\int_{\pd B(x(\xi_0),\del_1)}\tp{z - x\tp{\xi_0}}^n\tp{z - x\tp{\xi}}^{-1}\;\m{d}z\;\m{d}\xi.
\end{equation}
The Cauchy integral formula allows us to equate for every $\xi\in E$:
\begin{equation}\label{STEP_3}
    \f{1}{2\pi\ii}\int_{\pd B(x(\xi_0),\del_1)}\tp{z - x\tp{\xi_0}}^n\tp{z - x\tp{\xi}}^{-1}\;\m{d}z = \tp{x(\xi) - x\tp{\xi_0}}^n
\end{equation}
and hence the combination of~\eqref{STEP_2} and~\eqref{STEP_3} gives
\begin{equation}
    \int_{E}\Upomega_\Phi\tp{x} = \int_{E}\sum_{n\in\N}\f{1}{n!}\tp{x - x\tp{\xi_0}}^n\Phi^{(n)}\tp{x(\xi_0)}.
\end{equation}
Since this holds for every such $E\subseteq B(\xi_0,\del)$, the claims of the fourth item follow.

At last, we justify the fifth item. Due to the embedding $\mathcal{Y}\emb C^0_\loc\tp{\R^d; V}$ for any $\xi\in\R^d$ the linear map given by evaluation at $\xi$ is continuous. In turn, we may compose formula~\eqref{Dunford integral} (for any $x\in\mathcal{X}\tp{U}$) with any such bounded linear evaluation map and then invoke the Cauchy integral formula to equate
\begin{equation}
    \tsb{\Upomega_\Phi\tp{x}}\tp{\xi} = \f{1}{2\pi\ii}\int_{\Gamma_{\mu_x}}\Phi\tp{z}\tp{\xi}\tp{z - x\tp{\xi}}^{-1}\;\m{d}z = \Phi\tp{x\tp{\xi}}\tp{\xi}.
\end{equation}
This completes the proof.

\end{proof}

We next specialize the previous result to the case of real analytic functions. If $\mathcal{X}$ is an admissible unital Banach algebra in the sense of Definition~\ref{defn of admissable Banach algebras} we denote the $\R$-subspace of $\R$-valued functions according to $\mathcal{X}_{\R} = \tcb{x\in\mathcal{X}\;:\;\m{img}\tp{x}\subset\R}$; further, if $I\subseteq\R$ is an open interval then we define the $\mathcal{X}_{\R}-$open set $\mathcal{X}_\R\tp{I} = \tcb{x\in\mathcal{X}_{\R}\;:\;\Bar{\m{img}\tp{x}}\subset I}$.

\begin{coro}[Real analytic Nemytskii operators]\label{coro on compound analytic superposition real case}
    Let $V_\R$ denote a real finite dimensional vector space with complexification denoted by $V$. Let $\mathcal{Y}$ be a complex Banach space for which we have the continuous embedding $\mathcal{Y}\emb L^1_\loc\tp{\R^d; V}$ and let $\mathcal{Y}_\R = \mathcal{Y}\cap L^1_\loc\tp{\R^d;V_{\R}}$ denote the real subspace of real-valued functions. Also let $\mathcal{X}$ be an admissible unital Banach algebra in the sense of Definition~\ref{defn of admissable Banach algebras} in which the pointwise product $\mathcal{Y}\times\mathcal{X}\to\mathcal{Y}$ is continuous. Let $a,b\in[-\infty,\infty]$ with $a<b$,  $I = (a,b)$, $I_\ep = (a+\ep,b - \ep)$ for $\ep\in(0,1)$, and $\phi:I\to\mathcal{Y}_{\R}$ be real analytic. The following hold.
    \begin{enumerate}
        \item For all $x\in\mathcal{X}_{\R}\tp{I}$, there exists $\Upomega_\phi\tp{x}\in\mathcal{Y}_{\R}$ such that for any choice of standard regularizing operators $\tcb{S_\epsilon}_{0<\epsilon<1}$ and any sequence $\tcb{\epsilon_n}_{n=0}^\infty\subset\tp{0,1}$ satisfying $\epsilon_n\to0$ as $n\to\infty$ we have that for almost every $\xi\in\R^d$ 
        \begin{equation}
            \lim_{n\to\infty}\tsb{S_{\epsilon_n}\phi}\tp{x\tp{\xi}}\tp{\xi} = \tsb{\Upomega_\phi\tp{x}}\tp{\xi}.
        \end{equation}
        \item The induced Nemytskii map $\Upomega_\phi:\mathcal{X}_{\R}\tp{I}\to\mathcal{Y}_{\R}$ is real analytic.
        \item For any $\ep\in\tp{0,1}$ the restricted mapping $\Upomega_\phi:\mathcal{X}_{\R}\tp{I_\ep}\to\mathcal{Y}_\R$ maps bounded sets to bounded sets.
        \item For each $x\in\mathcal{X}_{\R}\tp{I}$ and $\xi_0\in\R^d$ there exists $\del>0$ such that in the set $B(\xi_0,\del)$ we have the a.e. equality
        \begin{equation}
            \Upomega_\phi\tp{x} = \sum_{n\in\N}\f{1}{n!}\tp{x - x\tp{\xi_0}}^n\phi^{(n)}\tp{x\tp{\xi_0}}
        \end{equation}
        where the series on the right is absolutely convergent in $L^1\tp{B\tp{\xi_0,\del},V_{\R}}$ and almost everywhere pointwise absolutely convergent; $\phi^{(n)}\tp{t}\in\mathcal{Y}_\R$ denotes the $n^{\m{th}}$ derivative of $\phi$ at $t\in\tp{a,b}$.
        \item Suppose we have the continuous embedding $\mathcal{Y}\emb C^0_\loc\tp{\R^d;\C}$ as in the fifth item of Proposition~\ref{prop on compound analytic superposition complex case}. Then
        \begin{equation}
            \tsb{\Upomega_{\phi}\tp{x}}\tp{\xi} = \phi\tp{x\tp{\xi}}\tp{\xi}\text{ for all }\xi\in\R^d\text{ and }x\in\mathcal{X}_{\R}\tp{I}.
        \end{equation}
    \end{enumerate}
\end{coro}
\begin{proof}
    The stated claims follow from the complex version, Proposition~\ref{prop on compound analytic superposition complex case} with minor additional arguments.

    Indeed, we may find a holomorphic extension of $\phi$, $\Phi:U\to\mathcal{Y}$ with $(a,b)\subset U\subset\C$ open and simply connected and $\Phi = \phi$ on $(a,b)$. The function $\phi$ is $\mathcal{Y}_{\R}$-valued real analytic and thus at each point $q\in(a,b)$, there exists $\rho_q\in\R^+$ and a sequence $\tcb{c_{q,n}}_{n\in\N}\subset\mathcal{Y}_{\R}$ such that for all $t \in(q-\rho_q,q+\rho_q)\cap(a,b)$ we have
    \begin{equation}
        \phi(t)=\sum_{n=0}^\infty c_{q,n}\tp{t-q}^n
    \end{equation}
    where the series is, in fact, absolutely and uniformly convergent for (complex-valued) $t \in B_{\C}(q,\rho_q)$. In this way we obtain a holomorphic extension of $\phi$
    \begin{equation}
        \Phi:U\to\C,\quad U=\bigcup_{q\in(a,b)}B_{\C}\tp{q,\rho_q}\quad\text{with}\quad\Phi=\phi\text{ on }(a,b)\subset U.
    \end{equation}
    The domain $U$ is manifestly simply connected as it retracts to an interval.
    
    Proposition~\ref{prop on compound analytic superposition complex case} applies for $\Phi$ and we obtain a complex analytic operator $\Upomega_\Phi$ on the set $\mathcal{X}\tp{U}$ with values in $\mathcal{Y}$. $\Upomega_\phi$ is then defined via restriction to $\mathcal{X}_{\R}(I)$. The limiting formula~\eqref{the limit computation of the weirdo composition} shows actually that if $x\in\mathcal{X}_{\R}\tp{I}$, then $\Upomega_\phi(x)\in\mathcal{Y}_\R$. Hence $\Upomega_{\phi}$ is real analytic; the remaining claims are directly inherited from the corresponding properties of $\Upomega_\Phi$ and simple $\R$-valued considerations.
\end{proof}

The previous result may be further specialized to study the outer composition of $\R$-valued real analytic functions with $\R$-valued members of admissible Banach algebras.

\begin{coro}[Real analytic Nemytskii operators - simple case]\label{coro on analytic superposition real case}
    Let $\mathcal{X}$ be an admissible unital Banach algebra in the sense of Definition~\ref{defn of admissable Banach algebras}, $a,b\in[-\infty,\infty]$ with $a<b$, $I = (a,b)$, $I_\ep = (a+\ep,b-\ep)$ for $\ep\in(0,1)$, and $\phi:I\to\R$ be real analytic. The following hold.
    \begin{enumerate}
        \item For all $x\in\mathcal{X}_{\R}\tp{I}$ the pointwise composition $\phi(x) = \phi\circ x:\R^d\to\R$ belongs to $\mathcal{X}_{\R}$.
        \item The induced Nemytskii map $\phi:\mathcal{X}_{\R}\tp{I}\to\mathcal{X}_{\R}$ given by $x\mapsto\phi\tp{x}$ is real analytic.
        \item For any $\ep\in(0,1)$ the restricted mapping $\phi:\mathcal{X}_{\R}\tp{I_\ep}\to\mathcal{X}_{\R}$ maps bounded sets to bounded sets to bounded sets.
    \end{enumerate}
\end{coro}
\begin{proof}
    There is the natural embedding of $\R\emb\mathcal{X}_{\R}$ given by $t\mapsto t\cdot 1$ where $1\in\mathcal{X}_{\R}$ is the unit (the constant function identically unity). In this way we may view the given function $\phi$ as a real analytic mapping $I\to\mathcal{X}_{\R}$. This allows us to invoke Corollary~\ref{coro on compound analytic superposition real case} with the target space $\mathcal{Y} = \mathcal{X}\emb C^0_{\loc}\tp{\R^d;\C}$ and obtain a real analytic Nemytskii map $\Upomega_\phi:\mathcal{X}_{\R}\tp{I}\to\mathcal{X}_{\R}$ which maps bounded subsets of $\mathcal{X}_{\R}\tp{I_\ep}$ into bounded subsets of $\mathcal{X}_{\R}$ for all $\ep\in(0,1)$.

    To complete the proof we need only establish that $\Upomega_\phi\tp{x} = \phi\circ x$ for all $x\in\mathcal{X}_{\R}\tp{I}$. But this is an immediate consequence of the fifth item of Corollary~\ref{coro on compound analytic superposition real case}.
\end{proof}

\subsection{Basic nonlinear analysis in Sobolev spaces}\label{basic nonlinear analysis in weighted sobolev spaces}

We begin with the Gagliardo-Nirenberg interpolation estimates.
\begin{thm}[Gagliardo-Nirenberg]\label{thm on gagliardo nirenberg}

    Let $\Gamma=\R^d$ or $\Gamma=\T^d$. Let $1\le p,q,t\le\infty$ and $s,r\in\N^+$ with $1\le r<s$ and suppose
    \begin{equation}
        \f{r}{s}\f{1}{p}+\bp{1-\f{r}{s}}\f{1}{q}=\f{1}{t}.
    \end{equation}
    For all $f\in\tp{L^q\cap\dot{W}^{s,p}}\tp{\Gamma}$ we have that $f\in\dot{W}^{r,t}\tp{\Gamma}$ with the estimate
       \begin{equation}\label{strong form in Rd GNI}
        \tnorm{\grad^rf}_{L^t}\lesssim\tnorm{f}_{L^q}^{1-r/s}\tnorm{\grad^s f}^{r/s}_{L^p}
    \end{equation}
    where the implicit constant depends on $d$, $p$, $q$, $t$, $s$, and $r$.
\end{thm}
\begin{proof}
For the case $\Gamma=\R^d$ see, for instance, Theorem 12.85 in Leoni~\cite{MR3726909}. The case $\Gamma=\T^d$ can be derived from the $\R^d$ case by considering an `extension' mapping of the form
\begin{equation}
    \tp{\mathfrak{E}f}(x)=\varphi(x)f(x\;\m{mod}\;\Z^d),\quad x\in\R^d
\end{equation}
with $\varphi\in C^\infty_{\m{c}}(\R^d)$ satisfying $\varphi=1$ on $[0,1]^d\subset\R^d$. This mapping has the property that for any $\mathfrak{k}\in\N$, $\mathfrak{p}\in[1,\infty]$, and $g\in W^{\mathfrak{k},\mathfrak{p}}\tp{\T^d}$ it not only holds that $\mathfrak{E}g\in W^{\mathfrak{k},\mathfrak{p}}\tp{\R^d}$, but also
\begin{equation}
    \tnorm{g}_{W^{\mathfrak{k},\mathfrak{p}}\tp{\T^d}}\lesssim\tnorm{\mathfrak{E}g}_{W^{\mathfrak{k},\mathfrak{p}}\tp{\R^d}}\lesssim\tnorm{g}_{W^{\mathfrak{k},\mathfrak{p}}\tp{\T^d}}
\end{equation}
for implicit constants depending only on $d$, $\mathfrak{k}$, $\mathfrak{p}$, and the choice of $\varphi$. So applying the $\R^d$-Gagliardo-Nirenberg interpolation inequality to $\mathfrak{E}f$ gives an initial `inhomogeneous' $\T^d$ estimate of the form
\begin{equation}\label{inhomogeneous Td}
    \tnorm{f}_{W^{r,t}}\lesssim\tnorm{f}_{L^q}^{1 - r/s}\tnorm{f}_{W^{s,p}}^{r/s}.
\end{equation}
The homogeneous form~\eqref{strong form in Rd GNI} then follows by applying~\eqref{inhomogeneous Td} with $f$ replaced by $f - \int_{\T^d}f$ and using mean-zero Poincar\'e inequality.
\end{proof}

Armed with Theorem~\ref{thm on gagliardo nirenberg}, we can now establish some important product estimates.
\begin{thm}[High-low product estimates]\label{thm on high low product estimates}
    Let $\Gamma=\R^d$ or $\Gamma=\T^d$. Suppose that $s\in\N^+$, $1\le p_0,p_1,q_0,q_1,t\le\infty$ satisfy 
    \begin{equation}\label{exponent relationships}
        \f{1}{t}=\f{1}{p_0}+\f{1}{q_1}=\f{1}{p_1}+\f{1}{q_0}.
    \end{equation}
    Let $f_0\in\tp{L^{q_0}\cap W^{s,p_0}}\tp{\Gamma}$ and $f_1\in\tp{L^{q_1}\cap W^{s,p_1}}\tp{\Gamma}$. The pointwise product $f_0f_1$ belongs to $W^{s,t}\tp{\Gamma}$ and the following estimate holds
    \begin{equation}
        \tnorm{f_0f_1}_{W^{s,t}}\lesssim\tnorm{f_0}_{L^{q_0}}\tnorm{f_1}_{W^{s,p_1}}+\tnorm{f_0}_{W^{s,p_0}}\tnorm{f_1}_{L^{q_1}}.
    \end{equation}
    with implicit constants depending only on $d$, $s$, $t$, $p_0$, $p_1$, $q_0$, and $q_1$.
\end{thm}
\begin{proof}
    We note first that thanks to an equivalent norm and the Leibniz rule we have the bound
    \begin{equation}
        \tnorm{f_0f_1}_{W^{s,t}}\lesssim\tnorm{f_0f_1}_{L^t}+\sum_{|\al|=s}\sum_{\be\le\al}\tnorm{\pd^{\al-\be}f_0\pd^\be f_1}_{L^t}.
    \end{equation}
    By the H\"older inequality we have $\tnorm{f_0f_1}_{L^t}\le\tnorm{f_0}_{L^{p_0}}\tnorm{f_1}_{L^{q_1}}$ and we have
    \begin{equation}
        \tnorm{\pd^{\al-\be}f_0\pd^\be f_1}_{L^t}\le\tnorm{f_0}_{W^{|\al|-|\be|,t_0}}\tnorm{f_1}_{W^{|\be|,t_1}}
    \end{equation}
    with $t_0,t_1\in[1,\infty]$ are defined according to
    \begin{equation}
        \f{1}{t_0}=\f{|\al|-|\be|}{s}\f{1}{p_0}+\bp{1-\f{|\al|-|\be|}{s}}\f{1}{q_0},\quad\f{1}{t_1}=\f{|\be|}{s}\f{1}{p_1}+\bp{1-\f{|\be|}{s}}\f{1}{q_1}\quad\text{and satisfy}\quad\f{1}{t}=\f{1}{t_0}+\f{1}{t_1}
    \end{equation}
    thanks to hypothesis~\eqref{exponent relationships}. Then, by Theorem~\ref{thm on gagliardo nirenberg} we have
    \begin{equation}
        \tnorm{f_0}_{W^{|\al|-|\be|,t_0}}\lesssim\tnorm{f_0}_{L^{q_0}}^{|\be|/s}\tnorm{f_0}_{W^{s,p_0}}^{1-|\be|/s},\quad\tnorm{f_1}_{W^{|\be|,t_1}}\lesssim\tnorm{f_1}_{L^{q_1}}^{1-|\be|/s}\tnorm{f_1}_{W^{s,p_1}}^{|\be|/s}
    \end{equation}
    and so by Young's inequality for products (in the case that $|\be|\not\in\tcb{0,s}$)
    \begin{equation}
        \tnorm{f_0}_{W^{|\al|-|\be|,t_0}}\tnorm{f_1}_{W^{|\be|,t_1}}\lesssim\tnorm{f_0}_{L^{q_0}}\tnorm{f_1}_{W^{s,p_1}}+\tnorm{f_0}_{W^{s,p_0}}\tnorm{f_1}_{L^{q_1}}.
    \end{equation}
    Synthesizing the above yields the claimed estimate.
\end{proof}

We now establish important results which relates to the analytic Nemytskii operator material in Appendix~\ref{appendix on analytic composition in unital Banach algebras}.

\begin{prop}[Admissibility in the periodic case]\label{prop on admissability in the periodic setting}
    Given any $s\in\N$ the $p\in[1,\infty]$ the function space $\tp{C^0\cap L^\infty\cap W^{s,p}}\tp{\T^d;\C}$ is an admissible unital Banach algebra in the sense of Definition~\ref{defn of admissable Banach algebras}.
\end{prop}
\begin{proof}
    Thanks to Theorem~\ref{thm on high low product estimates} we are assured that $\tp{C^0\cap L^\infty\cap W^{s,p}}\tp{\T^d;\C}$ is indeed a Banach algebra - moreover it is unital where the constant function $1$ is the unit. Now given $f\in\tp{C^0\cap L^\infty\cap W^{s,p}}\tp{\T^d;\C}$ it is easy to see that $\C\setminus\m{spec}\tp{f}\subseteq\C\setminus\Bar{\m{img}\tp{f}}$ and so we only need to focus on establishing the opposite inclusion with the required structured estimates. Let $z\in\C\setminus\Bar{\m{img}\tp{f}}$ and set $\ep=\m{dist}_{\C}\tp{\Bar{\m{img}\tp{f}},z}>0$. Evidently the function $(z-f)^{-1}$ belongs to $\tp{C^0\cap L^\infty}\tp{\T^d;\C}$ with $\tnorm{(z-f)^{-1}}_{L^\infty}\le\ep^{-1}$. Since $\tp{L^\infty\cap\dot{W}^{s,p}}\tp{\T^d;\C}\emb W^{s,p}\tp{\T^d;\C}$, it therefore remains to estimate the $\dot{W}^{s,p}\tp{\T^d;\C}$ seminorm of $(z-f)^{-1}$. We do this by first invoking Fa\`a di Bruno's formula to obtain the estimate
    \begin{equation}\label{(((___)))}
        \tnorm{\tp{z-f}^{-1}}_{\dot{W}^{s,p}}\lesssim\sum_{j=1}^s\sum_{k_1+\cdots+k_j=s}\ep^{-j}\tnorm{\grad^{k_1}f\otimes\cdots\otimes\grad^{k_j}f}_{L^p}.
    \end{equation}
    Next, we write $1/p=\sum_{i=1}^j1/t_{k_i}$ with $t_{k_i}=(s/k_i)p\in[1,\infty]$ and use H\"older's inequality followed by Theorem~\ref{thm on gagliardo nirenberg} to deduce
    \begin{equation}\label{)))---(((}
        \tnorm{\grad^{k_1}f\otimes\cdots\otimes\grad^{k_j}f}_{L^p}\lesssim\prod_{i=1}^j\tnorm{f}_{L^\infty}^{1-k_i/s}\tnorm{f}_{W^{s,p}}^{k_i/s}=\tnorm{f}_{L^\infty}^{j-1}\tnorm{f}_{W^{s,p}}.
    \end{equation}
    Combining~\eqref{(((___)))} with~\eqref{)))---(((} and the aforementioned reduction gives the resolvent estimate
    \begin{equation}
        \tnorm{\tp{z-f}^{-1}}_{W^{s,p}}\lesssim\ep^{-1}+\max\tcb{1,\ep^{-s}}\tbr{\tnorm{f}_{W^{s,p}}}^s
    \end{equation}
    with an implicit constant depending only on $s$, $p$, and $d$. This is a bound of the form~\eqref{resolvent estimates} and so the proof is complete.
\end{proof}

\begin{prop}[Admissibility in the solitary case]\label{prop on preliminary admissibility in the solitary case}
    Given any $r\in[1,\infty]$ the function space $\tp{C^0\cap L^\infty\cap\dot{W}^{1,r}}\tp{\R^d;\C}$ is an admissible unital Banach algebra in the sense of Definition~\ref{defn of admissable Banach algebras}.
\end{prop}
\begin{proof}
    This is a simple calculation as the derivative of the resolvent obeys the formula $\grad\tp{\tp{z-f}^{-1}}=\tp{z-f}^{-2}\grad f$ for any $f\in\tp{C^0\cap L^\infty\cap\dot{W}^{1,r}}\tp{\R^2}$ and $z\not\in\Bar{\m{img}\tp{f}}$ and so it is obvious that $(z-f)^{-1}$ belongs to the same space with the correct structured estimates.
\end{proof}

\subsection{Harmonic analysis in weighted square summable Sobolev spaces}\label{appendix on harmonic analysis in weighted square summable Sobolev spaces}

In this appendix we record a selection of important tools from Harmonic analysis adapted to the weighted Sobolev spaces of Definition~\ref{defn of weighted Sobolev spaces} in the Hilbert case of $p=2$. We begin with an elementary multiplier theorem which lends us some rather important estimates on band limited functions.

\begin{prop}[Multiplier bounds and weighted Bernstein inequalities]\label{prop on a multiplier bound an weighted Bernstein inequalities}
    Let $s\in\N$, $\del\in[0,\infty)$. The following hold.
    \begin{enumerate}
        \item Let $k=\lceil\del\rceil\in\N$ and $m\in W^{k,\infty}\tp{\R^d;\C}$. For all $f\in H^s_\del\tp{\R^d;\C}$ we have $m(D)f\in H^{s}_\del\tp{\R^d;\C}$ with the estimate
        \begin{equation}\label{type 1 multiplier bounds}
            \tnorm{m(D)f}_{H^s_\del}\lesssim\tnorm{m}_{W^{k,\infty}}\tnorm{f}_{H^s_\del}.
        \end{equation}
        \item Now suppose that $f\in L^2_\del\tp{\R^d;\C}$ and $\lambda\in[1,\infty)$.
        \begin{enumerate}
            \item If $\m{supp}\mathscr{F}[f]\subset B(0,\lambda)$ then $f\in H^s_\del\tp{\R^d;\C}$ and $\tnorm{f}_{H^s_\del}\lesssim\lambda^s\tnorm{f}_{L^2_\del}$.
            \item If $\m{supp}\mathscr{F}[f]\subset B(0,\lambda)\setminus\Bar{B(0,\lambda/2)}$, then $f\in H^s_\del\tp{\R^d;\C}$ and $\lambda^s\tnorm{f}_{L^2_\del}\lesssim\tnorm{f}_{H^s_\del}\lesssim\lambda^s\tnorm{f}_{L^2_\del}$.
        \end{enumerate}
    \end{enumerate}
    All implicit constants in the inequalities above depend only on $d$, $s$, and $\del$.
\end{prop}
\begin{proof}
    Let us begin by proving the first item. By using the norm equivalence of the first item of Proposition~\ref{prop on equivalent norm on weighted square summable Sobolev spaces} followed by the boundedness of the product map $W^{k,\infty}\times H^\del\to H^\del$ we obtain
    \begin{equation}
        \tnorm{m(D)f}_{H^s_\del}\asymp\bp{\sum_{|\al|\le s}\tnorm{m\mathscr{F}[\pd^\al f]}_{H^\del}^2}^{1/2}\lesssim\tnorm{m}_{W^{k,\infty}}\tnorm{f}_{H^s_\del}.
    \end{equation}
    This proves the first item.

    Let us now prove the second item. We fix $\varphi\in C^\infty_{\m{c}}\tp{B(0,2)}$ with $\varphi=1$ on $\Bar{B(0,1)}$. If $f\in L^2_\del\tp{\R^2;\C}$ satisfies  $\m{supp}\mathscr{F}[f]\subset B(0,\lambda)$, then $f=\varphi(D/\lambda)f$. By using the equivalent norm of Proposition~\ref{prop on equivalent norms in the weighted Sobolev spaces} we then find that
    \begin{equation}\label{[[]]}
        \tnorm{f}_{H^s_\del}\asymp\bp{\sum_{|\al|\le s}\tnorm{\varphi(D/\lambda)\pd^\al f}^2_{L^2_\del}}^{1/2}=\bp{\sum_{|\al|\le s}\tnorm{m_{\al,\lambda}(D)f}^2_{L^2_\del}}^{1/2}
    \end{equation}
    for the multipliers $m_{\al,\lambda}\tp{\xi}=\varphi(\xi/\lambda)\tp{2\pi\ii\xi}^\al$. It is elementary to estimate $\tnorm{m_{\al,\lambda}}_{W^{k,\infty}}\lesssim\lambda^s$ and so by invoking the bounds~\eqref{type 1 multiplier bounds} on each term in the right hand side's sum in~\eqref{[[]]} we obtain $\tnorm{f}_{H^s_\del}\lesssim\lambda^s\tnorm{f}_{L^2_\del}$. Thus part (a) of the second item is justified.

    To prove part (b) it only remains to show the complementary lower bound. If $f\in L^2_\del\tp{\R^d;\C}$ satisfies $\m{supp}\mathscr{F}[f]\subset B(0,\lambda)\setminus\Bar{B(0,\lambda/2)}$, then we have the identity $f=\tp{1-\varphi(D/4\lambda)}\varphi(D/\lambda)f$. We can decompose this multiplier via
    \begin{equation}
        n_{\al,\lambda}\tp{\xi}=\tp{1-\varphi\tp{\xi/4\lambda}}\varphi\tp{\xi/\lambda}\f{\xi^\al}{\tp{2\pi\ii}^s|\xi|^{2s}}\quad\text{which satisfy}\quad\sum_{|\al|=s}n_{\al,\lambda}\tp{D}\pd^\al=\tp{1-\varphi(D/4\lambda)}\varphi\tp{D/\lambda}.
    \end{equation}
    Therefore, through this decomposition, the elementary bounds $\tnorm{n_{\al,\lambda}}_{W^{k,\infty}}\lesssim\lambda^{-s}$, and estimate~\eqref{type 1 multiplier bounds} we find
    \begin{equation}\label{((()))}
        \tnorm{f}_{L^2_\del}\le\sum_{|\al|=s}\tnorm{n_{\alpha,\lambda}(D)\pd^\al f}_{L^2_\del}\lesssim\lambda^{-s}\sum_{|\al|=s}\tnorm{\pd^\al f}_{L^2_\del}.
    \end{equation}
    Now, thanks to Proposition~\ref{prop on equivalent norms in the weighted Sobolev spaces} again, the right hand side of~\eqref{((()))} is bounded above by $\lambda^{-s}\tnorm{f}_{H^{s}_\del}$. 
\end{proof}

The weighted inequalities for band-limited functions in the second item of Proposition~\ref{prop on a multiplier bound an weighted Bernstein inequalities} are the basis of the following Littlewood-Paley results.

\begin{prop}[Littlewood-Paley]\label{prop on littlewood paley}
    Let $V$ be a finite dimensional complex inner product space, $s\in\N$, $\del\in[0,\infty)$, and set $k=\lceil\del\rceil$. Assume that $\tcb{\varphi_j}_{j\in\N}\subset W^{k,\infty}\tp{\R^d;\C}$ satisfy $\m{supp}\tp{\varphi_0}\subset B(0,4)$, for $j>0$ $\m{supp}\tp{\varphi_j}\subset B(0,2^{j+2})\setminus\Bar{B\tp{0,2^{j-2}}}$, and $\sup_{j\in\N}\tnorm{\varphi_j}_{W^{k,\infty}}<\infty$. The following hold.
    \begin{enumerate}
        \item For all $f\in H^s_\del\tp{\R^d;V}$ we have the estimate
        \begin{equation}\label{first littlewood paley}
            \bp{\sum_{j\in\N}4^{js}\tnorm{\varphi_j(D)f}^2_{L^2_\del}}^{1/2}\lesssim\tp{\sup_{j\in\N}\tnorm{\varphi_j}_{W^{k,\infty}}}\tnorm{f}_{H^s_\del}
        \end{equation}
        for an implicit constant depending only on the dimension, $s$, and $\del$.
        \item Assume additionally that $\sum_{j\in\N}\varphi_j=1$; then we may bound for all $f\in H^s_\del\tp{\R^d;V}$
        \begin{equation}\label{second littlewood paley}
            \tnorm{f}_{H^s_\del}\lesssim\bp{\sum_{j\in\N}4^{js}\tnorm{\varphi_j(D)f}^2_{L^2_\del}}^{1/2}
        \end{equation}
        again with implicit constants depending only on the dimension, $s$, and $\del$.
        \item Again suppose that $\sum_{j\in\N}\varphi_j=1$. For all $f\in\mathscr{S}^\ast\tp{\R^d;V}$ it holds that $f\in H^s_\del\tp{\R^d;V}$ if and only if the quantity on the right hand side of~\eqref{second littlewood paley} is finite; moreover, we have the equivalence of norms
        \begin{equation}
            \tnorm{f}_{H^s_\del}\asymp\bp{\sum_{j\in\N}4^{js}\tnorm{\varphi_j(D)f}^2_{L^2_\del}}^{1/2}
        \end{equation}
        with implicit constants depending only on the dimension, $s$, $\del$, and $\tcb{\varphi_j}_{j\in\N}$.
    \end{enumerate}
\end{prop}
\begin{proof}
    Let us begin by proving the first item. For brevity, we shall restrict our attention to the case $\del<1$, as the cases $\del\ge 1$ follow from a nearly identical, but more cumbersome argument. Thanks to the second item of Proposition~\ref{prop on a multiplier bound an weighted Bernstein inequalities} for each $j\in\N$ we have the estimate
    \begin{equation}\label{R0}
        2^{js}\tnorm{\varphi_j(D)f}_{L^2_\del}\lesssim\tnorm{\varphi_j(D)f}_{H^s_\del}.
    \end{equation}
    Now to handle $\tnorm{\varphi_j(D)f}_{H^s_\del}$, we shall use the specific choice of inner product from the second item of Proposition~\ref{prop on equivalent norm on weighted square summable Sobolev spaces}, which leads us to
    \begin{multline}\label{long and ugly}
        \sum_{j\in\N}\tnorm{\varphi_j(D)f}^2_{H^s_\del}\asymp\sum_{j\in\N}\sum_{|\al|\le s}\tbr{\varphi_j\mathscr{F}\tp{\pd^\al f},\varphi_j\mathscr{F}[\pd^\al f]}^\star_{H^\del}\\=\sum_{j\in\N}\sum_{|\al|\le s}\int_{\R^d}\bp{|\varphi_j(\xi)\mathscr{F}[\pd^\al f](\xi)|^2+\bp{\int_{B(0,1)}\f{|\varphi_j(\xi+h)\mathscr{F}[\pd^\al f](\xi+h)-\varphi_j(\xi)\mathscr{F}[\pd^\al f](\xi)|^2}{|h|^{d+2\del}}\;\m{d}h}}\;\m{d}\xi.
    \end{multline}
    The integral over $\R^d$ above has zeroth order terms and averaged difference quotient terms. For the former type, we gain by the support hypotheses on $\tcb{\varphi_j}_{j\in\N}$ the bound
    \begin{equation}
        \bnorm{\sum_{j\in\N}\tabs{\varphi_j}^2}_{L^\infty}\lesssim\tp{\sup_{j\in\N}\tnorm{\varphi_j}_{L^{\infty}}}^2
    \end{equation}
    and hence may conclude that
    \begin{equation}\label{R1}
        \sum_{j\in\N}\sum_{|\al|\le s}\int_{\R^d}\tabs{\varphi_j\tp{\xi}\mathscr{F}[\pd^\al f]\tp{\xi}}^2\;\m{d}\xi\lesssim\tp{\sup_{j\in\N}\tnorm{\varphi_j}_{L^\infty}}^2\tnorm{f}^2_{H^s_\del}.
    \end{equation}
    The averaged difference quotient terms in~\eqref{long and ugly} are somewhat more involved - we initially split according to
    \begin{multline}
        \sum_{j\in\N}|\varphi_j(\xi+h)\mathscr{F}[\pd^\al f](\xi+h)-\varphi_j(\xi)\mathscr{F}[\pd^\al f](\xi)|^2\lesssim\bp{\sum_{j\in\N}|\varphi_j(\xi+h)|^2}|\mathscr{F}[\pd^\al f](\xi+h)-\mathscr{F}[\pd^\al f](\xi)|^2\\+\bp{\sum_{j\in\N}|\varphi_j(\xi+h)-\varphi_j(\xi)|^2}|\mathscr{F}[\pd^\al f](\xi)|^2.
    \end{multline}
    Then, by using $|h|\le 1$ and the support conditions on $\tcb{\varphi_j}_{j\in\N}$ one readily verifies the bounds
    \begin{equation}
        \sup_{\substack{\xi\in\R^d\\|h|\le 1}}\sum_{j\in\N}|\varphi_j(\xi+h)|^2\lesssim\tp{\sup_{j\in\N}\tnorm{\varphi_j}_{L^\infty}}^2,\quad\sup_{\substack{\xi\in\R^d\\|h|\le 1}}\f{1}{|h|^2}\sum_{j\in\N}|\varphi_j(\xi+h)-\varphi_j(\xi)|^2\lesssim\tp{\sup_{j\in\N}\tnorm{\varphi_j}_{W^{1,\infty}}}^2
    \end{equation}
    which lead us to the estimate
    \begin{multline}\label{R2}
         \sum_{j\in\N}|\varphi_j(\xi+h)\mathscr{F}[\pd^\al f](\xi+h)-\varphi_j(\xi)\mathscr{F}[\pd^\al f](\xi)|^2\\\lesssim\tp{\sup_{j\in\N}\tnorm{\varphi_j}_{W^{k,\infty}}}^2\ssb{|\mathscr{F}[\pd^\al f](\xi+h)-\mathscr{F}[\pd^\al f](\xi)|^2+|h|^2|\mathscr{F}[\pd^\al f](\xi)|^2}.
    \end{multline}
    The $|h|^2$ term above introduces the integrable singularity $|h|^{-d+2-2\del}$. Therefore, it holds
    \begin{equation}\label{R3}
        \sum_{j\in\N}\sum_{|\al|\le s}\int_{\R^d}\int_{B(0,1)}\f{|\varphi_j(\xi+h)\mathscr{F}[\pd^\al f](\xi+h)-\varphi_j(\xi)\mathscr{F}[\pd^\al f](\xi)|^2}{|h|^{d+2\del}}\;\m{d}h\;\m{d}\xi\lesssim\tp{\sup_{j\in\N}\tnorm{\varphi_j}_{W^{1,\infty}}}^2\tnorm{f}^2_{H^s_\del}.
    \end{equation}
    Upon synthesizing~\eqref{long and ugly} with the bounds, \eqref{R0}, \eqref{R1}, \eqref{R2}, and~\eqref{R3} we arrive at estimate~\eqref{first littlewood paley} and so the proof of the first item is complete.

    In proving the second item the inner product of Proposition~\ref{prop on equivalent norm on weighted square summable Sobolev spaces} will again be quite useful. By the hypotheses on $\tcb{\varphi_j}_{j\in\N}$ we have that $f=\sum_{j\in\N}\varphi_j(D)f$ and hence
    \begin{equation}
        \tnorm{f}^2_{H^s_\del}\asymp\tbr{f,f}_{H^s_\del}^\bigstar=\sum_{i,j\in\N}\tbr{\varphi_i(D)f,\varphi_j(D)f}_{H^s_\del}^\bigstar=\sum_{j\in\N}\sum_{|\iota|\le 6}\tbr{\varphi_j(D)f,\varphi_{j+\iota}(D)}^\bigstar_{H^s_\del}
    \end{equation}
    where the final equality follows from the third item of Proposition~\ref{prop on equivalent norm on weighted square summable Sobolev spaces} (and we are using the convention that if $j+\iota<0$ then $\varphi_{j+\iota}=0$). Therefore, by Cauchy-Schwarz and the second item of Proposition~\ref{prop on a multiplier bound an weighted Bernstein inequalities} we have
    \begin{equation}
        \tnorm{f}^2_{H^s_\del}\lesssim\sum_{j\in\N}\tnorm{\varphi_j(D)f}^2_{H^s_\del}\lesssim\sum_{j\in\N}4^{js}\tnorm{\varphi_j(D)f}^2_{L^2_\del}.
    \end{equation}
    This is exactly~\eqref{second littlewood paley}.

    The third item follows by a simple synthesis of the first two items - we omit the details for brevity.
\end{proof}

Proposition~\ref{prop on littlewood paley} leads us to the following generalization of the spaces from Definition~\ref{defn of weighted Sobolev spaces} in the case $p=2$ to non integer regularity indices $s\in\R$.

\begin{defn}[Generalized weighted square summable Sobolev spaces]\label{defn of generalized weighted Sobolev spaces}
    Let $V$ be a finite dimensional complex inner product space, $s\in\R$, $\del\in[0,\infty)$, and $k=\lceil\del\rceil$. We define
    \begin{equation}\label{norm on the generalized weighted spaces}
        {^\star}H^s_\del\tp{\R^d;V}=\tcb{f\in\mathscr{S}^\ast\tp{\R^d;V}\;:\;\tnorm{f}_{{^\star}H^s_\del}<\infty},\quad\tnorm{f}_{{^\star}H^s_\del}=\bp{\sum_{j\in\N}4^{js}\tnorm{\upvarphi_j(D)f}^2_{L^2_\del}}^{1/2},
    \end{equation}
    where $\tcb{\upvarphi_j}_{j\in\N}\subset W^{k,\infty}\tp{\R^d;\C}$ satisfy the dyadic annular support conditions as in Proposition~\ref{prop on littlewood paley}, $\sup_{j\in\N}\tnorm{\upvarphi_j}_{W^{k,\infty}}<\infty$, and $1=\sum_{j\in\N}\upvarphi_j$.
\end{defn}
\begin{rmk}\label{rmk on gen = reg on natty}
    The third item of Proposition~\ref{prop on littlewood paley} shows that whenever $s\in\N$ we have the coincidence ${^\star}H^s_\del\tp{\R^d;V}=H^s_\del\tp{\R^d;V}$.
\end{rmk}

The final part of this subsection is dedicated to proving a strengthening of the multiplier theorem from the first item of Proposition~\ref{prop on a multiplier bound an weighted Bernstein inequalities}. For this we require the following lemma.

\begin{lem}[Hardy's inequality]\label{lem on hardys inequality}
    Let $\del\in[0,d/2)$. There exists a $C\in\R^+$ such that for any function $f\in H^\del\tp{\R^d;\C}$ we have the estimate
    \begin{equation}\label{Hardy's inequality for fractional sobolev spaces}
        \bp{\int_{\R^d}\f{|f(x)|^2}{|x|^{2\del}}\;\m{d}x}^{1/2}\le C\tnorm{f}_{H^\del}.
    \end{equation}
\end{lem}
\begin{proof}
    Theorem 2.57 in Bahouri, Chemin, and Danchin~\cite{MR2768550} proves inequality~\eqref{Hardy's inequality for fractional sobolev spaces} for the subspace of functions $f\in H^\del\tp{\R^d}$ with $0\not\in\m{supp}\mathscr{F}[f]$. In the general case, we can take $f\in H^\del\tp{\R^d;\C}$ and split $f=\varphi(D)f+(1-\varphi(D))f$ with $\varphi\in C^\infty_{\m{c}}\tp{B(0,2)}$ a radial function satisfying $\varphi=1$ on $\Bar{B(0,1)}$. Then by using the triangle inequality and the fact that $|x|^{-2\del}\le 1$ for $|x|\ge 1$:
    \begin{equation}
        \bp{\int_{\R^d}\f{|f(x)|^2}{|x|^{2\del}}\;\m{d}x}^{1/2}\le\bp{\int_{B(0,1)}\f{|\varphi(D)f(x)|^2}{|x|^{2\del}}\;\m{d}x}^{1/2}+\bp{\int_{\R^d\setminus B(0,1)}|\varphi(D)f|^2}^{1/2}+C\tnorm{(1-\varphi(D))f}_{H^\del}.
    \end{equation}
    Hence
    \begin{equation}
        \bp{\int_{\R^d}\f{|f(x)|^2}{|x|^{2\del}}\;\m{d}x}^{1/2}\lesssim\tnorm{\varphi(D)f}_{L^\infty}+\tnorm{f}_{H^\del}\lesssim\tnorm{f}_{H^\del}.
    \end{equation}
    Note in the above we are tacitly using Bernstein's inequalities.
\end{proof}

\begin{prop}[Improved multiplier bounds]\label{prop on improved multiplier bounds}
    Let $V$ be a finite dimensional complex inner product space, $s,r\in\R$, $\del\in[0,d/2)$, $k=\lceil\del\rceil$, and $m\in W^{k,\infty}_{\m{loc}}\tp{\R^d\setminus\tcb{0};\C}$ a symbol obeying the estimates
    \begin{equation}\label{the hypotheses on the multiplier}
        \underset{\xi\in\R^d}{\m{esssup}}\bsb{\tbr{\xi}^{-r}\sum_{j=0}^k\min\tcb{1,|\xi|^j}|\grad^jm(\xi)|}=[m]_{k,r}<\infty.
    \end{equation}
    Then for any $f\in{^{\star}}H^{s+r}_\del\tp{\R^d;V}$ we have that $m(D)f\in{^\star}H^{s}\tp{\R^d;V}$ along with the estimate
    \begin{equation}
        \tnorm{m(D)f}_{{^\star}H^s_\del}\lesssim[m]_{k,r}\tnorm{f}_{{^\star}H^{s+r}_\del},
    \end{equation}
    with implicit constants depending only on the dimension, $\del$, $s$, and $r$.
\end{prop}
\begin{proof}
    Let $\tcb{\upvarphi_j}_{j\in\N}\subset W^{k,\infty}\tp{\R^d;\C}$ denote the dyadic partition of unity as in Definition~\ref{defn of generalized weighted Sobolev spaces} and denote the `fattened projectors' by
    \begin{equation}
        \tilde{\upvarphi}_j=\sum_{\substack{i\in\N\\|i-j|\le 6}}\upvarphi_i\quad\text{and observe that}\quad\upvarphi_j=\tilde{\upvarphi}_j\upvarphi_j.
    \end{equation}
    This leads us to defining the localized multipliers:
    \begin{equation}\label{the localized multipliers}
        m_j=\tilde{\upvarphi}_jm\quad\text{and we compute using~\eqref{the hypotheses on the multiplier} that}\quad
        \begin{cases}
            \sum_{i=0}^k\tnorm{|X|^i\grad^im_j}_{L^\infty}\lesssim[m]_{k,r}&\text{if }j<10,\\
            \tnorm{m_j}_{W^{k,\infty}}\lesssim 2^{jr}[m]_{k,r}&\text{if }j\ge10.
        \end{cases}
    \end{equation}
    So we learn from the above and the definition of the norm~\eqref{norm on the generalized weighted spaces} that
    \begin{multline}\label{the splitting}
        \tnorm{m(D)f}_{{^\star}H^s_\del}=\bp{\sum_{j\in\N}4^{sj}\tnorm{m_j(D)\upvarphi_j(D)f}^2_{L^2_\del}}^{1/2}\lesssim\bp{\sum_{j=0}^9\tnorm{m_j(D)\upvarphi_j(D)f}^2_{L^2_\del}}^{1/2}\\+\bp{\sum_{j=10}^\infty4^{sj}\tnorm{m_j(D)\upvarphi_j(D)f}^2_{L^2_\del}}^{1/2}=\bf{I}+\bf{II}.
    \end{multline}
    We shall estimate $\bf{II}$ first, as it is simpler. Due to the estimates of~\eqref{the localized multipliers} on $m_j$ for $j\ge 10$ we are free to invoke the first item of Proposition~\ref{prop on a multiplier bound an weighted Bernstein inequalities} to obtain 
    \begin{equation}
\bf{II}=\bp{\sum_{j=10}^\infty4^{sj}\tnorm{m_j(D)\upvarphi_j(D)f}^2_{L^2_\del}}^{1/2}\lesssim\bp{\sum_{j=10}^\infty4^{\tp{s+r}j}\tnorm{\upvarphi_j(D)f}^2_{L^2_\del}}^{1/2}\lesssim\tnorm{f}_{{^\star}H^{s+r}_\del}.
    \end{equation}

    Now we turn our attention to estimating $\bf{I}$ in~\eqref{the splitting}. We shall prove that for each $j\in\tcb{0,\dots,9}$ we have the bound
    \begin{equation}
        \tnorm{m_j(D)\upvarphi_j(D)f}_{L^2_\del}\lesssim\tnorm{\upvarphi_j(D)f}_{L^2_\del}\quad\text{from which it follows that}\quad\bf{I}\lesssim\tnorm{f}_{{^\star}H^{s+r}_\del}.
    \end{equation}
    Evidently, it is sufficient to show that for any $n\in L^\infty\tp{\R^d;\C}$ satisfying $[n]_{k,0}<\infty$ (recall~\eqref{the hypotheses on the multiplier}) then we have for all $g\in L^2_\del\tp{\R^d;V}$ the estimate $\tnorm{n(D)g}_{L^2_\del}\lesssim[n]_{k,0}\tnorm{g}_{L^2_\del}$. We shall only explicitly write out the case $\del<1$, as the cases $\del\ge 1$ can be felled with similar arguments. We use the equivalent norm from the first item of Proposition~\ref{prop on equivalent norm on weighted square summable Sobolev spaces} to estimate
    \begin{equation}
        \tnorm{n(D)g}_{L^2_\del}\lesssim\tnorm{n\mathscr{F}[g]}_{L^2}+\bp{\int_{\R^d}\int_{B(0,1)}\f{|n(\xi+h)\mathscr{F}[g](\xi+h)-n(\xi)\mathscr{F}[g](\xi)|^2}{|h|^{d+2\del}}\;\m{d}h\;\m{d}\xi}.
    \end{equation}
    Clearly we have $\tnorm{n\mathscr{F}[g]}_{L^2}\lesssim[n]_{1,0}\tnorm{g}_{L^2_\del}$ and so we only need to study further the averaged difference quotients term above. We make the usual splitting
    \begin{equation}
        n(\xi+h)\mathscr{F}[g](\xi+h)-n(\xi)\mathscr{F}[g](\xi)=n(\xi+h)\tp{\mathscr{F}[g](\xi+h)-\mathscr{F}[g](\xi)}+\tp{n(\xi+h)-n(\xi)}\mathscr{F}[g](\xi)
    \end{equation}
    and so
    \begin{multline}\label{LUAD}
        \bp{\int_{\R^d}\int_{B(0,1)}\f{|n(\xi+h)\mathscr{F}[g](\xi+h)-n(\xi)\mathscr{F}[g](\xi)|}{|h|^{d+2\del}}\;\m{d}h\;\m{d}\xi}\lesssim[n]_{1,0}\tnorm{\mathscr{F}[g]}_{H^\del}\\+\bp{\int_{\R^d}\tabs{\mathscr{F}[g](\xi)}^2\bp{\int_{B(0,1)}\f{|n(\xi+h)-n(\xi)|^2}{|h|^{d+2\del}}\;\m{d}h}\;\m{d}\xi}^{1/2}.
    \end{multline}
    To continue, we shall prove the pointwise estimate
    \begin{equation}\label{FUAD}
        \int_{B(0,1)}\f{|n(\xi+h)-n(\xi)|^2}{|h|^{d+2\del}}\;\m{d}h\lesssim\tp{[n]_{1,0}|\xi|^{-\del}}^2.
    \end{equation}
    Decompose $B(0,1)=A_\xi\cup B_\xi$ with $A_\xi=\tcb{h\in B(0,1)\;:\;|h|\le|\xi|/2}$ and $B_\xi=\tcb{h\in B(0,1)\;:\;|h|>|\xi|/2}$. In the region $A_\xi$ we exploit the difference with the fundamental theorem of calculus
    \begin{equation}
        |n(\xi+h)-n(\xi)|\lesssim[n]_{1,0}|h|\int_0^1|\xi+\tau h|^{-1}\;\m{d}\tau\lesssim[n]_{1,0}|h||\xi|^{-1}
    \end{equation}
    and so by $\del<1$ we have
    \begin{equation}\label{QUAD}
        \int_{A_\xi}\f{|n(\xi+h)-n(\xi)|^2}{|h|^{d+2\del}}\;\m{d}h\lesssim[n]^2_{1,0}|\xi|^{-2}\int_{B(0,|\xi|)}|h|^{-d-2\del+2}\;\m{d}h\lesssim\tp{[n]_{1,0}|\xi|^{-\del}}^2.
    \end{equation}
    On the other hand, in the region $B_\xi$ we simply bound $|n(\xi+h)-n(\xi)|\lesssim[n]_{1,0}$ without exploiting the difference:
    \begin{equation}\label{WUAD}
        \int_{B_{\xi}}\f{|n(\xi+h)-n(\xi)|^2}{|h|^{d+2\del}}\;\m{d}h\lesssim[n]_{1,0}^2\int_{B(0,1)\setminus B(0,|\xi|/2)}|h|^{-d-2\del}\;\m{d}h\lesssim\tp{[n]_{1,0}|\xi|^{-\del}}^2.
    \end{equation}
    Equations~\eqref{QUAD} and~\eqref{WUAD} give~\eqref{FUAD}. Now we may return to the final term in~\eqref{LUAD} and invoke Lemma~\ref{lem on hardys inequality}:
    \begin{equation}
        \bp{\int_{\R^d}\tabs{\mathscr{F}[g](\xi)}^2\bp{\int_{B(0,1)}\f{|n(\xi+h)-n(\xi)|^2}{|h|^{d+2\del}}\;\m{d}h}\;\m{d}\xi}^{1/2}\lesssim[n]_{1,0}\bp{\int_{\R^d}\f{|\mathscr{F}[g](\xi)|^2}{|\xi|^{2\del}}\;\m{d}\xi}^{1/2}\lesssim[n]_{1,0}\tnorm{g}_{L^2_\del}.
    \end{equation}
    By collecting the above the claimed estimate on $n(D)g$ is verified and with that the proof is complete.
\end{proof}
\section{Derivation of the shallow water equations with bathymetry}\label{appendix on derivation of SWE w/ bathymetry}

In this appendix we compute the shallow water system of equations~\eqref{time-dependent variable-batheymetry shallow water equations} through a rescaling and asymptotic expansion procedure beginning with the free boundary Navier-Stokes equations over a smooth graphical bathymetry.  Our derivation is largely inspired by those of Marche~\cite{MR2281291}, Bresch~\cite{MR2562163}, and Mascia~\cite{mascia_2010}, but ours includes bathymetry, forcing, and a generalized parameter scaling.  The latter means that we can view the shallow water system as an approximation of the Navier-Stokes system in the regimes of either long wavelength or small depth, but our technique actually allows for intermediate scaling regimes as well.

Sections~\ref{a_der_1} and~\ref{a_der_2} set up the Navier-Stokes equations in a form amenable to asymptotic expansion while in Section~\ref{a_der_3} we finally compute the system of shallow water equations. We note that throughout this appendix we frequently abuse notation by neglecting our variables' time dependence.

\subsection{The free boundary Navier-Stokes system with bathymetry}\label{a_der_1}

The physical spatial dimension is denoted by $n\in\tcb{2,3}$. We decompose our differential operators into horizontal and vertical parts through the following notation $\pmb{\grad}=\tp{\pd_1,\dots,\pd_{n-1},\pd_n}$ and $\grad=\tp{\pd_1,\dots,\pd_{n-1}}$ so that $\pmb{\grad}=\tp{\grad,\pd_n}$. We also set $\Delta=\pd_1^2+\cdots+\pd_{n-1}^2$, $\pmb{\Delta}=\Delta+\pd_n^2$.

The free surface and the bathymetry shall be described graphically and are represented by the functions $\zeta,\be:\R^{n-1}\to\R$ (respectively); these are assumed to satisfy a strict `no penetration' condition: $\be < \zeta$. Using these functions we define the bulk domain
\begin{equation}\label{bulk domain}
    \Omega[\be,\zeta]=\tcb{(x,y)\in\R^{n-1}\times\R\;:\;\be(x)<y<\zeta(x)}
\end{equation}
along with its lower and upper boundaries
\begin{equation}\label{lower and upper boundaries}
    \Sigma[\be]=\tcb{(x,y)\in\R^{n-1}\times\R\;:\;y=\be(x)},\quad \Sigma[\zeta]=\tcb{(x,y)\in\R^{n-1}\times\R\;:\;y=\zeta(x)}.
\end{equation}
The fluid velocity and pressure are defined in the bulk and are denoted by $w:\Omega[\be,\zeta]\to\R^n$ and $r:\Omega[\be,\zeta]\to\R$. 

We allow the fluid to be acted upon, in addition to the gravitational force in the $-e_n$ direction with acceleration strength $g\in\R^+$, generic bulk and interface forces represented by the vector $F:\Omega[\be,\zeta]\to\R^n$ and the symmetric 2-tensor $T:\Sigma[\zeta]\to\R^{n\times n}_{\m{sym}}$. The fluid also experiences the internal forces of viscosity and surface tension; the coefficients of which are $\mu>0$ and $\varsigma\ge 0$. On the bottom the fluid experiences a frictional force which is encoded via the (laminar) Navier slip boundary condition with strength coefficient $\al\ge0$.

The free boundary Navier-Stokes equations are given by:
\begin{equation}\label{formulation in Eulerian coordinates}
    \begin{cases}
    \pd_tw+w\cdot\pmb{\grad} w + \pmb{\grad} r-\mu\pmb{\Delta} w=F,\quad \pmb{\grad}\cdot w=0&\text{in }\Omega[\be,\zeta],\\
    -\tp{r-\mu\pmb{\mathbb{D}}w}\mathcal{N}_\zeta+\tp{g\zeta-\varsigma\mathcal{H}(\eta)}\mathcal{N}_\zeta=T\mathcal{N}_\zeta,\quad \pd_t\zeta-w\cdot\mathcal{N}_\zeta=0&\text{on }\Sigma[\zeta],\\
    (I-\nu_\be\otimes\nu_\be)\tp{\mu\pmb{\mathbb{D}}w \nu_\be-\al w}=0,\quad  w\cdot \nu_\be=0&\text{on }\Sigma[\be].
    \end{cases}
\end{equation}
In the above we set $\mathcal{N}_\zeta=(-\grad\zeta,1)$, $\mathcal{N}_\be=(-\grad\be,1)$, $\nu_\be=-\mathcal{N}_\be/|\mathcal{N}_\be|$, $\pmb{\mathbb{D}}w=\pmb{\grad}w+\pmb{\grad}w^{\m{t}}$, and $\mathcal{H}(\eta)=\grad\cdot\tp{\tp{1+|\grad\zeta|^2}^{-1/2}\grad\zeta}$.

Our task now is to rewrite system~\eqref{formulation in Eulerian coordinates} into a horizontal-vertical decomposed form. We shall decompose the velocity vector field into its vertical and horizontal components labeled as $w=(U,V)$ with $U:\Omega[\beta,\zeta]\to\R^{n-1}$ and $V=\Omega[\beta,\zeta]\to\R$. The symmetrized gradient of $w$ shall be decomposed according to
\begin{equation}
    \pmb{\mathbb{D}}w=\bpm\mathbb{D}U&\grad V+\pd_n U\\\grad V+\pd_n U&2\pd_n V\epm\quad\text{where}\quad\mathbb{D}U=\grad U+\grad U^{\m{t}}.
\end{equation}
In terms of these decomposed variables, the equations~\eqref{formulation in Eulerian coordinates} are equivalently written as
\begin{equation}\label{broken down system of equations}
    \begin{cases}
        \pd_tU+U\cdot\grad U+V\pd_n U-\mu\Delta U-\mu\pd_n^2 U+\grad r=F_{\|}&\text{in }\Omega[\be,\zeta],\\
        \pd_t V+U\cdot\grad V+V\pd_n V-\mu\Delta V-\mu\pd_n^2V+\pd_nr=F_n&\text{in }\Omega[\be,\zeta],\\
        \grad\cdot U+\pd_nV=0&\text{in }\Omega[\be,\zeta],\\
        (-r+g\zeta-\varsigma\mathcal{H}(\zeta))\bpm-\grad\zeta\\1\epm\\\quad+\mu\bpm-\mathbb{D}U\grad\zeta+\grad V+\pd_nU\\-(\grad V+\pd_nU)\cdot\grad\zeta+2\pd_nV\epm=\bpm-\Xi\grad\zeta+\xi\\-\grad\zeta\cdot\xi+\varphi\epm&\text{on }\Sigma[\zeta],\\
        \pd_t\zeta=V-U\cdot\grad\zeta&\text{on }\Sigma[\zeta],\\
        V-U\cdot\grad\be=0&\text{on }\Sigma[\be],\\
        \f{2\grad\be}{1+|\grad\be|^2}\pd_nV+\sp{I-2\f{\grad\be\otimes\grad\be}{1+|\grad\be|^2}}\tp{\pd_nU+\grad V}&\\\quad-\sp{I-\f{\grad\be\otimes\grad\be}{1+|\grad\be|^2}}\mathbb{D}U\grad\be=\f{\al}{\mu}\tp{1+|\grad\be|^2}^{1/2}U&\text{on }\Sigma[\be],\\
        \f{2|\grad\be|^2}{1+|\grad\be|^2}\pd_n V+\f{1-|\grad\be|^2}{1+|\grad\be|^2}(\grad V+\pd_nU)\cdot\grad\be&\\\quad-\f{1}{1+|\grad\be|^2}\mathbb{D}U\grad\be\cdot\grad\be=\f{\al}{\mu}\tp{1+|\grad\be|^2}^{1/2}U\cdot\grad\be&\text{on }\Sigma[\be],
    \end{cases}
\end{equation}
where we have decomposed the force data as
\begin{equation}
    F=(F_{\|},F_n)\quad\text{and}\quad T=\bpm \Xi&\xi\\\xi&\varphi\epm.
\end{equation}

\subsection{Rescaled and flattened equations}\label{a_der_2}

We let $\ep\in\R^+$ denote a (small) dimensionless scaling factor, $\theta\in[0,1]$ an interpolation parameter, and we introduce (using the notation $x\in\R^{n-1}$ and $y\in\R$) the following new rescaled variables $u$, $v$, $p$, $\eta$, $b$, $\bf{g}$, $\bf{a}$, $\sig$, $f_{\|}$, $f_n$, $\Upxi$, $\upxi$, and $\upvarphi$:
\begin{multline}\label{rescaling ansatz}
    U(t,x,y)=\ep^\theta u(\ep^{2\theta}t,\ep^\theta x,\ep^{\theta-1}y),\;V(t,x,y)=\ep^{1+\theta}v(\ep^{2\theta}t,\ep^\theta x,\ep^{\theta-1}y),\\r(t,x,y)=\ep^{2\theta}p(\ep^{2\theta}t,\ep^\theta x,\ep^{\theta-1}y),\;\zeta(t,x)=\ep^{1-\theta}\eta(\ep^{2\theta}t,\ep^\theta x),\;\be(x)=\ep^{1-\theta}b(\ep^\theta x),\\g=\ep^{3\theta-1}\bf{g},\;\al=\ep^{1+\theta}\bf{a},\;\varsigma=\ep^{\theta-1}\sig,\;F_{\|}(t,x,y)=\ep^{3\theta}f_{\|}(\ep^{2\theta}t,\ep^\theta x,\ep^{\theta-1}y),\\F_n(t,x,y)=\ep^{1+3\theta}f_n(\ep^{2\theta}t,\ep^\theta x,\ep^{\theta-1}y),\;
    \Xi(t,x,y)=\ep^{2\theta}\Upxi(\ep^{2\theta}t,\ep^\theta x,\ep^{\theta-1}y),\\\xi(t,x,y)=\ep^{1+2\theta}\upxi(\ep^{2\theta}t,\ep^\theta x,\ep^{\theta-1}y),\;\varphi(t,x,y)=\ep^{2\theta}\upvarphi(\ep^{2\theta}t,\ep^\theta x,\ep^{\theta-1}y).
\end{multline}
The parameter $\theta$ is interpolating between a small depth approximation ($\theta=0$) and a long wavelength approximation ($\theta=1$).

The equivalent equations satisfied by the rescaled variables are given as follows.
\begin{equation}\label{rescaled equations, free boundary form, with bathymetry}
        \begin{cases}
        \ssb{-\mu\pd_n^2u}+\ep^2\ssb{\pd_tu+u\cdot\grad u+v\pd_n u-\mu\Delta u+\grad p-f_{\|}}=0&\text{in }\Omega[b,\eta],\\
        \ssb{-\mu\pd_n^2v+\pd_np}+O(\ep^2)=0&\text{in }\Omega[b,\eta],\\
        \grad\cdot u+\pd_n v=0&\text{in }\Omega[b,\eta],\\
        \ssb{\mu\pd_nu}+\ep^2\ssb{\tp{p-\bf{g}\eta+\sig\Delta\eta}\grad\eta-\mu\mathbb{D}u\grad\eta+\mu\grad v+\Upxi\grad\eta-\upxi}+O(\ep^4)=0&\text{on }\Sigma[\eta],\\
        \ssb{-p+\bf{g}\eta-\sig\Delta\eta-\mu\grad\eta\cdot\pd_nu+2\mu\pd_nv-\upvarphi}+O(\ep^2)=0&\text{on }\Sigma[\eta],\\
        \pd_t\eta=v-u\cdot\grad\eta&\text{on }\Sigma[\eta],\\
        v-u\cdot\grad b=0&\text{on }\Sigma[b],\\
        \ssb{\pd_n u}+\ep^2\ssb{-(\bf{a}/\mu)u+\grad v+2\grad b\pd_nv-2\pd_nu\cdot\grad b\grad b-\mathbb{D} u\grad b}+O(\ep^4)=0&\text{on }\Sigma[b],\\
        O(1)=O(1)&\text{on }\Sigma[b].
    \end{cases}
\end{equation}
All terms hidden in the expressions $O(\ep^k)$ for $k\in\tcb{0,2,4}$ are easily computed, but we do not write them above as they play no role in the derivation that follows; in particular the last equation in~\eqref{rescaled equations, free boundary form, with bathymetry} is completely irrelevant. Note that after this rescaling the dependence on the interpolation parameter $\theta$ drops out.

Our next task is to flatten these rescaled equations of~\eqref{rescaled equations, free boundary form, with bathymetry} to a simple slab domain $\Omega=\R^{n-1}\times(0,1)$. We utilize the linear homotopy flattening mapping
\begin{equation}
    \mathfrak{F}_{\eta,b}:\Omega\to\Omega[b,\eta],\quad\mathfrak{F}_{\eta,b}(x,z)=(x,z\eta(x)+(1-z)b(x)),\quad(x,z)\in\Omega.
\end{equation}
We then consider the equations as written in terms of the flattened unknowns $\bf{u}:\Omega\to\R^{n-1}$, $\bf{v}:\Omega\to\R$, $\bf{p}:\Omega\to\R$
\begin{equation}
    \bf{u}=u\circ\mathfrak{F}_{\eta,b},\quad\bf{v}=v\circ\mathfrak{F}_{\eta,b},\quad\bf{p}=p\circ\mathfrak{F}_{\eta,b}.
\end{equation}
The following geometric quantities and differential operators are then derived from the flattening map:
\begin{equation}
    \pmb{\grad}\mathfrak{F}_{\eta,b}=\bpm I_{n-1}&0_{(n-1)\times 1}\\z\grad\eta+(1-z)\grad b&\eta-b\epm,\quad J_{\eta,b}=\det\mathfrak{F}_{\eta,b}=\eta-b,\;K_{\eta,b}=1/J_{\eta,b}=(\eta-b)^{-1},
\end{equation}
and
\begin{equation}
    \mathcal{A}_{\eta,b}=\tp{\pmb{\grad}\mathfrak{F}_{\eta,b}}^{-\m{t}}=\bpm I_{(n-1)\times(n-1)}&-K_{\eta,b}\tp{z\grad\eta+(1-z)\grad b}\\0_{1\times(n-1)}&K_{\eta,b}\epm.
\end{equation}
Write $\pmb{\grad}^{\mathcal{A}_{\eta,b}}=\mathcal{A}_{\eta,b}\pmb{\grad}$  and $\pmb{\grad}^{\mathcal{A}_{\eta,b}}=\tp{\grad^{\mathcal{A}_{\eta,b}},\pd_n^{\mathcal{A}_{\eta,b}}}$ so that (for $j\in\tcb{1,\dots,n-1}$)
\begin{equation}
    \grad^{\mathcal{A}_{\eta,b}}=\grad-K_{\eta,b}\tp{z\grad\eta+(1-z)\grad b}\pd_n,\quad\pd_j^{\mathcal{A}_{\eta,b}}=\pd_j-K_{\eta,b}(z\pd_j\eta+(1-z)\pd_jb)\pd_n,\quad\pd_n^{\mathcal{A}_{\eta,b}}=K_{\eta,b}\pd_n.
\end{equation}
We also set $\pd_t^{\mathcal{A}_{\eta,b}}=\pd_t-z\pd_t\eta K_{\eta,b}\pd_n$, $\Delta^{\mathcal{A}_{\eta,b}}=\sum_{j=1}^{n-1}\tp{\pd_j^{\mathcal{A}_{\eta,b}}}^2$, and $\tp{\mathbb{D}^{\mathcal{A}_{\eta,b}}\bf{u}}_{ij}=\pd_j^{\mathcal{A}_{\eta,b}}\bf{u}_i+\pd_i^{\mathcal{A}_{\eta,b}}\bf{u}_j$.

We write the equations (equivalent to~\eqref{rescaled equations, free boundary form, with bathymetry}) for the flattened variables as
\begin{equation}\label{bathyflat}
\begin{cases}
    [-\mu K_{\eta,b}^2\pd_n^2\bf{u}]+\ep^2\ssb{\pd_t^{\mathcal{A}_{\eta,b}}\bf{u}+\bf{u}\cdot\grad^{\mathcal{A}_{\eta,b}}\bf{u}+\bf{v}K_{\eta,b}\pd_n\bf{u}-\mu\Delta^{\mathcal{A}_{\eta,b}}\bf{u}+\grad^{\mathcal{A}_{\eta,b}}\bf{p}-f_{\|}\circ\mathfrak{F}_{\eta,b}}=0&\text{in }\Omega,\\
    \ssb{-\mu K_{\eta,b}^2\pd_n^2\bf{v}+K_{\eta,b}\pd_n\bf{p}}+O(\ep^2)=0&\text{in }\Omega,\\
    \grad^{\mathcal{A}_{\eta,b}}\cdot\bf{u}+K_{\eta,b}\pd_n\bf{v}=0&\text{in }\Omega,\\
    \ssb{\mu K_{\eta,b}\pd_n\bf{u}}+\ep^2\ssb{\tp{\bf{p}-\bf{g}\eta+\sig\Delta\eta}\grad\eta-\mu\mathbb{D}^{\mathcal{A}_{\eta,b}}\bf{u}\grad\eta+\mu\grad^{\mathcal{A}_{\eta,b}}\bf{v}+\Upxi\circ\mathfrak{F}_{\eta,b}\grad\eta-\upxi\circ\mathfrak{F}_{\eta,b}}&\\\quad+O(\ep^4)=0&\text{on }\Sigma_1,\\
    \ssb{-\bf{p}+\bf{g}\eta-\sig\Delta\eta-\mu K_{\eta,b}\grad\eta\cdot\pd_n\bf{u}+2\mu K_{\eta,b}\pd_n\bf{v}-\upvarphi\circ\mathfrak{F}_{\eta,b}}+O(\ep^2)=0&\text{on }\Sigma_1,\\
    \pd_t\eta=\bf{v}-\bf{u}\cdot\grad\eta&\text{on }\Sigma_1,\\
    \bf{v}-\bf{u}\cdot\grad b=0&\text{on }\Sigma_0,\\
    \ssb{K_{\eta,b}\pd_n\bf{u}}+\ep^2\ssb{-(\bf{a}/\mu)\bf{u}+\grad^{\mathcal{A}_{\eta,b}}\bf{v}+2K_{\eta,b}\grad b\pd_n\bf{v}-2K_{\eta,b}\grad b\grad b\cdot\pd_n\bf{u}-\mathbb{D}^{\mathcal{A}_{\eta,b}}\bf{u}\grad b}+O(\ep^4)=0&\text{on }\Sigma_0,\\
    O(1)=O(1)&\text{on }\Sigma_0,
\end{cases}
\end{equation}
where in the above we have labeled the top and bottom boundaries of $\Omega$ as $\Sigma_i=\R^{n-1}\times\tcb{i}$ for $i\in\tcb{0,1}$.

\subsection{Expansion and shallow water equations}\label{a_der_3}

With system~\eqref{bathyflat} in hand we have now set the stage for the derivation to truly begin. To continue, we make the asymptotic expansion ansatz 
\begin{equation}\label{ansatz f}
\tp{\eta,\bf{u},\bf{v},\bf{p}}=\sum_{\nu=0}^\infty\ep^{2\nu}\tp{\eta_\nu,u_\nu,v_\nu,p_\nu}
\end{equation}
for the solution to system~\eqref{bathyflat}. In other words we write each of the unknowns $\eta$, $\bf{u}$, $\bf{v}$, and $\bf{p}$ as a power series in $\ep^2$. Our goal now is to collect the resulting equations corresponding to like powers of $\ep$ and then perform manipulations until we isolate a closed system for the lowest order unknowns $\eta_0$, $u_0$, $v_0$, and $p_0$. As it turns out the equations for $v_0$ and $p_0$ `trivialize' and what we are effectively left with is a closed system for $\eta_0$ and $u_0$.

After plugging in the ansatz~\eqref{ansatz f} into system~\eqref{bathyflat} we now collect the relevant zeroth order equations.
\begin{equation}\label{The zerothy bathy}
    \begin{cases}
        -\mu K_{\eta_0,b}\pd_n^2u_0=0&\text{in }\Omega,\\
        -\mu K_{\eta_0,b}\pd_n^2v_0+K_{\eta_0,b}\pd_np_0=0&\text{in }\Omega,\\
        \grad^{\mathcal{A}_{\eta_0,b}}\cdot u_0+K_{\eta_0,b}\pd_nv_0=0&\text{in }\Omega,\\
        \mu K_{\eta_0,b}\pd_nu_0=0&\text{on }\Sigma_1,\\
        -p_0+\bf{g}\eta_0-\sig\Delta\eta_0-\mu K_{\eta_0,b}\grad\eta\cdot\pd_n u_0+2\mu K_{\eta_0,b}\pd_nv_0-\upvarphi\circ\mathfrak{F}_{\eta_0,b}=0&\text{on }\Sigma_1,\\
        \pd_t\eta_0=v_0-u_0\cdot\grad\eta_0&\text{on }\Sigma_1,\\
        v_0-u_0\cdot\grad b=0&\text{on }\Sigma_0,\\
        K_{\eta_0,b}\pd_nu_0=0&\text{on }\Sigma_0.
    \end{cases}
\end{equation}

The first and fourth equations of~\eqref{The zerothy bathy} imply that $\pd_nu_0=0$ and so  $u_0$ is a function of $t$ and $x$. Next, we combine the third and seventh equations to deduce that
\begin{equation}\label{expression for the v0}
    \pd_n v_0=-(\eta_0-b)\grad\cdot u_0\quad\text{and hence}\quad v_0=u_0\cdot\grad b-z(\eta_0-b)\grad\cdot u_0.
\end{equation}
In particular, we deduce that $\pd_n^2v_0=0$; therefore, from this fact and the second equation of~\eqref{The zerothy bathy} we have that $\pd_np_0=0$. Then, from the fifth equation we obtain the following expression for the pressure
\begin{equation}\label{the pressure approximation}
    p_0=\bf{g}\eta_0-\sig\Delta\eta_0-2\mu\grad\cdot u_0-\upvarphi(\cdot,\eta_0).
\end{equation}
We come to the sixth equation of~\eqref{The zerothy bathy} and insert~\eqref{expression for the v0} to deduce
\begin{equation}\label{the continuity equation for bathy}
\pd_t\eta_0=u_0\cdot\grad b-(\eta_0-b) \grad\cdot u_0-u_0\cdot\grad\eta_0=-\grad\cdot\tp{(\eta_0-b)u_0}. 
\end{equation}
This is the shallow water free surface transport equation.

At this point, we have used all of the equations in~\eqref{The zerothy bathy} and so it is time to see what information can be gleaned from the next order. The relevant second order equations of expansion~\eqref{ansatz f} in system~\eqref{bathyflat} are as follows:
\begin{equation}\label{some of the first order bathy}
    \begin{cases}
        -\mu K_{\eta_0,b}^2\pd_n^2u_1+\pd_tu_0+u_0\cdot\grad u_0-\mu\Delta u_0+\grad p_0-f_{\|}\circ\mathfrak{F}_{\eta,b}=0&\text{in }\Omega,\\
        \mu K_{\eta_0,b}\pd_nu_1+(p_0-\bf{g}\eta_0+\sig\Delta\eta_0)\grad\eta_0-\mu\mathbb{D}u_0\grad\eta_0+\mu\grad^{\mathcal{A}_{\eta_0,b}}v_0+\Upxi(\cdot,\eta_0)\grad\eta_0-\upxi(\cdot,\eta_0)=0&\text{on }\Sigma_1,\\
        \mu K_{\eta_0,b}\pd_nu_1-\bf{a}u_0+\mu\grad^{\mathcal{A}_{\eta_0,b}}v_0+2\mu K_{\eta_0,b}\grad b\pd_nv_0-\mu\mathbb{D}u_0\grad b=0&\text{on }\Sigma_0.
    \end{cases}
\end{equation}

Now we multiply the first equation in~\eqref{some of the first order bathy} by $J_{\eta_0,b}$, integrate from $0$ to $1$ in the $z$-variable, and then use the fundamental theorem of calculus on the $\pd_n^2u_1$ term. This results in
\begin{equation}\label{eq1}
    (\eta_0-b)\bp{\pd_t u_0+u_0\cdot\grad u_0-\mu\Delta u_0+\grad p_0-\int_0^1f_{\|}(\cdot,z\eta_0+(1-z)b)\;\m{d}z}=\mu K_{\eta_0,b}\tp{\pd_nu_1(\cdot,1)-\pd_nu_1(\cdot,0)}.
\end{equation}
This expression is the momentum equation of the shallow water equations in disguise. We need only unpack it by examining the boundary values of $\pd_nu_1$ on the right hand side and plugging in the previously derived expressions for $p_0$ and $v_0$ (see~\eqref{the pressure approximation} and~\eqref{expression for the v0}).

Now by the second equation in~\eqref{some of the first order bathy} we have (at $z=1$):
\begin{equation}\label{eq2}
    \mu K_{\eta_0,b}\pd_n u_1=\tp{2\mu\grad\cdot u_0+\upvarphi(\cdot,\eta_0)}\grad\eta_0+\mu\mathbb{D}u_0\grad\eta_0+\mu\eta_0\grad\grad\cdot u_0-\mu\grad\grad\cdot\tp{bu_0}-\Upxi(\cdot,\eta_0)\grad\eta_0+\upxi(\cdot,\eta_0)=0.
\end{equation}
In the above we have tacitly used the identity for $z=1$:
\begin{equation}
    \grad^{\mathcal{A}_{\eta_0,b}}v_0=\grad v_0-K_{\eta_0,b}\grad\eta_0\pd_nv_0=\grad\tp{\grad\cdot\tp{bu_0}-\eta_0\grad\cdot u_0}+\grad\cdot u_0\grad\eta_0=\grad\grad\cdot\tp{bu_0}-\eta_0\grad\grad\cdot u_0.
\end{equation}
From the third equation in~\eqref{some of the first order bathy} we deduce that (at $z=0$)
\begin{equation}\label{eq3}
    \mu K_{\eta_0,b}\pd_n u_1=\bf{a}u_0-\mu\grad(u_0\cdot\grad b)-\mu\grad b\grad\cdot u_0+2\mu\grad b\grad\cdot u_0+\mu\mathbb{D}u_0\grad b=0.
\end{equation}
In~\eqref{eq3} we have used the identity for $z=0$:
\begin{equation}
    \grad^{\mathcal{A}_{\eta_0,b}}v_0=\grad\tp{u\cdot\grad b}+\grad b\grad\cdot u_0.
\end{equation}

We isolate the coefficient of $\mu$ in the synthesis of expressions~\eqref{eq1}, \eqref{eq2}, and~\eqref{eq3}. This will give us the form of the dissipation term:
\begin{multline}
    (\eta_0-b)\Delta u_0+2(\eta_0-b)\grad\grad\cdot u_0+2\grad\cdot u_0\grad\eta_0+\mathbb{D}u_0\grad\eta_0+\eta_0\grad\grad\cdot u_0-\grad\grad\cdot\tp{bu_0}\\+\grad\tp{u_0\cdot\grad b}+\grad b\grad\cdot u_0-2\grad b\grad\cdot u_0-\mathbb{D}u_0\grad b=(\eta_0-b)\Delta u_0+2(\eta_0-b)\grad\grad\cdot u_0+2\grad(\eta_0-b)\grad\cdot u_0\\+\mathbb{D}u_0\grad\tp{\eta_0-b}+(\eta_0-b_0)\grad\grad\cdot u_0=\grad\cdot\tp{(\eta_0-b)\mathbb{S}u_0},
\end{multline}
with $\mathbb{S}u_0=\mathbb{D}u_0+2(\grad\cdot u_0)I$. Therefore, we synthesize the above to arrive at the shallow water momentum equation:
\begin{multline}\label{shallow mom}
    (\eta_0-b)\tp{\pd_tu_0+u_0\cdot\grad u_0+\grad\tp{\bf{g}\eta_0-\sig\Delta\eta_0}}+\bf{a}u_0-\mu\grad\cdot\tp{(\eta_0-b)\mathbb{S}u_0}\\=\int_{b}^{\eta_0}f_{\|}(\cdot,y)\;\m{d}y+(\eta_0-b)\grad\tp{\upvarphi(\cdot,\eta_0)}+\upvarphi(\cdot,\eta_0)\grad\eta_0-\Upxi(\cdot,\eta_0)\grad\eta_0+\upxi(\cdot,\eta_0),
\end{multline}
By combining~\eqref{shallow mom} with the transport equation~\eqref{the continuity equation for bathy} we find the form of the equations~\eqref{time-dependent variable-batheymetry shallow water equations} mentioned in Section~\ref{intro_-0}.

\section*{Acknowledgments}

We would like to express our gratitude to Ming Chen, Sam Walsh, and Miles Wheeler for the helpful discussions about global implicit function theorems. We would also like to thank Junichi Koganemaru for his assistance with computing techniques.


\bibliographystyle{abbrv}
\bibliography{bib.bib}
\end{document}